\documentclass[11pt]{amsart}
\usepackage{amsmath,amssymb}
\usepackage{graphicx}
\usepackage{amsfonts,amscd,latexsym,epsfig, psfrag}

\usepackage{hyperref}
\usepackage[all]{xy}

\newcommand{\er}{{\Diamond}}
\newcommand{\uU}{{\underline{U}}}

\newcommand{\Aa}{{\mathcal A}}

\newcommand{\Mor}{{\rm Mor}}
\newcommand{\Obj}{{\rm Obj}}

\newcommand{\und}{\underline}
\newcommand{\ul}{\underline}
\renewcommand{\Hat}{\widehat}

\newcommand{\HBb}{\Hat{\Bb}}

\newcommand{\PSL}{{\rm PSL}}
\newcommand{\pbar}{{\ov {\p}_J}}
\newcommand{\Cc}{{\mathcal C}}
\newcommand{\Kk}{{\mathcal K}}

\newcommand{\Jj}{{\mathcal J}}

\newcommand{\im}{{\rm im\,}}
\newcommand{\less}{{\smallsetminus}}

\newcommand{\bla}{{\bigl\langle}}
\newcommand{\bra}{{\bigl\rangle}}

\newcommand{\supp}{{\rm supp\,}}
\newcommand{\TU}{{\Tilde U}}
\newcommand{\p}{{\partial}}
\newcommand{\al}{{\alpha}}

\newcommand{\be}{{\beta}}

\newcommand{\om}{{\omega}}
\newcommand{\eps}{{\varepsilon}}
\newcommand{\de}{{\delta}}

\newcommand{\ga}{{\gamma}}
\newcommand{\Ga}{{\Gamma}}
\newcommand{\io}{{\iota}}

\newcommand{\la}{{\lambda}}
\newcommand{\La}{{\Lambda}}
\newcommand{\si}{{\sigma}}
\newcommand{\Si}{{\Sigma}}

\newcommand{\Uu}{{\mathcal U}}
\newcommand{\Bb}{{\mathcal B}}
\newcommand{\Ww}{{\mathcal W}}

\newcommand{\Mm}{{\mathcal M}}

\newcommand{\Tt}{{\mathcal T}}
\newcommand{\oMm}{{\overline {\Mm}}}
\newcommand{\ov}{\overline}
\newcommand{\wh}{\widehat}

\newcommand{\id}{{\rm id}}
\newcommand{\rd}{{\rm d}}
\newcommand{\rT}{{\rm T}}
\newcommand{\Exp}{{\rm Exp}}
\renewcommand{\Tilde}{\widetilde}

\newcommand{\TV}{{\Tilde V}}

\newcommand{\coker}{{\rm coker\,}}

\newcommand{\Ee}{{\mathcal E}}
\newcommand{\Ii}{{\mathcal I}}

\newcommand{\Vv}{{\mathcal V}}

\newcommand{\N}{{\mathbb N}}

\newcommand{\Q}{{\mathbb Q}}
\newcommand{\R}{{\mathbb R}}
\newcommand{\C}{{\mathbb C}}
\newcommand{\E}{{\mathbb E}}

\newcommand{\Z}{{\mathbb Z}}

\newcommand{\Hom}{{\rm Hom}}

\newcommand{\Nn}{{\mathcal N}}
\newcommand{\Pp}{{\mathcal P}}

\newcommand{\ev}{{\rm ev}}

\newcommand{\bB}{{\bf B}}
\newcommand{\bC}{{\bf C}}
\newcommand{\bG}{{\bf G}}

\newcommand{\bE}{{\bf E}}
\newcommand{\bK}{{\bf K}}
\newcommand{\bX}{{\bf X}}
\newcommand{\bz}{{\bf z}}
\newcommand{\bZ}{{\bf Z}}

\newcommand{\s}{{\mathfrak s}}

  \newcommand{\uVv}{{\und \Vv}}
  \newcommand{\uCc}{{\und \Cc}}
  \newcommand{\ud}{{\und d}}
  
  \newcommand{\us}{{\und{\mathfrak s}}}

    \newcommand{\uC}{{\underline{C}}}

      \newcommand{\uN}{{\und N}}
    
   \newcommand{\uV}{{\underline{V}}}
  
   \newcommand{\uphi}{{\underline{\phi}}}
      
\newcommand{\uKk}{{\underline{\Kk}}}

\newtheorem{theorem}{Theorem}[subsection]
\newtheorem{thm}[theorem]{Theorem}

\newtheorem{lemma}[theorem]{Lemma}
\newtheorem{proposition}[theorem]{Proposition}
\newtheorem{prop}[theorem]{Proposition}

\newtheorem{definition}[theorem]{Definition}
\newtheorem{defn}[theorem]{Definition}
\newtheorem{example}[theorem]{Example}

\newtheorem{remark}[theorem]{Remark}
\newtheorem{rmk}[theorem]{Remark}

\numberwithin{figure}{subsection}
\numberwithin{equation}{subsection}
\newcommand{\MS}{{\medskip}}

\newcommand{\NI}{{\noindent}}

\newcommand{\ti}{\tilde}
\newcommand{\Ti}{\widetilde}

\newcommand{\gr}{\operatorname{graph}}
\newcommand{\pr}{{\rm pr}}

\newcommand{\lm}{\Lambda^{\rm max}\,}

\newenvironment{enumlist}
   { \begin{list} {}
         {  \setlength{\itemsep}{.5ex} \setlength{\leftmargin}{0ex} } }
   { \end{list} }

   \newcounter{qcounter}

\newenvironment{itemlist}
   { \begin{list} {$\bullet$}
         {  \setlength{\itemsep}{.5ex} \setlength{\leftmargin}{2.5ex} } }
   { \end{list} }

\newcommand*{\longhookleftarrow}{\ensuremath{\leftarrow\joinrel\relbar\joinrel\rhook}}
\newcommand*{\longhookrightarrow}{\ensuremath{\lhook\joinrel\relbar\joinrel\rightarrow}}

\newcommand{\leftsub}[2]{{\vphantom{#2}}_{#1}{#2}}

\newcommand\quotient[2]{
        \mathchoice
            {% \displaystyle
                \text{\raise1ex\hbox{$#1$}\Big/\lower1ex\hbox{$#2$}}
            }
            {% \textstyle
                #1\,/\,#2
            }
            {% \scriptstyle
                #1\,/\,#2
            }
            {% \scriptscriptstyle
                #1\,/\,#2
            }
    }

\newcommand\quot[2]{
                \text{\raise1ex\hbox{$#1$}/\lower1ex\hbox{$\scriptstyle#2$}}
  }

\newcommand\quo[2]{
                \text{\raise1ex\hbox{$#1\!\!$}/\lower1ex\hbox{$\!\scriptstyle#2$}}
  }

\newcommand\qu[2]{
                \text{\raise.8ex\hbox{$\scriptstyle#1\!$}/\lower.8ex\hbox{$\!\scriptstyle#2$}}
  }

\newcommand\qq[2]{
                \text{\raise.8ex\hbox{$#1\!$}/\lower.8ex\hbox{$#2$}}
}

\title[Smooth Kuranishi atlases with trivial isotropy]{The fundamental class of smooth Kuranishi atlases with trivial isotropy}
 \author{Dusa McDuff}
\address{Department of Mathematics,
 Barnard College, Columbia University}
\email{dusa@math.columbia.edu}
\author{Katrin Wehrheim}
\address{Department of Mathematics, UC Berkeley}
\email{katrin@math.berkeley.edu}
\thanks{partially supported by NSF grants  
DMS 0905191, DMS 1308669 and DMS 0844188}
\begin{document}
\maketitle

\begin{abstract}
Kuranishi structures were introduced in the 1990s by Fukaya and Ono for the purpose of assigning a virtual cycle to moduli spaces of pseudoholomorphic curves that cannot be regularized by geometric methods. 
Their core idea was to build such a cycle by patching local finite dimensional reductions.
The first sections of this paper discuss topological, algebraic and analytic challenges that arise 
in this program.

We then develop a theory of Kuranishi atlases and cobordisms that transparently resolves these challenges, for simplicity concentrating on the case of trivial isotropy. 
In this case, we assign to a cobordism class of additive weak Kuranishi atlases both a virtual moduli cycle 
(VMC -- a cobordism class of smooth manifolds) and a virtual fundamental class (VFC -- a Cech homology class). 
We moreover show that such Kuranishi atlases exist on simple Gromov-Witten moduli spaces and develop the technical results in a manner that easily transfers to more general settings. 
\end{abstract}

\tableofcontents
%%%%%%%%%%%%%%%%%%%%%%%%%%%%%%%%%%%%%%%%%%%%%%%%%%%%%%%%%%%%%%%%%%%%%%%%%%%%%%%%%%%%%%%%%%%%%%%%%%%%%%%%%%%%%%%%%%
\section{Introduction}
%%%%%%%%%%%%%%%%%%%%%%%%%%%%%%%%%%%%%%%%%%%%%%%%%%%%%%%%%%%%%%%%%%%%%%%%%%%%%%%%%%%%%%%%%%%%%%%%%%%%%%%%%%%%%%%%%%

Kuranishi structures were introduced to symplectic topology by Fukaya and Ono \cite{FO},
and refined by Joyce \cite{J1}, in order to extract homological data from compactified moduli spaces of holomorphic maps in cases where geometric regularization approaches such as
perturbations of the almost complex structure do not yield a smooth structure on the moduli space.
These geometric methods generally cannot handle curves that are nowhere injective.  The first instance in which it was important to overcome these limitations
was the case of nowhere injective spheres, which are then multiply covered and
have nontrivial isotropy.\footnote
{This is not the case for discs. For example a disc with boundary on the equator can wrap two and a half times around the sphere. This holomorphic curve, called the lantern, has trivial isotropy. }  
Because of this, the development of virtual transversality techniques in \cite{FO}, and the related work by Li and Tian \cite{LT}, was focussed on dealing with finite isotropy groups, while some algebraic, topological, and analytic issues were not resolved. 

The goal of this paper is to explain these issues, provide a framework for resolving them, and demonstrate this in the first nontrivial case. To that end we focus on the most fundamental issues, which are already present in applying virtual transversality techniques to moduli spaces of holomorphic spheres without nodes or nontrivial isotropy.
We give a survey of regularization techniques in symplectic topology in Section~\ref{s:fluff}, pointing to some general analytic issues in Sections~\ref{ss:geom} --\ref{ss:kur}, and discussing the specific 
algebraic and topological issues of the Kuranishi approach in Sections~\ref{ss:alg} and \ref{ss:top}.
The main analytic issue in each regularization approach is in the construction of transition maps for a given moduli space, where one has to deal with the lack of differentiability of the reparametrization action on infinite dimensional function spaces discussed in Section~\ref{s:diff}.
When building a Kuranishi atlas on a moduli space, this issue also appears in a sum construction for basic charts on overlaps, and has to be dealt with separately for each specific moduli space.
We explain the construction of basic Kuranishi charts, their sums, and transition maps in the case of spherical Gromov--Witten moduli spaces in Section~\ref{s:construct}, outlining the proof of a more precise version of the following in Theorem~\ref{thm:A2}.
\MS

\NI {\bf Theorem A.}\,\,{\it  Let $(M,\om,J)$ be a symplectic manifold with tame almost complex structure, and let $\Mm(A,J)$ be the space of simple $J$-holomorphic maps $S^2\to M$ in class $A$ with one marked point, modulo reparametrization. 
If $\Mm(A,J)$ is compact (e.g.\ if $A$ is \lq\lq $\omega$-minimal"), then we construct an open cover
$\Mm(A,J)= \bigcup_{i=1,\ldots,N} F_i$ by ``footprints'' of basic Kuranishi charts $(\bK_i)_{i=1,\ldots,N}$,
that are compatible in the following sense:

For any tuple $(\bK_i)_{i\in I}$ of basic charts with overlapping footprints, whose obstruction spaces $E_i$ satisfy a ``transversality condition'', there is transition data as follows:
We construct a ``sum chart'' $\bK_I$ with obstruction space  $\prod_{i\in I}E_i$ and footprint  $F_{I}=\bigcap_{i\in I} F_i \subset \Mm(A,J)$, such that a restriction of each basic chart $\bK_i|_{F_I}$ includes into $\bK_I$ by a coordinate change.
Moreover, we construct coordinate changes $\Hat\Phi_{IJ}$ from $\bK_I$ to $\bK_J$ for each $I\subset J$, so that the collection of basic Kuranishi charts and tranisition data ${(\bK_I, \Hat\Phi_{IJ})}$ forms an additive weak Kuranishi atlas in the sense of Definitions~\ref{def:K}, \ref{def:Ku2}; 
}
\MS

The abstract notions of Kuranishi chart, restriction, and coordinate change are introduced in detail in Section~\ref{s:chart}. Throughout, we simplify the discussion by assuming that all isotropy groups are trivial.
In that special case our basic definitions largely follow \cite{FO,J1}, though avoiding notions of germs.
We then introduce in Section~\ref{s:Ks} a new notion of Kuranishi atlas as a covering family of basic charts together with transition data satisfying a cocycle condition involving an inclusion requirement on the domains of the coordinate changes.
At this point one could already try to construct compatible transverse perturbations of the sections in each Kuranishi chart.
However, there is no guarantee that the perturbed zero set modulo transition maps is a closed manifold, in particular Hausdorff -- which is an essential requirement in the construction of a {\it virtual moduli cycle}, which should be a cycle in an appropriate homology theory representing the {\it virtual fundamental class} of $X$.
This first topological issue, along with many others, is remedied by our theory of topological Kuranishi atlases \cite{MW:top}, which is universally applicable to regularization approaches of Kuranishi type (involving e.g.\ isotropy, boundary and corners, or lack of differentiability). 
Thus Sections~\ref{s:chart}, \ref{s:Ks}, \ref{ss:red} are mostly an expository demonstration of the ease of adapting \cite{MW:top} to a specific differentiable setting -- in our case yielding a self-contained development of the theory of smooth Kuranishi atlases and cobordisms with trivial isotropy, in which most proofs are quoted directly from \cite{MW:top}. 

In particular, we construct a {\it virtual neighbourhood} of the moduli space, with Hausdorff topology, in which the perturbed zero set modulo transition maps is a compact subset.
This construction in \cite{MW:top} requires {\it tameness} of the atlas, in particular a {\it strong cocycle condition} in which the domain of a composition of coordinate changes equals the domain of a direct coordinate change. However, the coordinate changes arising from sum constructions as in Theorem~A generally only satisfy a {\it weak cocycle condition} on the overlap of domains, thus yielding a {\it weak Kuranishi atlas}.
On the other hand, sum constructions naturally provide an {\it additivity} property for the obstruction spaces, and \cite{MW:top} develops a shrinking process that refines filtered weak Kuranishi atlases to tame Kuranishi atlases. 
While the implementation of these results in our setting in Section~\ref{ss:tame} is lengthy due to the development of the language of tameness, shrinkings, etc., it only requires one nontrivial proof: additivity implies filtration. 
Similarly, Section~\ref{ss:red} transfers the notion of {\it reductions} from \cite{MW:top} to our context, which drastically reduces the complexity of compatibility conditions in the construction of perturbations.
Moreover, both shrinkings and reductions are shown to be unique up to a notion of {\it concordance} that is developed in Section~\ref{ss:Kcobord} as special case of {\it Kuranishi cobordism}.

Based on these algebraic and topological preparations, Section~\ref{ss:sect} develops the notion of {\it precompact transverse (cobordism) perturbations} and shows that the perturbed zero sets are closed manifolds resp.\ cobordisms. Then the main technical effort in this paper is to prove existence and uniqueness of these perturbations in Section~\ref{ss:const}. 
Next, the main conceptual effort is the development of a theory of orientations in Section~\ref{ss:vorient}.
Both of these main efforts are presented in a way that should allow for direct extensions to many other differentiable settings, as we demonstrate in the case of nontrivial isotropy in \cite{MW:iso}.
Finally, Section~\ref{ss:VFC} proves the following Kuranishi regularization theorem.

\MS
\NI {\bf Theorem B.}\,\,{\it
Let $\Kk$ be an oriented, $d$-dimensional, additive weak Kuranishi atlas with trivial isotropy on a compact metrizable space $X$. Then $\Kk$ determines 
\begin{itemize}
\item
a {\bf virtual moduli cycle (VMC)}, that is a cobordism class of smooth, oriented, compact $d$-dimensional manifolds;
\item
a {\bf virtual fundamental class (VFC)}, that is an element $[X]^{vir}_\Kk$ in the \v{C}ech homology group $\check{H}_d(X;\Q)$.
\end{itemize}
Both depend only on the oriented, additive weak cobordism class of $\Kk$.}
\MS

Precise statements are given in Theorems~\ref{thm:VMC1} and \ref{thm:VMC2}.
We use rational, rather than integer, \v{C}ech homology, since we need a continuity property explained in Remark~\ref{rmk:Cech}.  
A novel point here is that the virtual fundamental class can be realized as an actual homology class on 
the moduli space $X$, which can then be pushed forward by e.g.\ evaluation maps.
Previous constructions of the VFC were also ``virtual'' in the sense that they only constructed pushforwards of the VFC.
Moreover, they require compatible and smooth extensions of the evaluation maps to the full domains of the Kuranishi charts, rather than just a continuous map defined on $X$, which induces maps from the zero sets.

\MS
\NI {\bf Extensions:}
We prove Theorem~B in a narrative that should be applicable to any differentiable refinement of a notion of topological Kuranishi atlases with the above two features (existence and uniqueness of precompact transverse perturbations, and a coherent theory of orientations). For example, in the case of nontrivial isotropy in \cite{MW:iso}, the proofs only differ in the use of multivalued perturbations, which are obtained by pullback of precompact transverse perturbations constructed in Section~\ref{ss:const}. Their effect is to replace closed manifolds as perturbed solution sets with weighted branched manifolds, which have a rational fundamental class. So the resulting VMC is a cobordism class of weighted branched manifolds, whereas the VFC remains a rational \v{C}ech homology class.

One limiting factor to the applicability of Kuranishi regularization (in fact any abstract regularization approach) is that one must start off from a compactified moduli space, and that all singular curves in the compactification, however ``nongeneric'', need to be covered with Kuranishi charts. 
At the moment, this excludes applications to both the ASD-with-Lagrangian-boundary-condition and Quilt-with-strip-shrinking moduli spaces studied by the second author, since their compactifications have not (yet) been constructed, let alone given local Fredholm descriptions near the putative exotic bubbling phenomena.
\MS

\noindent
{\bf Organization:} 
The following remarks in \S\ref{ss:back} on the context of this project, together with Sections~\ref{s:fluff} and \ref{s:diff}, provide a survey of regularization techniques in symplectic topology and their pitfalls. Section~\ref{s:construct} continues this discussion for the specific example of Kuranishi atlases for genus zero Gromov--Witten moduli spaces, and also outlines an approach to proving Theorem~A.
All of these sections are essentially self-contained and can be read in any order.
The main technical parts of the paper, Sections~\ref{s:chart}, \ref{s:Ks}, \ref{s:red}, and  \ref{s:VMC} , are independent of the previous sections, but strongly build on each other as well as \cite{MW:top} towards a proof of Theorem~B. 
For readers not familiar with this subject, we recommend Section~\ref{s:fluff} (with \S\ref{ss:poly}, \S\ref{ss:alg} skippable) as introduction to these technical parts. 
In order to make our exposition as self-contained as possible, we import the definitions and results of \cite{MW:top} by restating them in the present context. References to the corresponding content of \cite{MW:top} is given at the beginning of definitions resp.\ in the proof or results.
Readers familiar with \cite{MW:top} should be able to skim Sections~\ref{s:chart}, \ref{s:Ks}, \ref{ss:red} fairly quickly, taking note of only a few new concepts and results, of which only the first two are needed for the VMC/VFC construction in this paper:

\begin{itemlist}
\item
Definition~\ref{def:change} introduces an index condition, which Lemma~\ref{le:change} shows to be equivalent to the tangent bundle condition of \cite{J1,FOOO}, and which is compatible with composition by Lemma~\ref{le:cccomp}.
\item
Definition~\ref{def:Ku2} introduces a notion of additivity for (weak) Kuranishi atlases, which by Lemma~\ref{le:Ku3} induces a canonical filtration on the underlying (weak) topological Kuranishi atlas.
\item
Definition~\ref{def:Kcomm} introduces a notion of commensurability between additive weak Kuranishi atlases, which by Lemma~\ref{lem:cobord1}~(iii) implies additive weak concordance.
\item
Example~\ref{ex:nonlin} shows that the obstruction bundle $\pr_\Kk:|\bE_\Kk|\to |\Kk|$ of a Kuranishi atlas may fail to have well defined linear structures on the fibers, though tameness guarantees compatible linear structures by Proposition~\ref{prop:Khomeo}.
\item
Lemma~\ref{le:phitrans} shows that transition maps in tame Kuranishi atlases have tightly controlled images that intersect transversely.
\item
Proposition~\ref{prop:red} associates to any reduction of a tame Kuranishi atlas a Kuranishi atlas that -- while neither additive nor tame -- satisfies the strong cocycle condition.
\end{itemlist}

\medskip
\noindent
{\bf Acknowledgements:}
We would like to thank
Mohammed Abouzaid,
Kenji Fukaya,
Tom Mrowka,
Kaoru Ono,
Yongbin Ruan,
Dietmar Salamon,
Bernd Siebert,
Cliff Taubes,
Gang Tian,
and
Aleksey Zinger
for encouragement and enlightening discussions about this project,
and Jingchen Niu for pointing out some gaps in an earlier version.
We moreover thank MSRI, IAS, BIRS and SCGP for hospitality.

%%%%%%%%%%%%%%%%%%%%%%%%%%%%%%%%%%%%%%%%%%%%%%%%%%
\subsection{Background, outlook, and relation to other regularization approaches} \label{ss:back} \hspace{1mm}\\ \vspace{-3mm}
%%%%%%%%%%%%%%%%%%%%%%%%%%%%%%%%%%%%%%%%%%%%%%%%%%

Since our project revisits almost twenty years old, much used theories, this section describes its background, motivations, and outlook beyond our work, as well as 
relations to old and new work since the original version of this paper appeared as \cite{MW0}.\footnote{
To address misinterpretations of \cite{MW0} as merely regularizing simple Gromov-Witten moduli spaces and hence of little general interest, we developed the universally applicable parts into a general theory of topological Kuranishi atlases \cite{MW:top} whose usefulness should be evident.
The present paper consists of the expository parts of \cite{MW0} -- which still seem timely and are meant to provide context for the ongoing discussions of regularization approaches -- and technical parts that construct a virtual moduli cycle for a smooth Kuranishi atlas with trivial isotropy, using techniques that 
-- as we will see in \cite{MW:iso} -- 
generalize easily to more interesting cases
such as nontrivial isotropy.
}

\smallskip\NI
{\bf Background:}
Following Gromov's seminal work \cite{GRO}, the construction of Gromov-Witten invariants in the symplectic setting was first developed in the 1980s in settings where the moduli spaces, for appropriate choice of almost complex structure, carry a natural fundamental class or pseudocycle.
Comparisons with the algebro-geometric setting, in which the Gromov-Witten spaces carry a ``virtual fundamental class'', soon indicated that one should also be able to define such a class in general symplectic settings (and for the large variety of moduli spaces of pseudoholomorphic curves), using more abstract regularization approaches based on the local description of moduli spaces as zero sets of Fredholm sections.
Various approaches were proposed in the 1990s by Fukaya--Ono \cite{FO}, Li--Tian \cite{LT}, Liu--Tian \cite{LiuT}, Ruan~\cite{Ruan}, Siebert \cite{Sieb}. In the 2000s, the geometric methods were refined by Cieliebak--Mohnke \cite{CM} (with further developments in \cite{Gerst,Io,IoP,TZ}), while the abstract approaches were extended by e.g.\ Chen--Tian \cite{CT}, Chen--Li \cite{CL}, Fukaya--Oh--Ohta-Ono \cite{FOOO}, Joyce \cite{J1}, Lu \cite{Lu}, Lu--Tian \cite{LuT}
to include a growing variety of moduli spaces and localization tools. However, these are all variations of either an obstruction bundle approach or a Kuranishi approach, as explained in \S\ref{ss:approach}.
A third type of abstract regularization approach is still being developed by Hofer--Wysocki--Zehnder [HWZ1--5].

\smallskip\NI
{\bf Motivations:}
In a 2009 talk at MSRI \cite{w:msritalk}, the second author posed foundational questions on all these abstract regularization approaches.
The first author, who had been uneasily aware of analytic problems with the approach of \cite{LiuT}, the basis of her expository article \cite{Mcv}, decided that now was the time to clarify the constructions once and for all.
We found that within the obstruction bundle framework used in \cite{Mcv} (which is most closely related to the obstruction theory of algebraic geometry, as explained in \S\ref{ss:approach}, \S\ref{ss:LTBS}) we could not overcome the issue of lack of differentiability of the reparametrization action. This enters both via local slices of the action or Deligne--Mumford type spaces of domains and maps, and is discussed in detail in \S\ref{s:diff}.
When attempting to resolve these issues via a Kuranishi approach (which focusses on finite dimensional reductions as explained in \S\ref{ss:kur}), we soon found the differentiability issue in the compatibility of charts, but realized that this issue could be resolved by geometric construction of obstruction spaces, as we explain in \S\ref{s:construct}.
However, in making the abstract framework explicit, we needed to resolve ambiguities in the notion of a Kuranishi structure, concerning the precise meaning of germ of coordinate changes and the cocycle condition, discussed in \S\ref{ss:alg}.
More generally, we found it difficult to find a notion of Kuranishi structure that on the one hand clearly has a virtual fundamental class (some version of Theorem B), and on the other hand arises from fairly simple analytic techniques for holomorphic curves (some version of Theorem A).
One issue that we will mention only briefly in \S\ref{ss:approach} is the lack of smoothness of the standard gluing constructions, which affects the smoothness of the Kuranishi charts near nodal or broken curves.
The topological issues mentioned above and discussed in detail in \S\ref{ss:top} are more fundamental and surprising since there had been little doubt even in our minds that the perturbative construction of an Euler class for orbibundles should have a straightforward generalization to ``patching local Euler classes arising from a cover by local finite dimensional reductions''.

Most of these oversights seem to happen when only an oversimplified model case -- such as the Euler class of an orbifold bundle or an equivariant Fredholm section with compact zero set -- is worked out in detail, and the extension to an actually relevant setting is merely sketched or left to intuition.
This motivated our decision to give a completely explicit VFC construction in the simplest relevant and nontrivial case.
We have moreover found the topic of regularization of moduli spaces to lack the ``structural stability'' of other topics in symplectic geometry, in which a reasonable set of ideas almost always has a rigorous proof in its span, and imprecisions can be corrected locally. This is likely due to the otherwise safe intuitions from physics and algebraic geometry failing to have traction on this topic. As a result, our theory had to undergo constant global changes until the last technical piece was in place.
This is our reason for insisting on excruciating precision in each definition and step of proof.

\smallskip\NI
{\bf Relations:}
As the core of our work \cite{MW0} was nearing completion, we alerted Fukaya et al and Joyce to some of the issues we had uncovered. The ensuing discussion resulted in new versions of their approaches \cite{FOOO12,Jd,J2} and also motivated a new version of the Kuranishi approach by Pardon~\cite{pardon}, while additional work on the obstruction bundle approach appeared in \cite{CLW1,CLW2,Liu}.
We will comment on all these approaches in \S\ref{s:fluff} though we have not verified any of these papers in sufficient detail to endorse their correctness.
Here we compare the basic Kuranishi notions.

While the previous definitions of Kuranishi structures in \cite{FO,J1} are algebraically inconsistent as explained in \S\ref{ss:alg}, our approach is compatible with the notions of \cite{FOOO,FOOO12} in the case of trivial isotropy.  Indeed we show in Remark~\ref{rmk:otherK} how to obtain a Kuranishi structure in the latter sense from a weak Kuranishi atlas.  
However, the two approaches differ significantly when isotropy is nontrivial; see \cite{MW:iso,McL}.
One can make an analogy with the development of the theory of orbifolds:  The approach of \cite{FOOO12} is akin to Satake's definition of a $V$-manifold, while our definitions are much closer to the idea of describing an orbifold as the realization of an \'etale proper groupoid.
In our view, weak Kuranishi atlases in the sense of Definition~\ref{def:K} are the natural outcomes of constructions of compatible finite dimensional reductions, and we see a clear abstract path from an atlas to a VMC.  Constructing a weak atlas involves checking only a finite number of consistency conditions for the coordinate changes, while uncountably many such conditions must be checked if one tries to construct a Kuranishi structure directly.  

The notion of implicit atlas in \cite{pardon} is essentially our notion of tame Kuranishi atlas, even in the case of nontrivial isotropy. While we obtain tameness by an abstract refinement process from a much weaker structure, \cite{pardon} aims to obtain this directly from canonical analytic descriptions of the moduli space -- at the expense of a differentiable structure in the Kuranishi charts. 
We will further compare the different variations of the Kuranishi approach in Remarks~\ref{rmk:JBS} and \ref{rmk:otherK}.

\smallskip\NI
{\bf Outlook:}
The present Kuranishi regularization Theorem B applies only to Gromov--Witten moduli spaces that contain neither nodal nor multiply covered curves (as shown in Theorem A).
However, a generalization of our approach to other moduli spaces of closed pseudoholomorphic curves with Gromov compactification only requires two distinct additions to both the abstract theory and the constructions on a moduli space:

\begin{itemlist}
\item
Multiply covered curves yield local finite dimensional reductions in which a nontrivial isotropy group acts. 
In \cite{MW:iso}, we capture this abstractly in a notion of Kuranishi atlases 
with nontrivial isotropy and extend our VMC/VFC construction to this case.
This notion captures significantly more information than the notions of Kuranishi structures in \cite{FOOO,J1}, but we show in \cite{MW:GW,Mcn} how it naturally arises from genus zero Gromov--Witten moduli spaces. 
For an outline see the August 2013 lecture~\cite{McL}. 
(This requires recasting the constructions of \S\ref{s:construct} in terms of stabilizations rather than local obstruction bundles, and a systematic addition of marked points on which the isotropy groups act by permutation.)

\item
Nodal curves have neighbourhoods described by gluing constructions.
These again yield Kuranishi charts with trivial or nontrivial isotropy, but the smooth structures in different gluing charts are usually not preserved by coordinate changes. This issue needs to be addressed either by constructing more compatible smooth structures, or by proving a regularization theorem for Kuranishi atlases with less compatible smooth structures.
The classical gluing theory in e.g.\ \cite{MS} yields Kuranishi charts with stratified smooth structures, and we checked that these are preserved by coordinate changes. 
We also believe that it should be feasible to extend our VMC/VFC constructions to this case, though it will require refined notions of stratified smoothness with chain rules. We are not planning to work on this extension but will be happy to assist others if the need arises.
Instead, the gluing theorems from polyfold theory yield smooth Kuranishi atlases with nontrivial isotropy, to which our VMC/VFC constructions in \cite{MW:iso} apply directly. 
In fact, \cite{Yang} announced a general construction of smooth Kuranishi structures from a proper Fredholm section in a polyfold bundle. However, the price to pay by using polyfold theory (apart from the temptation of using its own regularization theorem directly) is that it uses a smooth structure on the Deligne-Mumford spaces (constructed in \cite{HWZ:DM} with globally rescaled gluing parameters) which is not compatible with their complex structure.
\end{itemlist}

\NI
An intermediate approach to the gluing issue is being taken by Castellano \cite{Cast1,Cast2}, who proves a gluing theorem that -- after appropriate rescaling of the gluing parameters -- yields stratified smooth Kuranishi atlases with $\Cc^1$-differentiability across strata, to which the VMC/VFC constructions given here and in \cite{MW:iso} apply with minor modifications. He moreover shows that the resulting genus zero Gromov--Witten invariants satisfy the standard axioms.
We believe that $\Cc^1$-Kuranishi atlases (which then automatically carry a VMC/VFC) for other moduli spaces of closed holomorphic curves can be constructed analogously, though each case requires a geometric construction of local slices as well as obstruction bundles specific to the setup, and careful gluing analysis.

An extension of our approach to moduli spaces which involve boundary nodes or breaking/buildings, as in Floer theories, SFT, or the construction of $A_\infty$-structures, would require -- beyond the construction of compatible Kuranishi charts on any given moduli space -- two more additions to the abstract theory:

\begin{itemlist}
\item
The gluing constructions near boundary nodes and breaking yield boundary and corners when moduli spaces are regular. 
For the regularization of general moduli spaces, our notion of Kuranishi cobordism should have a 
straightforward generalization that allows for corners and yields (branched weighted) manifolds with boundary and corners as VMC.
However, this would require a generalization of the notion of collared boundary in Definition~\ref{def:Ycob}, where two distinct boundary components (the incoming and outgoing end of a cobordism) have disjoint collars. When allowing for corners, the main boundary strata will still be disjoint but have overlapping collars. Since our notion of Kuranishi cobordism requires the charts and coordinate changes to have product form on collars, the corner version will require each corner stratum to have a collar homeomorphic to a product with $[0,\eps)^k$ -- corresponding to gluing parameters, and arising from the overlap of boundary collars.
Moreover, the various corner collars will need to be compatible in the sense that e.g.\ the $[0,\eps)^3$ collars induced from two different ``orders of gluing'' (each arising from an overlap of a boundary stratum and a corner stratum with $[0,\eps)^2$ collar) are the same.
On the one hand, this is necessary to generalize our construction of relative perturbations in Proposition~\ref{prop:ext2}. 
On the other hand, this requires a construction of Kuranishi charts from associative gluing maps in the sense of \cite{w:Morse}. However, to the best of our knowledge the present gluing constructions in the literature (including the gluing maps arising from the polyfold approach) do not naturally satisfy associativity.
(The construction in \cite{w:Morse} crucially uses the Morse flow and Euclidean normal form near critical points.)

\item
More globally, the gluing constructions identify the boundary strata of each moduli space with (fibered) products of other moduli spaces of similar type, and the VMC/VFC construction is required to be compatible with these ``gluing operations'' in order to obtain the intended algebraic structures such as $d^2=0$ in Floer theory, or the $A_\infty$-relations.
Thus the regularization has to solve the additional task of respecting the fiber product structure on the boundary. In perturbative approaches, this issue is also known as constructing coherent perturbations and has to be addressed separately in each specific geometric setting since it requires a hierarchy of moduli spaces which permits one to construct the perturbations iteratively.
In the construction of the Floer differential on a finitely generated complex, such an iteration can be performed using an energy filtration thanks to the algebraically simple gluing operation.
However, once one deals with homotopies of data or wants to prove independence from the choice of perturbations, compatibility with the gluing operation usually excludes transversality relative to the boundary strata -- in particular, $1$-dimensional moduli spaces arising from homotopies can intersect corner strata of arbitrarily high degeneracy.\footnote{This ``diagonal relator problem'' occurs whenever curves can be glued to themselves, starting with circle-valued Morse theory as in \cite{Hu}. In geometric regularization approaches it is often avoided by constructing direct continuation maps instead, but this option does not exist for abstract perturbations.}
This has been resolved in some special cases by a refined gluing analysis \cite{Hu} or artificial deformation of the gluing operation \cite{SEID}, but there does not seem to be a general understanding, let alone solution method, for this issue.
\end{itemlist}

\NI
It seems to us that polyfold theory is better suited to regularize moduli spaces which involve boundary nodes or breaking. Since it avoids finite dimensional reductions, it only requires constructions of pregluing maps -- which are naturally associative. It also offers weaker notions of transversality relative to the boundary stratification and provides an analytic framework in which the obstruction bundle gluing analysis \cite{Hu} can be generalized to settings in which there is at most one way of gluing a curve to itself \cite{jiayong}.

Finally, many applications of pseudoholomorphic curve invariants require equivariant regularization. 
For example, Floer's proof of the Arnold conjecture \cite{floer} argues with an $S^1$--action by reparametrizations on the Floer trajectory space $\Mm$ for an autonomous Hamiltonian, whose fixed points (and hence only solutions in dimension $0$) are the Morse trajectories.
When geometric (automatically $S^1$-equivariant) regularization fails, the argument was translated into abstract regularization terms by \cite{FO,LiuT} roughly as follows: The compactified Floer trajectory space $\oMm$ is equipped with a Kuranishi atlas of index $0$ and an $S^1$--action whose fixed point set $F\subset\oMm$ are the Morse trajectories, at which the Kuranishi section is transverse. This induces a Kuranishi atlas on $(\oMm\less F)/S^1$ that has index $-1$ and thus allows for a perturbation with empty solution set. Pulling this back to $\oMm$ yields a perturbation whose only solutions are $F$.
In order to make such a proof rigorous in our perturbative framework, one has to deal with the following challenges:

\begin{itemlist}
\item
A notion of $S^1$--action on a Kuranishi atlas should involve compatible $S^1$--actions on the Kuranishi domains and obstruction spaces with respect to which the sections are equivariant.
When -- as in the Arnold case -- the sections are transverse at the $S^1$--fixed points in the zero set, and the action on the zero set is otherwise free, then one can expect the existence of an equivariant transverse perturbation. Its construction via a quotient of the Kuranishi atlas would have to shrink domains appropriately to avoid all $S^1$--fixed points, not just those in the zero set.
\item
Since the compactified Floer trajectory space $\oMm$ may contain irregular solutions of all kinds, such as broken trajectories, or trajectories with sphere bubbles, the $S^1$--equivariant Kuranishi charts have to be constructed near all kinds of solutions. 
This requires choices of obstruction spaces that are equivariant under the non-differentiable $S^1$--action as well as gluing constructions that are compatible with the diagonal $S^1$-action on broken Floer trajectories.
\end{itemlist}

Theories addressing these points are now announced in \cite{FOOO12,pardon}.
Again, it seems to us that polyfold theory is better suited to achieve equivariant regularization since the first challenge only requires another generalization of a classical theorem in finite dimensional differential geometry -- something that has already been achieved in many instances for the polyfold framework -- and the second challenge is absent since the natural $S^1$--action on the ambient polyfold is already scale-smooth and compatible with pregluing.

In summary, the Kuranishi approach seems to be less technologically sophisticated and thus mostly has value in the Gromov--Witten setting, especially in very geometric situations such as \cite{MT}, or in situations very close to algebraic geometry such as \cite{Mcu, Z2}.
Our project aims to develop the needed theory in the simplest way possible, using basic tools from general and differential topology rather than sheaf theory or sophisticated category theory as in \cite{pardon,J2}.

%%%%%%%%%%%%%%%%%%%%%%%%%%%%%%%%%%%%%%%%%%%
\section{Regularizations of holomorphic curve moduli spaces}  \label{s:fluff}
%%%%%%%%%%%%%%%%%%%%%%%%%%%%%%%%%%%%%%%%%%%

One of the central technical problems in the theory of holomorphic curves, which provides many of the modern tools in symplectic topology, is to construct algebraic structures by extracting homological information from moduli spaces of holomorphic curves in general compact symplectic manifolds $(M,\om)$.
We will refer to this technique as {\it regularization} and note that it requires 
several
distinct components:
Some perturbation technique is used to achieve {\it transversality}, which gives the moduli space a smooth structure. In order for this to induce a count or chain, the perturbation also has to preserve {\it compactness} and a suitable version of {\it Hausdorffness} of the moduli space.
Moreover, some type of cobordism technique is used to achieve {\it invariance}, i.e.\ independence of the resulting homological information from the choices involved.

The aim of this section is to give an overview of the different regularization approaches in the case of genus zero Gromov--Witten invariants $\bla \al_1,\ldots,\al_k\bra_{A} \in \Q$.
These are defined as a generalized count of $J$-holomorphic genus $0$ curves in class $A\in H_2(M;\Z)$ that meet $k$ representing cycles of the homology classes $\al_i\in H_*(M)$.
This number should be independent of the choice of $J$ in the contractible space of $\om$-compatible almost complex structures, and of the cycles representing $\alpha_i$.
For complex structures $J$ one can work in the algebraic setting, in which the curves are cut out by holomorphic functions on $M$, but general symplectic manifolds do not support an integrable $J$.  For non-integrable $J$, the approach introduced by Gromov \cite{GRO} is to view the (pseudo-)holomorphic curves as maps to $M$ satisfying the Cauchy--Riemann PDE, modulo reparametrizations by automorphisms of the complex domain.

To construct the Gromov--Witten moduli spaces of holomorphic curves, one starts out with the typically noncompact quotient space
$$
\Mm_k(A,J) := \bigl\{ \bigl( f: S^2 \to M, \bz\in (S^2)^k
\less \Delta \bigr) \,\big|\, f_*[S^2]=A , \pbar f = 0 \bigr\} / \PSL(2,\C)
$$
of equivalence classes of tuples $(f,\bz)$, where $f$ is a $J$-holomorphic map, the marked points
$\bz= (z_1,\ldots z_k)$ are pairwise disjoint, and the equivalence relation is given by the reparametrization action $\ga\cdot(f,\bz)=(f\circ\ga,\ga^{-1}(\bz))$
of the M\"obius group $\PSL(2,\C)$.
This space is contained (but not necessarily dense) in the compact moduli space $\oMm_{k}(A,J)$ formed by the equivalence classes of $J$-holomorphic genus $0$ stable maps $f:\Si \to M$ in class $A$ with $k$ pairwise disjoint marked points.  There is a natural evaluation map
\begin{equation} \label{eq:ev}
\ev: \oMm_{k}(A,J)\to M^k, \quad
[\Si,f,(z_1,\ldots,z_k)]\mapsto \bigl(f(z_1),\ldots,f(z_k)\bigr),
\end{equation}
and one expects the Gromov--Witten invariant
$$
\bla \al_1,\ldots,\al_k\bra_{A}:=\ev_*[\oMm_{k}(A,J)]\cap (\al_1\times\ldots\times \al_k)
$$
to be defined as intersection number of a homology class $\ev_*[\oMm_{k}(A,J)]\in H_*(M;\Q)$ with the class $\al_1\times\ldots\times \al_k$.
The construction of this homology class requires a {\it regularization} of $\oMm_{k}(A,J)$.
In Sections~\ref{ss:geom} - \ref{ss:kur} we give a brief overview of the approaches using geometric means or an abstract polyfold setup, and review the fundamental ideas behind Kuranishi 
structures.
Sections~\ref{ss:alg} and~\ref{ss:top} then discuss the algebraic and topological issues in constructing a virtual fundamental class from a Kuranishi 
structure or atlas.

%%%%%%%%%%%%%%%%%%%%%%%%%%%%%%%%%%%%%
\subsection{Geometric regularization} \label{ss:geom}\hspace{1mm}\\ \vspace{-3mm}
%%%%%%%%%%%%%%%%%%%%%%%%%%%%%%%%%%%%%

For some special classes of symplectic manifolds, the regularization of holomorphic curve moduli spaces can be achieved by a choice of the almost complex structure $J$, or more generally a perturbation of the Cauchy--Riemann equation $\pbar f =0$ that preserves the symmetry under reparametrizations, and whose Hausdorff compactification is given by nodal solutions.
Note that these properties are generally not preserved by perturbations of a nonlinear Fredholm operator such as $\pbar$, so this approach requires a class of perturbations
that preserves the geometric properties of $\pbar$.

The construction of Gromov--Witten invariants from a regularization of $\oMm_{k}(A,J)$ most easily fits into this approach if $A$ is a  homology  class on which $\om(A)>0$ is minimal, since then $A$
cannot be represented by a multiply covered or nodal holomorphic sphere.
For short, we call such $A$ $\om$-{\it minimal.}
In this case $\Mm_{k}(A,J')$ is smooth for generic $J'$, and compact if $k\le 3$.
More generally, this approach applies to all spherical Gromov--Witten invariants in semipositive symplectic manifolds, since in this case it is possible to compactify the image $\ev(\Mm_{k}(A,J'))$ by adding codimension-$2$ strata.
Full details for this construction can be found in \cite{MS}.
The most general form of this {\bf geometric regularization approach} proceeds in the following steps.

\begin{enumlist}
\item{\bf Fredholm setup:}
Write the (not necessarily compact) moduli space $\Mm = \si^{-1}(0)/{\rm Aut}$ as the
quotient, by an appropriate reparametrization group ${\rm Aut}$, of an equivariant smooth Fredholm section $\si:\Hat\Bb\to\Hat\Ee$ of a Banach vector bundle $\Hat\Ee\to\Hat\Bb$.
For example, $\Mm=\Mm_k(A,J)$ is cut out from $\Hat\Bb=W^{m,p}(S^2,M)$ by the Cauchy--Riemann operator $\si=\pbar$, which is equivariant with respect to ${\rm Aut}=\PSL(2,\C)$.

\item {\bf Geometric perturbations:}
Find a Banach manifold $\Pp\subset\Ga^{\rm Aut}(\Hat\Ee)$ of {\it equivariant }
sections for which the perturbed sections $\si+p$ have the same 
Fredholm and 
compactness properties as $\si$.
For example, the contractible set $\Jj^\ell$ of compatible $\Cc^\ell$-smooth almost complex structures for $\ell\geq m$ provides equivariant sections $p=\overline{\partial}_{J'}-\pbar$ for all $J'\in\Jj^\ell$. Moreover, $J'$-holomorphic curves also have a Gromov compactification $\oMm_k(A,J')$.

\item{\bf Sard--Smale:}
Check transversality of the section $(p,f)\to (\si + p)(f)$ to deduce that the universal moduli space
$$
\Mm(\Pp):=
{\textstyle \bigcup_{p\in\Pp}} \; \{p\}\times (\si + p)^{-1}(0) \;\subset\; \Pp\times\Hat\Bb
$$
is a differentiable Banach manifold. (In the example it is $\Cc^\ell$-differentiable.)
Then the Sard--Smale theorem applies to the projection $\Mm(\Pp)\to\Pp$ when the differentiability $\ell$ is sufficiently high -- larger than the Fredholm index of $\si$.
It provides a comeagre subset $\Pp^{\rm reg}\subset\Pp$ of regular values, for which the perturbed sections $\si_p:=\si+p$ are transverse to the zero section.
For holomorphic curves and perturbations given by $\Jj^\ell$, this transversality holds if all holomorphic maps are {\it somewhere injective}.
For $\om$-minimal  $A$ this weak form of injectivity
is a consequence of unique continuation (cf.\ \cite[Chapter~2]{MS}),
but for general Gromov--Witten moduli spaces this step only
applies to the subset $\Mm_k^*(A,J)$ of simple (i.e.\ not multiply covered) curves.
\item {\bf Quotient:}
For $p\in \Pp^{\rm reg}$, the perturbed zero set $\si_p^{-1}(0)\subset \Hat\Bb$  is
a smooth manifold by the implicit function theorem.
If, moreover, the action of ${\rm Aut}$ on $\si_p^{-1}(0)$ is smooth, free, and properly discontinuous, then the moduli space $\Mm^p := \si_p^{-1}(0) / {\rm Aut}$ is a smooth manifold.
For holomorphic curves, the smoothness of the action can be achieved if all solutions of $\overline{\partial}_{J'}f=0$ are smooth. For that purpose one can use e.g.\ the Taubes' trick to find regular perturbations given by smooth $J'$.

\item {\bf Compactification:}
For Gromov--Witten moduli spaces with $A$ $\om$-minimal and $k\le 3$, the previous steps already give $\Mm_k(A,J')$ the structure of a compact smooth manifold.
Thus the Gromov--Witten invariants can be defined using its fundamental class $[\Mm_{k}(A,J')]$.
In the semipositive case, the previous steps give $\Mm_k^*(A,J')$ a smooth structure such that
$\ev:\Mm_k^*(A,J')\to M^k$ defines a pseudocycle.
Indeed, its image is compact up to $\ev\bigl(\oMm_{k}(A,J') \less  \Mm_{k}^*(A,J')\bigr)$ which is given by the images of nodal and multiply covered maps.
Since the underlying simple curves are regular and of lower Fredholm index, these additional sets are smooth and of codimension at least $2$, so that they do not contribute to the homological boundary of  the image $\ev(\Mm_k^*(A,J'))$.

A more general approach for showing this pseudocycle property is to use {\bf gluing techiques}, which also apply to moduli spaces whose regularization is expected to have boundary.
Generally, one obtains a compactification $\oMm\,\!^p$ of the perturbed moduli space by constructing gluing maps into $\Mm^p$, whose images cover the complement of a compact set, and which are compatible on their overlaps.
For example, the gluing construction in the Gromov--Witten case, roughly speaking, provides local homeomorphisms
$$
((1,\infty)\times S^1)\,\!^{N} \times \Nn^N_k(A,J') \;\hookrightarrow\; \Mm_{k}(A,J')
$$
for each regular moduli space $\Nn^{N}_k(A,J')$ of stable curves with $N$ nodes.
A Gromov compactification $\oMm_k(A,J')$ is then constructed by completing each cylinder to a disc $\bigl((1,\infty)\times S^1\bigr) \cup \{\infty\}$, where we identify the added set $ \{\infty\}\times  \Nn^{N}_k(A,J') $ with a stratum of nodal curves in the compactification $\Nn^N_k(A,J') \subset \oMm_{k}(A,J')$.
Then $\oMm_{k}(A,J')$ is compact in the Gromov topology.
In the semipositive case, this gluing construction can be applied to moduli spaces of simple nodal curves (i.e.\ without multiply covered or repeated components) to obtain a partial compactification that carries a fundamental class. 
However, a general construction of a fundamental class $[\oMm_{k}(A,J')]$ from this gluing procedure would require transversality for moduli spaces $\Nn^{N}_k(A,J')$ of non-simple stable curves, which cannot always be achieved by choice of $J'$.

\item {\bf Invariance:}
To prove that invariants extracted from the perturbed moduli space $\oMm\,\!^p$ are well defined, one chooses $\Pp$ to be a connected neighbourhood of the zero section and constructs a cobordism between $\oMm\,\!^{p_0}$ and $\oMm\,\!^{p_1}$ for any regular pair $p_0,p_1\in\Pp^{\rm reg}$ by repeating the last five steps for the section $[0,1]\times\Hat\Bb\to\Hat\Ee$, $(t,b)\mapsto (\si+p_t)(b)$ for any smooth path $(p_t)_{t\in[0,1]}\in\Pp$.
In the semipositive Gromov--Witten example, the same argument is applied to find a pseudochain with boundary $\ev(\Mm_{k}^*(A,J_0)) \sqcup \ev(\Mm_{k}^*(A,J_1))$.
\end{enumlist}

\begin{remark}\label{rmk:GWmult} \rm
For Gromov-Witten theory, the evaluation map \eqref{eq:ev} generally does not represent a well defined rational homology class in $M^k$.  Although $\oMm_{k}(A,J)$ is compact and has a well understood formal dimension $d$ (given by the Fredholm index of $\pbar$ minus the dimension of the automorphism group), it need not be a manifold or orbifold of dimension $d$ for any~$J$.  Indeed it may contain subsets of dimension larger than $d$ consisting of stable maps with a component that is a
multiple cover on which $c_1(f^*\rT S^2)$ is negative.
In the case of spherical Gromov--Witten theory on manifolds with $[\om]\in H^2(M;\Q)$, it is possible to avoid this problem by first finding a consistent way to ``stabilize the domain'' to obtain a global description of the moduli space that involves no reparametrizations, and then allowing a richer class of perturbations; cf.\ \cite{CM}.
Though it is quite possible that this method can be extended to the higher genus case as claimed in 
\cite{Gerst,Io,IoP}, some potential pitfalls with this approach are pointed out in \cite{TZ}.
$\hfill\er$
\end{remark}

The main nontrivial steps in the geometric approach, which need to be performed in
careful detail for any given moduli space, are the following.
\begin{itemlist}
\item Each setting requires a different, precise definition of a Banach space of perturbations. Note in particular that spaces of maps with compact support in a given open set are not compact. The proof of transversality of the universal section is very sensitive to the specific geometric setting, and in the case of varying $J$ requires each holomorphic map to have suitable injectivity properties.
\item The gluing analysis is a highly nontrivial Newton iteration scheme and
should have an abstract framework that does not seem to be available at present. In particular, it requires surjective linearized operators, and so only applies after perturbation. Moreover, gluing of noncompact spaces requires uniform quadratic estimates, which do not hold in general.
Finally, injectivity of the gluing map does not follow from the Newton iteration and needs to be checked in each geometric setting.
\end{itemlist}

%%%%%%%%%%%%%%%%%%%%%%%%%%%%%%%%%%%%%%%%%%%%%%%%%%
\subsection{Approaches to abstract regularization} \hspace{1mm}\\ \vspace{-3mm}
\label{ss:approach}
%%%%%%%%%%%%%%%%%%%%%%%%%%%%%%%%%%%%%%%%%%%%%%%%%%

In order to obtain a regularization scheme that is generally applicable to holomorphic curve moduli spaces, it seems to be necessary to work with abstract perturbations $f\mapsto p(f)$ that need not be differential operators. Thus we recast the question more abstractly into one of regularizing a
compactification of the quotient of the zero set of a Fredholm operator.\footnote{
As pointed out by Aleksey Zinger, the ``Gromov compactification'' of a moduli space of holomorphic curves in fact need not even be a compactification in the sense of containing the holomorphic curves with smooth domains as dense subset. For example, it could contain an isolated nodal curve.
}
From this abstract differential geometric perspective, the geometric regularization scheme provides a highly nontrivial generalization of the well known finite dimensional regularization based on Sard's theorem, see e.g.\ \cite[ch.2]{GuillP}.

\label{finite reg}
\medskip
\noindent
{\bf Finite Dimensional Regularization Theorem:}  {\it
Let $E\to B$ be a finite dimensional vector bundle, and let $s:B\to E$ be a smooth section such that $s^{-1}(0)\subset B$ is compact.
Then there exists a compactly supported, smooth perturbation section $p:B\to E$ such that $s+p$ is transverse to the zero section, and hence $(s+p)^{-1}(0)$ is a smooth manifold.
Moreover, $[(s+p)^{-1}(0)]\in H_*(B,\Z)$ is independent of the choice of such perturbations.
}

\begin{remark}\rm   \label{equivariant BS}
(i)
Using multisections, this theorem generalizes to equivariant sections under a finite group action, yielding  
branched manifolds 
as regularized spaces and thus a well defined rational homology class $[(s+p)^{-1}(0)]\in H_*(B,\Q)$.
\MS

\NI (ii)
For nontrivial Lie groups $G$ acting by bundle maps on $E$, equivariance and transversality are in general contradictory requirements on a section.
Only if $G$ acts smoothly, freely, and properly on $B$ and $E$,  can
one obtain $G$-equivariant transverse sections by pulling back transverse sections of $E/G\to B/G$.
%
%NOTE TO SELF
%
%namely, if $\si:B/G\to E/G$ is transverse, define $\ti\si(b):= e$ where $\si(b)=Ge$, $e\in E_b$
%
\MS

\NI (iii)
Finite dimensional regularization also holds for noncompact zero sets $s^{-1}(0)$,
but the homological invariance of the zero set fails in the simplest examples.
\MS

\NI (iv) There have been several attempts to extend this theorem to the case of a Fredholm section $s:\Hat\Bb\to \Hat \Ee$ of a Banach (orbi)bundle 
\cite{Lu,LuT,CT}.
However, in their global form these do not apply to most Gromov--Witten moduli spaces, and when localized they run into serious problems concerning the smoothness of coordinate changes and lack of suitable cut off functions that we discuss in \S\ref{ss:LTBS}.
$\hfill\er$
\end{remark}

Note here that no typical moduli spaces of holomorphic curves, nor even the moduli spaces in gauge theory or Morse theory, have a currently available description as the zero set of a Fredholm section in a Banach groupoid bundle.
In the case of holomorphic curves or Morse trajectories, the first obstacle to such a description is the differentiability failure of the 
reparametrization action 
on any Sobolev space of maps 
explained in \S\ref{ss:nodiff}.
In gauge theory, the action of the gauge group typically is smooth, but in all 
theories the typical moduli spaces are compactified by gluing constructions, for which there is not even a natural description as a zero set in a topological vector bundle.

In comparison, the geometric regularization approach works with a smooth section $\si:\Hat\Bb\to\Hat\Ee$ of a Banach bundle, which has a noncompact solution set $\si^{-1}(0)$ and is equivariant under the action of a noncompact Lie group.
From an abstract topological perspective, the nontrivial achievement of this approach is that it produces equivariant transverse perturbations and a well defined homology class by compactifying quotients of perturbed spaces, rather than by directly perturbing the compactified moduli space.

\begin{remark}\rm
Another notable analytic feature of the perturbations obtained by altering $J$ is that they preserve the compactness and Fredholm properties of the nonlinear differential operator, despite changing it nonlinearly in highest order. Indeed, in local coordinates, $\si= \partial_s + J \partial_t$ is a first order operator, and changing $J$ to $J'$ amounts to adding another first order operator $p=(J'-J)\partial_t$.
This preserves the Fredholm operator since it preserves ellipticity of the symbol. In general, one retains Fredholm properties only with lower order perturbations, i.e.\ by adding a compact operator to the linearization.
For the Cauchy--Riemann operator, that would mean an operator involving no derivatives, e.g.\ $p(f)=X\circ f$ given by the pullback of a vector field $X:M\to{\rm T}M$.
Note also that the compactness properties of solution sets of nonlinear operators are generally not even preserved under lower order perturbations that are supported in a neighbourhood of a compact solution set, since in the infinite dimensional setting such neighbourhoods are never compact.
$\hfill\er$
\end{remark}

This discussion shows that a regularization scheme for general holomorphic curve moduli spaces needs to work with more abstract perturbations and directly on the compactified moduli space -- i.e.\ after quotienting and taking the Gromov compactification. The following approaches,
which are further discussed in \S\ref{ss:LTBS}, \S\ref{ss:kur}, \S\ref{ss:poly} respectively, 
 are currently used in symplectic topology.

\begin{enumlist}
\item{The {\bf global obstruction bundle approach}} as 
introduced by Liu-Tian and Siebert \cite{LiuT, Sieb, Mcv} aims to extend 
techniques from algebraic geometry and gauge theory to
holomorphic curve settings, by means of a weak orbifold structure on a suitably stratified Banach space completion of the space of equivalence classes of smooth stable maps.

\item{The {\bf Kuranishi approach}}, introduced by Fukaya-Ono \cite{FO} and implicitly Li-Tian \cite{LT} in the 1990s, aims to construct a virtual fundamental class from finite dimensional reductions of the equivariant Fredholm problem and gluing maps near nodal curves.

\item{The {\bf polyfold approach}}, developed by Hofer-Wysocki-Zehnder in [HWZ1--5], aims to generalize the finite dimensional regularization theorem so that it applies directly to the compactified moduli space, by expressing it as the zero set of a smooth section.
\end{enumlist}

\MS

The first two approaches are also referred to as virtual transversality.
They have been used for concrete calculations of Gromov--Witten invariants, e.g.\ \cite{MT,Mq} by building a VMC using geometrically meaningful perturbations.
The third approach is more functorial and produces a VMC with significantly more structure, e.g.\ as a cobordism class of smooth weighted branched manifolds in the case of Gromov--Witten invariants \cite{HWZ:GW}. 
This allows one to define the invariants of, for example, symplectic field theory (SFT) on the chain level. 
The book \cite{FOOO} uses the Kuranishi approach to a similar end in the construction of chain level Lagrangian Floer theory.
We will make no further comments on chain level theories.
Instead, let us compare how the different approaches handle the fundamental analytic issues.

\medskip\noindent
{\bf Dividing by the automorphism group:}
Unlike the smooth action of the (infinite dimensional) gauge group on Sobolev spaces of connections, the reparametrization groups (though finite dimensional) do not act differentiably on any known Banach space completion of spaces of smooth maps (or pairs of domains and maps from them); see \S\ref{s:diff}.
In the global obstruction bundle approach this causes a significant differentiability failure in the relation between local charts in 
\cite{LiuT}, and hence in the survey article 
\cite{Mcv} and subsequent papers such as 
\cite{CL,CT,Lu,LuT}.
For more details of the problems here and some proposed solutions, see \S\ref{ss:LTBS}.
This differentiability issue was not mentioned in \cite{FO,LT}.
However, as we explain in detail in \S\ref{ss:gw}, it needs to be addressed when defining charts that combine two or more basic charts since this must be done in the Fredholm setting {\it before} passing to a finite dimensional reduction.
We make this explicit in our notion of ``sum chart", but the same construction is used implicitly in \cite{FO,FOOO}, and now more explicitly in \cite{FOOO12}.
In this setting, it can be overcome by working with special obstruction bundles, as we outline in \S\ref{ss:gw}.
In the polyfold approach, this issue is resolved by replacing the notion of smoothness in 
Banach spaces by a notion of scale-smoothness which applies to the reparametrization action.
To implement this, one must redevelop linear as well as nonlinear functional analysis in the scale-smooth category.

\medskip
\noindent
{\bf Gromov compactification:}
Sequences of holomorphic maps can develop various kinds of singularities: bubbling (energy concentration near a point), breaking (energy diverging into a noncompact end of the domain -- sometimes also induced by stretching in the domain), buildings (parts of the image diverging into a noncompact end of the target -- sometimes also induced by stretching the domain at a hypersurface).
These limits are described as tuples of maps from various domains, capturing the Hausdorff limit of the images. In quotienting by reparametrizations, note that for the limit object this group is a substantially larger product of various reparametrization groups.

In the geometric and virtual regularization approaches, charts near the singular limit objects are constructed by gluing analysis, which involves a pregluing construction and a Newton iteration. The pregluing creates from a tuple of holomorphic maps a single map from a nonsingular domain, which solves the Cauchy--Riemann equation up to a small error. The Newton iteration then requires quadratic estimates for the linearized Cauchy--Riemann operator to find a unique exact solution nearby.
In principle, the construction of a continuous gluing map should always be possible along the lines of \cite{MS}, though establishing the quadratic estimates is nontrivial in each setting. However, additional arguments specific to each setting are needed to prove surjectivity, injectivity, and openness of the gluing map.
Moreover, while homeomorphisms to their image suffice for the geometric regularization approach, 
many of the virtual regularization approaches 
require stronger differentiability of the gluing map; e.g.\ smoothness in \cite{FO,FOOO,J1}.  
Exceptions are e.g.\ \cite{Sieb,pardon}, which only use gluing theorems along the lines of \cite{MS}. 

None of \cite{LT,LiuT,FO,FOOO} give all details for the construction of a gluing map.
For closed nodal curves, \cite[Chapter~10]{MS} constructs continuous gluing maps in full detail, but 
does not claim that the glued curves depend differentiably on the gluing parameter $a\in\C$ as $a\to 0$.
However, our study of the transition maps indicates that the analysis of \cite{MS} in fact establishes smoothness within each stratum and global continuity of these stratum differentials.
Alternatively, one might rescale $|a|$ to obtain more differentiability across strata such as $\Cc^1$ gluing maps used in \cite{Ruan}.
However, \cite{CL} pointed out that this $\Cc^1$ structure is not intrinsic, so extra care is needed to ensure that they are preserved under coordinate changes. 
In the genus zero case, \cite{Cast1} uses a uniform rescaling of cross ratios to obtain a $\Cc^1$ structure, 
while \cite{FOOO12} aim to achieve a smooth structure by rescaling of gluing parameters.
Recent work in \cite{CLW2} also explores under what conditions one might be able to abstractly construct a (global but not necessarily unique) smooth structure on a stratified space.

The polyfold approach reinterprets the pregluing construction as the chart map for an ambient space $\Tilde \Bb$ which contains the compactified moduli space, essentially making the quadratic estimates part of the definition of a Fredholm operator on this space. The Newton iteration is replaced by an abstract implicit function theorem for transverse Fredholm operators in this setting. The injectivity and surjectivity issues then only need to be dealt with at the level of pregluing. Here injectivity fails dramatically but in a way that can be reinterpreted in terms of a generalization of a Banach manifold chart, where the usual model domain of an open subset in a Banach space is replaced by a relatively open subset in the image of a scale-smooth retraction of a scale-Banach space.
This makes it necessary to redevelop differential geometry in the context of retractions and scale-smoothness.

%%%%%%%%%%%%%%%%%%%%%%%%%%%%%%%%%%%%%%%%%%%%%%%%%%
\subsection{
The polyfold regularization approach
}  \hspace{1mm}\\ \vspace{-3mm}
\label{ss:poly}
%%%%%%%%%%%%%%%%%%%%%%%%%%%%%%%%%%%%%%%%%%%%%%%%%%

In the setting of holomorphic maps with trivial isotropy (but allowing for general compactifications by e.g.\ nodal curves), the result of the entirely abstract development of scale-smooth nonlinear functional analysis and retraction-based differential geometry is the following direct generalization
of the finite dimensional regularization theorem, see \cite{HWZ3}.
The following is the relevant version for trivial isotropy, in which ambient spaces have the structure of an M-polyfold --- a generalization of the notion of Banach manifold, essentially given by charts in open subsets of images of retraction maps, and scale-smooth transition maps between the ambient spaces of the retractions.

\medskip
\noindent
{\bf M-polyfold Regularization Theorem:} {\it
Let $\Tilde{\Ee}\to\Tilde{\Bb}$ be a strong M-polyfold bundle, and let $s:\Tilde{\Bb}\to\Tilde{\Ee}$ be a scale-smooth Fredholm section such that $s^{-1}(0)\subset\Tilde{\Bb}$ is compact.
Then there exists a class of perturbation sections $p:\Tilde{\Bb}\to\Tilde{\Ee}$ supported near $s^{-1}(0)$ such that $s+p$ is transverse to the zero section, and hence $(s+p)^{-1}(0)$ carries the structure of a smooth finite dimensional manifold.
Moreover, $[(s+p)^{-1}(0)]\in H_*(\Tilde\Bb,\Z)$ is independent of the choice of such perturbations.
}

\medskip
For dealing with nontrivial, finite isotropies, \cite{HWZ3} transfers this theory to a groupoid setting to obtain a direct generalization of the orbifold version of the finite dimensional regularization theorem.
It is these groupoid-type ambient spaces $\Tilde\Bb$, whose object and morphism spaces are M-polyfolds, that are called polyfolds.
These abstract regularization theorems should be compared with the definition of Kuranishi atlas and the abstract construction of a virtual fundamental class for any Kuranishi atlas that will be outlined in the following sections.
While the language of polyfolds and the proof of the regularization theorems in \cite{HWZ1,HWZ2,HWZ3} is highly involved, it seems to be developed in full detail and is readily quotable.
A survey of the basic philosophy and language is now available in \cite{gffw}.

Just as in the construction of a Kuranishi atlas for a given holomorphic curve moduli space discussed in \S\ref{ss:gw}, the application of the polyfold regularization approach still requires a description of the compactified moduli space as the zero set of a Fredholm section in a polyfold bundle.
It is here that the polyfold approach promises the most revolutionary advance in regularization techniques. Firstly, fiber products of moduli spaces with polyfold descriptions are naturally described as zero sets of a Fredholm section over a product of polyfolds. For example, one can obtain a polyfold setup for the PSS morphism by combining the polyfold setup for SFT with a smooth structure on Morse trajectory spaces, see \cite{afw:arnold}.
Secondly, Hofer--Wysocki--Zehnder are currently working on formalizing a ``modular''
approach to the polyfold axioms in such a way that the analytic setup can be given locally in domain and target for every singularity type. With that, the polyfold setup for a new moduli space that combines previously treated singularities in a different way would merely require a Deligne--Mumford type theory for the underlying spaces of domains and targets.

\begin{remark} \rm \label{polyfold BS checklist}
While the polyfold framework is a very powerful method for constructing algebraic invariants from holomorphic curve moduli spaces, it also has some pitfalls in geometric applications.
\begin{itemlist}
\item
Some caution is required with arguments involving the geometric properties of solutions after regularization.
The reason for this is that
the perturbed solutions do not solve a PDE but an abstract compact perturbation of the Cauchy--Riemann equation. Essentially, one can only work with the fact that the perturbed solutions can be made to lie arbitrarily close to the unperturbed solutions in any metric that is compatible with the scale-topology (e.g.\ any $\Cc^k$-metric in the case of closed curves).
\item
Despite reparametrizations acting scale-smoothly on spaces of maps, the question of equivariant regularization for smooth, free, proper actions remains nontrivial due to the interaction with retractions, i.e.\ gluing constructions.
In the example of the $S^1$-action on spaces of Floer trajectories for an autonomous Hamiltonian, the unregularized compactified Floer trajectory spaces of virtual dimension $0$ may contain broken trajectories. The corresponding stratum of the quotient space,
$$
\oMm(p_-,p_+)/S^1 \;\supset\; {\textstyle \bigcup_q} \bigl(\oMm(p_-,q) \times \oMm(q,p_+)\bigr)/S^1 ,
$$
is an $S^1$-bundle over the fiber product $\oMm(p_-,q)/S^1\times \oMm(q,p_+)/S^1$ of quotient spaces, rather than the fiber product itself.
Due to these difficulties, as yet, there is no quotient theorem for polyfolds, and hence no understanding of when a description of $\oMm$ as zero set of an $S^1$-equivariant Fredholm section would induce a description of $\oMm/S^1$ as zero set of a Fredholm section with smaller Fredholm index. Moreover, such a quotient would not even immediately induce an equivariant regularization of the Floer trajectory spaces compatible with gluing. 
$\hfill\er$
\end{itemlist}
\end{remark}

%%%%%%%%%%%%%%%%%%%%%%%%%%%%%%%%%%%%%%%%%%%%%%%%%%
\subsection{
The Kuranishi regularization approach
}  \hspace{1mm}\\ \vspace{-3mm} \label{ss:kur}
%%%%%%%%%%%%%%%%%%%%%%%%%%%%%%%%%%%%%%%%%%%%%%%%%%

Continuing the notation of \S\ref{ss:geom}, the basic idea of 
the Kuranishi approach to regularization 
is to describe the compactified moduli space $\oMm$ by local finite dimensional reductions of the ${\rm Aut}$-equivariant section $\si:\Hat\Bb \to \Hat\Ee$, and by gluing maps near the nodal curves. 
There are different ways to formalize the compatibility of finite dimensional reductions, yielding notions of ``Kuranishi structure'' in \cite{FO}, ``smooth resolution'' in \cite{LT}, or ``Kuranishi atlas'' in our work, but all 
proceed along the lines of the following steps.

\MS

\begin{enumlist}
\item {\bf Compactness:}
Equip the compactified moduli space $\oMm$ with a compact, metrizable topology; namely as the Gromov compactification of $\Mm=\si^{-1}(0)/{\rm Aut}$.

\item {\bf Equivariant Fredholm setup:}
As in the geometric approach, a significant subset $\Mm\subset\oMm$ of the compactified moduli space is given as the zero set of a Fredholm section modulo a finite dimensional Lie group,
\vspace{-10mm}
\[
\begin{aligned}
&\phantom{\Mm} \\
&\phantom{\Mm} \\
& \Mm = \frac{\si^{-1}(0)}{{\rm Aut}} ,
 \end{aligned}
 \qquad\qquad
\xymatrix{
 \Hat\Ee   \ar@(ul,dl)_{\textstyle \rm Aut}\ar@{->}[d]     \\
 \Hat\Bb \ar@(ul,dl)_{\textstyle \rm Aut} \ar@/_1pc/[u]_{\textstyle \si}
}
\]
One can now relax the assumption of $\rm Aut$ acting freely to the requirement that
 the isotropy subgroup $\Ga_{f}:=\{ \ga \in {\rm Aut} \,|\, \ga\cdot f = f \}$ be finite for every solution
 $f\in\si^{-1}(0)$.\MS

\item {\bf Finite dimensional reduction:}
Construct {\bf basic Kuranishi charts}
for every ${[f]\in\Mm}$,
\vspace{-5mm}
\[
\begin{aligned}
&\phantom{\Mm} \\
&\phantom{\Mm} \\
 \Mm \; \overset{\psi_f}{\longhookleftarrow} \;\frac{\ti s_f^{-1}(0)}{\Ga_f} ,
 \end{aligned}
 \qquad\qquad
\xymatrix{
*+[r]{\Tilde E_f }  \ar@(ul,dl)_{\textstyle \Ga_f}
  \ar@{->}[d]    \\
 *+[r]{U_f} \ar@(ul,dl)_{\textstyle \Ga_f} \ar@/_1pc/[u]_{\textstyle \ti s_f}
}
\]
which depend on a choice\footnote{
In practice the Kuranishi data will be constructed from many choices, including that of a representative. So we try to avoid false impressions by using the subscript $f$ rather than $[f]$.
}
of representative $f$, and consist of the following data:
\begin{itemize}
\item
the {\bf domain} $U_f$ is a finite dimensional 
manifold (constructed from a local slice of the ${\rm Aut}$-action on a thickened solution space
$\{g \,|\, \pbar g \in \Hat E_{f}\}$);
\item
the {\bf obstruction bundle} $\Ti E_f =\Hat E_f|_{U_f}\to U_f$, a finite rank vector bundle (constructed from the cokernel of the linearized Cauchy--Riemann operator at $f$), which is isomorphic $\Ti E_f\cong U_f\times E_f$ to a trivial bundle, whose fiber $E_f$ we call the {\bf obstruction space};
\item
the {\bf section} $\ti s_{f}: U_{f} \to \Tilde E_{f}$ (constructed from $g \mapsto \pbar g$),
which induces a (usually smooth, but at least continuous) map $s_f : U_f\to E_f$ in the trivialization;
\item
the {\bf isotropy group} $\Ga_{f}$ acting on $U_{f}$ and $E_{f}$ such that $\ti s_{f}$ is equivariant; \item
the {\bf footprint map} $\psi_{f}: \ti s_{f}^{-1}(0)/\Ga_{f} \to \oMm$, a homeomorphism to a neighbourhood of $[f]\in\oMm$ (constructed from $\{g \,|\, \pbar g=0 \}
\ni g \mapsto [g]$).
\end{itemize}
A detailed outline of this construction for $\oMm_1(A,J)$ with $\Ga_f=\{\rm id\}$ is given in \S\ref{ss:Kchart}.

\item {\bf Gluing:}
Construct basic Kuranishi charts covering $\oMm\less \Mm$ by combining finite dimensional reductions with gluing analysis similar to the geometric approach.

\item {\bf Compatibility:}
Given a finite cover of $\oMm$ by the footprints of basic Kuranishi charts $\bigl(\bK_i = (U_i,E_i,\Ga_i,s_i,\psi_i)\bigr)_{i=1,\ldots,N}$, construct transition data satisfying suitable compatibility conditions.
In the case of trivial isotropies $\Ga_i=\{{\rm id}\}$, any notion of compatibility will have to induce the following minimal transition data for any element $[g]\in \im\psi_i\cap \im\psi_j\subset \oMm$ in an overlap of two footprints:
\begin{itemize}
\item
a
{\bf transition Kuranishi chart} $\bK^{ij}_g = (U^{ij}_g,E^{ij}_g,s^{ij}_g,\psi^{ij}_g)$ whose footprint $\im\psi^{ij}_g \subset  \im\psi_i\cap \im\psi_j$ is a neighbourhood of $[g]\in\oMm$;
\item
{\bf coordinate changes} $\Phi^{i,ij}_g : \bK_i \to \bK^{ij}_g$ and $\Phi^{j,ij}_g :\bK_j \to \bK^{ij}_g$ consisting of embeddings and linear injections
$$
\phi^{\bullet,ij}_g :\;  U_\bullet \supset V^{\bullet,ij}_g\; \longhookrightarrow\; U^{ij}_g , \qquad
\Hat\phi^{\bullet,ij}_g :\;  E_\bullet \; \longhookrightarrow\; E^{ij}_g \qquad
\text{for}\; \bullet = i,j
$$
which extend $\phi^{\bullet,ij}_g|_{\psi_\bullet^{-1}(\im\psi^{ij}_g)} = (\psi^{ij}_g)^{-1}\circ \psi_\bullet$ to open subsets $V^{\bullet,ij}_g \subset U_\bullet$ such that
$$
s^{ij}_g \circ \phi^{\bullet,ij}_g   \; =\;  \Hat\phi^{\bullet,ij}_g\circ s_\bullet  \qquad
\text{for}\; \bullet = i,j .
$$
\end{itemize}
At this point the approaches differ in making further requirements on the charts such as differentiability, an index or ``tangent bundle" condition, coordinate changes between multiple overlaps, and cocycle conditions; see \S\ref{ss:alg}, \S\ref{ss:top} for a discussion. Our notion of {\bf Kuranishi atlas} in \S\ref{ss:Ksdef} involves a collection of smooth basic charts, transition charts, and coordinate changes satisfying an index condition and cocycle conditions.

The analytic challenges in constructing compatible Kuranishi charts for a given holomorphic curve moduli space are discussed in \S\ref{s:diff}, \S\ref{s:construct}.

\item {\bf Abstract Regularization:}
For a suitable notion of Kuranishi data (involving a covering of $\oMm$ by finite dimensional reductions, suitable transition data, and compatibility conditions), there should be a
{\bf Kuranishi Regularization Theorem} along the lines of 

\smallskip
\noindent
{\it Kuranishi data $\Kk$ on a compact space $\oMm$ induce a virtual fundamental class $[\oMm]_\Kk^{\rm vir}$.}\smallskip

This is an abstract result, proven without reference to pseudoholomorphic curves, by ``patching local Euler classes". This patching is either done on the level of (multivalued) transverse perturbations $s_f': U_f\to E_f$ of the sections in the Kuranishi charts, or their zero sets, or in a sheaf theoretic way as in \cite{pardon}.

\item {\bf Invariance:}
Prove that $[\oMm]_\Kk^{\rm vir}$ is independent of the different choices in the previous steps, in particular the choice of local slices and obstruction bundles. This involves the construction of  
Kuranishi data 
for $[0,1]\times\oMm $ that restricts to two given choices $\Kk^0$ on $\{0\}\times \oMm$ and $\Kk^1$ on $\{1\}\times \oMm$. Then an abstract 
Kuranishi 
cobordism theory  
should imply $[\oMm]_{\Kk^0}^{\rm vir}=[\oMm]_{\Kk^1}^{\rm vir}$.
\end{enumlist}

The construction of a Kuranishi atlas for a given holomorphic curve moduli space is explained in more detail in \S\ref{s:construct}.
The rest of the paper, \S\ref{s:chart}--\S\ref{s:VMC}, then 
develops our specific Kuranishi atlas approach.
For that purpose we restrict to the case of trivial isotropy groups $\Ga_{f}=\{{\rm id}\}$ in all
Kuranishi charts. This simplifies constructions in two ways.
First, it significantly simplifies the form of the coordinate changes between different charts, and hence simplifies the structure of the atlas.
Second, for trivial isotropy one can construct the virtual fundamental class from the zero sets of perturbed sections $s_f + \nu_f \approx s_f$ that are transverse, $s_f+\nu_f\pitchfork 0$,
rather than replacing each $\Ga_f$-equivariant section $s_f$ with a transverse multisection.

We devote a great deal of attention to this aspect of the theory since this part was 
particularly shortchanged in the early work \cite{FO,LT} on the subject
while, as we explain in \S\ref{ss:top}, there are
many interesting topological questions to deal with.

%%%%%%%%%%%%%%%%%%%%%%%%%%%%%%%%%%%%%%%%%%%%%%%%%%
\subsection{Algebraic issues in the use of germs for Kuranishi structures}  \hspace{1mm}\\ \vspace{-3mm}  \label{ss:alg}
%%%%%%%%%%%%%%%%%%%%%%%%%%%%%%%%%%%%%%%%%%%%%%%%%%

A natural approach, adopted in \cite{FO} and more recently \cite{J1}, for formalizing the compatibility of Kuranishi charts is to work with germs of charts and coordinate changes.
Recent discussions have led to an agreement that this approach has serious algebraic issues in making sense of a cocycle condition for germs of coordinate changes, which we explain here. This issue is rooted in the fact that only the footprints of Kuranishi charts have invariant meaning, so that
the coordinate changes between Kuranishi charts are fixed by the charts only on the zero sets. Thus in the definition of a germ of charts the traditional equivalence of maps with common restriction to a smaller domain is extended by equivalence of maps that are intertwined by a diffeomorphism of the domains. This leads to an ambiguity in the definition of germs of coordinate changes between germs of charts.
As a result, germs of coordinate changes are defined as conjugacy classes of coordinate changes with respect to diffeomorphisms of the domains that fix the zero sets. However, in this setting the composition of germs is ill defined, so there is no meaningful cocycle condition.
Alternatively, one might want to view (charts, coordinate changes, equivalences of coordinate changes) as a $2$-category with ill-defined $2$-composition. Either way, there is no general procedure for extracting the data necessary for a construction of a VFC:
 a finite set of charts and coordinate changes that satisfy the cocycle condition.
In the following, we spell out in complete detail the usual definitions of germs and point out the algebraic issues that arise from the equivalence under conjugation.

\MS

To simplify notation let us (incorrectly) pretend that all obstruction spaces are finite rank subspaces $E_f\subset\Ee$ of the same space and the linear maps $\Hat\phi$ in the coordinate changes are restrictions of the identity. We moreover assume that all isotropy groups are trivial $\Ga_f=\{{\rm id}\}$ and only consider germs of charts and coordinate changes at a fixed point $p\in\oMm$.
In the following all neighbourhoods are required to be open.

\MS\NI
To the best of our understanding, \cite{FO,J1} define a germ of Kuranishi chart as follows.

\begin{itemlist}
\item
A Kuranishi chart consists of a neighbourhood $U\subset\R^k$ of $0$ for some $k\in\N$, a map $s:U\to E \subset\Ee$ with $s(0)=0$, and an embedding $\psi:s^{-1}(0)\to\oMm$ with $\psi(0)=p$,
$$
\oMm \;\overset{\psi}{\longhookleftarrow}\; s^{-1}(0)
\;\subset\;
U \;\overset{s}{\longrightarrow}\; E.
$$
\item
Two Kuranishi charts $(U_1,s_1,\psi_1)$, $(U_2,s_2,\psi_2)$ are equivalent if the transition map
$$
\psi_2^{-1}\circ\psi_1 :\; s_1^{-1}(0) \;\supset\; \psi_1^{-1}(\im\psi_2)
\;\longrightarrow \; \psi_2^{-1}(\im\psi_1) \;\subset\; s_2^{-1}(0)
$$
extends to a diffeomorphism $\theta: U'_1\to U'_2$ between neighbourhoods $U_i' \subset U_i$ of $0$ that intertwines the sections $s_1|_{U'_1}= s_2|_{U'_2}\circ \theta$.
\item
A germ of Kuranishi chart at $p$ is an equivalence class of Kuranishi charts.
\item[$\mathbf{\bigtriangleup} \hspace{-2.08mm} \raisebox{.3mm}{$\scriptscriptstyle !$}\,$]
Note that $s_1= s_2\circ \theta$ does not necessarily determine the diffeomorphism $\theta$ except on the (usually singular and not dense) zero set.
Hence there may exist auto-equivalences, i.e.\ a nontrivial diffeomorphism $\theta: U'_1\to U'_2$ between restrictions $U'_1,U'_2\subset U$ of the same Kuranishi chart $(U,\ldots)$, satisfying $\theta|_{s^{-1}(0)}={\rm id}$ and $s= s\circ \theta$.
\end{itemlist}

Next, one needs to define the notion of a coordinate change between two germs of Kuranishi charts $[U_I,s_I,\psi_I]$ and $[U_J,s_J,\psi_J]$.\footnote{
In the notation of the previous section,
an example of a required coordinate change is one for index sets $I=\{i\}$, $J=\{i,j\}$, where
$[U_I, \ldots]$ denotes the germ at $[f]$ induced by $(U_i,\ldots)$, and $[U_J, \ldots]$ denotes the germ at $[f]$ induced by $(U^{ij}_f, \ldots)$.
}
It is here that ambiguities in the compatibility conditions appear, so we give what seems like the most natural definition, which is at least closely related to \cite{FO,J1}.

\begin{itemlist}
\item
A coordinate change $(U_{IJ},\phi_{IJ}) : (U_I,s_I,\psi_I)\to (U_J,s_J,\psi_J)$ between Kuranishi charts consists of a neighbourhood $U_{IJ}\subset U_I$ of $0$ and an embedding $\phi_{IJ}:U_{IJ} \hookrightarrow U_J$ that extends the natural transition map and intertwines the sections,
$$
\phi_{IJ}|_{s_J^{-1}(0)\cap U_{IJ}} \;=\; \psi_J^{-1}\circ\psi_I , \qquad
s_J|_{U_{IJ}}  \; =\; s_I \circ \phi_{IJ} .
$$
\item
Let $(U_{I,1},s_{I,1},\psi_{I,1})\sim(U_{I,2},s_{I,2},\psi_{I,2})$ and $(U_{J,1},s_{J,1},\psi_{J,1})\sim (U_{J,2},s_{J,2},\psi_{J,2})$ be two pairs of equivalent Kuranishi charts. Then two coordinate changes
\begin{align*}
& (U_{IJ,1},\phi_{IJ,1}) : (U_{I,1},s_{I,1},\psi_{I,1})\to (U_{J,1},s_{J,1},\psi_{J,1}) \\
\text{and}\quad &
(U_{IJ,2},\phi_{IJ,2}) : (U_{I,2},s_{I,2},\psi_{I,2})\to (U_{J,2},s_{J,2},\psi_{J,2})
\end{align*}
are equivalent if there exist diffeomorphisms $\theta_I: U'_{I,1}\to U'_{I,2}$ and $\theta_J: U'_{J,1}\to U'_{J,2}$ between smaller neighbourhoods of $0$ as in the definition of equivalence of Kuranishi charts (i.e.\
$s_{I,1}|_{U'_{I,1}}= s_{I,2}|_{U'_{I,2}}\circ \theta_I$ and $s_{J,1}|_{U'_{J,1}}= s_{J,2}|_{U'_{J,2}}\circ \theta_J$) that intertwine the coordinate changes on a neighbourhood of $0$,
$$
\theta_J \circ \phi_{IJ,1} =  \phi_{IJ,2} \circ \theta_I  .
$$
\item
A germ of coordinate changes between germs of Kuranishi structures at $p$ is an equivalence class of coordinate changes.
\item[$\mathbf{\bigtriangleup} \hspace{-2.08mm} \raisebox{.3mm}{$\scriptscriptstyle !$}\,$]
As a special case, two coordinate changes $\phi_{IJ}, \phi'_{IJ} : (U_{I},\ldots )\to (U_{J},\ldots)$
 between the same Kuranishi charts are equivalent
if there exist auto-equivalences $\theta_I: U'_{I,1}\to U'_{I,2}$ and $\theta_J: U'_{J,1}\to U'_{J,2}$ such that
$\theta_J \circ \phi_{IJ} =  \phi'_{IJ} \circ \theta_I  $.
\item[$\mathbf{\bigtriangleup} \hspace{-2.08mm} \raisebox{.3mm}{$\scriptscriptstyle !$}\,$]
Given a germ of coordinate change $[U_{IJ},\phi_{IJ}] : [U_I,\ldots]\to [U_J,\ldots]$ and choices of representatives $(U'_I,\ldots), (U'_J,\ldots)$ of the germs of charts, a representative of the coordinate change now only exists between suitable restrictions $(U''_I\subset U'_I,\ldots), (U''_J\subset U'_J,\ldots)$, and even with fixed choice of restrictions may not be uniquely determined.
\end{itemlist}

\MS\NI
Finally, it remains to make sense of the cocycle condition for germs of coordinate changes.
At this point \cite{FO,J1} simply write equations such as $[\Phi_{JK}]\circ [\Phi_{IJ}] = [\Phi_{IK}]$ on the level
of conjugacy classes of maps, which do not make strict sense.
The following is an attempt to phrase the cocycle condition on the level of germs, but we will see that it falls short of implying the existence of compatible choices of representatives that is required for the construction of a VMC.

\begin{itemlist}
\item
Let $[U_{I},s_{I},\psi_{I}]$, $[U_{J},s_{J},\psi_{J}]$, and $[U_{K},s_{K},\psi_{K}]$ be germs of Kuranishi charts. Then we say that a triple of germs of coordinate changes $[U_{IJ},\phi_{IJ}], [U_{JK},\phi_{JK}],[U_{IK},\phi_{IK}]$ satisfies the cocycle condition if there exist representatives of the coordinate changes
between representatives $(U_{I},\ldots)$, $(U_{J},\ldots)$, $(U_{K},\ldots)$ of the charts,
\begin{align*}
 (U_{IJ},\phi_{IJ}) : \;(U_{I},s_{I},\psi_{I})&\to (U_{J},s_{J},\psi_{J}) , \\
 (U_{JK},\phi_{JK}) : (U_{J},s_{J},\psi_{J})&\to (U_{K},s_{K},\psi_{K}) , \\
 (U_{IK},\phi_{IK}) : \;(U_{I},s_{I},\psi_{I})&\to (U_{K},s_{K},\psi_{K}) ,
\end{align*}
such that on a neighbourhood of $0$ we have
\begin{align} \label{algcc}
\phi_{JK} \circ \phi_{IJ} = \phi_{IK} .
\end{align}
\item[$\mathbf{\bigtriangleup} \hspace{-2.08mm} \raisebox{.3mm}{$\scriptscriptstyle !$}\,$]
Note that the above cocycle condition for some choice of representatives does not imply a cocycle condition for different choices of representatives. For example, suppose that $\phi_{IJ},\phi_{JK},\phi_{IK}$ satisfy \eqref{algcc}, and consider other representatives
$$
\phi_{IJ}' = \theta_J \circ \phi_{IJ} \circ \theta_I^{-1} ,  \quad
\phi_{JK}' = \Theta_K \circ \phi_{JK} \circ \Theta_J^{-1}
$$
given by auto-equivalences $\theta_I,\theta_J,\Theta_J, \Theta_K$.
Then these fit into a cocycle condition
$$
\phi_{JK}' \circ \phi_{IJ}' \;=\;
\bigl( \Theta_K \circ \phi_{JK} \circ \Theta_J^{-1} \bigr)
\circ
\bigl( \theta_J \circ \phi_{IJ} \circ \theta_I^{-1} \bigr)  \;=\; \phi'_{IK} \;\in\; [\phi_{IK}]
$$
only if $\Theta_J=\theta_J$ and $\phi'_{IK} = \Theta_K \circ \phi_{IK} \circ \theta_I^{-1}$.
That is, the choice of one representative in the cocycle condition between three germs of coordinate changes essentially fixes the choice of the other two representatives.
This causes problems as soon as one considers the
compatibility of four or more coordinate changes.
\end{itemlist}
\MS

Now suppose that a Kuranishi structure on $\oMm$ is given by 
germs of charts at each point and 
germs of coordinate changes between each suitably close pair of points, satisfying a cocycle condition. Then the fundamentally important first step towards the construction of a VMC is the claim of \cite[Lemma~6.3]{FO} that any such Kuranishi structure has a ``good coordinate system".
The latter, though the definitions in \cite{FO,FOOO} are slightly ambiguous, is a finite cover of $\oMm$ by partially ordered charts (where two charts should be comparable iff the footprints intersect) with coordinate changes according to the partial order, and satisfying a weak cocycle condition.
In order to extract such a finite cover from a tuple of germs of charts and germs of coordinate changes, one makes a choice of representative in each equivalence class of charts and picks a finite subcover.
The first nontrivial step is to make sure that these representatives were chosen sufficiently small for coordinate changes between them to exist in the given germs of coordinate changes.
The second  crucial step is to make specific choices of representatives of the coordinate changes such that the cocycle condition is satisfied.
However, \cite[(6.19.4)]{FO} does not address the need to choose specific, rather than just sufficiently small, representatives.
 In order to reduce the number of constraints,  this would require a rather special structure of the overlaps of charts.
In general, the choice of a representative for $[\phi_{AB}]$ would affect the choice of
representatives for $[\phi_{CA}]$ or $[\phi_{AC}]$ for all $C$ with $\dim U_C\leq \dim U_B$, and for $[\phi_{BC}]$ or $[\phi_{CB}]$ when $\dim U_C \geq \dim U_A$.
These are algebraic issues, governed by the intersection pattern of the charts.

One approach to solving these algebraic issues could be to replace the definition of Kuranishi structure by that of a good coordinate system. However, we know of no direct way to construct such ordered covers and explicit cocycle conditions for a given moduli space $\oMm$.
The problem of composing conjugacy classes were also realized by Fukaya et al.\, leading to new basic definitions in \cite{FOOO} that avoid germs. However, Joyce tried to formalize germs again in \cite{J1}, and both \cite{FOOO,J1} refer to \cite{FO} for existence of a good coordinate system. 
After this paper was publicized, both author teams developed new definitions and proofs in~\cite{FOOO12,Jd,J2}.

Our approach solves both problems by defining the notion of a Kuranishi atlas as a weaker version of a good coordinate system --- without a partial ordering on the charts, but satisfying an explicit cocycle condition --- that can in practice be constructed.
We then construct an analog of a ``good coordinate system" in 
Theorem~\ref{thm:red} by an abstract refinement of the Kuranishi atlas that we call a reduction.

\begin{rmk}\rm \label{rmk:JBS}
(i)
One potential attraction of the notion of germs of Kuranishi charts is that for moduli spaces $\oMm$ arising from a Fredholm problem, there could be the notion of a ``natural germ" of charts at a point
$[f]\in\oMm$ given by the finite dimensional reductions at any representative $f$.
However, the present definition of germ does not provide a notion of equivalence between finite dimensional reductions with obstruction spaces of different dimension.
So the only natural choice would be to require obstruction spaces to have minimal rank at $f$.
But with such a choice it is not clear how to make compatible choices of the 
coordinate changes.   As we will see, given two different charts at $[f]$ there is usually no natural  choice of a coordinate change from one to another; the natural maps arise by including each of them into a bigger chart (here called their sum).
Such a construction 
takes one quickly out of the class of minimal germs.

\MS\NI
(ii)
Another approach to resolving the algebraic challenges of capturing the ambiguity of local finite dimensional reductions is proposed in \cite{Jd} using so-called ``d-orbifolds'', which have more algebraic properties than Kuranishi structures. Roughly speaking, this notion replaces conjugacy by a more restrictive relation between coordinate changes that does have a 2-categorical interpretation unlike (spaces, maps, conjugacy).
Given such a categorical structure, we expect the VFC construction to be similar to ours. However, it does not offer a direct approach to regularizing moduli spaces, but instead seems to require special types of Kuranishi structures or polyfold Fredholm sections as starting point:
\cite[Thm.15.6]{Jd} claims ``virtual class maps'' for d-orbifolds under an additional ``(semi)effectiveness'' assumption, which in our understanding could only be obtained from the constructions of \cite[Thm.16.1]{Jd} under the assumption of obstruction spaces on which the isotropy action is trivial. 
But note that this is almost a case of equivariant transversality in which e.g.\ the polyfold setup should allow for a global finite dimensional reduction as smooth section of an orbifold bundle along the lines of \cite{Yang}.

The notion of ``Kuranishi space'' in \cite{J2} is closely related to this setup, as of now does not assume ``(semi)effectiveness'', but also seems to make no claims of inducing virtual fundamental classes, except via an equivalence to ``d-orbifolds''.
$\hfill\er$
\end{rmk}

%%%%%%%%%%%%%%%%%%%%%%%%%%%%%%%%%%%%%%%%%%%%%%%%%
\subsection{Topological issues in the construction of a virtual fundamental class}  \hspace{1mm}\\ \vspace{-3mm}  \label{ss:top}
%%%%%%%%%%%%%%%%%%%%%%%%%%%%%%%%%%%%%%%%%%%%%%%%%

After one has solved the analytic issues involved in constructing compatible basic Kuranishi
charts as defined in \S\ref{ss:kur} for a given moduli space $\oMm$, the further difficulties in constructing the virtual fundamental class $[\oMm]^{\rm vir}$ 
are all essentially topological, though their solution will impose further requirements on the construction of a Kuranishi atlas.
The basic idea
for constructing a VMC is to 
``patch local Euler classes''.
This is usually\footnote{
Although the guiding principle of Kuranishi regularization has always been ``patching local homology classes'', a sheaf-theoretic formalization of this approach has only recently been proposed by \cite{pardon}.
While it requires less differentiability (since homology classes can be transferred via continuous maps, whereas the implicit function theorem for transverse perturbations requires continuous differentiability), it has to deal with the same topological issues described in the following -- just in a more abstract, sheaf-theoretic guise. Ultimately, it requires essentially the same refinements of Kuranishi data as our approach.
}
achieved by 
transverse perturbations  
$s_i\approx s_i+\nu_i \pitchfork 0$
of the section in each basic chart, such that the smooth zero sets modulo a relation given by the transition data provide a regularization of the moduli space
$$
\oMm \,\!^\nu := \; \quotient{{\underset{{i=1,\ldots,N}}{\textstyle \bigsqcup}
}\; (s_i+\nu_i)^{-1}(0)} { \sim}
$$
There are various notions of regularization; the common features (in the case of trivial isotropy and empty boundary) are that $\oMm\,\!^\nu$ should be a CW complex with a distinguished homology class $[\oMm\,\!^\nu]$ (e.g.\ arising from an orientation and triangulation), and that in some sense this class should be independent of the choice of perturbation~$\nu$.
For example, \cite{FO,FOOO} require that for any CW complex $Y$ and continuous map $f:\oMm \to Y$ that extends compatibly to the Kuranishi charts, the induced map $f:\oMm\,\!^\nu \to Y$ is a cycle, whose homology class is independent of the choice of $\nu$ and extension of $f$.
The basic issues in any regularization are that we need to make sense of the equivalence relation 
induced by transition data 
and ensure that the zero set of a transverse perturbation is not just locally smooth (and hence can be triangulated locally), but also that the transition data glues these local charts to a compact Hausdorff space without boundary. These properties are crucial for obtaining a global triangulation and thus well defined cycles.
For simplicity we aim here for the strongest version of regularization, giving $\oMm \,\!^\nu$ the structure of an oriented, compact, smooth manifold, which is unique up to cobordism.
That is, we wish to realize $\oMm \,\!^\nu$ as an abstract compact 
manifold as follows.
(We simplify here by deferring the discussion of orientations to the end of this section.)

\begin{definition} \label{def:mfd}
An abstract compact smooth manifold of dimension $d$ consists of
\begin{itemlist}
\item[{\bf (charts)}]
a finite disjoint union $\underset {{i=1,\ldots,N}}{\bigsqcup} V_i$ of open subsets $V_i\subset \R^d$,
\item[{\bf (transition data)}] 
for every pair $i,j\in\{1,\ldots,N\}$ an open subset $V_{ij} \subset V_i$ 
and a smooth embedding $\phi_{ij} : V_{ij} \hookrightarrow V_j$
such that $V_{ji} = \phi_{ij}(V_{ij})$ and $V_{ii} = V_i$,
\end{itemlist}
satisfying the {\bf cocycle condition}
$$
\phi_{jk} \circ \phi_{ij} = \phi_{ik} \qquad\text{on}\;\; \phi_{ij}^{-1}(V_{jk}) \subset V_{ik}
\qquad\quad
\forall i,j,k\in\{1,\ldots,N\},
$$
and such that the induced topological space
\begin{equation} \label{quotient}
\quotient{{\textstyle \underset{{i=1,\ldots,N}}{\bigsqcup}} V_i}{\sim}
\qquad\text{with}\quad
  x \sim y \;:\Leftrightarrow\; \exists \; i,j :  y=\phi_{ij}(x)
\end{equation}
is Hausdorff and compact.
\end{definition}

Note here that it is easy to construct examples of charts and transition data that satisfy the cocycle condition but fail to induce a Hausdorff  space, e.g.\ $V_1=V_2=(0,2)$ with $V_{12}=V_{21}=(0,1)$ and $\phi_{12}(x)=\phi_{21}(x)=x$ does not separate the points $1\in V_1$ and $1\in V_2$.
However, if we rephrase the data of charts and transition maps in terms of groupoids, then, as we now show, the Hausdorff property of the quotient is simply equivalent to a properness condition.
In this paper we take a groupoid to be a topological category whose morphisms are invertible, whose spaces of objects and morphisms are smooth manifolds, and whose structure maps (encoding source, target, composition, identity, and inverse) are local diffeomorphisms.
Such groupoids are often called {\it \'etale}. For further details see e.g.\ \cite{ALR}.

\begin{rmk}\label{rmk:grp}\rm
A collection of charts and transition data satisfying the cocycle condition as in Definition~\ref{def:mfd} induces a topological groupoid $\bG$, that is a category with
\begin{itemlist}
\item
 the topological space of objects $\Obj=\Obj_\bG = {\textstyle{\bigsqcup}_{i=1,\ldots,N}} V_i$
induced by the charts,
\item
 the topological space of morphisms $\Mor=\Mor_\bG= {\textstyle{\bigsqcup}_{i,j=1,\ldots,N}} V_{ij}$
induced by the transition domains, with
\begin{itemize}
\item[-] source map $s: \Mor \to \Obj$, $(x\in V_{ij}) \mapsto (x\in V_i)$,
and target map $ t: \Mor \to \Obj$, $(x\in V_{ij}) \mapsto (\phi_{ij}(x)\in V_j)$ induced
by the transition maps,
\item[-] composition $\Mor \leftsub{t}{\times}_s \Mor\to \Mor$,
$\bigl( (x\in V_{ij}), (y\in V_{jk})\bigr) \mapsto (x\in V_{ik})$ if $\phi_{ij}(x)=y$,
which is well defined by the cocycle condition,
\item[-]
identities $\Obj \to \Mor, x \mapsto x\in V_{ii} \cong V_i$, and inverses $\Mor \to \Mor, (x\in V_{ij}) \mapsto (\phi_{ij}(x)\in V_{ji})$, again well defined by the cocycle condition.
\end{itemize}
\end{itemlist}
Moreover, $\bG$ has the following properties.
\begin{itemlist}
\item[{\bf (nonsingular)}]
For every $x\in\Obj$ the isotropy group $\Mor(x,x)=\{\rm id_x\}$ is trivial.
\item[{\bf (smooth)}]
The object and morphism spaces $\Obj$ and $\Mor$ are smooth manifolds.
\item[{\bf (\'etale)}]
All structure maps are local diffeomorphisms.
\end{itemlist}

\NI
The quotient space \eqref{quotient} is now given as the
{\it realization}
of the groupoid $\bG$, that is
$$
|\bG| := \Obj_\bG/\sim \qquad\text{with}\qquad
x \sim y \; \Leftrightarrow\; \Mor(x,y) \neq \emptyset .
$$
This realization is a compact manifold iff $\bG$ has the following additional properties.
\begin{itemlist}
\item[{\bf (proper)}]  The product of the 
source and target map $s\times t : \Mor \to \Obj \times \Obj$ is proper,  i.e.\ preimages of compact sets are compact.
%
%NOTE TO SELF
%i.e.\ if $V_{ij}\ni x_n \to x_\infty$ and $\phi_{ij}(x_n) \to y_\infty \in V_j$ then $x_\infty\in V_{ij}$
%
\item[{\bf (compact)}]
$|\bG|$ is compact.
%
%NOTE TO SELF
%?? i.e.\ $V_i \setminus \bigcup_{j\neq i} V'_{ij}$ is compact, where $V'_{ij}\subset V_{ij}$ and $V'_{ij} \cap \phi_{ji}(V'_{ji}) = \emptyset$
%
$\hfill\er$
\end{itemlist}
\end{rmk}

Now let $\oMm$ be a compact moduli space, equipped with basic Kuranishi charts $(\bK_i = (U_i,E_i,s_i,\psi_i))_{i=1,\ldots,N}$ with trivial isotropy $\Gamma_i=\{\id\}$, whose footprints $F_i: = \psi_i(s_i^{-1}(0))\subset \oMm$ cover $\oMm$.
The regularization goal is to associate to this data an abstract compact manifold in the sense of Remark~\ref{rmk:grp}.
To understand the challenges, it is useful to formulate these charts in terms of two categories, 
the base category called $\bB_\Kk$, formed from the domains $U_i$, and the bundle category $\bE_\Kk$ formed from the obstruction bundles.
The morphism spaces in both will arise from some type of transition maps between the basic charts.
The projections $U_i\times E_i \to U_i$ and sections $s_i$ should then induce a projection functor $\pi_\Kk:\bE_\Kk\to \bB_\Kk$ and a section functor $\s_\Kk: \bB_\Kk \to \bE_\Kk$.
Further, the footprint maps $\psi_i$ induce a surjection $\psi_\Kk: \s_\Kk^{-1}(0)\to \oMm$ from the zero set onto the moduli space. This induces natural morphisms in the subcategory $\s_\Kk^{-1}(0)$, given by
$$
\psi_j^{-1} \circ \psi_i \, : \; s_i^{-1}(0)\cap \psi_i^{-1}(F_i\cap F_j) \;\longrightarrow\; s_j^{-1}(0) .
$$
If we use only these morphisms and their lifts to $s_i^{-1}(0)\times\{0\}\subset U_i\times E_i$, then composition in the categories $\bB_\Kk, \bE_\Kk$ is well defined, $\pi_\Kk, \s_\Kk, \psi_\Kk$ are functors, and $|\psi_\Kk|:|s_\Kk^{-1}(0)|\to \oMm$ is a homeomorphism, which identifies the unperturbed moduli space $\oMm$ with a subset of the realization $|\bB_\Kk|$. However, these morphism spaces may be highly singular, so that the structure maps are merely local homeomorphisms between 
singular subsets of the object spaces and hence merely topological morphism spaces.
This structure is insufficient for any regularization approach: Patching of homology classes at least requires continuous transition maps defined on open subsets, patching of perturbed zero sets requires morphisms between objects outside of the zero sets, and a well defined notion of transverse perturbation of the section $\s_\Kk$ requires differentiable transition maps.
Hence any notion of Kuranishi data requires an extension of the transition maps $\psi_j^{-1} \circ \psi_i$ to diffeomorphisms between submanifolds of the domains. 

Recall here that the domains of the charts $U_i$ may not have the same dimension, since one can only expect the Kuranishi charts $\bK_i $ to have constant index $d= \dim U_i-\dim E_i$. 
Hence transition data is generally given by ``transition charts'' $\bK^{ij}$ and coordinate changes  $\bK_i \to \bK^{ij} \leftarrow \bK_j$ which in particular involve embeddings from open subsets $U_i^{ij}\subset U_i, U_j^{ij}\subset U_j$ into $U^{ij}$. (Here we simplify the notion from \S\ref{ss:kur} by assuming that a single transition chart covers the overlap $F_i\cap F_j$.)
Now one could appeal to Sard's theorem to find a transverse perturbation in each basic chart,
$$
\nu=(\nu_i:U_i\to E_i)_{i=1,\ldots,N} \qquad\text{with}\quad s_i+\nu_i \pitchfork 0  \quad\forall i=1,\ldots,N ,
$$
and use this to regularize $\oMm \cong |\s_\Kk^{-1}(0)|=\qu{\s_\Kk^{-1}(0)}{\sim}$. Here the relation $\sim$ is given by morphisms, so the regularization ought to be the realization $|(\s_\Kk+\nu)^{-1}(0)| = \qu{(\s_\Kk + \nu)^{-1}(0)}{\sim}$ of a subcategory. 
Hence the perturbations $\nu_i$ need to be compatible with the morphisms, i.e.\ transition maps.
Given such compatible transverse perturbations, one obtains the charts and transition data for an abstract manifold as in Definition~\ref{def:mfd}, but still needs to verify the cocycle condition, Hausdorffness, compactness, and an invariance property to obtain a generalization of the finite dimensional regularization theorem on page \pageref{finite reg} to Kuranishi  
data 
along the following lines.

\medskip
\noindent
{\bf Kuranishi Regularization:} {\it
Let $\s_\Kk: \bB_\Kk\to\bE_\Kk$ be a Kuranishi section of index~$d$ such that $|\s_\Kk^{-1}(0)|$ is compact. Then there exists a class of smooth perturbation functors $\nu:\bB_\Kk\to\bE_\Kk$ such that the subcategory $\bZ^\nu:=(\s_\Kk+\nu)^{-1}(0)$ carries the structure of an abstract compact smooth manifold of dimension $d$ in the sense of Definition~\ref{def:mfd}.
Moreover, $[\, |\bZ^\nu | \,]\in H_d(\bB_\Kk,\Z)$ is independent of the choice of such perturbations.
}

\medskip

However, in general there is no theorem of this precise form, since the topological issues discussed below require various refinements of the 
setup and construction. 
In order to expose these issues, we need to make the notion of Kuranishi data or Kuranishi section more precise, which involves a number of choices.\footnote{
In each step -- e.g.\ when formalizing the compatibility conditions, aiming to achieve Hausdorffness, or trying to control compactness -- there are many options of mathematically rigorous structure (or property or construction) that one may choose to formalize the Kuranishi regularization intuition. In our experience, any one issue is naturally resolved by one or another simple structure, but the multitude of issues requires a multitude of structures which are impossible to achieve simultaneously in practice.
Moreover, leaving any room for interpretation in the formalization of a structure will most certainly tempt one to use contradictory interpretations later on in the construction (see for example the multitude of topologies that -- after many false identifications -- we are forced to deal with separately in \S\ref{ss:Ksdef} and \S\ref{ss:tame}). Thus the challenge here is not to resolve a single issue but to simultaneously resolve all issues in the regularization construction within a single mathematically coherent framework that moreover applies to moduli spaces in practice.
}
But while there are many possible choices of formal Kuranishi notions, the topological issues (or at least the resulting requirements on the structure) remain universal. The following chooses our approach of perturbative regularization in a Kuranishi atlas to explain the issues, but can also be read as explaining the motivations for our choices of formal structure.
It does in fact develop somewhat naturally once one embraces the categorical formulation -- which was first introduced in our work \cite{MW0} but at this point is very widely adopted.\footnote{A failure to fit into categorical language in fact usually points to algebraic issues as explained in \S\ref{ss:alg}.}

\MS\NI
{\bf Compatibility:}
In order to obtain well defined transition maps, i.e.\ a space of morphisms in $\bZ^\nu$ with well defined composition, the perturbations $\nu_i$ clearly need to be compatible.
Since $(s_i + \nu_i)^{-1}(0)$ and $(s_j + \nu_j)^{-1}(0)$ are not naturally identified via $\psi_j^{-1} \circ \psi_i$ for $\nu\not\equiv 0$, this requires that one include in $\bB_\Kk$  the choice of specific transition data between the basic charts. 
Next, since the intersection of these embeddings $\phi_i^{ij}:U_i^{ij} \to U^{ij}$ and $\phi_j^{ij}:U_j^{ij} \to U^{ij}$ in the ``transition chart'' $\bK^{ij}$ is not controlled, the direct transition map
$(\phi_j^{ij})^{-1} \circ \phi_i^{ij}$ may not have a smooth extension that is defined on an open set.  
Therefore, we do not want to consider such maps to be morphisms in $\bB_\Kk$ since that would violate the \'etale property.
Instead, we include $\bK^{ij}$ into the set of charts, and ask that the pushforward of each perturbation $\nu_i,\nu_j$ extend to to a perturbation $\nu^{ij}$.  But now one must consider triple 
overlaps,
and so on.

The upshot is that, as well as the system of basic charts $(\bK_i)_{i=1,\ldots,N}$ with footprints $F_i$,  one is led to consider a full collection of transition charts $(\bK_I)_{I\subset\{1,\ldots,N\}}$ with footprints $F_I: = \cap _{i\in I} F_i$.
To make a category, each of these ``sum charts" should have a chosen domain $U_I$, which is a smooth manifold, and the objects in $\bB_\Kk$ should be $\sqcup_I U_I$.  Further
the morphisms in the category should come from coordinate changes 
between these charts, which in particular involve embeddings $\phi_{IJ}:U_{IJ} \to U_J$ of open subsets $U_{IJ}\subset U_I$. 
Thus the space  $\Mor_{\bB_\Kk}$ will be the disjoint union of the domains $U_{IJ}$ of these coordinate changes over all relevant pairs $I,J$.\footnote{
Detailed definitions of the categories $\bB_\Kk$ and $\bE_\Kk$ along these lines can be found in Definition~\ref{def:catKu}.
It is worth noting here that we do not require the domains of charts to be open subsets of Euclidean space.  We could achieve this for the basic charts, since these can be arbitrarily small. However, for the transition charts one may need to make a choice between having a single sum chart for each overlap and having sum charts whose domains are topologically trivial.
We construct the former in Theorem~\ref{thm:A2}.
}

For this to form a category, all composites must exist, which is equivalent to the    
cocycle condition $\phi_{JK}\circ\phi_{IJ}=\phi_{IK}$ including the condition that domains be chosen such that $\phi_{IJ}^{-1}(U_{JK}) \subset U_{IK}$.
However, natural constructions as in \S\ref{ss:gw} only satisfy the cocycle condition on the overlap $\phi_{IJ}^{-1}(U_{JK}) \cap U_{IK}$.
Thus already the construction of an equivalence relation from the transition data requires a refinement of the choice of domains.
We resolve this issue by a shrinking procedure that iteratively chooses subsets of each $U_I$ and $U_{IJ}$. It is developed in \cite{MW:top} to be universally applicable to Kuranishi-type settings, and in our specific setting is summarized in Theorem~\ref{thm:K}.
If we assume the cocycle condition, then $\bB_\Kk$ satisfies all properties of a nonsingular groupoid except 
\begin{itemize}
\item[-]  we do not assume that inverses exist;
\item[-] the \'etale condition is relaxed to require that the structure maps are smooth embeddings rather than diffeomorphisms.
\end{itemize}
We write $|\Kk|$ for the realization $\Obj_{\bB_\Kk}/\!\!\sim$ of $\bB_\Kk$, where $\sim$ is the equivalence relation generated by the morphisms, and denote by $\pi_\Kk:\Obj_{\bB_\Kk}\to|\Kk|$ 
the projection.
We show in Lemma~\ref{le:realization}~(iv) that the natural inclusion $\io_\Kk: \oMm\to |\Kk|$ is a homeomorphism to its image. 
Therefore we may think of $|\Kk|$ as a {\it virtual neighbourhood} of $\oMm$.

\MS\NI
{\bf Hausdorff property:}  For a category such as  $\bB_\Kk$, it is no longer true that
the properness of $s\times t$ implies that its realization $|\Kk|$ is Hausdorff; cf. \cite[Example~2.4.4]{MW:top}. Therefore, the easiest way to ensure that  the realization of the perturbed zero set $|\bZ^\nu|$ 
is Hausdorff is to make $|\Kk|$ Hausdorff and check that the inclusion $|\bZ^\nu|\to |\Kk|$ is continuous.
This Hausdorff property (or more general properness conditions in the case of nontrivial isotropy) 
was not addressed in the literature \cite{FO,FOOO,J1}.
The resolution by building a Hausdorff ambient space was first introduced in our work \cite{MW0} but has also been universally adopted since, e.g.\ in \cite{FOOO12}. 
We have thus formulated our resolution of this issue in a universally applicable Kuranishi setting in \cite{MW:top}, which shows that our shrinking procedure can achieve a notion of {\it tameness} (see Definition~\ref{def:tame}), which is a very strong form of the cocycle condition that controls the morphisms in $\bB_\Kk$. 
In the present context, Theorem~\ref{thm:K} shows that these tame shrinkings exist if the original Kuranishi atlas is additive, that is if the obstruction spaces $E_i$ of the basic charts are suitably transverse. 
Moreover, the realization of a tame Kuranishi atlas $\Kk$ is not only Hausdorff but also has the property that the natural maps $U_I\to |\Kk|$ are homeomorphisms to their image. This means that we can construct a perturbation over $|\Kk|$ by working with its pullbacks to each chart.
A further shrinking construction moreover guarantees the existence of a suitable metric on $|\Kk|$, which is also crucial in the construction of perturbations.

\MS\NI
{\bf Compactness:}  Unfortunately, even when we can make $|\Kk|$ Hausdorff, it is almost never locally compact or metrizable.  In fact, a typical local model is the subset $S$ of $\R^2$ formed by the union of the line $y=0$ with the half plane  $x>0$, with $\io_\Kk(\oMm) = \{y=0\}$, but given the  topology as a quotient of the disjoint union $\{y=0\}\sqcup \{x>0\}$. As we show in Example~\ref{ex:Khomeo}, the quotient topology on $S$ is not metrizable, and even in the weaker subspace topology from $\R^2$ the zero set $\io_\Kk(\oMm)$ does not have a locally compact  neighbourhood in $|\Kk|$.
Therefore ``sufficiently small'' perturbations $\nu$ cannot guarantee compactness of the perturbed zero set $|\bZ^\nu|$.
Instead, the challenge is to find subsets of $|\Kk|$ containing $\io_\Kk(\oMm)$ that are compact 
and -- while not open -- are still large enough to contain the zero sets of appropriately perturbed sections $\s_\Kk+\nu$.

Similar to the Hausdorff property, this compactness was asserted in the classical literature \cite{FO,FOOO,J1} by quoting an analogy to the construction of an Euler class of orbibundles -- in which both issues are resolved by a locally compact ambient space.
Again, our work \cite{MW0} was first to point out the lack of local compactness and offer a resolution by means of containing the perturbed zero set in a precompact subset of the virtual neighbourhood.
This approach is also adopted across the more recent publications and formulated in a universally applicable Kuranishi setting in \cite{MW:top}.
In the present context, Theorem~\ref{prop:zeroS0} shows that sequential compactness can be guaranteed for  perturbed zero sets that are contained $\pi_\Kk\bigl((s+\nu)^{-1}(0)\bigr)\subset \pi_\Kk(\Cc)$ in the image of a precompactly nested pair of ``reductions'' $\Cc\sqsubset \Vv\sqsubset \Obj(\bB_\Kk)$ (which are introduced below).
One can think of $\pi_\Kk(\Cc)$ as a kind of neighbourhood of $\io_\Kk(\oMm)$, but, even though $\Cc$ is an open subset of $\Obj_{\bB_\Kk}$, the image $\pi_\Kk(\Cc)$  is not open in $|\Kk|$ because the different components of $\Obj_{\bB_\Kk}$ have different dimensions. For example, if $|\Kk|=S$ as above then $\pi_\Kk(\Cc)$ could be the union of $\{y=0, x<2\}$ with $\{x>1, |y|<1\}$. 
(In this example, since $\oMm$ is not compact, we cannot expect $\pi_\Kk(\Cc)$ to be precompact, but its closure is locally compact.)
 
\MS\NI
{\bf Construction of 
perturbations:} We aim to construct the perturbation $\nu$ by finding
a compatible family of local perturbations $\nu_I$ in each chart $\bK_I$. 
Thus, if basic charts $\bK_i$ and $\bK_j$ have nontrivial overlap (i.e. their footprints intersect) and we start by defining $\nu_i$, the most naive approach is to try to  extend the partially defined perturbation
$\nu_i\circ (\phi_i^{ij})^{-1} \circ \phi_j^{ij}$  over $U_j$.
But, as seen above, the image of $(\phi_j^{ij})^{-1} \circ \phi_i^{ij}$ might be too singular to allow for an extension, and since $\phi_i^{ij}$ and $\phi_j^{ij}$ have overlapping images in $U^{ij}$ it does not help to rephrase this in terms of finding an extension of the pushforwards of these sections to $U^{ij}$.
Thus one needs some notion of a ``good coordinate system" on $\oMm$, 
in which all compatibility conditions between the perturbations are given by pushforwards with embeddings. That is, two charts $\bK_I$ and $\bK_J$ either have no overlap or
there is a direct coordinate change $\phi_{IJ}$ or $\phi_{JI}$ from one to the other.
Thus two overlapping charts must have order comparable indices: either $I\subset J$ or $J\subset I$.
A first notion of ``good coordinate system'' was introduced in \cite{FO,FOOO}, with refinements in \cite{FOOO12}; see \S\ref{ss:alg} and Remark~\ref{rmk:otherK}~(v).
We achieve the required ordering -- again in \cite{MW:top} for a universally applicable Kuranishi setting -- by constructing a {\it reduction} $\Vv\subset \Obj_{\bB_\Kk}$ as stated in Theorem~\ref{thm:red}.
This does not provide another Kuranishi atlas or collection of compatible charts but merely is a subset of the domain spaces that covers the unperturbed moduli space $\pi_\Kk\bigl(\Vv\cap s_\Kk^{-1}(0)\bigr) = \io_\Kk( \oMm)$
and whose parts project to disjoint subsets $\pi_\Kk(\Vv\cap U_I)\cap\pi_\Kk(\Vv\cap U_J)=\emptyset$ in the virtual neighbourhood $|\Kk|$ unless there is a direct coordinate change between $\bK_I$ and $\bK_J$. 
Since the unidirectional coordinate changes induce an ordering, this allows for an iterative approach to constructing compatible perturbations $\nu_I$.
However, this construction in Proposition~\ref{prop:ext} is still very delicate and requires great control over the perturbation $\nu$ since, to ensure compactness of the zero set, 
we must construct it so that 
the perturbed zero set $\pi_\Kk\bigl((s+\nu)^{-1}(0)\bigr)$ is contained in a precompact but generally not open set $\pi_\Kk(\Cc)$ that arises from a precompact subset $\Cc\sqsubset \Vv$.
In particular, this construction requires a suitable metric on $\pi_\Kk(\Vv)$, cf.\ Definition~\ref{def:metric}, 
which raises the additional difficulty of working with different topologies since -- as explained above -- 
the natural quotient topologies are almost never metrizable.

\MS\NI{\bf Regularity of  
perturbations:}
In order to deduce the existence of transverse perturbations in a single chart from Sard's theorem, the section must be $\Cc^k$, where $k\ge 1$ is larger than the index of the Kuranishi atlas.
(This was also overlooked in \cite{FO} and \cite{J1}.)
For applications to pseudoholomorphic curve moduli spaces this means that either a refined gluing theorem with controls of the derivatives must be proven, 
or a perturbation theory in stratified smooth Kuranishi data is needed.

Moreover, when extending a transverse section $(\phi_{IJ})_*(s_I+\nu_I)$ from the image of the embedding $\phi_{IJ}$ to the rest of $U_J$, we must control its behavior in directions normal to the submanifold $\im \phi_{IJ}$ so that zeros of $s_I+\nu_I:U_I\to E_I$ in $U_{IJ}$ correspond to transverse zeros of $s_J+\nu_J:U_J\to E_J$.
That is, the derivative $\rd (s_J+\nu_J)$ must induce a surjective map from the normal bundle of $\phi_{IJ}(U_{IJ})\subset U_J$ to $E_J/\Hat\phi_{IJ}(E_I)$.
If this is to be satisfied at the intersection of several embeddings to $U_J$, then the construction of transverse sections necessitates a tangent bundle condition, which was introduced in \cite{J1} and then adopted in \cite{FOOO}.
We reformulate it as an {\it index condition} relating the kernel and cokernels of $\rd s_I$ and $\rd s_J$ and can then extend transverse perturbations by requiring 
vanishing derivatives
in the normal directions to all embeddings $\phi_{IJ}$
(see Definition~\ref{def:sect}).

\MS\NI{\bf Uniqueness  up to cobordism:}  
Another crucial requirement on the perturbation constructions is that the resulting manifold (in the case of trivial isotropy) is unique modulo cobordism. This requires considerable effort since it does not just pertain to nearby sections of one bundle, but to sections constructed with respect to different metrics in different shrinkings and reductions of the 
Kuranishi atlas.
Finally, in applications to pseudoholomorphic curve moduli spaces, a notion of equivalence between different Kuranishi atlases is needed. Contrary to the finite dimensional charts for 
regular 
moduli spaces, or the Banach manifold charts for ambient spaces of maps, two Kuranishi charts for the same moduli space may not be directly compatible. 
Instead, we introduce in \S\ref{ss:Kcobord} a notion of {\it commensurability} by a common extension. 
In the application to Gromov-Witten moduli spaces \cite{MW:GW}, we expect to obtain this equivalence from an infinite dimensional index condition relating the linearized Kuranishi section to the linearized Cauchy-Riemann operator.
Any construction method for Kuranishi atlases on a Gromov-Witten moduli space 
(as outlined in \S\ref{s:construct}) should yield atlases that are commensurate in this sense, though 
this does not seem to be discussed in the literature yet.

In order to prove invariance of the abstract VFC construction, however, it suffices to work with a weaker notion of concordance for Kuranishi atlases, which is defined by considering atlases over the product $[0,1]\times \oMm$.
We develop the requisite general cobordism theory in \S\ref{ss:Kcobord}, and then need to generalize the shrinking, reduction, and perturbation constructions to a relative setting, interpolating between fixed data for $\{0,1\}\times \oMm$. 
Again the general categorical setting -- rather than a base manifold with boundary -- 
and the complexity of interrelated choices
introduce unanticipated subtleties into these constructions,
despite previous literature deeming a theory with boundary (and in fact corners, which we do not address) a straightforward extension.

\MS
\NI {\bf Orientability:}
The dimension condition $\dim U_I-\dim E_I = \dim U_J-\dim E_J=:d$ together with the fact that each $s_I+\nu_I$ is transverse to $0$ implies that
the zero sets  of $s_I+\nu_I$ and $s_J+\nu_J$ both have dimension $d$, so that the embedding
$\phi_{IJ}$ does restrict to a local diffeomorphism between these local zero sets.
Thus, if all the above conditions hold, then the zero sets $(s_I+\nu_I)^{-1}(0)$ and morphisms induced by the coordinate changes do form an \'etale proper groupoid $\bZ^\nu$.
To give its realization $|\bZ^\nu|$ a well defined fundamental cycle, it remains to orient the local zero sets compatibly, i.e.\ to pick compatible nonvanishing sections of the determinant line bundles $\La^d\bigl(\ker \rd(s_I+\nu_I)\bigr)$.
These should be induced from a notion of orientation of the Kuranishi atlas, i.e.\ of sections of the unperturbed determinant line bundles $\La^{\rm max} \ker \rd s_I \otimes \bigl(\La^{\rm max} \coker \rd s_I\bigr)^*$, which are compatible with fiberwise isomorphisms induced by the embeddings $\phi_{IJ}: U_{IJ}\to U_J$.

To construct this {\it determinant line bundle} $\det(\s_\Kk)$ of the Kuranishi atlas in Proposition~\ref{prop:det0}, we have to compare trivializations of determinant line bundles that arise from stabilizations by trivial bundles of different dimension.  
As recently pointed out in \cite{Z3}, there are several ways to choose local trivializations that are compatible with all necessary structure maps. We use one that is different from both the original and the revised construction in \cite[Theorem~A.2.2]{MS}, since these lead to sign incompatibilities.
Our construction, though, does coincide with ordering conventions in the construction of a canonical K-theory class on the space of linear operators between fixed finite dimensional spaces.\footnote{Thanks to Thomas Kragh for illuminating discussions on the topic of determinant bundles.
}
Finally, we use intermediate determinant bundles $\La^{\rm max} \rT U_I \otimes \bigl(\La^{\rm max} E_I\bigr)^*$ in Proposition~\ref{prop:orient1} to transfer an orientation of $\det(\s_\Kk)$ to~$\bZ^\nu$.

We develop this theory from scratch since previous treatments were at best sketches based on an assumption of structure maps being compatible with any kind of canonical conventions, and which initially even overlooked the crucial tangent bundle condition.

\MS
Putting everything together, we finally conclude in Theorem~\ref{thm:VMC1} that every
oriented weak additive $d$-dimensional Kuranishi atlas $\Kk$ with trivial isotropy determines a unique cobordism class of  oriented $d$-dimensional compact manifolds, that is represented by the zero sets of a suitable class of  admissible sections.
Theorem~\ref{thm:VMC2} interprets this result in more intrinsic terms, defining
a \v{C}ech homology class on $\oMm$, which we call the {\it virtual fundamental class} (VFC).
This is a stronger notion than in \cite{FO,FOOO}, where a virtual fundamental cycle is supposed to associate to any ``strongly continuous map'' $f:\oMm\to Y$ a cycle in $Y$. On the other hand, our notion of VFC does not yet provide a ``pull-push'' construction as needed for e.g.\ the construction of a chain level $A_\infty$ algebra in \cite{FOOO} by pullback of cycles via evaluation maps $\ev_1,\ldots,\ev_k:\oMm\to L$ and pushforward by another evaluation $\ev_0:\oMm\to L$;
for recent work on this see \cite{FOOO15}.

Summarizing, our definition of a Kuranishi atlas is designed to make it possible both to construct them in applications, such as Gromov--Witten moduli spaces, and to prove that they have natural virtual fundamental cycles.  
Since we were forced to make essential changes to almost all global notions and constructions, we compare our notion of Kuranishi atlas to the various notions of Kuranishi structures in Remark~\ref{rmk:otherK}.
We cannot directly compare these notions with the approach in \cite{LT} for lack of a concrete comparable structure, as explained in the final remark below.

\begin{rmk}\rm  
Our paper makes rather few references to 
\cite{LT} because that deals mostly with gluing and isotropy; in other respects it is very sketchy.  For example, it does not mention any of the analytic details in \S\ref{s:construct}. Its Theorem~1.1  constructs the oriented Euler class of a ``generalized Fredholm bundle" $[s:\Tilde\Bb\to \Tilde\Ee]$, avoiding the Hausdorff question
 by assuming that there is a global finite dimensional bundle $\Tilde\Bb\times F$ that maps onto the local approximations.  However, in the Gromov--Witten situation this is essentially never the case (even if there is no gluing) since $\Tilde\Bb$ is a quotient of the form $\Hat\Bb/G$.  Therefore, we must work in the situation described in \cite[Remark~3]{LT}, and here they just say that the extension to this case is easy, without further comment.
Also, the proof that the structure described in Remark 3 actually exists even in the simple Gromov--Witten case that we consider in \S\ref{s:construct} lacks almost all detail, c.f.\ the proof of Proposition 2.2. 
Their idea is to build a global object from a covering family of basic charts using sum charts (see condition (iv) at the beginning of \S1) and partitions of unity to extend sections. 
The paper \cite{LiuT} explains this idea with much more clarity, but unfortunately, because it does not pass to  finite dimensional reductions, it makes serious analytic errors which we explain in \S\ref{ss:LTBS}.
There are also serious difficulties with using partitions of unity in this context that cannot be easily circumvented by passing to finite dimensional reductions. Therefore at present it is unclear to us whether this construction can be correctly carried out.
$\hfill\er$
\end{rmk}

%NOTE TO SELF:
%
% Kai says it might be interesting to compare with
%
%Whitney, H.; Bruhat, F.
%Quelques propri�t�s fondamentales des ensembles analytiques-r�els. (French)
%Comment. Math. Helv. 33 1959 132�160.
%

%%%%%%%%%%%%%%%%%%%%%%%%%%%%%%%%%%%%%%%%%%%%%%
\section{Differentiability issues in abstract regularization approaches}
\label{s:diff}
%%%%%%%%%%%%%%%%%%%%%%%%%%%%%%%%%%%%%%%%%%%%%%

Any abstract regularization procedure for holomorphic curve moduli spaces needs to deal with the fundamental analytic difficulty of the reparametrization action, which has been often overlooked in symplectic topology. We thus explain in \S\ref{ss:nodiff} the relevant differentiability issues in the example of spherical curves with unstable domain.
In a nutshell, the reparametrization $f\mapsto f\circ\ga$ with a fixed diffeomorphism $\ga$ is smooth on infinite dimensional function spaces, but the action $(\ga,f)\mapsto f\circ\ga$ of any nondiscrete family of diffeomorphisms fails even to be differentiable in any standard Banach space topology.
In geometric regularization techniques, this difficulty is overcome by regularizing the space of parametrized holomorphic maps in such a way that it remains invariant under reparametrizations. Then the reparametrization action only needs to be considered on a finite dimensional manifold, where it is smooth. It has been the common understanding that by stabilizing the domain or working in finite dimensional reductions one can overcome this differentiability failure in more general situations.
We will explain in \S\ref{ss:DMdiff} that reparametrizations nevertheless need to be dealt with in establishing compatibility of constructions in local slices, in particular between charts near nodal curves and local slices of regular curves.
In particular, we will show the difficulties in the
global obstruction bundle approach in \S\ref{ss:nodiff}, and for the Kuranishi atlas approach will see explicitly in \S\ref{ss:Kcomp} that the action on infinite dimensional function spaces needs to be dealt with when establishing compatibility of local finite dimensional reductions.
Finally, \S\ref{ss:eval} explains additional smoothness issues in dealing with evaluation maps.

%%%%%%%%%%%%%%%%%%%%%%%%%%%%%%%%%%%%%%%%%%%%%%%%%%%%%%%
\subsection{Differentiability issues arising from reparametrizations}  \hspace{1mm}\\ \vspace{-3mm}
\label{ss:nodiff}
%%%%%%%%%%%%%%%%%%%%%%%%%%%%%%%%%%%%%%%%%%%%%%%%%%%%%%%

The purpose of this section is to explain the implications of the fact that the action of a nondiscrete automorphism group ${\rm Aut}(\Si)$ on a space of maps $\{ f: \Si \to M \}$ by reparametrization is not continuously differentiable in any known Banach metric.
In particular,
the space
$$
\{ f: \Si \to M \,|\, f_*[\Si]\neq 0 \}/{\rm Aut}(\Si),
$$
of equivalence classes of (nonconstant) smooth maps from a fixed domain modulo repara\-metrization of the domain, has no known completion with differentiable Banach orbifold structure.
We discuss the issue in the concrete case of the moduli space $\oMm_{1}(A,J)$ of $J$-holomorphic spheres with one marked point.\footnote{
In order to understand how any given abstract regularization technique deals with the differentiability issues caused by reparametrizations, one can test it on the example of spheres with one marked point. This is a realistic test case since since sphere bubbles will generally appear in any compactified moduli space (before regularization).
}
For the sake of simplicity let us assume that the nonzero class $A\in H_2(M)$ is such that it excludes bubbling and multiply covered curves a priori, so that no nodal solutions are involved and all isotropy groups are trivial. In that case one can describe the moduli space
\begin{align*}
\oMm_{1}(A,J)
&:= \bigl\{ (z_1,f) \in S^2 \times \Cc^\infty(S^2,M) \,\big|\, f_*[S^2]=A, \pbar f = 0 \bigr\} / \PSL(2,\C) \\
& \cong
\bigl\{ f \in \Cc^\infty(S^2,M) \,\big|\, f_*[S^2]=A, \pbar f = 0 \bigr\} / G_\infty
\end{align*}
as the zero set of the section $f\mapsto \pbar f$ in an appropriate bundle over the quotient
$$
\HBb/G_\infty \quad\text{of}\quad
\Hat\Bb:= \bigl\{ f \in \Cc^\infty(S^2,M) \,\big|\, f_*[S^2]=A \bigr\}$$
by the reparametrization action $f\mapsto f\circ\ga$ of
$$G_\infty : = \{\ga\in \PSL(2,\C) \,|\, \ga(\infty)=\infty\}.
$$
The quotient space $\HBb/G_\infty$ inherits the structure of a Fr\'echet manifold, but note that the action on any Sobolev completion
\begin{align} \label{action}
\Theta: G_\infty \times W^{k,p}(S^2,M) \to W^{k,p}(S^2,M), \quad
(\gamma,f) \mapsto f\circ\gamma
\end{align}
does not even have directional derivatives at maps $f_0\in W^{k,p}(S^2,M) \less W^{k+1,p}(S^2,M)$ since the differential\footnote{
Here the tangent space to the automorphism group $\rT_{\rm Id}G_\infty \subset\Ga(\rT S^2)$ is the finite dimensional space of holomorphic (and hence smooth) vector fields $X:S^2 \to \rT S^2$ that vanish at $\infty\in S^2$.
}
\begin{align}\label{eq:actiond}
{\rm D}\Theta ({\rm Id},f_0) : \;
\rT_{\rm Id}G_\infty
\times W^{k,p}(S^2, f_0^*\rT M) &\;\longrightarrow\; W^{k,p}(S^2, f_0^*\rT M) \\
 {(X,\xi)} \qquad\qquad\qquad\;\;\; &\;\longrightarrow\; \;\;   \xi + \rd f_0 \circ X \nonumber
\end{align}
is well defined only if $\rd f_0$ is of class $W^{k,p}$.
In fact, even at smooth points $f_0\in{\mathcal C}^\infty(S^2,M)$, this ``differential'' only provides directional derivatives of \eqref{action}, for which the rate of linear approximation depends noncontinuously on the direction. Hence \eqref{action} is not classically differentiable at any point.

\begin{remark} \rm  \label{BSdiff}
To the best of our knowledge, the differentiability failure of \eqref{action} persists in all other completions of $\Cc^\infty(S^2,M)$ to a Banach manifold -- e.g.\ using H\"older spaces.
The restriction of \eqref{action} to $\Cc^\infty(S^2,M)$ does have directional derivatives, and the differential is continuous in the $\Cc^\infty$ topology. Hence one could try to deal with \eqref{action} as a smooth action on a Fr\'echet manifold.
Alternatively, one could equip $\Cc^\infty(S^2,M)$ with a (noncomplete) Banach metric. Then
$$
\Theta:G_\infty \times \Cc^\infty(S^2,M) \to \Cc^\infty(S^2,M)
$$  has a bounded differential operator
$$
{\rm D}\Theta (\ga,f) : \rT_\ga G_\infty \times \Cc^\infty(S^2, f^*\rT M) \to \Cc^\infty(S^2, \ga^*f^*\rT M).
$$
 However, the differential fails to be continuous with respect to $f\in  \Cc^\infty(S^2,M)$ in the Banach metric.
Now continuous differentiability could be achieved by restricting to a submanifold $\Cc\subset\Cc^\infty(S^2,M)$ on which the map $f\mapsto \rd f$ is continuous. However, in e.g.\ a Sobolev or
H\"older metric, the identity operator  $\rT\Cc\to\rT\Cc$ would then be compact,
 so that $\Cc$ would have to be finite dimensional.

Finally, one could observe that the action \eqref{action} is in fact $\Cc^\ell$ when considered as
a map $G_\infty \times W^{k+\ell,p}(S^2,M) \to W^{k,p}(S^2,M)$. This might be useful for fixing the differentiability issues in the virtual regularization approaches with additional analytic arguments. In fact, this is essentially the definition of scale-smoothness developed in \cite{HWZ1} to deal with reparametrizations directly in the infinite dimensional setting.
$\hfill\er$
\end{remark}

It has been the common understanding that virtual regularization techniques deal with the differentiability failure of the reparametrization action by working in finite dimensional reductions, in which the action is smooth.
We will explain below for the global obstruction bundle approach, and in \S\ref{ss:Kcomp} for the Kuranishi atlas approach, that the action on infinite dimensional spaces nevertheless needs to be dealt with in establishing compatibility of the local finite dimensional reductions.
In fact, as we show in \S\ref{s:construct}, the existence of a consistent set of such finite dimensional reductions with finite isotropy groups for a Fredholm section that is equivariant under a nondifferentiable group action is highly nontrivial.
For most holomorphic curve moduli spaces, even the existence of not necessarily compatible reductions relies heavily on the fact that,
despite the differentiability failure, the action of the reparametrization groups generally do have local slices.
However, these do not result from a general slice construction for Lie group actions on a Banach manifold, but from an explicit geometric construction using transverse slicing conditions.

We now explain this construction, and subsequently show that it only defers the differentiability failure to the transition maps \eqref{transition} between different local slices.

\medskip

In order to construct local slices for the action of $G_\infty$ on a Sobolev completion of $\Hat\Bb$,
$$
\Hat\Bb^{k,p}:= \bigl\{ f \in W^{k,p}(S^2,M) \,\big|\, f_*[S^2]=A \bigr\},
$$
we will assume $(k-1)p>2$ so that $W^{k,p}(S^2)\subset \Cc^1(S^2)$.
Then any element of $\Hat\Bb^{k,p}/G_\infty$ can be represented as $[f_0]$, where the parametrization $f_0\in W^{k,p}(S^2,M)$ is chosen so that $\rd f_0(t)$ is injective for $t=0, 1 \in S^2=\C\cup\{\infty\}$.
With such a choice, a neighbourhood of $[f_0]\in \Hat\Bb^{k,p}/G_\infty$ can be parametrized by $[\exp_{f_0}(\xi)]$, for $\xi$ in a small ball in the subspace
$$
\bigl \{  \xi\in W^{k,p}(S^2,f_0^*\rT M) \ | \  \xi(t)\in \im \rd f_0(t)^\perp \; \text{for}\; t=0,1\bigr\}.
$$
Moreover, the map $\xi \mapsto [\exp_{f_0}(\xi)]$ is injective up to an action of the finite {\bf isotropy group}
$$
G_{f_0} = \{ \ga \in G_\infty \,|\, f_0\circ\ga = f_0 \} .
$$
In other words, for sufficiently small $\eps>0$, a $G_{f_0}$-quotient of
\begin{equation}\label{eq:slice}
\Bb_{f_0}:= \bigl\{ f\in \Hat\Bb^{k,p} \,\big|\, d_{W^{k,p}}(f,f_0)<\eps , f(0)\in Q_{f_0}^0 , f(1) \in Q_{f_0}^1 \bigr\}
\end{equation}
is a local slice for the action of $G_\infty$, where for some $\delta>0$
\begin{equation} \label{eq:hypsurf}
Q_{f_0}^t=\bigl\{\exp_{f_0(t)} (\xi) \,\big|\,  \xi \in \im \rd f_0(t)^\perp , |\xi|<\delta \bigr\}
\subset M
\end{equation}
 are codimension $2$ submanifolds transverse to the image of $f_0$ in two extra marked points $t=0,1$.
 For simplicity we will in the following assume that the isotropy group $G_{f_0} =\{{\rm id}\}$ is trivial, and that the submanifolds $Q_{f_0}^{t}$ can be chosen so that $f_0^{-1}(Q_{f_0}^{t})$ is unique for $t=0,1$.
Then, for sufficiently small $\eps>0$, the intersections $\im f\pitchfork Q_{f_0}^t$ are unique and transverse for all elements of $\Bb_{f_0}$. This proves that $\Bb_{f_0}$ is a {\it local slice} to the action of $G_\infty$ in the following sense.

\begin{lemma} \label{lem:slice}
For every $f_0\in\Hat\Bb^{k,p}$ such that $\rd f_0(t)$ is injective for $t=0,1$ and $G_{f_0} =\{{\rm id}\}$, there exist $\eps,\delta>0$ such that $\Bb_{f_0} \to \Hat\Bb^{k,p}/G_\infty$, $f\mapsto [f]$ is a homeomorphism to its image.
\end{lemma}

\begin{remark} \rm \label{rmk:unique}
If $f_0$ is pseudoholomorphic with closed domain, then trivial isotropy implies somewhere injectivity, see \cite[Chapter~2.5]{MS}; however this is not true for general smooth maps or other domains.
Thus to prove Lemma~\ref{lem:slice} for general $f_0$ with trivial isotropy, we must deal with the case of non-unique intersections. In that case one obtains unique transverse intersections for $f\approx f_0$ in a neighbourhood of the chosen points in $f_0^{-1}(Q_{f_0}^{t})$ and can prove the same result. 
We defer the details to \cite{MW:GW}, where we also prove an orbifold version of Lemma~\ref{lem:slice} in the case of  nontrivial isotropy.
In that case, one must define the local action of the isotropy group with some care.  However, it is always defined by a formula such as \eqref{transition}, and so in general is no more differentiable than the transition maps \eqref{transition} below.
$\hfill\er$
\end{remark}

The topological embeddings $\Bb_{f}\to \Hat\Bb^{k,p}/G_\infty$ of the local slices provide a cover of $\Hat\Bb^{k,p}/G_\infty$ by Banach manifold charts.
The transition map between two such Banach manifold
charts centered at $f_0$ and $f_1$ is given in terms of the local slices by
\begin{equation}\label{transition}
\Gamma_{f_0,f_1} :\;
\Bb_{f_0,f_1}:=
 \Bb_{f_0}\cap G_\infty\Bb_{f_1}
\; \longrightarrow \;
\Bb_{f_1}, \qquad f\longmapsto f\circ \gamma_f ,
\end{equation}
where $\gamma_f\in G_\infty$ is uniquely determined by $\gamma_f(t)\in f^{-1}(Q_{f_1}^{t})$ for $t=0,1$ by our choice of $\Bb_{f_1}$.
Here the differentiability of the map
\begin{equation} \label{gf}
W^{k,p}(S^2,M)\to G_\infty, \quad f\mapsto \gamma_f
\end{equation}
is determined by that of the intersection points  with the slicing conditions for $t=0,1$,
$$
 W^{k,p}(S^2,M)\to S^2, \quad f \mapsto f^{-1}(Q_{f_1}^{t}) .
$$
By the implicit function theorem, these maps are $\Cc^\ell$-differentiable if $k>\ell + 2/p$ such that $W^{k,p}(S^2)\subset \Cc^\ell(S^2)$.
However, the transition map also involves the action \eqref{action}, and thus is non-differentiable at some simple examples of $f\in W^{k,p}\less W^{k+1,p}$, no matter how we pick $k,p$.

\begin{lemma} \label{le:Gsmooth}
Let $B\subset \Bb_{f_0}$ be a finite dimensional submanifold of $\Bb_{f_0}$ with the $W^{k,p}$-topology, and assume that it lies in the subset of smooth maps, $B\subset\Cc^\infty(S^2,M)\cap  \Bb_{f_0}$.
Then the transition map \eqref{transition} restricts to a smooth map
$$
B \cap G_\infty\Bb_{f_1}
\; \longrightarrow \; \Bb_{f_1}, 
\qquad f \;\longmapsto\; \Gamma_{f_0,f_1}(f) = f\circ \gamma_f .
$$
\end{lemma}
\begin{proof}
Since $B$ is finite dimensional, all norms on $\rT B$ are equivalent. In particular, we equip $B$ with the $W^{k,p}$-topology in which it is a submanifold of $\Bb_{f_0}$. Then the embeddings $B\to \Cc^\ell(S^2,M)$ for all $\ell\in\N$ are continuous and hence the above discussion shows that the map $B\to G_\infty$, $f\mapsto\ga_f$ given by restriction of \eqref{gf} is smooth.
To prove smoothness of $\Gamma_{f_0,f_1}|_{B \cap G_\infty\Bb_{f_1}}$ it remains to establish smoothness of the restriction of the action $\Theta$ in \eqref{action} to 
$$
\Theta_B \,: \; G_\infty \times B \;\longrightarrow \; W^{k,p}(S^2,M), \qquad (\ga,f) \;\longmapsto\;  f\circ \ga .
$$
For that purpose first note that continuity in $f\in B$ is elementary since, after embedding $M\hookrightarrow \R^N$, this is a linear map in $f$. Continuity in $\ga$ for fixed $f\in\Cc^\infty(S^2,M)$ follows from uniform bounds on the derivatives of $f$ (and could also be extended to infinite dimensional subspaces of $W^{k,p}(S^2,M)$ by density of the smooth maps). This proves continuity of $\Theta_B$.

Generalizing \eqref{eq:actiond}, with $\rT_{\ga_0} G_\infty\subset\Ga(\ga_0^*\rT S^2)$ the space of holomorphic (and hence smooth) sections $X:S^2 \to \ga_0^*\rT S^2$ that vanish at $\infty\in S^2$, the differential of $\Theta_B$ is 
\begin{align*}
{\rm D}\Theta_B (\ga_0,f_0) : \;
\rT_{\ga_0} G_\infty \times \rT_{f_0} B &\;\longrightarrow\; W^{k,p}(S^2, f_0^*\rT M) \\
 {(X,\xi)} \qquad&\;\longrightarrow\; \;\;   \xi\circ\ga_0 + \rd f_0 \circ X .
\end{align*}
It exists and is a bounded operator at all $(\ga_0,f_0)\in G_\infty \times B$ since by assumption $f_0$ is smooth, so it remains to analyze the regularity of this operator family under variations in $G_\infty \times B$.
Denoting by $L(E,F)$ the space of bounded linear operators $E\to F$, the second term,
$$
B \;\to\; L\bigl(\rT_{\ga_0} G_\infty , W^{k,p}(S^2, f_0^*\rT M)\bigr), \quad
f_0 \;\mapsto\; (\rd f_0)_* 
\qquad\text{given by}\;   (\rd f_0)_* X = \rd f_0 \circ X ,
$$
is smooth on the finite dimensional submanifold $B$ because 
$\Cc^\ell(S^2,\R^N) \to W^{k,p}(\rT S^2,\R^N)$, $f_0\mapsto \rd f_0$ is a bounded linear map for sufficiently large $\ell$, and the $\Cc^\ell$-norm on $B\subset \Cc^\infty(S^2,\R^N)$ is equivalent to the $W^{k,p}$-norm.
The first term,
\begin{equation}\label{thetag}
G_\infty \;\to\; L\bigl(\rT_{f_0} B , W^{k,p}(S^2, f_0^*\rT M)\bigr) , \quad
\ga_0 \;\mapsto\; \theta_{\ga_0} \qquad\text{given by}\;\theta_{\ga_0}(\xi) = \xi\circ\ga_0 ,
\end{equation}
is of the same type as $\Theta_B$, hence continuous by the above arguments. 

This proves continuous differentiability of $\Theta_B$. Then continuous differentiability of the first term \eqref{thetag} follows from the same general statement about differentiability of reparametrization by $G_\infty$, and thus implies continuous differentiability of ${\rm D}\Theta_B$. Iterating this argument, we see that all derivatives of $\Theta_B$ are continuous, and hence $\Theta_B$ is smooth, as claimed.
Note however that this argument crucially depends on the finite dimensionality of $B$ to obtain continuity for the second term of ${\rm D}\Theta_B$.
\end{proof}

An important observation here is that the Cauchy--Riemann operator
$$
\pbar : \; \Hat\Bb^{k,p} \;\longrightarrow\; \Hat\Ee:= {\textstyle \bigcup_{f\in\Hat\Bb^{k,p}}} W^{k-1,p}(S^2,\Lambda^{0,1}f^*\rT M)
$$
restricts to a smooth section $\pbar:\Bb_{f_i}\to\Hat\Ee|_{\Bb_{f_i}}$ in each local slice.
The bundle map
$$
\Hat\Ga_{f_0,f_1} : \; \Hat\Ee|_{\Bb_{f_0,f_1}}  \;\longrightarrow\;  \Hat\Ee|_{\Bb_{f_1}}, \qquad
 \Hat\Ee_f  \;\ni\; \eta \;\longmapsto\; \eta \circ \rd \ga_f^{-1}  \;\in\; \Hat\Ee_{f\circ\ga_f} ,
$$
intertwines the Cauchy--Riemann operators in different local slices,
$$
\Hat\Ga_{f_0,f_1} \circ \pbar = \pbar \circ \Ga_{f_0,f_1} .
 $$
However, general perturbations of the form $\pbar + \nu : \Bb_{f_1} \to \Hat\Ee|_{\Bb_{f_1}}$, where $\nu$ is a  $\Cc^1$ section of the bundle $\Hat\Ee|_{\Bb_{f_1}}$, are {\it not} pulled back to $\Cc^1$ sections of $\Hat\Ee|_{\Bb_{f_0,f_1}}$ by $\Hat\Ga_{f_0,f_1}$ since
\begin{equation}\label{trans2}
\Hat\Ga_{f_0,f_1}^{-1}\circ \nu \circ \Ga_{f_0,f_1} : \;  f \;\longmapsto\; \nu(f\circ\ga_f) \circ \rd\ga_f^{-1}
\end{equation}
does not depend differentiably on the points $f$ in the base $\Bb_{f_0}$.
In equations \eqref{graphsp} and \eqref{graph} we  give a geometric construction of a special class of sections $\nu$ that do behave well under this pullback.

%%%%%%%%%%%%%%%%%%%%%%%%%%%%%%%%%%%%%%%%%%%%%%%%%%%%%%
\subsection{Differentiability issues in obstruction bundle regularization approaches}   \hspace{1mm}\\ \vspace{-3mm}
 \label{ss:LTBS}
%%%%%%%%%%%%%%%%%%%%%%%%%%%%%%%%%%%%%%%%%%%%%%%%%%%%%%

The lack of differentiability in \eqref{transition} and \eqref{trans2} poses a significant problem in the obstruction bundle approach to regularizing holomorphic curve moduli spaces.
This approach views the Cauchy--Riemann operator $\pbar:\Ti\Bb\to\Ti\Ee$ as a section of a topological vector bundle over an ambient space $\widetilde\Bb$ of stable $W^{k,p}$-maps modulo reparametrization, with a $W^{k,p}$-version of Gromov's topology. It requires a ``partially smooth structure'' on this space, in particular a smooth structure on each stratum. For example, the open stratum in the present Gromov--Witten example is $\Hat\Bb^{k,p}/G_\infty\subset\Ti\Bb$, for which smooth orbifold charts, isotropy actions, and transition maps are explicitly claimed in \cite[Proposition~2.15]{Lu} and implicitly claimed in \cite{LiuT}.
The latter paper does not even prove continuity of isotropy and transition maps, though an 
argument was supplied by Liu for the 2003 revision of \cite[\S6]{Mcv}.  However, continuity does not suffice to preserve the differentiability of perturbation sections in local trivializations of $\Ti\Ee\to\Ti\Bb$ under pullback by isotropy or transition maps.

Another serious problem with this approach is its use of cutoff functions to extend  sections
defined on infinite dimensional local slices such as $\Bb_{f_0}$ to other local slices.
Since these cutoff sections are still intended to give Fredholm perturbations of $\pbar$, the cutoff functions must be $\Cc^1$ and remain so under coordinate changes.  The paper \cite{LiuT} gives no details here.
A construction is given in \cite[Appendix~D]{LuT}, but this paper implicitly assumes an invariant notion of smoothness on the strata of $\Ti\Bb$, thus generally does not apply to Gromov--Witten moduli spaces.
Similarly, \cite[\S5]{CT} introduces a notion of ``Fredholm system"  that in its global form is irrelevant to most Gromov--Witten moduli spaces due to the lack of smooth structure on $\Ti\Bb$, and when localized runs into the same problems to do with smoothness of coordinate changes.
(Note also that \cite{CT} should replace the properness assumption, which implicitly assumes finite dimensions, by directly assuming compactness of the zero set.) 
These issues are now being explored in a number of recent preprints, e.g.\ \cite{Liu,CLW1},  
which also propose partial solutions in special cases (excluding nodes or isotropy), though details are hard to follow. They do seem similar to the approach in \cite{Sieb}, which may well be rigorous in this special case, but has known issues in the cases of nodes or isotropy.

Siebert also aims for a Banach orbifold structure on a space of equivalence classes of maps in \cite[Theorem 5.1]{Sieb}.
However, his notion is that of topological orbifold, i.e.\ with continuous transition maps. Indeed, his construction of local slices uses a averaged\footnote{In the case of isotropy, this averaging procedure is problematic since it may not preserve continuity.}
version of the slicing condition in \eqref{eq:slice}; thus the transition maps have the same form as \eqref{transition}, and hence fail differentiability.
However, the local smooth structures do allow for a notion of Fredholm section, which can be suitably stabilized to yield a finite dimensional reduction, which in turn yields a well defined Euler class.
Next, in order to deal with the lack of local smooth structure near a nodal curve, \cite{Sieb} notes that classical differentiability does hold in all but finitely many directions.\footnote{
The classical differentiability in all directions other than the gluing parameters seems to be a crucial fact that has recently also been used to prove the Fredholm property in applications of the polyfold approach  \cite{HWZ:GW,W:fred}.} 
The construction of a ``localized Euler class'' in \cite[Thm.1.21]{Sieb} then only requires a section whose differential varies continuously in the operator norm, even in the nondifferentiable directions. 
Unfortunately, at least in the fairly standard analytic setup of \cite{HWZ:GW,W:fred}, this is not the case.

\begin{remark} \rm
In summary, the global obstruction bundle approach requires the description of the compactified moduli space as the zero set of a single global section. This section needs to have at least local differentiability properties on a global ambient space, which is stratified by infinite dimensional quotient spaces. This is a meaningful requirement on the main stratum $\Hat\Bb^{k,p}/G_\infty$, where local smooth structures exist (just are generally not compatible). Near nodal strata, the smooth structure and differentiability requirements across strata are less clear.

One way to resolve the compatibility issue would be to use the scale calculus of polyfold theory, in which the action \eqref{action} and the coordinate changes \eqref{transition} are scale-smooth, hence $\Hat\Bb^{k,p}/G_\infty$ has the structure of a scale-Banach manifold -- the simplest nontrivial example of a polyfold. It is conceivable that the constructions of \cite{LiuT,Sieb,Mcv} can be made rigorous by replacing Banach spaces with scale-Banach spaces, smoothness with scale-smoothness, and all standard calculus results (e.g.\ chain rule and implicit function theorem) with their correlates in scale calculus.
However, the regularization constructions near nodal curves will likely also require a compatibility of strata-wise smooth structures, i.e.\ appropriate gluing analysis.
$\hfill\er$
\end{remark}

%%%%%%%%%%%%%%%%%%%%%%%%%%%%%%%%%%%%%%%%%%%%%%%%%%%%%%%%%%%%%%%%%%%%%%%%%%
\subsection{Differentiability issues in general holomorphic curve moduli spaces}  \hspace{1mm}\\ \vspace{-3mm}
\label{ss:DMdiff}
%%%%%%%%%%%%%%%%%%%%%%%%%%%%%%%%%%%%%%%%%%%%%%%%%%%%%%%%%%%%%%%%%%%%%%%%%%

The purpose of this section is to explain that the differentiability issues discussed in the previous sections pertain to any holomorphic curve moduli space for which regularization is a nontrivial question.
The only exception to the differentiability issues are compactified moduli spaces that can be expressed as subspace of tuples of maps and complex structures on a {\it fixed domain},
\begin{equation}\label{safe}
\oMm = \bigl\{ (f,j) \in \Cc^\infty(\Si,M)\times {\mathfrak C}_\Si   \,\big|\, \overline{\partial}_{j,J} f = 0 \bigr\} ,
\end{equation}
where ${\mathfrak C}_\Si$ is a compact manifold of complex structures on a fixed smooth surface $\Si$.
In particular, this does not allow one to divide out by any equivalence relation of the type
\begin{equation} \label{rep}
(f,j) \sim (f\circ\phi, \phi^* j),\qquad \forall  \phi\in{\rm Diff}(\Si).
\end{equation}
For moduli spaces of this form, regularization can be achieved by the simplest geometric approach; namely choosing a generic domain-dependent almost complex structure $J:\Si \to \Jj(M,\om)$,
 with no further quotient or compactification needed.
One rare example of this setting is the $3$-pointed spherical Gromov--Witten moduli space $\oMm_3(A,J)$  for a class $A$ which excludes bubbling by energy arguments, since the parametrization can be fixed by putting the marked points at $0,1,\infty\in \C\cup\{\infty\}\cong S^2$, thus setting $\Si=S^2$ and ${\mathfrak C}_\Si=\{i\}$ in \eqref{safe}.
A similar setup exists for tori with $1$ or disks with $3$ marked points in the absence of bubbling, but we are not aware of further meaningful examples.
Generally, the compactified holomorphic curve moduli spaces are of the form
$$
\bigl\{ (\Si,\bz,f) \,\big|\, (\Si,\bz) \in {\mathfrak R}, f: \Si \to M, \overline{\partial}_{J} f = 0 \bigr\}  / \sim
$$
with
$$
 (\Si,\bz, f) \sim (\Si' , \phi^{-1}(\bz), f\circ\phi) \qquad \forall \phi:\Si'\to\Si .
$$
Here ${\mathfrak R}$ is some space of Riemann surfaces $\Si$ with a fixed number $k\in\N_0$ of pairwise distinct marked points $\bz\in\Si^k$, which contains regular as well as broken or nodal surfaces.
In important examples (Floer differentials and one point Gromov--Witten invariants arising from disks or spheres) all domains $(\Si,\bz)$ are unstable, i.e.\ have infinite automorphism groups.
If the regular domains are stable, unless bubbling is a priori excluded, ${\mathfrak R}/\!\!\sim$ is still not a 
Deligne--Mumford space since one has to allow nodal domains $(\Si,\bz)$ with unstable components to describe sphere or disk bubbles.

On the complement of nodal surfaces, these moduli spaces have local slices of the form \eqref{safe}
with additional marked points $\bz\in \Si^k \less \Delta$.
In the case of unstable domains, the slices are constructed by stabilizing the domain with additional marked points given by intersections of the map with auxiliary hypersurfaces.
In the case of stable domains, the slices are constructed by pullback of the complex structures to a fixed domain $\Si$, or fixing some of the marked points.
In fact, stable spheres, tori, and disks have a single slice covering the interior of the Deligne--Mumford space $\{(\Si,\bz) \;\text{regular, stable}\}/\!\!\sim$ given by fixing the surface and letting all but $3$ resp.\ $1$ marked point vary.
Using such slices, the differentiability issue of reparametrizations still appears in many guises:
\begin{enumerate}
\item
The transition maps between different local slices
-- arising from different choices of fixed marked points or auxiliary hypersurfaces --
are reparametrizations by biholomorphisms that vary with the marked points or the maps.
The same holds for local slices arising from different reference surfaces, unless the two families of diffeomorphisms to the reference surface are related by a fixed diffeomorphism, and thus fit into a single slice.
\item
A local chart for ${\mathfrak R}$ near a nodal domain is constructed by gluing the components of the nodal domain to obtain regular domains. Transferring maps from the nodal domain to the nearby regular domains involves reparametrizations of the maps that vary with the gluing parameters.
\item
The transition map between a local chart near a nodal domain and a local slice of regular domains is given by varying reparametrizations. This happens because the local chart produces a family of Riemann surfaces that varies with gluing parameters, whereas the local slice has a fixed reference surface.
\item
Infinite automorphism groups act on unstable components of nodal domains.
\end{enumerate}

The geometric regularization approach deals with issues (i), (iii), and (iv) by dealing with the biholomorphisms between domains only after equivariant transversality is achieved. This is possible only in restricted geometric settings; in particular it fails unless multiply covered spheres can be excluded in (iv).
Similarly, the geometric approach deals with issue (ii) by making gluing constructions only on finite dimensional spaces of smooth solutions that are cut out transversely.
We show in 
\S\ref{s:diff} and \S\ref{s:construct}
that these issues are highly nontrivial to deal with in abstract regularization approaches.
In the polyfold approach described in \S\ref{ss:poly}, it is solved by introducing the notion of scale-smoothness for maps between scale-Banach spaces, in which the reparametrization action is smooth.
The other approaches have no systematic way of dealing with a symmetry group that acts nondifferentiably.

\begin{remark}\label{rmkLCM}\rm
One notable partial solution of the differentiability issues is the construction of Cieliebak-Mohnke \cite{CM} for genus $0$ Gromov--Witten moduli spaces in integral symplectic manifolds. 
(A related approach is proposed and discussed in \cite{Io,IoP,TZ}.)
They use a fixed set of symplectic hypersurfaces to construct a global slice to the equivalence relation \eqref{rep} on the compactified moduli space.
Almost complex structures that depend on the hypersurface intersection points then provide a rich enough set of perturbations to construct a pseudocycle as in the semipositive case, which does not require gluing analysis.
This method fits into the geometric approach as described in \S\ref{ss:geom} by working with a larger set of perturbations. It avoids gluing analysis by contenting itself with the construction of a pseudocycle that is sufficient for defining Gromov--Witten invariants.
The existence of suitable hypersurfaces is a special geometric property of the symplectic manifold and the type of curves considered. 
It has been used in other restricted geometric situations, for example \cite{CM2}.
However, as mentioned in Remark~\ref{rmk:GWmult}, it is not yet clear whether it can be extended to higher genus curves as claimed in \cite{Gerst}. 
$\hfill\er$
\end{remark}

%%%%%%%%%%%%%%%%%%%%%%%%%%%%%%%%%%%%%%%%%%%%%
\subsection{Smoothness issues arising from evaluation maps} \hspace{1mm}\\ \vspace{-3mm} \label{ss:eval}
%%%%%%%%%%%%%%%%%%%%%%%%%%%%%%%%%%%%%%%%%%%%

Another less dramatic differentiability issue in the regularization of holomorphic curve moduli spaces arises from evaluating maps at varying marked points.
This concerns evaluation maps of the form
$$
{\rm ev_i}: \;
\bigl\{ \bigl(\Si, \bz=(z_1,\ldots,z_k) ,f \bigr) \,\big|\,  \ldots \bigr\}/\!\!\sim \;\;\longrightarrow\;  M ,
\qquad
[\Si,\bz,f] \;\longmapsto\;  f(z_i)
$$
in situations when they need to be regularized while the moduli space is being constructed,
e.g.\ if they need to be transverse to submanifolds of $M$ or are involved in its definition via fiber products.
In those cases, the evaluation map needs to be included in the setup of a Fredholm section. However, on infinite dimensional function spaces its regularity depends on the Banach norm on the function space.
As a representative example, the map
\begin{equation} \label{evmap}
{\rm ev}: \; S^2 \times \Cc^\infty(S^2,M) \;\longrightarrow\; M , \qquad
(z,f) \;\longmapsto\; f(z)
\end{equation}
is $\Cc^\ell$ with respect to a Banach norm on $\Cc^\infty(S^2,M)$ only if the corresponding Banach space of functions, e.g.\ $\Cc^k(S^2)$ or $W^{k,p}(S^2)$, embeds continuously to $\Cc^\ell(S^2)$, e.g.\ if $k\geq\ell$ resp.\ $(k-\ell)p>2$. This can be seen from the explicit form of the differential
$$
{\rm D}_{(z_0,f_0)}{\rm ev}: \; T_{z_0}S^2 \times \Cc^\infty(S^2,f_0^*\rT M) \;\longrightarrow\; \rT_{f_0(z_0)}M , \qquad
(Z,\xi) \;\longmapsto\;  \rd f_0 (Z) + \xi(z_0)  ,
$$
whose regularity is ruled by the regularity of $\rd f_0$.
We will encounter this issue in the construction of a smooth domain for a Kuranishi chart in \S\ref{ss:gw}, where the evaluation maps are used to express the slicing conditions that provide local slices to the reparametrization group. There we are able to deal with the lack of smoothness of \eqref{evmap} by first constructing a ``thickened solution space'', which is a finite dimensional manifold consisting of smooth maps and marked points that do not satisfy the slicing condition yet. Then the slicing conditions can be phrased in terms of the evaluation restricted to a finite dimensional submanifold of $\Cc^\infty(S^2,M)$.
This operator is smooth, but now it is nontrivial to establish its  transversality.

\begin{lemma} \label{le:evsmooth}
Let $B\subset W^{k,p}(S^2,M)$ be a finite dimensional submanifold, and assume that it lies in the subset of smooth maps, $B\subset\Cc^\infty(S^2,M)$.
Then the evaluation map \eqref{evmap} restricts to a smooth map
$$
\ev_B \; : \; S^2 \times B
\; \longrightarrow \; M , 
\qquad (z,f) \;\longmapsto\; f(z) .
$$
\end{lemma}
\begin{proof}
We will prove this by an iteration similar to the proof of Lemma~\ref{le:Gsmooth}, with Step~$\ell$ asserting that maps of the type
\begin{equation}\label{type}
{\rm Ev} \;:\;
 \C \times \Cc^\ell(\C, \R^n )  \; \longrightarrow \; \R^n , \qquad (z,f) \;\longmapsto\; f(z) 
\end{equation}
are $\Cc^\ell$. 
In Step $0$ this proves continuity of the evaluation \eqref{evmap} on Sobolev spaces $W^{k,p}(S^2,M)$ that continuously embed to $\Cc^0(S^2,M)$.
In Step $\ell $ this proves that $\ev_B$ is $\Cc^\ell $  if 
we can check that the inclusion $B\hookrightarrow \Cc^\ell (S^2, M)$ is smooth when $B$ is equipped with the $W^{k,p}$-topology. Indeed, embedding $M\hookrightarrow\R^N$, this is the restriction of a linear map, which is bounded (and hence smooth) since $B$ is finite dimensional.
Hence to prove smoothness of $\ev_B$ it remains to perform the iteration.

Continuity in Step $0$ holds since we can estimate, given $\eps>0$,
\begin{align*}
\bigl| f(z) - f'(z') \bigr|
&\;\le\; 2 \| f - g\|_{\Cc^0} + \bigl| g(z) - g(z') \bigr| + \bigl|f(z') - f'(z')\bigr| \\
&\;\le\; \tfrac 12 \eps + \|\rd g\|_\infty |z-z'| + \|f - f'\|_{\Cc^0}  \;\le\;  \eps ,
\end{align*}
where we pick $g\in\Cc^1(\C,\R^n)$ sufficiently close to $f$, and then obtain the $\eps$-estimate for $(f',z')$ sufficiently close to $(f,z)$.

To see that Step $\ell $ implies Step $\ell +1$ we express the differential ${\rm D}_{(z_0,f_0)}\,{\rm Ev}:(Z,\xi)\mapsto \xi(z_0) + \rd_{z_0} f_0(Z)$ as sum of two operator families. The first family, 
$$
\C  \;\longrightarrow\; L\bigl( \Cc^{\ell +1}(\C, \R^n )  , \R^n \bigr) , \qquad
z_0 \;\longmapsto\;  {\rm Ev}(\cdot, z_0)  ,
$$
can be written as composition of $\C  \to L\bigl( \Cc^{\ell }(\C, \R^n )  , \R^n \bigr)$, $z_0 \mapsto {\rm Ev}(\cdot, z_0)$, which is $\Cc^\ell $ by Step $\ell $, 
with the bounded linear operator $L\bigl( \Cc^\ell (\C, \R^n )  , \R^n \bigr) \to L\bigl( \Cc^{\ell +1}(\C, \R^n )  , \R^n \bigr)$ given by restriction.
The second family,
$$
\C \times \Cc^{\ell +1}(\C, \R^n )  \;\longrightarrow\;  L\bigl( \C , \R^N \bigr) , \qquad
(z_0, f_0) \;\longmapsto\;  \rd_{z_0} f_0,
$$
can be written as composition of the linear map 
\begin{equation}\label{bugger}
\C \times \Cc^{\ell +1}(\C, \R^n )  \;\longrightarrow\; \C\times \Cc^\ell (\C,\bigl(\C, L(\C,\R^n) \bigr), \qquad 
(z_0, f_0) \;\longmapsto\;  (z_0, \rd f_0) ,
\end{equation}
which is a bounded linear operator hence smooth, and the evaluation map
$$
\C \times  \Cc^\ell (\C,\bigl(\C, L(\C,\R^n) \bigr) \;\longrightarrow\; L(\C,\R^n), \qquad
(z_0, \eta) \;\longmapsto\; \eta(z_0) ,
$$
which is of the type \eqref{type} dealt with in Step $\ell $, hence also $\Cc^\ell $ by iteration assumption.
This proves that the differential of evaluation maps of type ${\rm Ev} : \C \times \Cc^{\ell +1}(\C,\R^n) \to \R^n$  is $\Cc^\ell $, i.e.\ the maps are $\Cc^{\ell +1}$, which finishes the iteration step and hence proof of smoothness of $\ev_B$.

Again note that this argument makes crucial use of the finite dimensionality of $B$ to obtain continuity of the embeddings $B\hookrightarrow \Cc^k(S^2, M)$ for all $k\in\N$. This embedding for $k=\ell$ is used to conclude that $\ev_B$ is $\Cc^\ell$. Moreover, the iteration step from $\Cc^{\ell}$ to $\Cc^{\ell+1}$ requires an increased differentiability index $k=\ell+1$ for the embedding in order to obtain boundedness of \eqref{bugger}.
\end{proof}

%%%%%%%%%%%%%%%%%%%%%%%%%%%%%%%%%%%%%%%%%%%%%
\section{On the construction of compatible finite dimensional reductions}  \label{s:construct}
%%%%%%%%%%%%%%%%%%%%%%%%%%%%%%%%%%%%%%%%%%%%%

This section gives a general outline of the construction of a Kuranishi atlas on a given moduli space of holomorphic curves, concentrating on the issues of  dividing by the reparametrization action and making charts compatible. We thus use the example of
the Gromov--Witten moduli space $\oMm_{1}(A,J)$ of $J$-holomorphic curves of genus $0$ with one marked point, and assume that the nonzero class $A\in H_2(M)$ is such that it excludes bubbling and multiply covered curves a priori.
(For example, $A$ could be $\om$-minimal as assumed in \S\ref{ss:geom}.)
This allows us to use the framework of smooth Kuranishi atlases with trivial isotropy, that is developed in
\S\ref{s:chart}--\S\ref{s:VMC} of this paper.
We do not claim that this is a general procedure for regularizing other moduli
spaces of holomorphic curves, but it does provide a guideline for similar constructions.
In particular, the analysis explained here requires rather few changes in order to deal with nontrivial isotropy (see \cite{MW:GW,Mcn}); however, dealing with nodal curves of course requires a gluing theorem.

Recall that in this simplified setting the compactified Gomov--Witten moduli space
\begin{align*}
\oMm_{1}(A,J)
&:=  \bigl\{ (z_1=\infty, f) \in S^2\times \Cc^\infty(S^2,M) \,\big|\, f_*[S^2]=A, \pbar f = 0 \bigr\} / G_\infty
\end{align*}
is the solution space of the Cauchy--Riemann equation modulo reparametrization by
$$
G_\infty : = \{\ga\in \PSL(2,\C) \,|\, \ga(\infty)=\infty\}.
$$
We begin by discussing the construction of basic Kuranishi charts for $\oMm_{1}(A,J)$ in \S\ref{ss:Kchart}, where we find that an abstract approach runs into differentiability issues in reducing to a local slice for  the action of $G_\infty$. 
Proposition~\ref{prop:A1} shows that this problem can be overcome by using the infinite dimensional local slices that are constructed geometrically in \S\ref{ss:nodiff}.
In \S\ref{ss:Kcomp} we discuss the compatibility of a pair of basic Kuranishi charts, 
showing again that a  sum chart and coordinate changes cannot be constructed abstractly 
(e.g.\ from the given basic charts), but require specifically constructed obstruction bundles, which transfer well under the action of $G_\infty$.
Finally in \S\ref{ss:gw} we explain how to construct a Kuranishi atlas for $\oMm_1(A,J)$, going into considerable detail because previous work (such as \cite{FO,FOOO,LT}) does not clearly 
address the compatibility of Kuranishi charts.
This will prove Theorem~A in the introduction.

%%%%%%%%%%%%%%%%%%%%%%%%%%%%%%%%%%%%%%%%%%%%
\subsection{Construction of basic Kuranishi charts} \label{ss:Kchart}  \hspace{1mm}\\ \vspace{-3mm}
%%%%%%%%%%%%%%%%%%%%%%%%%%%%%%%%%%%%%%%%%%%%

The construction of basic Kuranishi charts for the Gromov--Witten moduli space $\oMm_{1}(A,J)$ requires local finite dimensional reductions of the Cauchy--Riemann operator
\begin{equation} \label{eq:dbar}
\pbar : \; \Hat\Bb^{k,p}= W^{k,p}(S^2,M) \;\longrightarrow\; \Hat\Ee:= {\textstyle \bigcup_{f\in\Hat\Bb^{k,p}}} W^{k-1,p}(S^2,\Lambda^{0,1}f^*\rT M),
\end{equation}
and simultaneously a reduction of the noncompact Lie group $G_\infty$ to a finite isotropy group; namely the trivial group in the case considered here.
We begin by giving an abstract construction of a finite dimensional reduction for an abstract equivariant Fredholm section. Note that by the previous discussion, holomorphic curve moduli spaces do not exactly fall into this abstract setting.
However, our purpose is to demonstrate 
the necessity of dealing with the reparametrization group in infinite dimensional settings.

\begin{remark}\rm
To simplify the reading of the following sections, let us explain our notational philosophy.
We use curly letters for locally noncompact spaces and roman letters for finite dimensional spaces.
We also use the hat superscript to denote spaces on which an automorphism group acts, or the slicing conditions are not (yet) applied. For example, $\Bb_{f_0} \subset \Hat \Bb^{k,p}$ is an infinite dimensional local slice in a Banach manifold $\Hat \Bb^{k,p}$ of maps,
the {\it local thickened solution space} $\Hat U$ is a finite dimensional submanifold of $\Hat \Bb$,
 and $U\subset \Hat U$ is the subset satisfying a slicing condition.

For bundles we again use curly letters if the fibers are infinite dimensional and roman letters if they are finite dimensional, with hats indicating that the base is infinite dimensional and tildes indicating that it is finite dimensional.
For example, $\Hat\Ee\to\Hat\Bb^{k,p}$ is a bundle with infinite dimensional fibers over a Banach manifold, while $\Hat E \subset \Hat\Ee|_{\Hat\Bb}$ has finite dimensional fibers $\Hat E |_f$ over points $f\in\Hat\Bb$ in an open subset $\Hat\Bb \subset \Hat\Bb^{k,p}$.
We will always write the fiber at a point as a restriction $\Ti E|_f$, since we require subscripts for other purposes.
Namely, when constructing a finite dimensional reduction near a point $f_0$, we use $f_0$ as subscript for the domains $U_{f_0}$ and restrictions of the bundles $\Ti E_{f_0}=\Hat E|_{U_{f_0}}$.
Moreover, we denote by $E_{f_0}$ a finite dimensional vector space isomorphic to the fibers
$(\Ti E_{f_0})|_f$ of $\Ti E_{f_0}$.

Finally, the symbol $\approx$  is used to mean ``sufficiently close to". Thus for $\ga\in G_\infty$, the set
$\{\ga\approx id\}$ is a neighbourhood of the identity.
$\hfill\er$
\end{remark}

\begin{lemma} \label{le:fobs}
Suppose that $\si:\Hat\Bb\to\Hat\Ee$ is a smooth Fredholm section that is equivariant under the smooth, free, proper action of a finite dimensional Lie group $G$.
For any $f\in\si^{-1}(0)$ let $E_f\subset \Hat\Ee|_f$
be a finite rank complement of $\im {\rm D}_f\si\subset \Hat\Ee|_f$,
and let $\rT_f (Gf)^\perp \subset \ker{\rm D}_f\si$ be a complement of the tangent space of the $G$-orbit inside the kernel.
There exists a smooth map $s_f: W_f \to E_f$ on a neighbourhood $W_f\subset \rT_f (Gf)^\perp$ of~$0$ and a homeomorphism $\psi_f: s_f^{-1}(0)\to \si^{-1}(0)/G$ to a neighbourhood of $[f]$.
\end{lemma}
\begin{proof}
Let $\Hat E \subset\Hat\Ee|_{\Hat \Vv}$ be the trivial extension of $E_f\subset \Hat\Ee|_f$
given by a local trivialization $\Hat\Ee|_{\Hat \Vv} \cong \Hat\Vv \times \Hat\Ee|_f$
over an open neighbourhood $\Hat\Vv\subset\Hat\Bb$ of $f$.
Then $\Pi\circ\si : \Hat\Vv \to \Hat\Ee_{\Hat\Vv}/\Hat E$
is a smooth Fredholm operator that is transverse to the zero section. Thus by the implicit function theorem the thickened solution space
$$
\Hat U_f := \{ \, g\in \Hat\Vv \,|\, \si(g)\in \Hat E \, \} \; \subset \Hat\Bb
$$
is a submanifold of finite dimension ${\rm ind}\,{\rm D}_f\si + {\rm rk}\,E_f$. In particular, for small $\Hat\Vv$, there is an exponential map $\rT_f \Hat U_f \supset \Hat W_f \to \Hat U_f$. More precisely, this is a diffeomorphism
$$
\exp_f: \; \ker {\rm D}_f \si \;\supset\; \Hat W_f \;\overset{\cong}{\longrightarrow} \; \{\, g\in \Hat\Vv \,|\, \si(g)\in \Hat E \,\} \;=\; \Hat U_f
$$
from a neighbourhood $\Hat W_f\subset \ker {\rm D}_f \si$
of $0$ with $\exp_f(0)=f$ and $\rd_0\exp_f : \ker{\rm D}_f\si \to \rT_f \Hat\Bb$ the inclusion.
Note here that we chose the minimal obstruction space $E_f$ so that
$$
\rT_f \Hat U_f \;=\;
({\rm D}_f\si)^{-1}(E_f) \;=\; \ker {\rm D}_f(\Pi\circ\si) \;=\;  \ker {\rm D}_f\si.
$$
Via this exponential map we then obtain maps
\begin{align*}
\Hat s : \;\Hat W_f &\to \exp_f^*\Hat E,
\qquad\qquad\;\, \xi\mapsto \si(\exp_f(\xi))  , \\
\Hat\psi : \;\Hat s^{-1}(0) &\to \si^{-1}(0)/G, \quad\quad \xi\mapsto [\exp_f(\xi)]
\end{align*}
such that the section $\Hat s$ is smooth and  $\Hat\psi$ is continuous with image $[\Hat\Vv\cap\si^{-1}(0)]$.
Restricting to the complement of the infinitesimal action, $W_f:= \Hat W_f \cap \rT_f (Gf)^\perp$, and
trivializing
$\exp_f^*\Hat E \cong \Hat W_f \times E_f$ we obtain a smooth map $s_f$ and a continuous map $\psi_f$,
\begin{align*}
s_f:= \Hat s|_{\rT_f (Gf)^\perp} \;\; &: \; \quad W_f \to E_f,  \\
\psi_f:= \Hat\psi_f|_{\rT_f (Gf)^\perp} &: \; s_f^{-1}(0) \to \si^{-1}(0)/G .
\end{align*}
We need to check that $\psi_f$ is injective i.e.\  that every orbit of $G$ in $\Hat W_f$ intersects $\exp_f(s_f^{-1}(0))$ at most once.
We claim that this holds for $\Hat \Vv$ sufficiently small. By contradiction suppose $s_f^{-1}(0) \ni \xi_i, \xi'_i\to 0$, $\ga_i\in G\less\{{\rm id}\}$ satisfy $\ga_i\cdot \exp_f(\xi_i)= \exp_f( \xi'_i)$.
By continuity of $\exp_f$ this implies $(\ga_i\cdot\exp_f(\xi_i),\exp_f(\xi'_i))\to (f,f)$, and properness of the action implies $\ga_i\to\ga_\infty\in G$ for a subsequence. Since the action is also free, we have $\ga_\infty={\rm id}$.
This will constitute a contradiction once we have proven that the ``local action'' $\{\ga\approx{\rm id}\} \times W_f \to \si^{-1}(0)/G$ is injective on a sufficiently small neighbourhood of $({\rm id},0)$.
So far we have only used the differentiability of the $G$-action at a fixed point $f\in\Hat\Bb$ to define
$\rT_f(Gf)$. However, the proof of injectivity of the local action as well as local surjectivity of $\psi_f$ will rely heavily on the continuous differentiability of the $G$-action $G\times\Hat\Bb \to \Hat\Bb$.
(Intuitively, the problem is that our slice is given by a condition involving
a derivative of the $G$ action at $f$, and so is well behaved only if this derivative varies continuously with $f$.)

To finish the proof of the homeomorphism property of $\psi_f$ we pick $\Hat\Vv$ sufficiently small that $\Hat U_f$ is covered by a single submanifold chart (i.e.\  a chart for $\Hat\Bb$ in a Banach space, within which $\Hat U_f$ is mapped to a finite dimensional subspace).
Then we can extend $\exp_f$ to an exponential map on the ambient space, i.e.\ a diffeomorphism
from a neighbourhood $\Hat\Ww_f\subset \rT_f\Hat\Bb$ of $\Hat W_f$,
$$
\Exp_f: \;   \Hat\Ww_f \;\overset{\cong}{\longrightarrow} \; \Hat\Vv   \qquad \text{with} \quad \Exp_f|_{\Hat W_f} = \exp_f , \quad \rd_0\Exp_f ={\rm id}_{\rT_f \Hat\Bb}.
$$
Note that the existence of such an extension at least requires continuous differentiability of the submanifold   $\Hat U_f$ and the map $\exp_f$.
Next, we also crucially use the continuous differentiability of the action $G\times\Hat\Bb\to\Hat\Bb$ to deduce that, for $\Hat \Vv$ sufficiently small, by the implicit function theorem
\begin{equation} \label{GBS}
\{ \ga \in G \,|\, \ga \approx {\rm id} \} \;\times\; \bigr(\Hat\Ww_f \cap  \rT_f (Gf)^\perp\bigl) \;\longrightarrow\; \Hat\Bb , \qquad
(\ga, \xi ) \;\longmapsto\; \ga\cdot \Exp_f(\xi)
\end{equation}
is a diffeomorphism to a neighbourhood of $f\in\Hat\Bb$.
The injectivity of \eqref{GBS} then implies that $\ga_i\cdot \exp_f(\xi_i) \neq \exp_f( \xi'_i)$ for  $\ga_i\neq{\rm id}$, which finishes the proof of injectivity of $\psi_f$.
More generally, the local diffeomorphism \eqref{GBS} implies that
$$
\Psi : \,\Hat\Ww_f \cap  \rT_f (Gf)^\perp \;\to\; \Hat\Bb/G, \qquad \xi\mapsto [\Exp_f(\xi)]
$$
is a homeomorphism to a neighbourhood $\Uu \subset  \Hat\Bb/G$ of $[f]$ (which in general is a proper subset of $[\Hat\Vv]$).
In particular, its image contains $\Hat\psi_f(\Hat s_f^{-1}(0))\cap\Uu=[\si^{-1}(0)]\cap\Uu$, and by construction
$$
\Psi \bigl( \Hat\Ww_f \cap  \rT_f (Gf)^\perp \bigr)
\;\cap\; \Hat\psi_f(\Hat s_f^{-1}(0))
\;=\; \Hat\psi_f\Bigl( \Hat\Ww_f \cap  \rT_f (Gf)^\perp \cap \Hat s_f^{-1}(0) \Bigr)
\;=\; \psi_f( s_f^{-1}(0) ) .
$$
This finally implies that the restriction $\psi_f =  \Psi|_{s_f^{-1}(0)}$ is a homeomorphism from  $s_f^{-1}(0)$ to the neighbourhood $\Uu\cap[\si^{-1}(0)] \subset \si^{-1}(0)/G$ of $[f]$, which completes the proof.
\end{proof}

\begin{remark} \rm \label{FOglue}
If the reparametrization action were smooth, then the above proof would be the construction of basic Kuranishi charts in \cite{FO} for the case of unstable domains.
However, their arguments focus on dealing with nodes in stable domains and do not address the issues arising from differentiability failure of reparametrizations, see \cite{kw:web}.

These issues can likely be overcome in several ways: We use in \S\ref{s:construct} geometrically explicit local slices and obstruction spaces, and thus explicit understanding of the transition maps. The same approach is used in \cite{FOOO12}.
Alternatively, \cite{CLW1} propose the use of more abstract slices and obstruction spaces together with explicit properties of the reparametrization action.
$\hfill\er$
\end{remark}

In contrast to the differentiability failure of the reparametrization action discussed above, note that the gauge action on spaces of connections is generally smooth.
Hence Lemma~\ref{le:fobs} applies in gauge theoretic settings, with an infinite dimensional group $G$, and abstractly provides finite dimensional reductions or Kuranishi charts for the moduli spaces.
On the other hand, the differentiability issues in the construction of Kuranishi charts (and in particular coordinate changes between them) can likely only be resolved by using a geometrically explicit local slice $\Bb_{f}\subset \Hat\Bb^{k,p}$ such as \eqref{eq:slice}.
At the time of our original publication of this work \cite{MW0}, this was mentioned in various places throughout the literature, e.g.\ \cite[Appendix]{FO}, but the following analytic details had not been presented.
\MS

In our situation, 
the arguments of Lemma~\ref{le:fobs}
do not apply because the reparametrization group does not act differentiably.  
Instead, our construction of a Kuranishi chart near $[f_0]\in\oMm_1(A,J)$ is based on the existence of the geometric slice defined in Lemma~\ref{lem:slice}, and will depend on the choice of
\begin{itemlist}
\item
a representative $f_0\in[f_0]$;
\item
hypersurfaces $Q^0:=Q_{f_0}^0,Q^1:=Q_{f_0}^1 \subset M$ as in \eqref{eq:hypsurf}, and $\eps_{f_0}>0$ inducing a local slice
$$
\Bb_{f_0}:= \bigl\{ f\in \Hat\Bb^{k,p} \,\big|\, d_{W^{k,p}}(f,f_0)<\eps_{f_0} , f(0)\in Q_{f_0}^0 , f(1) \in Q_{f_0}^1 \bigr\} \;\subset \; \Hat\Bb^{k,p};
$$
\item
an obstruction space $E_{f_0}\subset\Hat\Ee|_{f_0}$ 
that covers the cokernel of the linearization at $f_0$ of the Cauchy-Riemann section \eqref{eq:dbar}, that is 
$\im {\rm D}_{f_0}\pbar  + E_{f_0} = \Hat\Ee|_{f_0}$;
\item
an extension of $E_{f_0}$ to a trivialized finite rank obstruction bundle $\Hat\Vv_{f_0}\times E_{f_0} \cong \Hat E_{f_0}  \subset\Hat\Ee|_{\Hat\Vv_{f_0}}$ over a neighbourhood $\Hat\Vv_{f_0}\subset\Hat \Bb^{k,p}$ of the slice $\Bb_{f_0}$.
\end{itemlist}

With that we can construct the Kuranishi chart as a local finite dimensional reduction of the Cauchy--Riemann operator
$\pbar : \Bb_{f_0} \to \Hat\Ee|_{\Bb_{f_0}}$ in the slice to the action of $G_\infty$.
Note in the following that this construction requires the extension of the obstruction bundle $\Hat E_{f_0}$ to an open set of $\Hat\Bb^{k,p}$.

\begin{prop} \label{prop:A1}
For a sufficiently small slice $\Bb_{f_0}$, the subspace of generalized holomorphic maps with respect to the obstruction bundle $\Hat E_{f_0}$ is a finite dimensional manifold
\begin{equation}\label{Uf0}
U_{f_0}:=\bigl\{ f\in\Bb_{f_0} \,\big|\, \pbar f \in \Hat E_{f_0} \bigr\} .
\end{equation}
Moreover, $\Ti E_{f_0}:=\Hat E_{f_0}|_{U_{f_0}}\cong U_{f_0}\times E_{f_0}$ forms the bundle of a Kuranishi chart, whose smooth section and footprint map (a homeomorphism to a neighbourhood of $[f_0]$) are
$$
\begin{array}{rll}
\ti s_{f_0} \,:\; U_{f_0} &\to \; \Hat E_{f_0}|_{U_{f_0}}, & \quad f\mapsto \pbar f , \\
\psi_{f_0} \,: \; \ti s_{f_0}^{-1}(0) = \bigl\{ f\in \Bb_{f_0} \,\big|\, \pbar f =0\bigr\}&\to\; \oMm_{1}(A,J),& \quad
f\mapsto [f] .
\end{array}
$$
\end{prop}

\begin{proof}
%
%NOTE TO SELF
%
%One approach for proving smoothness of \eqref{Uf0} is to expresses it as zero set of the Fredholm section $\Bb_{f_0}\to \Hat\Ee / \Hat E, f\mapsto [\pbar f]$. Its linearized operator at $f_0$ is onto since $\rT_{f_0} \Hat\Bb^{k,p} = \rT_{f_0} \Bb_{f_0} \oplus \rT_{f_0} (G_\infty\cdot f_0)$ with $\rT_{f_0} (G_\infty\cdot f_0)\subset\ker {\rm D}_{f_0}$, and hence ${\rm D}_{f_0} (\rT_{f_0} \Bb_{f_0}) + \Hat E_f = \Hat\Ee_f$.
%Note that one should take care to only make this argument on the dense subset $f_0\in W^{k+1,p}(S^2,M)\subset\Hat\Bb^{k,p}$ such that $ \rT_{f_0} (G_\infty\cdot f_0) = \rd f_0 ({\rm Lie}G_\infty)\subset\rT\Hat\Bb^{k,p}$ is well defined.
%
We combine the local slice conditions and the perturbed Cauchy--Riemann equation to express $U_{f_0}$ as the zero set of
\begin{align*}
\Hat\Bb^{k,p} \;\supset\;
 \bigl\{ f \,\big|\, d_{W^{k,p}}(f,f_0)<\eps_{f_0} \bigr\}
&\;\longrightarrow\;
 \bigl( \Hat\Ee / \Hat E_{f_0}\bigr) \times
(\rT_{f_0(0)} Q^{0})^\perp\times (\rT_{f_0(1)} Q^{1})^\perp ,\\
f \quad &\longmapsto
\Bigl([\pbar f], \Pi^\perp_{Q^{0}}(f(0)), \Pi^\perp_{Q^{1}}(f(1))\Bigr),
\end{align*}
with projections $\Pi^\perp_{Q^t}$ near $f_0(t)$ along $Q^t$ to $T_{f_0(t)}(Q^t)^\perp$. 
Since the choice of $\Hat E_{f_0}$ guarantees that the linearized Cauchy-Riemann operator ${\rm D}_f\pbar$ maps onto $\Hat\Ee_f/\Hat E_{f_0}$ for $f=f_0$, and thus for nearby $f\approx f_0$, we obtain transversality of the full operator for sufficiently small $\eps_{f_0}>0$ if the linearized operator at $f_0$ maps the kernel of ${\rm D}_{f_0}\pbar$ onto the second and third factor. That is, we claim surjectivity of the map
$$
R_{f_0} : \;
\ker{\rm D}_{f_0}\pbar \;\ni\; \delta f \mapsto \bigl(\rd\Pi^\perp_{Q^{0}}(\delta f(0)), \rd\Pi^\perp_{Q^{1}}(\delta f(1))\bigr) .
$$
To check this, we can use the inclusion $\rT_{f_0} (G_\infty\cdot f_0)\subset\ker{\rm D}_{f_0}\pbar$ of a tangent space to the orbit together with the surjectivity of the infinitesimal action on two
points,
$$
\rT_{\rm id} G_\infty \; \to \; \rT_0 S^2 \times \rT_1 S^2
,\qquad \xi \; \mapsto \; \bigl(\xi(0), \xi(1) \bigr) .
$$
Combining these facts with $(\rT_{f_0(t)} Q^{t})^\perp =\im \rd f_0(t)$ we obtain transversality from
$$
R_{f_0} \bigl( \rT_{f_0} (G_\infty\cdot f_0) \bigr) \;=\;
\bigl(\rd\Pi^\perp_{Q^{0}} \times \rd\Pi^\perp_{Q^{1}}\bigr)
\bigl(\im \rd f_0(0) \times \im \rd f_0(1)\bigr) .
$$
This approach circumvents the differentiability failure of the $G_\infty$-action by working with the explicit local slice $\Bb_{f_0}$, which is analytically better behaved.
Moreover, the homeomorphism $\psi_{f_0}$ is given by restriction of the local homeomorphism $\Bb_{f_0}\to\Hat\Bb^{k,p}/G_\infty$ from Lemma~\ref{lem:slice}.
Finally, we need to find a trivialization of the obstruction bundle $\Ti E_{f_0}:=\Hat E_{f_0}|_{U_{f_0}}\cong U_{f_0}\times E_{f_0}$. For that purpose we choose $\eps_{f_0}>0$ even smaller. The effect of this on the bundle $\Ti E_{f_0}$ is a restriction to smaller neighbourhoods of $f_0$. Thus for sufficiently small $\eps_{f_0}>0$ the bundle over a smaller domain $U_{f_0}$ can be trivialized.
\end{proof}

A Kuranishi chart in the exact sense of  Definition~\ref{def:chart} can be obtained from the trivialization
$\Ti E_{f_0}\cong U_{f_0} \times E_{f_0}$.
However, to emphasize the geometric meaning of our constructions we
continue to use the notation for Kuranishi charts given in \S\ref{ss:kur} in terms of a
bundle $\Tilde E_f\to U_f$ with section $\tilde s$.

%%%%%%%%%%%%%%%%%%%%%%%%%%%%%%%%%%%%%%%%%%%%%%%%%%%%%%%%%%%%%%%%%%%%%%%%%%%%%%
\subsection{Compatibility of Kuranishi charts} \label{ss:Kcomp}  \hspace{1mm}\\ \vspace{-3mm}
%%%%%%%%%%%%%%%%%%%%%%%%%%%%%%%%%%%%%%%%%%%%%%%%%%%%%%%%%%%%%%%%%%%%%%%%%%%%%%

As in \S\ref{ss:kur} we oversimplify the formalism by saying that basic Kuranishi charts
$$
\bigl( \; \ti s_{f_i} : U_{f_i}\to \Ti E_{f_i}  \;,\;  \psi_{f_i} : \ti s_{f_i}^{-1}(0)\hookrightarrow \oMm_{1}(A,J) \;\bigr) \qquad \text{for}\; i=0,1 ,
$$
as constructed in the previous section from obstruction bundles  $\Hat E^i:=\Hat E_{f_i}$ over neighbourhoods of local slices $\Bb_{f_i}$,
are {\bf compatible} if the following transition data exists for every element in the overlap $[g_{01}]\in \im\psi_{f_0}\cap \im\psi_{f_1}\subset \oMm_{1}(A,J)$:
\begin{enumerate}
\item
a Kuranishi chart
$\quad\displaystyle
\bigl( \; \ti s_{g_{01}} : U_{g_{01}}\to \Ti E_{g_{01}}  \;,\;  \psi_{g_{01}} : \ti s_{g_{01}}^{-1}(0)\hookrightarrow \oMm_{1}(A,J) \;\bigr)
$\\
whose footprint $\im\psi_{g_{01}} \subset \im\psi_{f_0}\cap \im\psi_{f_1}$ is a neighbourhood of $[g_{01}]\in\oMm_1(A,J)$;
\item
for $i=0,1$ the transition map
arising from the footprints,
$$
\phi|_{\psi_{f_i}^{-1}(\im\psi_{g_{01}})} := \;
\psi_{g_{01}}^{-1}\circ \psi_{f_i} : \; \ti s_{f_i}^{-1}(0) \;\supset\; \psi_{f_i}^{-1}(\im\psi_{g_{01}}) \; \overset{\cong}{\longrightarrow}\; \ti s_{g_{01}}^{-1}(0)
$$
extends to a coordinate change consisting of an open neighbourhood $V_i\subset U_{f_i}$
of $\psi_{f_i}^{-1}(\im\psi_{g_{01}})$ and an embedding and linear injection in the trivialization $\Ti E_{f_i}\cong U_{f_i}\times E_{f_i}$ that intertwine the sections $\ti s_\bullet$,
$$
\phi :\; U_{f_i} \supset V_i \; \longhookrightarrow\; U_{g_{01}} , \qquad
\Hat\phi :\; E_{f_i} \; \longrightarrow\; E_{g_{01}} .
$$
\end{enumerate}

For notational convenience we will continue to construct the Kuranishi charts so that the domains have a canonical embedding $U \hookrightarrow \Bb^{k,p}/G_\infty$ (given by $f\mapsto [f]$ from a local slice $\Bb\subset \Bb^{k,p}$) which identifies the homeomorphism $\psi : s^{-1}(0) \hookrightarrow \oMm_1(A,J)$ with the identity on $\oMm_1(A,J)\subset \Bb^{k,p}/G_\infty$.
However, we will not use this ambient space for other purposes, since it has no direct generalization in the case of nodal curves.
In particular, the new domain $U_{g_{01}}$ cannot be constructed as an overlap of the domains $U_{f_i}$ since only the intersection of the possibly highly singular footprints $\im\psi_{f_0}\cap \im\psi_{f_1}\subset\oMm_1(A,J)$ has invariant meaning. Indeed, because the bundles $\Hat E^0, \Hat E^1$ may be quite different, the intersection $[U_{f_0}]\cap[U_{f_1}]\subset \Bb^{k,p}/G_\infty$ may only contain the intersection of footprints.
Moreover, the domains $U_{f_0},U_{f_1}\subset \Bb^{k,p}$ have no relation to each other beyond the fact that they are both spaces of perturbed solutions of the Cauchy--Riemann equation in a local slice.
Hence the Kuranishi chart (i) cannot be abstractly induced from the basic Kuranishi charts but needs to be constructed as another finite dimensional reduction of the Cauchy--Riemann operator.
With such a chart given, the transition map $\psi_{g_{01}}^{-1}\circ \psi_{f_i}$ between the zero sets is well defined, but its extension to a neighbourhood of $\psi_{f_i}^{-1}(\im\psi_{g_{01}})\subset \ti s_{f_i}^{-1}(0)$ in the domain $U_{f_i}$ also needs to be constructed. In fact, the need for this extension guides the construction of the chart.

For the rest of this section we will assume that the Kuranishi chart required in (i) can be constructed in the same way as the basic charts in \S\ref{ss:Kchart}, and explain which extra requirements are necessary to guarantee the existence of a coordinate change (ii).
The chart (i) will be determined by the following data:

\begin{itemize}
\item
a representative $g_{01}\in[g_{01}]$;
\item
hypersurfaces $Q_{g_{01}}^0,Q_{g_{01}}^1 \subset M$ and $\eps_{g_{01}}>0$ inducing a local slice $\Bb_{g_{01}}\subset\Hat\Bb^{k,p}$;
\item
a finite rank subspace $E_{g_{01}}\subset\Hat\Ee |_{g_{01}}$ such that $\im {\rm D}_{g_{01}}\pbar  + E_{g_{01}} = \Hat\Ee|_{g_{01}}$;
\item
an extension  to a trivialized finite rank subbundle $\Hat\Vv_{g_{01}}\times E_{g_{01}} \cong \Hat E^{01}: =  \Hat E_{g_{01}} \subset\Hat\Ee|_{\Hat\Vv_{g_{01}}}$ over a neighbourhood $\Hat\Vv_{g_{01}}\subset\Hat\Bb^{k,p}$ of $\Bb_{g_{01}}$.
\end{itemize}

\NI
The coordinate change (ii) requires the construction of the following for $i=0,1$
\begin{itemize}
\item
open neighbourhoods $V_i\subset U_{f_i}$ of $\psi_{f_i}^{-1}(\im\psi_{g_{01}})$;
\item
embeddings $\phi_i : V_i \hookrightarrow U_{g_{01}}$ and a bundle map
$\Hat\phi_i : \Ti E_{f_i}|_{V_i} \to \Ti E_{g_{01}}$
covering $\phi_i$ and constant on the fibers in a trivialization, such that
$$
\Hat\phi_i \circ \ti s_{f_i} = \ti s_{g_{01}} \circ \phi_i , \qquad  \psi_{f_i} = \psi_{g_{01}} \circ \phi_i .
$$
\end{itemize}

In the explicit construction, we have
$V_i\subset U_{f_i} \subset \Bb_{f_i}$ and $U_{g_{01}} \subset\Bb_{g_{01}}$
both identified
with subsets of $\Bb^{k,p}/G_\infty$, and in this identification the embedding $\phi_i:V_i \hookrightarrow U_{g_{01}}$ is required to restrict to the identity on $\im\psi_{g_{01}}\subset\im\psi_{f_i}$.
So the natural extension of $\phi_i$ to a neighbourhood of $\psi_{f_i}^{-1}(\im\psi_{g_{01}})\subset U_{f_i}$ should lift the identity on $\Bb^{k,p}/G_\infty$. That is, with the domains $V_i\subset U_{f_i}$ still to be determined, we fix $\phi_i$ to be the transition map \eqref{transition}
between the local slices,
$$
\phi_i:= \Ga_{f_i,g_{01}}|_{V_i} : \; V_i \to \Bb_{g_{01}} , \quad f \mapsto f\circ\ga^{01}_f  ,
f\circ\ga^{i}_f  ,
$$
where $\ga^{01}_f\in G_\infty$ is determined by $f\circ\ga^{01}_f\in\Bb_{g_{01}}$ .
Now in order for $\phi_i(V_i)$ to take values in $U_{g_{01}}$ we must have
\begin{equation}\label{givestrans}
\pbar f \in \Hat E^i|_{f} \;\Longrightarrow\; \pbar f\circ\rd\ga^{01}_f \in \Hat E^{01}|_{f\circ\ga^{01}_f} \qquad\forall \; f\in V_i.
\end{equation}
In particular for all $g\in \ti s_{g_{01}}^{-1}(0)$ we must have
\begin{equation} \label{E01 req}
\Hat E^0|_{g\circ \ga^0_g} \circ (\rd\ga^0_g)^{-1}
\;+\; \Hat E^1|_{g\circ\ga^1_g} \circ (\rd\ga^1_g)^{-1}  \;\subset\; \Hat E^{01}|_{g} ,
\end{equation}
where $\ga^i_g\in G_\infty$ is determined by
$
g\circ\ga^i_g\in\Bb_{f_i}.
$
 and $\Hat E^i\subset\Hat \Ee|_{\Hat\Vv_{f_i}}$ is the obstruction bundle extending $E_{f_i}$.
Note here that we at least have to construct $\Hat E^{01}\to\Bb_{g_{01}}$ as a smooth obstruction bundle over an infinite dimensional slice, since this induces the smooth structure on the domain
$U_{g_{01}} =\{g\in\Bb_{g_{01}} \,|\, \pbar g \in \Hat E^{01} \}$.
(In fact, the proof of Lemma~\ref{Uf0} uses the obstruction bundle over an open set in $\Hat\Bb^{k,p}$.)
However, we encounter several obstacles in constructing $\Hat E^{01}$ such that \eqref{E01 req} is satisfied near ${g_{01}\in\Bb_{g_{01}}}$.
\MS

\begin{itemlist}
\item[{\bf \qquad\, 1.)}]
The left hand side of \eqref{E01 req} involves the pullbacks of $(0,1)$-forms by the transition map 
% D I displyed this and added detail since earlier we wrote out $\Ga_{f_i,g_{01}}$
%$\Ga_{g_{01}, f_i} : \Bb_{g_{01}} \to \Bb_{f_i}$ 
$$
\Ga_{g_{01}, f_i} : \Bb_{g_{01}} \to \Bb_{f_i},\quad g\mapsto g\circ \ga_g^i
$$ 
between local slices.
In fact, it is no surprise that the reparametrizations enter crucially, since $\Hat E^0$ and $\Hat E^1$ are bundles over neighbourhoods of the local slices $\Bb_{f_0}$ and $\Bb_{f_1}$ respectively, which may have no intersection in $\Hat\Bb^{k,p}$ at all, although they do have an open intersection in the quotient $\Hat\Bb^{k,p}/G_\infty$.
Since the transition maps are not continuously differentiable, the pullback bundles
$$
\Ga_{g_{01}, f_i}^*\Hat E^i
\,:= \;{\textstyle \bigcup_{g\in\Bb_{g_{01}}}} \Hat E^i|_{g\circ\ga^i_g} \circ (\rd\ga^i_g)^{-1}
$$
will not be differentiable in general. Thus we must find a special class of obstruction bundles, on which the pullback by reparametrizations acts smoothly.
\item[{\bf \qquad\, 2.)}]
Even if the pullback bundles $\Ga_{g_{01}, f_0}^*\Hat E^0$ and $\Ga_{g_{01}, f_1}^*\Hat E^1$ are differentiable, their fibers can have wildly varying intersections over $\Bb_{g_{01}}$. Here the diameter of the local slice can be chosen arbitrarily small, but it will always be locally noncompact. So it is unclear whether there even exists a finite rank subbundle of $\Hat\Ee|_{\Bb_{g_{01}}}$ that contains both pullback bundles.
To ensure this we must assume transversality at $g_{01}$,
$$
\bigl(\Hat E^0|_{g_{01}\circ(\ga^0_{g_{01}})^{-1}} \circ\rd\ga^0_{g_{01}} \bigr)
\cap
\bigl( \Hat E^1|_{g_{01}\circ(\ga^1_{g_{01}})^{-1}} \circ\rd\ga^1_{g_{01}} \bigr)  \;=\; \{ 0 \}  .
$$
\end{itemlist}

If the requirements in 1.) and 2.) are satisfied, then the sum of obstruction bundles
\begin{align*}
\Hat E^{01} &\,:=\;
\Ga_{g_{01}, f_0}^*\Hat E^0 \oplus \Ga_{g_{01}, f_1}^*\Hat E^1\\
&\;=\;
{\textstyle \bigcup_{g\in\Bb_{g_{01}}} } \bigl\{ \nu^0\circ (\rd\ga^0_g)^{-1} + \nu^1\circ (\rd\ga^1_g)^{-1}  \,\big|\, \nu^i\in \Hat E^i|_{g\circ\ga^{01}_g} \bigr\}
\end{align*}
is a smooth, finite rank subbundle of $\Hat\Ee$ over a local slice $\Bb_{g_{01}}$ of sufficiently small diameter $\eps_{g_{01}}>0$.
Under these assumptions, the constructions of \S\ref{ss:Kchart} provide a Kuranishi chart for a neighbourhood of $[g_{01}]\in \oMm_1(A,J)$, which we also call {\bf sum chart} since it is given by a sum of obstruction bundles. Its domain and section are
$$
\ti s_{g_{01}} : \; U_{g_{01}} :=\{g\in\Bb_{g_{01}} \,|\, \pbar g \in \Hat E^{01} \}  \;\to\; \Hat E^{01} , \qquad g \mapsto \pbar g ,
$$
and the embedding $\Bb_{g_{01}}\to\Hat\Bb^{k,p}/G_\infty$ of the local slice restricts to a homeomorphism into the moduli space,
$$
\psi_{g_{01}}: \ti s_{g_{01}}^{-1}(0) \to  \oMm_1(A,J) , \qquad g\mapsto [g].
$$
Moreover, we already fixed the embeddings $\phi_i = \Ga_{f_i,g_{01}}$ and can read off from \eqref{givestrans} the corresponding embedding of obstruction bundles
$$
\Hat\phi_i:
\Hat E^i|_{f} \to \Hat E^{01}|_{f\circ\ga_f},
\qquad
\nu \mapsto  \nu\circ \rd\ga_f.
$$
Since this should be a constant linear map $E_{f_i}\to E_{g_{01}}$ in some trivialization
$\Hat E^{01}\cong U_{g_{01}}\times E^{01}_{g_{01}}$, the  trivialization map
$T^{01}(g) :\Hat E^{01}|_g \to E_{g_{01}}$ must be given by
$$
T^{01}(g) \,:\;
\sum_{i=0,1} \nu^i\circ (\rd\ga^i_g)^{-1}
\;\mapsto\;
\sum_{i=0,1} \Bigl( T^i(g_{01}\circ\ga^i_{g_{01}} ) ^{-1} \,T^i(g\circ\ga^i_g ) \; \nu^i \Bigr)
\circ (\rd\ga^i_{g_{01}})^{-1}
$$
in terms of the trivializations $T^i(f) :\Hat E^i|_f\overset{\cong}\to E_{f_i}$ of its factors.
In fact, this shows exactly what it means for the sum bundle $\Hat E^{01} = \Ga_{g_{01}, f_0}^*\Hat E^0 \oplus \Ga_{g_{01}, f_1}^*\Hat E^1$ to be smooth.

\MS

We now summarize the preceding discussion in the context of a tuple of $N$ charts  $\bigl(\bK_i = (U_{f_i},E_{f_i},s_{f_i},\psi_{f_i})\bigr)_{i=1,\ldots,N}$.
Generalizing conditions (i) and (ii) at the beginning of this section, we find that if these arise from obstruction bundles $\Hat E^i\to \Hat\Vv_{f_i}$ over neighbourhoods of local slices $\Bb_{f_i}$, the minimally necessary compatibility conditions require us to construct for every index subset $I\subset\{1,\ldots,N\}$ and every element $[g_0]\in\bigcap_{i\in I}\im\psi_i\subset\oMm_1(A,J)$ in the overlap of footprints
\begin{enumerate}
\item
 a {\bf sum chart} $\bK_{I,g_0}$ with obstruction space $E_{I,g_0} \cong
 \prod_{i\in I} E_{f_i}$, whose footprint $\im\psi_{I,g_0} \subset \bigcap_{i\in I}\im\psi_{f_i}$ is a neighbourhood of $[g_0]$;
\item
coordinate changes $\bigl(\bK_i \to \bK_{I,g_0}\bigr)_{i\in I}$ that extend the
transition maps
$\psi_{I,g_0}^{-1}\circ \psi_{f_i}$.
\end{enumerate}

\NI
The construction of a virtual fundamental class $[\oMm_1(A,J)]^{\rm vir}$ from a cover by compatible basic Kuranishi charts $\bigl(\bK_i \bigr)_{i=1,\ldots,N}$ in addition requires fixed choices of the above transition data, and further coordinate changes $\bK_{I,g_0}\to\bK_{J,h_0}$ satisfying a cocycle condition; see \S\ref{ss:top}.
The main difficulty is to ensure that the sum charts are well defined.
The details of their construction are dictated by the existence of coordinate changes from the basic charts.
This construction is so canonical that coordinate changes between different sum charts exist essentially automatically, and satisfy the weak cocycle condition.
By the discussion in the case of two charts,
the following conditions on the choice of basic Kuranishi charts $\bigl(\bK_i = (U_{f_i},E_{f_i},s_{f_i},\psi_{f_i})\bigr)_{i=1,\ldots,N}$ ensure the existence of the sum charts (i) and transition maps (ii).

\begin{itemlist}
\item[{\bf \qquad\, Sum Condition I:}]
{\it For every $i\in \{1,\ldots,N\}$ let
$T^i(f) :\Hat E^i|_f\overset{\cong}\to E_{f_i}$ be induced by  the trivialization of the obstruction bundle.
Then for every $[g_0]\in\im\psi_i \cap\bigcap_{j\neq i}\im\psi_j$ and representative $g_0$ with sufficiently small local slice $\Bb_{g_0}$, the map}
\begin{align*}
\Bb_{g_{0}} \times E_{f_i} &\;\longrightarrow\; \quad \Hat\Ee \\
 (g, \nu_i )\quad & \;\longmapsto \; \bigl( T^i(g\circ\ga^i_g ) \, \nu_i \bigr) \circ (\rd\ga^i_g)^{-1}
\end{align*}
{\it is required to be smooth, despite the differentiability failure of $g\mapsto g\circ\ga^i_g$.}

An approach for satisfying this condition will be given in the next section.

\item[{\bf \qquad\, Sum Condition II:}]
{\it For every $I\subset\{1,\ldots,N\}$ and $[g]\in\bigcap_{i\in I}\im\psi_i$ we must ensure transversality
of the vector spaces
$\Hat E^i|_{g\circ(\ga^i_{g})^{-1}} \circ\rd\ga^i_{g}  = \bigl( T^i(g\circ(\ga^i_{g})^{-1})^{-1} E_{f_i}  \bigr) \circ\rd\ga^i_{g}$ for $i\in I$. That is, their sum needs to be a direct sum,
$$
\sum_{i\in I} \Hat E^i|_{g\circ(\ga^i_{g})^{-1}} \circ\rd\ga^i_{g}  \;=\;
\bigoplus_{i\in I} \Hat E^i|_{g\circ(\ga^i_{g})^{-1}} \circ\rd\ga^i_{g}    \quad\subset\;\Hat \Ee|_g .
$$
}

This means that, no matter how the obstruction bundles are constructed for each chart, the choices for a tuple need to be made ``transverse to each other''  along the entire intersection of the footprints
before transition data can be constructed.
\end{itemlist}

%%%%%%%%%%%%%%%%%%%%%%%%%%%%%%%%%%%%%%%%%%%%%%%%%%%%%%%%%%%%%%%%%%%%%%%%%%
\subsection{Sum construction for genus zero Gromov--Witten moduli spaces}  \hspace{1mm}\\ \vspace{-3mm}  \label{ss:gw}
%%%%%%%%%%%%%%%%%%%%%%%%%%%%%%%%%%%%%%%%%%%%%%%%%%%%%%%%%%%%%%%%%%%%%%%%%

The purpose of this section is to explain  
how to construct a Kuranishi atlas for the simplest genus zero Gromov--Witten moduli space, namely one with no nodal curves or only trivial isotropy, and thus prove Theorem A in the introduction.
Our approach combines the geometric perturbations of \cite{LT} with the gluing analysis of \cite{MS} and a natural idea (suggested to us by Cliff Taubes) for dealing with the failure of differentiability in the pullback construction for obstruction bundles:
We introduce varying marked points so that the pullback by $\Gamma_{g_{01},f_i}$ no longer depends on the infinite dimensional space of maps, instead depending on a finite number of parameters.
For the sum construction of two Kuranishi charts $(U_{f_i},\ldots)_{i=0,1}$ arising from finite rank bundles $\Hat E^i\to \Hat\Vv_{f_i}$ over neighbourhoods of local slices $\Bb_{f_i}$, let us for simplicity of notation work in a slice $\Bb_{g_{01}}\subset\Bb_{f_0}$ so that $\ga^0_g\equiv {\rm id}$.
Thus we construct the domain of the sum chart as
$$
U_{g_{01}} = \bigl\{  \bigl( g , \ul{w} \bigr) \in \Bb_{g_{01}}\times (S^2)^2 \,\big|\, \pbar g \in \Hat E^0 + \Gamma_{\ul w}^* \Hat E^1 ,
 \ul{w}=(w^{0} , w^{1}) \in D_{01},
g(w^{t})\in Q_{f_1}^{t} \bigr\} .
$$
Here $D_{01}\subset (S^2)^2$ is a neighbourhood of $\ul{w}_{01}:=(w_{01}^0, w_{01}^1)$ with $w_{01}^t=g_{01}^{-1}(Q_{f_1}^t)$, and $\Gamma_{\ul w} : g\mapsto g\circ \gamma_{\ul w}$ is the reparametrization with
\begin{equation}\label{gaw}
\gamma_{\ul w}\in G_\infty  \quad \text{given by} \quad
\gamma_{\ul w}(t)=w^{t} \quad \text{for}\; t=0,1.
\end{equation}
Observe that, with varying marked points, the map $(g,w) \mapsto g(w)$ still only has the regularity of $g$, see \S\ref{ss:eval}.
So the above $U_{g_{01}}$ is not cut out by a single smooth Fredholm section.
However, we may now consider the intermediate {\it thickened solution space}
$$
\Hat U_{g_{01}} = \bigl\{  ( g , \ul{w} ) \in \Bb_{g_{01}}\times D_{01} \,\big|\, \pbar g \in \Hat E^0 + \Gamma_{\ul w}^* \Hat E^1  \bigr\} \;\subset\; \Bb_{f_0}\times (S^2)^2,
$$
where we have not yet imposed the slicing conditions at the points $\ul{w}$.
Then the domain $U_{g_{01}}={\rm ev}^{-1}(Q_{f_1}^0\times Q_{f_1}^1)$ is cut out by the slicing conditions, which use the evaluation map on the finite dimensional thickened solution space:
$$
{\rm ev} : \; \Hat U_{g_{01}} \to (S^2)^2, \qquad  ( g , w^0, w^1 ) \mapsto  ( g(w^0) , g(w^1) ) .
$$
To check that this map is transverse to $Q_{f_1}^0\times Q_{f_1}^1$ at $(g_{01}, w_{01}^0, w_{01}^1)$, note that $\{0\}\times (\rT S^2)^2$ is tangent to the thickened solution space at this point
(crucially using the fact that the solution space
$\{g \,|\, \pbar g=0\}$
is $G_\infty$-invariant so that there is an infinitesimal action at $g_{01}$).
Moreover, at every point in the local slice $g\in\Bb_{f_1}$ we have $\im \rd_t g \pitchfork \rT_{g(t)} Q_{f_1}^t$, in particular at $g_{01}\circ\ga_{\ul w_{01}}$ with $\im \rd_t (g_{01}\circ\ga_{\ul w_{01}})= \im \rd_{w_{01}^t}g_{01}$.
Moreover, the evaluation map is smooth if we can ensure that the thickened solution space $\Hat U_{g_{01}}\subset \Cc^\infty(S^2,M)\times (S^2)^2$ contains only smooth functions.
Continuing the list of conditions on the choice of summable obstruction bundles from the previous section, this adds the following regularity requirement.

\begin{itemlist}
\item[{\bf \qquad\, Sum Condition III:}]
{\it The obstruction bundles $\Hat E^i\subset\Hat\Ee|_{\Hat\Vv_{f_i}}$ need to satisfy regularity,}
$$
\pbar g \in {\textstyle \sum_i } \, \Ga_{\ul w_i}^* \Hat E^i
\;\Longrightarrow \; g \in\Cc^\infty(S^2,M) .
$$
By elliptic regularity for $\pbar$, this holds if
$\Hat E^i|_{W^{\ell,p}\cap \Hat\Vv_{f_i}}\in W^{\ell,p}\cap\Hat\Ee$ for all $\ell\in\N$, or in terms of the trivializations $T^i(f):\Hat E^i_f \to E_{f_i}$ if the elements of $E_{f_i}$ are smooth
$1$-forms in $\Hat\Ee|_{f_i}$  and
$$
f\in W^{\ell,p} \; \Longrightarrow \;\im T^i(f)\subset W^{\ell,p} .
$$
This means that sections of $\Hat E^i$ are lower order, compact perturbations for $\pbar$, i.e.\ they are $sc^+$ in the language of scale calculus \cite{HWZ1}.
\end{itemlist}

\MS\NI
Finally, we need to ensure smoothness of the thickened solution space $\Hat U_{g_{01}}$, which can be viewed as the zero set of the section
$$
\Bb_{g_{01}} \times D_{01} \;\longrightarrow \; \Hat\Ee / ( \Hat E^0 + \Gamma^* \Hat E^1 ) , \qquad
(g,\ul w) \;\longmapsto \; \pbar g .
$$
Here the form of the summed obstruction bundle,
\begin{align}\label{eq:Gastar}
\Ga^* \Hat E^1 &\;=\; { \underset{\ul w \in (S^2)^2}{\textstyle{\bigcup}}}\Gamma_{\ul w}^* \Hat E^1 \; \longrightarrow \; \Bb_{g_{01}}\times D_{01}, \\ \notag
\bigl(\Ga^* \Hat E^1 \bigr)|_{(g,\ul w)} &\;=\; \bigl\{ \nu \circ \rd \ga_{\ul w}^{-1} \,\big|\, \ga_{\ul w} (t)= w^{t} , \nu \in \Hat E^1|_{g\circ\ga_{\ul w}} \bigr\},
\end{align}
is dictated by fixing the natural embedding
$\phi_1 : U_{f_1}\cap G_\infty U_{g_{01}} \to U_{g_{01}}$
given by  $f\mapsto (f\circ\ga_f^{-1}, \ga_f(0), \ga_f(1))$, where $f\circ\ga_f^{-1}\in\Bb_{g_{01}}$.
Its inverse map is $(g,\ul w)\mapsto g\circ\ga_{\ul w}$, which maps to a neighbourhood of $\Bb_{f_1}$.
While the extension of $\Hat E^1$ to a neighbourhood of $\Bb_{f_1}\subset\Hat\Bb^{k,p}$ so far was mostly for convenience in the proof of Lemma~\ref{Uf0}, it now becomes crucial for the construction of this ``decoupled sum bundle''.
In fact, as in that lemma, we will also extend $\Hat E_{g_{01}}= \Hat E^0 + \Gamma^* \Hat E^1$ to a neighbourhood $\Hat\Vv_{g_{01}}$ of $\Bb_{g_{01}}$ to induce the smooth structure on $\Hat U_{g_{01}}$.
With this setup, Sum Condition I becomes smoothness of the map involving the trivialization
$T^1(f):\Hat E^1(f)\to E_{f_1}$,
\begin{equation}\label{wantsmooth}
\Hat\Vv_{g_{01}} \times D_{01} \times E_{f_1} \;\longrightarrow\; \Hat\Ee , \qquad
 (g, \ul w ,\nu ) \;\longmapsto \; \bigl( T^1(g\circ\ga_{\ul w} ) \, \nu \bigr) \circ \rd \ga_{\ul w}^{-1} .
\end{equation}
This still involves reparametrizations $(g,\ga_{\ul w})\mapsto g\circ \ga_{\ul w}$, which are not differentiable in any Sobolev topology on $\Hat\Vv_{g_{01}}$, since $\ul w\in D_{01}$ and thus $\ga_{\ul w}$ is allowed to vary. Thus the compatibility of Kuranishi atlases requires a very special form of the trivialization $T^1$, i.e.\ very special obstruction bundles $\Hat E^i$.

\MS\NI
{\bf Geometric construction of obstruction bundles:}
To solve the remaining differentiability issue, we now follow the more geometric approach of \cite{LT} and construct obstruction bundles by pulling back finite rank subspaces
\begin{equation}\label{graphsp}
E^i\subset \Cc^\infty(\Hom^{0,1}_J(S^2,M))
\end{equation}
of the space of smooth sections of the bundle over $S^2\times M$ of $(j,J)$-antilinear maps $\rT S^2 \to \rT M$.
Given such a subspace and a neighbourhood $\Hat\Vv_{f_i}$ of a local slice, we hope to obtain an obstruction bundle
\begin{equation}\label{graph}
\Hat E^i := \;{\textstyle \bigcup_{f\in\Hat\Vv_{f_i}} } \bigl\{ \nu|_{\gr f} \;\big|\; \nu \in E^i\bigr\} \;\subset\; \Hat\Ee|_{\Hat\Vv_{f_i}}
\end{equation}
by restriction to the graphs $\nu|_{\gr f} \in \Hat\Ee|_f = W^{k-1,p}(S^2, \Lambda^{0,1}f^* \rT M )$
given by
$$
\nu|_{\gr f} (z) = \nu(z,f(z)) \in \Hom^{0,1}_J(\rT_zS^2,\rT_{f(z)}M) .
$$
The disadvantage of this construction is that we need to assume injectivity of the map
$$
E^i\ni \nu\mapsto \nu|_{\gr f}\in \Hat\Ee|_f
$$
for each $f\in\Hat\Vv_{f_i}$ to obtain fibers of constant rank.
On the other hand, the inverse trivialization of the obstruction bundle
$$
(T^i)^{-1}:
\Hat\Vv_{f_i} \times E^i   \to \Hat E^i|_f  , \qquad (f,\nu) \mapsto \nu|_{\gr f}
$$
is now a smooth map, satisfying the regularity requirement in Sum Condition III, since
on the finite dimensional space $E^i$ consisting of smooth sections
the composition on the domain with $f\in\Hat\Vv_{f_i}\subset W^{k,p}(S^2,M)$ is smooth.
In fact, the pullback $\Ga^*\Hat E^1$ in \eqref{eq:Gastar}
now takes the special form, with $\ga_{\ul w}$ from \eqref{gaw},
\begin{equation} \label{eq:Gamaw}
(g, \ul w ,\nu ) \mapsto \ga_{\ul w}^*\nu |_{\gr g}, \qquad
\ga_{\ul w}^*\nu (z,x)
 =  \nu ( \ga_{\ul w}^{-1}(z) , x ) \circ \rd_z \ga_{\ul w}^{-1} .
\end{equation}
This eliminates composition on the domain of infinite dimensional function spaces.
Indeed, we now have
$$
\ga_{\ul w}^*\nu |_{\gr g} (z)
 =  \nu ( \ga_{\ul w}^{-1}(z) , g(z) ) \circ \rd_z \ga_{\ul w}^{-1} ,
$$
whose derivatives in the directions of $g$ and $\ul{w}$ take forms that, unlike \eqref{eq:actiond}, do not involve derivatives of $g$.
Moreover, we will later make use of the special transformation of these obstruction bundles under the action of $\ga\in G_\infty$,
\begin{equation} \label{Eequivariant}
\ga_{\ul w}^*\nu |_{\gr g} \circ \rd\ga
\;=\;  \nu \bigl( \ga_{\ul w}^{-1}\circ\ga (\cdot) , g\circ \ga (\cdot) \bigr) \circ \rd \ga_{\ul w}^{-1} \circ \rd \ga
\;=\;  (\ga^{-1}\circ\ga_{\ul w})^*\nu |_{\gr g\circ \ga} .
\end{equation}
Thus we have replaced Sum Conditions I--III, including the highly nontrivial smoothness requirement in the previous section, by the following requirement for the compatibility of the geometrically constructed obstruction bundles.\footnote
{It is possible to circumvent these conditions by ``Fredholm stabilization" as described in  \cite{MW:GW,McL,Mcn}. 
 However, if we define the elements of the charts in the current naive way, they are essential.}

\begin{itemlist}
\item[{\bf \qquad Sum Condition I$'$:}]
{\it For every $i\in \{1,\ldots,N\}$ the obstruction bundle $\Hat E^i \subset \Hat\Ee|_{\Hat\Vv_{f_i}}$ is given by \eqref{graph} from a subspace $E^i\subset\Cc^\infty(\Hom^{0,1}_J(S^2,M))$ such that}
$$
E^i \to \Hat\Ee|_f , \;\;  \nu \mapsto \nu|_{\gr f} \quad \text{is injective} \quad \forall \; f\in\Hat\Vv_{f_i} .
$$
\item[{\bf \qquad Sum Condition II$'$:}]
{\it For every $I\subset\{1,\ldots,N\}$, $[g]\in\bigcap_{i\in I}\im\psi_i$ with representative $g\in\Bb_{f_{i_0}}$ for some $i_0\in I$, and marked points $\ul w_i \in D_{i_0 i}\subset (S^2)^2$ in neighbourhoods of $(g^{-1}(Q_{f_i}^t))_{t=0,1}$ resp.\  $D_{i_0 i_0}=\{(0,1)\}$,
we must ensure linear independence of $\bigl\{ \ga_{\ul w_i}^* \nu^i |_{\gr g}  \;\big|\;  \nu^i\in E^i \bigr\}$ for $i\in I$. That is, their sum must be a direct sum}
$$
\sum_{i\in I} \bigl\{ \ga_{\ul w_i}^* \nu^i |_{\gr g}  \;\big|\;  \nu^i\in E^i \bigr\}
 \;=\; 
\bigoplus_{i\in I} \bigl\{ \ga_{\ul w_i}^* \nu^i |_{\gr g}  \;\big|\;  \nu^i\in E^i \bigr\}
 \quad\subset\;\Hat \Ee|_g.
$$

\end{itemlist}

Satisfying these two conditions always requires making the choices of the obstruction spaces $E^i$ ``suitably generic''. If they are satisfied, then they provide a construction of sum charts and coordinate changes as we will state next.
At this point, we can also incorporate a further requirement from \S\ref{ss:top} into the compatibility condition (i) for a tuple of charts $(\bK_i)_{i=1,\ldots,N}$ by
constructing a single sum chart $\bK_{I,g_0}=\bK_I$ for each $I\subset\{1,\ldots, N\}$, whose footprint is the entire overlap of footprints $F_I:=\im\psi_I =  \bigcap_{i\in I} \im\psi_i$.
Moreover, we construct coordinate changes between any pair of tuples $I,J\subset\{1,\ldots,N\}$ with nonempty overlap $F_I\cap F_J\neq\emptyset$ that are, up to a choice of domains, directly induced from the basic charts. Thus our construction naturally satisfies the weak cocycle condition,
i.e.\ equality on overlap of domains as in \S\ref{ss:top}.\footnote
{Note here that $J\subset\{1,\ldots,N\}$ has a very different meaning from the almost complex structure that  determines the Gromov--Witten moduli space $\oMm_1(A,J)$.}

For the construction of sum charts, we will moreover make the following simplifying assumption that all intersections with the slicing hypersurfaces are unique. This can be achieved in sufficiently small neighbourhoods of any holomorphic sphere with trivial isotropy, see Remark~\ref{rmk:unique}.

 \begin{itemlist}
\item[{\bf \qquad Sum Condition IV$'$:}]
{\it For every $i\in \{1,\ldots,N\}$ we assume that the representative $[f_i]$, slicing conditions $Q^t_{f_i}$, size $\eps>0$ of local slice $\Bb_{f_i}$, and its neighbourhood $\Hat\Vv_{f_i}\subset\Hat\Bb^{k,p}$ are chosen such that
for all $g\in \Hat\Vv_{f_i}$ and $t=0,1$ the intersection $g^{-1}(Q^t_{f_i}) =: \{w_i^t(g)\}$ is a unique point and transverse, i.e.\ $\im\rd_{w_i^t(g)}g\pitchfork \rT_{w_i^t(g)} Q^t_{f_i}$.
}

Then the same holds for $g\in G_\infty \Hat\Vv_{f_i}$. Hence for any $i_0\in I \subset \{1,\ldots,N\}$
the local slice $\Bb_{f_{i_0}}$ embeds topologically (as a homeomorphism to its image, with inverse given by  the projection $\Bb_{f_{i_0}} \times (S^2)^{2|I|} \to \Bb_{f_{i_0}}$) into a space of maps and marked points by
\begin{align}
\label{embed}
\iota_{i_0, I} : \; \Bb_{f_{i_0}} &\;\longhookrightarrow\; \Hat\Bb^{k,p} \times (S^2)^{2|I|}   \\
g &\;\longmapsto\; \bigl( g ,
\ul w(g) \bigr) ,
\qquad\qquad\quad
\ul w(g):= \bigl( g^{-1}(Q_{f_i}^t) \bigr)_{i\in I,t=0,1}.
 \nonumber
\end{align}
\end{itemlist}

Note that the elements of $\im \io_{i_0,I}$ have the form $\bigl(g,\ul w(g)=(\ul w_i)_{i\in I}\bigr)$ with $\ul w_{i_0} = (0,1)$.
In the following we denote by $\ul w = (\ul w_i)_{i\in I}\in (S^2)^{2|I|}$ any tuple of $\ul w_i = (w_i^0, w_i^1)\in S^2\times S^2$, even if it is not determined by a map $g$.
Then ``$\forall i, t$'' will be shorthand for ``$\forall i\in I, t\in \{0,1\}$''.
After these preparations we finally prove Theorem A in the introduction.

\begin{thm} \label{thm:A2}
Suppose that the tuple of basic Kuranishi charts
$$
\bigl(\bK_i = (U_{f_i},E_{f_i},s_{f_i},\psi_{f_i})\bigr)_{i=1,\ldots,N}
$$
is constructed as in Proposition~\ref{prop:A1} from local slices $\Bb_{f_i}$ and subspaces
$$
E^i\subset \Cc^\infty(\Hom^{0,1}_J(S^2,M)),
$$
that  induce obstruction bundles $\Hat E^i$ over neighbourhoods
$\Hat\Vv_{f_i}\subset\Hat\Bb^{k,p}$ of $\Bb_{f_i}$.
Assume moreover that this data satisfies Sum Conditions {\rm I}$'$, {\rm II}$'$, and {\rm IV}$'$.
Then for every index subset $I\subset\{1,\ldots,N\}$ with nonempty overlap of footprints
$$
F_I := \;{\textstyle \bigcap_{i\in I}}\im \psi_i  \;\neq\; \emptyset
$$
we obtain the following transition data.
\begin{enumerate}
\item
Corresponding to each choice of $i_0\in I$ and sufficiently small open set
 $$
 \Hat\Ww_{I,i_0} \subset \Bigl(\Hat\Bb^{k,p} \times (S^2)^{2|I|}\Bigr)\cap \bigl\{(g,\ul w)
 \,\big|\, \ul w_{i_0} = (0,1)\bigr\}
 $$
that covers a neighbourhood of the footprint $F_I$ in the sense that
\begin{align*}
\quad
 \bigl\{ ( g , \ul w ) \in \Hat\Ww_{I,i_0} \,\big|\, \pbar g = 0 , \;
 g(w_i^t)\in Q_{f_i}^t,\; \forall i, t \,\bigr\}
= \iota_{i_0,I}( \psi_{i_0}^{-1}(F_I) )
\end{align*}
there is a {\bf sum chart} $\bK_{I}: = \bK_{I,i_0}$ with
\begin{itemize}
\item
domain
\[
\qquad\qquad
U_{I} := \bigl\{ \bigl( g , \ul w \bigr) \in \Hat\Ww_{I,i_0}\,\big|\,
\pbar g \in {\textstyle \sum_{i\in I}} \Gamma_{\ul w_i}^* \Hat E^i,   \,
 g(w_i^t)\in Q_{f_i}^t  \ \forall i,t \,
\bigr\} ,
\]
\item
obstruction space $\displaystyle \; E_I:= {\textstyle \prod_{i\in I}} E^i$,
\item
section
$\displaystyle \;
s_{I} : U_{I} \to E_I , \;  ( g , \ul{w})   \mapsto  (\nu^i)_{i\in I}$
given by
$$
\pbar g = {\textstyle \sum_{i\in I}} \, \ga_{\ul w_i}^* \nu^i |_{\gr g} ,
$$
\item
footprint map
$\displaystyle\; \psi_{I} : s_{I}^{-1}(0) \overset{\cong}\to F_I, \;
(g,\ul w) \mapsto [g]$.
\end{itemize}\MS
\item
For every $I\subset J$ and choice of $i_0\in I, j_0\in J$ as above,
a coordinate change
$\Hat\Phi_{IJ}:
\bK_{I} \to \bK_{J}$ is given by
\begin{itemize}
\item
a choice of domain
$\displaystyle \; V_{IJ} \subset U_{I}$ such that
\begin{align}
\label{choice VIJ}
\qquad\qquad\qquad  V_{IJ}\cap s_{I}^{-1}(0) = \psi_{I}^{-1}(F_J), \qquad V_{IJ} \subset
\iota_{i_0,I} \bigl( \Ga_{f_{j_0},f_{i_0}} \bigl( \iota_{j_0,J}^{-1} ( U_{J} ) \bigr)\bigr)
\end{align}
with the embeddings \eqref{embed} and the reparametrization ${\Ga_{f_{j_0},f_{i_0}}:\Bb_{f_{j_0}}\to \Bb_{f_{i_0}}}$ as in \eqref{transition},
\item
embedding  $\phi_{IJ} := \iota_{j_0,J} \circ \Ga_{f_{i_0},f_{j_0}} \circ \iota_{i_0,I}^{-1}$, that is
\footnote{
This map also equals
$\; \phi_{IJ}\bigl( g , (w_i^t)_{i\in I,t=0,1} \bigr) \;=\;
\bigl(\, g\circ \ga \,,\,  (\ga^{-1}(w_i^t))_{i\in I,t=0,1} \cup (g^{-1}(Q^t_{f_j}))_{j\in J\less I, t=0,1} \,\bigr)$,
where $\ga=\ga_{\ul w_{j_0}}=\ga_g\in G_\infty$ is determined by
$\Ga_{f_{i_0},f_{j_0}}(g) = g\circ \ga_g$ or equivalently $\ga_{\ul w_{j_0}}(t)=w_{j_0}^t \;\forall t$.
}
$$
\qquad\qquad
\phi_{IJ} : V_{IJ} \to U_{J}, \quad
\bigl( g ,
\ul w \bigr)
\mapsto
\bigl(\, \Ga_{f_{i_0},f_{j_0}}(g) \,,\, (g^{-1}(Q^t_{f_j}))_{j\in J, t=0,1} \,\bigr) ,
$$
\item
linear embedding $\Hat\phi_{IJ}:E_I \hookrightarrow E_J$ given by the natural inclusion.
\end{itemize}
\end{enumerate}
Moreover, any choice of $i_0\in I$ and open sets $\Hat\Ww_{I,i_0}$ for each $F_I\neq\emptyset$, and domains $V_{IJ}$ for each $F_I\cap F_J\neq\emptyset$ forms 
an additive weak Kuranishi atlas ${(\bK_I, \Hat\Phi_{IJ})}$ in the sense of 
Definitions~\ref{def:K}, \ref{def:Ku2}; 
in particular satisfying the weak cocycle condition
$$
\phi_{JK} \circ \phi_{IJ} = \phi_{IK} \qquad\text{on}\;\; V_{IK} \cap \phi_{IJ}^{-1}(V_{JK}) .
$$
\end{thm}

\begin{proof}
The sum charts $\bK_{I}$ will be constructed as in Proposition~\ref{prop:A1}. In fact, let us begin by showing that the necessary choices of neighbourhoods in (i) always exist.
Since $F_I\subset\oMm_1(A,J)$ is open and $\psi_{i_0}$ is a homeomorphism to $s_{i_0}^{-1}(0)\subset U_{i_0}\subset\Bb_{f_{i_0}}$, there exists an open set $\Bb_{I,i_0}\subset\Bb_{f_{i_0}}$
such that $\Bb_{I,i_0}\cap s_{i_0}^{-1}(0) = \bigl\{  g \in \Bb_{I,i_0} \,\big|\, \pbar g = 0 \bigr\} = \psi_{i_0}^{-1}(F_I)$. Next, since $\iota_{I,i_0}$ is an embedding to
$\Hat\Bb^{k,p} \times \bigl\{ (\ul w_{i}) \in (S^2)^{2|I|}   \,\big|\, \ul w_{i_0} = (0,1) \bigr\}$,
it contains an open set $\Hat\Ww_{I,i_0}$ such that $\Hat\Ww_{I,i_0} \cap\im\iota_{I,i_0} = \iota_{i_0,I}( \Bb_{I,i_0})$. Together, this implies the requirement in (i).
Note moreover that elements $(g,\ul w)\in \im \iota_{i_0,I}$ satisfy $g(w_i^t)\in 
Q^t_{f_i}$ and hence $\Hat\Ww_{I,i_0}$ can be chosen such that $g(w_i^t)$ lies in a given neighbourhood of the hypersurface  
$Q^t_{f_i}$ near $f_i(t)$ for any $(g,\ul w)\in \Hat\Ww_{I,i_0}$.

Next, note that Sum Condition II$'$ is assumed to be satisfied for $g\in\psi_{i_0}^{-1}(F_I)\subset\Bb_{f_{i_0}}$, and hence continues to hold for $\bigl(g,(\ul w_i) \bigr) \in\Hat\Ww_{I,i_0}$ in a sufficiently small neighbourhood of $\iota_{i_0,I}(\psi_{i_0}^{-1}(F_I))$. Thus we obtain a well defined bundle
\begin{equation}\label{HEI}
\Hat E_I \; \to \; \Hat\Ww_{I,i_0}   , \qquad
\Hat E_I|_{(g,\ul w)}:=
{\textstyle \sum_{i\in I}} \bigl( \Gamma_{\ul w_i}^* \Hat E^i \bigr)|_{g}  \;\subset\; \Hat\Ee|_g,
\end{equation}
where $\Gamma_{\ul w_i}^*\Hat E^i$ is as defined in \eqref{eq:Gamaw}.
In order to construct a Kuranishi chart $\bK_{I}$ with footprint $F_I$ from $\Hat E_I$ along the lines of Proposition~\ref{prop:A1}, we need to express the domain $U_I$ as the zero set of a smooth transverse Fredholm operator.
Recall here from \S\ref{ss:eval} that $\Cc^\infty(S^2,M) \times S^2 \ni (g, w_i^t) \mapsto g(w_i^t) \in M$ is not smooth in any standard Banach norm.
Hence we first
construct
the thickened solution space
\[
\Hat U_{I} := \bigl\{ ( g , \ul{w} ) \in \Hat\Ww_{I,i_0} \,\big|\, \; \pbar g \in \Hat E_I|_{(g,\ul w)}
\bigr\},
\]
which is the zero set of the smooth Fredholm operator
$$
\Hat\Ww_{I,i_0} \;\longrightarrow\;
\bigcup_{(g, \ul w)}\quotient{\Hat\Ee|_g }{ \Hat E_I |_{(g,\ul w)}}  ,
\qquad
( g , \ul{w})  \; \longmapsto\; [\pbar g] .
$$
We can achieve transversality of this operator by choosing $\Hat\Ww_{I,i_0}$ to be a sufficiently small neighbourhood of $\iota_{i_0,I}(\psi_{i_0}^{-1}(F_I))$, since $\pbar$ is transverse to $\Hat\Ee/\Hat E^{i_0}$ over $\psi_{i_0}^{-1}(F_I) \subset \Hat\Vv_{f_{i_0}}$, and for
$( g , \ul{w})\in  \iota_{i_0,I}(\psi_{i_0}^{-1}(F_I))$
we have
$\Hat E^{i_0}|_g \subset \Hat E_I |_{(g,\ul w)}$.

Finally, the domain $U_{I}\subset\Hat U_{I}$ is the zero set of the map
\begin{align} \label{BQW}
\Hat U_{I}  \quad &\;\longrightarrow\;  \underset{i\in I}{ \textstyle {\prod}}\bigl( (\rT_{f_i(0)} Q_{f_i}^{0})^\perp\times (\rT_{f_i(1)} Q_{f_i}^{1})^\perp \bigr) , \\
( g , \ul{w})
&\; \longmapsto\;
\underset{i\in I}{ \textstyle {\prod}} \bigl( \Pi^\perp_{Q_{f_i}^{0}}(g(w^0_i)), \Pi^\perp_{Q_{f_i}^{1}}(g(w^1_i))\bigr),
\nonumber
\end{align}
which is well defined for sufficiently small choice of $\Hat\Ww_{I,i_0}$, such that the $g(w^t_i)$ lie in the domain of definition of the projections $\Pi^\perp_{Q_{f_i}^{t}}$.
Moreover, this map is smooth, since by the regularity in Sum Condition III (which is satisfied by construction) we have $\Hat U_I\subset\Cc^\infty(S^2,M)$.
To see that it is transverse, it suffices to consider any given point $(g ,\ul{w}) \in \iota_{i_0,I}(\psi_{i_0}^{-1}(F_I))$, since transversality at these points persists in an open neighbourhood, and then $\Hat\Ww_{I,i_0}$ can be chosen sufficiently small to achieve transversality on all of $\Hat U_I$.
At these points we understand some parts of the tangent space $\rT_{(g ,\ul{w})}\Hat U_I$ because $\{ (f, \ul v ) \in \Hat\Ww_{I,i_0} \,|\,  \pbar f=0 \}$ is a subset of $\Hat U_I$ which contains $\iota_{i_0,I}(\psi_{i_0}^{-1}(F_I))$.
Hence we have $\bigl(\delta g, (\delta w_i)_{i\in I}\bigr) \in \rT_{(g, \ul w)} \Hat U_I$ for any $\delta g \in \ker{\rm D}_g\pbar$ and $\delta w_i \in \rT_{w_i}(S^2)^2$ with $\delta w_{i_0}=0$.
In particular, we have $\rT_{g}(G_\infty g) \times \{0\} \subset \rT_{(g, \ul w)} \Hat U_I$ since $\{ (f, \ul v ) \in \Hat\Ww_{I,i_0} \,|\,  \pbar f=0 \}$ is invariant under the action $\ga : (f,\ul v)\mapsto (f\circ \ga , \ul v)$ of $\{\ga\approx{\rm id}\}\subset G_\infty$, unlike the thickened solution space $\Hat U_I$ itself.
(Neither space is invariant under the more natural action $(f,\ul v)\mapsto (f\circ \ga , \ga^{-1}(\ul v))$ that will be important below, since at the moment $\ul w_{i_0}$ is fixed.)

Now the $i_0$ component of the linearized operator
of \eqref{BQW} at any point 
simplifies, since the marked points $w^t_{i_0}=t$ are fixed, to
\begin{equation}\label{linop0}
\rT_{(g, \ul w)} \Hat U_I \;\ni\; \bigl(\delta g, (\delta w_i)_{i\in I}\bigr)
 \;\mapsto\;  \bigl( \rd\Pi^\perp_{Q^{0}_{f_{i_0}}} \delta g (0) , \rd\Pi^\perp_{Q^{1}_{f_{i_0}}} \delta g (1) \bigr) .
\end{equation}
At points with $\pbar g=0$, its restriction to $\rT_{g}(G_\infty g) \times \{0\} \subset \rT_{(g, \ul w)} \Hat U_I$ is surjective by the same argument as in Proposition~\ref{prop:A1}, which uses the fact that $\im\rd_t g$ projects onto $(\rT_{f_{i_0}(t)} Q_{f_{i_0}}^{t})^\perp$ by the construction of the local slice $\Bb_{f_{i_0}}$ at $g\approx f_{i_0}$.
Next, 
the  
$j\in I\less\{i_0\}$ component of the linearized operator for fixed $\delta g$ is
\begin{equation}\label{linopi}
\rT_{(g, \ul w)} \Hat U_I \;\ni\; 
\bigl(\delta g, (\delta w_i)_{i\in I}\bigr)
\;\mapsto\;
\Bigl(  \rd\Pi^\perp_{Q^{t}_{f_j}}   \bigl(\delta g (w^t_j )  + \rd_{w^t_j} g ( \delta w^t_j)  \bigr)\Bigr)_{t=0,1}  .
\end{equation}
We claim that this is surjective for any given $\delta g \subset \rT_{g}(G_\infty g)$ (given by the surjectivity requirements for $i_0$), just by variation of $\delta w_j$.
Indeed, for $(g, \ul w)\in \iota_{i_0,I}(\psi_{i_0}^{-1}(F_I))$ we 
have $\bigl(\delta g, (\delta w_i)_{i\in I}\bigr) \in \rT_{(g, \ul w)} \Hat U_I$ for any $\delta g \subset \rT_{g}(G_\infty g)$ and $\delta w_i \in \rT_{w_i}(S^2)^2$.
Moreover, we have $\im\rd_{w^t_j} g = 
\im 
\rd_{t} (g\circ\ga_{\ul w_j})$, which projects onto $(\rT_{f_j(t)} Q_{f_j}^{t})^\perp$ by the construction of the local slice $\Bb_{f_j}$ at $g\circ\ga_{\ul w_j}\approx f_j$.
This proves surjectivity of \eqref{linopi} for $j\neq i_0$  
by variation of $\delta w_j$, 
and together with the surjectivity of\eqref{linop0} 
by variation of $\delta g$
proves transversality of \eqref{BQW} for sufficiently small $\Hat\Ww_{I, i_0}$.

Now that the domain $U_I$ is equipped with a smooth structure, we can construct a Kuranishi atlas $\bK_I$ as in Proposition~\ref{prop:A1} by pulling back the smooth section
$$
\ti s_I : U_I \;\to\; \Hat E_I|_{U_I} , \qquad  (g, \ul w) \;\mapsto\; \pbar g
$$
to the trivialization $\Hat E_I|_{U_I}\cong U_I \times E_I$ given by construction of the sum bundle.
The induced homeomorphism
$$
\psi_I :  \; \ti s_I^{-1}(0) \; \overset{\cong}{\longrightarrow} \;  F_I \;\subset\; \oMm_1(A,J) ,
\qquad  (g, \ul w)\;\mapsto\; [g]
$$
maps
$\ti s_I^{-1}(0) \subset \im\iota_{i_0,I}$
to the desired footprint since we chose the neighbourhoods $\Hat\Ww_{I, i_0}$ and $\Bb_{I,i_0}:=\iota_{i_0,I}^{-1}(\Hat\Ww_{I, i_0})\subset \Bb_{f_{i_0}}$ such that
$$
\psi_I( \ti s_I^{-1}(0) )
\;=\; \pr \bigl( \iota_{i_0,I}^{-1}(\ti s_I^{-1}(0)) \bigr)
\;=\;  \pr \bigl( \iota_{i_0,I}^{-1}(\Hat\Ww_{I,i_0}) \cap \pbar^{-1}(0) \bigr)
\;=\;  \pr \bigl( \psi_{i_0}^{-1}(F_I)  \bigr)
\;=\; F_I ,
$$
where $\pr :\Hat\Bb^{k,p}\to \Hat\Bb^{k,p}/G_\infty$ denotes the quotient.
This finishes the construction for (i).

To construct the coordinate changes, we can now forget the marked points, which were only a technical means to obtaining smooth sum charts.
For that purpose fix a pair $i_0\in I$ and note that the forgetful map $\Pi_I:\Hat\Bb^{k,p}\times (S^2)^{2|I|}\to\Hat\Bb^{k,p}$ is a left inverse to the embedding $\io_{i_0,I}:\Bb_{f_{i_0}}\hookrightarrow \Hat\Bb^{k,p}\times (S^2)^{2|I|}$ from \eqref{embed}, whose image contains the smooth finite dimensional domain $U_I\subset\Cc^\infty(S^2,M)\times (S^2)^{2|I|}$.
Hence it restricts to a topological embedding to a space of perturbed holomorphic maps in the slice,
\begin{align}\label{UB}
\Pi_I|_{U_I} : \;   U_I   \;\longrightarrow\;
B_{I,i_0} :=&\; \bigl\{ g\in \Bb_{f_{i_0}} \,\big|\, \exists \ul w \in (S^2)^{2|I|} : (g,\ul w) \in U_I \bigr\} \\
=&\; \bigl\{ g\in \Bb_{I,i_0} \,\big|\, \pbar g \in \Hat E_I |_{(g,\ul w(g))} \bigr\}   .  \nonumber
\end{align}
In fact, this is a smooth embedding since the forgetful map is smooth and we can check that the differential of the forgetful map $\Pi_I|_{U_I}$ is injective. Indeed, its kernel at $(g,\ul w)$  is the vertical part of the tangent space
$\rT_{(g,\ul w)} U_{I} \cap \bigl( \{0\} \times \rT_{\ul w} (S^2)^{2|I|}\bigr)$, which in terms of the linearized operators \eqref{linopi} is given by the kernel of
$$
\rT_{\ul w} (S^2)^{2|I|} \; \ni \;
(\delta w^t_i )_{i\in I, t=0,1} \;\longmapsto\;
\Bigl(  \rd_{g(w^t_i)}\Pi^\perp_{Q^{t}_{f_i}}   \bigl( \rd_{w^t_i} g ( \delta w^t_i)  \bigr)\Bigr)_{i\in I, t=0,1}
\;\in \; \underset{i\in I, t=0,1}{\textstyle\prod} \im \rd_t f_i .
$$
This operator is injective (and hence surjective) since by Sum Condition IV$'$
$$
\bigl(\im\rd_{w^t_i} g\bigr)\; \pitchfork\; \bigl(\rT_{w^t_i}Q^{t}_{f_i}\bigr)=\ker\rd_{g(w^t_i)}\Pi^\perp_{Q^{t}_{f_i}}.
$$
Thus $\io_{i_0,I}: B_{I,i_0}\to U_I$ is a diffeomorphism, and since it also intertwines the Cauchy--Riemann operator on the domains and the projection to $\oMm_1(A,J)$, this forms a map
$$
\Hat\Pi_{I,i_0}:=\bigl(\Pi_I|_{U_I} , \id_{E_I} \bigr): \; \bK_I \longrightarrow \bK^B_I,
$$
from the sum chart
$\bK_I=\bigl(\, U_I \,,\, \bigcup_{(g,\ul w)\in U_I} \Hat E_I|_{(g,\ul w)} \,,\, \ti s_I(g,\ul w)=\pbar g \,,\, \psi_I(g,\ul w)=[g] \,\bigr)$ to the Kuranishi chart
$$
\bK^B_I: = \bigl(\, B_{I,i_0} \,,\,
{\textstyle \bigcup_{g\in B_{I,i_0}}} \Hat E_I|_{(g,\ul w(g))} \,,\, \ti s(g)=\pbar g \,,\, \psi(g)=[g] \,\bigr).
$$
(Here we indicated the obstruction bundles before trivialization to $E_I$.)
The inverse map $\Hat\Pi_{I,i_0}^{-1}:=\bigl(\io_{i_0,I} , \id_{E_I} \bigr)$ is also a map between Kuranishi charts, and both are coordinate changes since the index condition is automatically satisfied when $\phi_{IJ}$ and $\Hat\phi_{IJ}$ are both diffeomorphisms.
Indeed, in this case, both target and domain in the tangent bundle condition \eqref{tbc} are trivial.

Next, we will obtain further coordinate changes $\Hat\Phi^I_{i_0 j_0} : (B_{I,i_0},\ldots) \to (B_{I,j_0},\ldots)$ for different choices of index $i_0,j_0\in I$. Here the choices of neighbourhoods $\Hat\Ww_{I,\bullet}$ induce neighbourhoods in the local slices $\Bb_{I,\bullet}\subset\Bb_{f_\bullet}$ such that
$B_{I,\bullet} :=\; \bigl\{ g\in \Bb_{f_\bullet} \,\big|\, \pbar g \in \Hat E_I |_{(g,\ul w(g))} \bigr\}$.
These domains are intertwined by the transition map between local slices $\Ga_{f_{i_0},f_{j_0}}$. Indeed, using the $G_\infty$-equivariance of the obstruction bundles \eqref{Eequivariant}, we have
$$
\pbar g = {\textstyle \sum_{i\in I}} \, \ga_{\ul w_i(g)}^*\nu^i |_{\gr g}
\quad\Longrightarrow\quad
\pbar(g\circ\ga) = {\textstyle \sum_{i\in I}}\,  \ga_{\ul w_i (g\circ \ga)}^*\nu^i |_{\gr g\circ\ga} ,
$$
where $\ga^{-1}\circ\ga_{\ul w_i} = \ga_{\ul w_i(g\circ\ga)}$ since $\ga^{-1}(\ga_{\ul w_i}(t))= \ga^{-1}( w_i^t)= w_i^t(g\circ\ga)$.
Thus we obtain a well defined map
$\Ga_{f_{i_0},f_{j_0}} : B_{I,i_0} \cap G_\infty\Bb_{I,j_0} \to  B_{I,j_0}$.
It is a topological embedding with open image, since its inverse is $\Ga_{f_{j_0},f_{i_0}} |_{B_{I,j_0} \cap G_\infty\Bb_{I,i_0}}$. In fact, it is a local diffeomorphism since both maps are smooth by Lemma~\ref{le:Gsmooth}.
The above also shows that this diffeomorphism intertwines the sections, given by the Cauchy--Riemann operator, and the footprint maps, given by the projection $g\mapsto [g]\in\oMm_1(A,J)$.
Since the index condition is automatic as above, we obtain the required coordinate change by
$\Hat\Phi^I_{i_0 j_0}:=\bigl(\, \Ga_{f_{i_0},f_{j_0}} \,,\, \id_{E_I} \,\bigr)$
with domain $B_{I,i_0} \cap G_\infty\Bb_{I,j_0} \subset B_{I,i_0}$.

With these preparations, a natural coordinate change  for $I\subsetneq J$ and any choice of $i_0\in I$, $j_0\in J$ arises from the composition of the above coordinate changes (all of which are local diffeomorphisms on the domains) with another natural coordinate change
$\Hat\Phi^{i_0}_{IJ} : (B_{I,j_0},\ldots) \to (B_{J,j_0},\ldots)$ given by the inclusion
$$
\phi^{j_0}_{IJ} := {\rm id}_{\Bb_{j_0}} : \; B_{I,j_0} \cap \Bb_{J,j_0} \;\hookrightarrow\; B_{J,j_0} .
$$
Again, this naturally intertwines the sections and footprint maps with
$$
s_I^{-1}(0)\cap B_{I,j_0} \cap \Bb_{J,j_0}= \psi_I^{-1}(F_J).
$$
To check the index condition for this embedding together with the linear embedding $\Hat\phi^{j_0}_{IJ} := {\rm id}_{E_I} : E_I \hookrightarrow E_J$ we express the tangent spaces to both domains in terms of the linearization of the Cauchy--Riemann operator on the local slice $\overline\partial: \Bb_{j_0} \to \Hat\Ee|_{\Bb_{j_0}}$. Comparing
$$
 \rT_g B_{I,j_0} =  ({\rm D}_g \overline{\partial}) ^{-1} \Bigl( \textstyle{\sum_{i\in I}} (\Ga_{\ul w_i(g)}^*\Hat E^i)|_g \Bigr) , \qquad
\rT_g B_{J,j_0} = ({\rm D}_g \overline\partial) ^{-1} \Bigl(  {\textstyle \sum_{j\in J}} (\Ga_{\ul w_j(g)}^*\Hat E^j)|_g \Bigr)
$$
as subsets of $\rT_g\Bb_{f_{j_0}}$, we can identify
$$
\quotient{\rT_g B_{J,j_0}}{\rd_g\phi^{j_0}_{IJ} \bigl( \rT_g B_{I,j_0} \bigr)}
\;=\;
 \quotient{ ({\rm D}_g \overline\partial) ^{-1}  \left( \textstyle{\sum_{j\in J\less I}} (\Ga_{\ul w_j(g)}^*\Hat E^j)|_g\right) }{\ker {\rm D}_g\overline\partial
}
$$
to see that the linearized section (given by the linearized Cauchy Riemann operator together with the trivialization of obstruction bundles) satisfies the tangent bundle condition \eqref{tbc}
\begin{eqnarray*}
{\rm D}_{g} \overline\partial \;:\;
\quotient{\rT_g B_{J,j_0}}{\rd_g\phi^{j_0}_{IJ} \bigl( \rT_g B_{I,j_0} \bigr)}
\;&\stackrel{\cong} \longrightarrow\; &
 {\textstyle\sum_{j\in J\less I}} (\Ga_{\ul w_j(g)}^*\Hat E^j)|_g \\
&&\qquad\quad \;\cong\;
\quotient{\Hat E_J|_{(g,(\ul w_j(g))_{j\in J})}}
{\Hat E_I|_{(g,(\ul w_i(g))_{i\in I})}}.
\end{eqnarray*}
Finally, we can compose the coordinate changes to
$$
\Hat\Phi_{IJ} \,:= \; \Hat\Pi_{J,j_0}^{-1} \circ \Hat\Phi^J_{i_0 j_0} \circ \Hat\Phi^{i_0}_{IJ} \circ \Hat\Pi_{I,i_0} \; : \;\;
\bK_I \; \to \; \bK_J .
$$
By Lemma~\ref{le:cccomp} this defines a coordinate change with the maximal domain
$$
\iota_{i_0,I}(B_{I,i_0}  \cap G_\infty \Bb_{J,j_0}) \;=\; \iota_{i_0,I} \bigl( \Ga_{f_{j_0},f_{i_0}} \bigl( \iota_{j_0,J}^{-1} ( U_{J} ) \bigr)\bigr),
$$
which we can restrict to any smaller choice of $V_{IJ}$ containing $\psi_I^{-1}(F_J)$.
The linear embedding, after the fixed trivialization of the bundle, is the trivial embedding $\Hat\phi_{IJ}: E_I\hookrightarrow E_J$, whereas the nonlinear embedding $\phi_{IJ}:=\iota_{i_0,I}^{-1}\circ \Ga_{f_{i_0},f_{j_0}}\circ \iota_{j_0,J}: V_{IJ} \to U_J$ of domains is given by the restriction to $V_{IJ}$ of the composition
$$
U_I \;\overset{\iota_{i_0,I}}{\longhookleftarrow}\; B_{I,i_0}  \cap G_\infty \Bb_{J,j_0} \; \xrightarrow[\cong]{\Ga_{f_{i_0},f_{j_0}}}\; B_{I,j_0} \cap \Bb_{J,j_0}
 \;\overset{\id_{\Bb_{j_0}}}{\longhookrightarrow}\; B_{J,j_0}
 \;\overset{\iota_{j_0,J}}{\longhookrightarrow}\; U_J .
$$
This completes the proof of (ii).
Finally, the  
additivity and cocycle conditions on the level of the linear embeddings $\Hat\phi_{IJ}$ 
hold by construction,
whereas the weak cocycle condition for the embeddings between the domains follows, since $\iota_{j_0,J}^{-1}\circ\iota_{j_0,J}=\id_{\Bb_{J,j_0}}$ from the cocycle property of the local slices,
$$
\Ga_{f_{j_0},f_{k_0}}\circ \Ga_{f_{i_0},f_{j_0}}  = \Ga_{f_{i_0},f_{k_0}}
\qquad\text{on}\; \Bb_{f_{i_0}}\cap \bigl(G_\infty\cdot\Bb_{f_{j_0}}\bigr) \cap \bigl(G_\infty \cdot \Bb_{f_{k_0}}\bigr) .
$$
This completes the proof of Theorem~\ref{thm:A2}.
\end{proof}

Note that we crucially use the triviality of the isotropy groups, in particular in the proof of the cocycle condition. Nontrivial isotropy groups cause additional indeterminacy, which has to be dealt with in the abstract notion of Kuranishi atlases. The construction of Kuranishi atlases with nontrivial isotropy groups for Gromov--Witten moduli spaces will in fact require a sum construction already for the basic Kuranishi charts.
We will give a more detailed proof of Theorem~\ref{thm:A2} in \cite{MW:GW}, where we will also treat nodal curves and deal with the case of isotropy or, more generally, nonunique intersections with the hypersurfaces $Q^t_{f_i}$.
There the chart domains will be defined using a notion of Fredholm stabilization which ensures that conditions equivalent to Sum Conditions I$'$ and II$'$ are automatically satisfied.

\begin{remark} \rm  \label{rmk:smart}
(i) We could choose the domains $U_{f_i}$ of the charts $\bK_i$ sufficiently small to be precompact open subsets of Euclidean spaces (of possibly different dimensions) for each $i\in I$. 
However, the domain $U_I$ of a sum chart cannot necessarily be constructed as open subset of a Euclidean space, since we require it to contain the full zero set $\psi_I^{-1}(F_I)$.

\MS\NI
(ii)
The actual idea behind the choice of local slice conditions and the introduction of further marked points is of course a stabilization of the domain in order to obtain a theory over the Deligne--Mumford moduli space of stable genus zero Riemann surfaces with marked points. So one might want to rewrite this approach invariantly and, when e.g.\ summing two charts, work over the Deligne--Mumford moduli space with five marked points instead of taking the points $(\infty,0,1,w_i^0,w_i^1)$.
Note however that one would need to make sure that the marked points $w^0,w^1\in S^2$ that we read off from intersection with the hypersurfaces $Q_{f_1}^{0},Q_{f_1}^{1}$ are disjoint from each other and from $\infty, 0,1$.
Thus, in order to be summable in this framework, the basic Kuranishi charts would have to be constructed from local slices with pairwise disjoint slicing conditions $Q_{f_i}^t$.

When properly handled, this approach does give a good framework for discussing coordinate changes. However one does need to take care not to obscure the analytic problems by introducing these further abstractions and notations.
Moreover, this abstraction does not yield another approach to constructing the coordinate changes.
If there is a rigorous approach using the Deligne--Mumford formalism, then in a local model near $(\infty,0,1, w^0_{01},w^1_{01})$,  it would take exactly the form discussed above.
\MS

\NI (iii)
The abstraction to equivalence classes of maps and marked points modulo automorphisms becomes crucial when one wants to extend the above approach to construct finite dimensional reductions near nodal curves, because the Gromov compactification exactly mirrors the construction of Deligne--Mumford space. 
While we will defer the details of this construction to \cite{MW:GW}, let us note that one can avoid the need to work with disjoint slicing conditions by working with several copies of Deligne--Mumford space in the sum charts. (This is the direct generalization of our approach above, where we do not require $(\infty,0,1,w_i^0,w_i^1)$ to be disjoint.)
\MS

\NI (iv)
In view of Sum Condition II$'$, one cannot expect any two given basic Kuranishi charts to have summable obstruction bundles and hence be compatible.
Thus even a simple moduli space such as $\oMm_{1}(A,J)$ does not have a canonical Kuranishi atlas.
Hence the construction of invariants from this space also involves constructing a Kuranishi atlas on the product cobordism $[0,1]\times \oMm_{1}(A,J)$ intertwining any two Kuranishi atlases for $\oMm_{1}(A,J)$ arising from different choices of basic charts and transition data. We will call the corresponding relation concordance.
Note here that one could construct basic charts of the ``wrong dimension'' simply by adding 
trivial finite dimensional factors to the abstract domains or obstruction spaces.
A natural and necessary condition for constructing a well defined concordance class of Kuranishi atlases with the ``expected dimension'' for $\oMm_{1}(A,J)$ is the following {\bf Fredholm index condition for charts} pointed out to us by Dietmar Salamon:\MS

{\it Each Kuranishi chart must in some sense identify the kernel $\ker\rd s_{f_i}$ and cokernel $(\im\rd s_{f_i})^\perp$ of the finite dimensional reduction with the kernel modulo the infinitesimal action $\ker\rd_{f_i}\pbar/{\scriptstyle \rT_{f_i}(G_\infty f_i)}$ and cokernel $(\im\rd_{f_i}\pbar)^\perp$ of the Cauchy--Riemann operator.}\MS

In fact, one might argue that this identification should be part of a Kuranishi atlas on a moduli space. However, this would require giving the abstract footprint of a Kuranishi atlas more structure than that of a compact metrizable topological space,
in order to keep track of the kernel and cokernel of the Fredholm operator that arises by linearization from the PDE that defines the moduli space.
Whether the index condition picks out a unique concordance class of Kuranishi atlases on $X$ is an interesting open question.
\MS

\NI 
(v)
The Fredholm index condition for Kuranishi charts, once rigorously formulated, should imply that any map between charts which satisfy the index condition should also satisfy the index condition for coordinate changes in Definition~\ref{def:change} (a reformulation of the tangent bundle condition introduced by \cite{J1}).
Conversely, a map between charts that satisfies the index condition for coordinate changes should also preserve the Fredholm index condition for charts.
More precisely, if $\Hat\Phi_{IJ} :\bK_I \to \bK_J$ is a map satisfying the index condition, and one of the charts $\bK_I$ or $\bK_J$ satisfies the Fredholm index condition, then both charts satisfy the Fredholm index condition.
$\hfill\er$
\end{remark}

%%%%%%%%%%%%%%%%%%%%%%%%%%%%%%%%%%%%%%%%%%%%%%%%%%%%%%%%%%%%%%
\section{Kuranishi charts and coordinate changes with trivial isotropy
}
\label{s:chart}
%%%%%%%%%%%%%%%%%%%%%%%%%%%%%%%%%%%%%%%%%%%%%%%%%%%%%%%%%%%%%%

Throughout this chapter, $X$ is assumed to be a compact and metrizable space.
This section defines Kuranishi charts with trivial isotropy for $X$ and coordinate changes between them.
The case of nontrivial isotropy is a fairly straightforward generalization using the language of groupoids, but the purpose of this paper is to clarify fundamental topological issues in the simplest example.
Hence we assume throughout that the charts have trivial isotropy and drop this qualifier from the wording.
Similarly, we only consider ``smooth charts'' whose domain are smooth manifolds, yet also will not add this qualifier.\footnote{
General holomorphic curve moduli spaces usually have boundary and corners (arising from breaking, buildings, or boundary nodes) and may have a stratified smooth structure (arising from lack of a natural smooth structure near interior nodes).
}
Our definitions are motivated by \cite{FO,FOOO,J1}. However, our insistence on specifying the domains of coordinate changes is new, as are our notion of Kuranishi atlas and interpretation in terms of categories in \S\ref{s:Ks} and 
all subsequent constructions from the virtual neighbourhood in 
\S\ref{ss:coord} to the virtual fundamental class in \S\ref{s:VMC}.

%%%%%%%%%%%%%%%%%%%%%%%%%%%%%%%%%%%%%%%%%%%%%%%%%%%%%%%%%%%%%%
\subsection{Charts, maps, and restrictions}\label{ss:chart}\hspace{1mm}\\ \vspace{-3mm}
%%%%%%%%%%%%%%%%%%%%%%%%%%%%%%%%%%%%%%%%%%%%%%%%%%%%%%%%%%%%%%

The core idea of Kuranishi regularization is to use local finite dimensional reductions, which are formalized in the following notion of local chart, which specializes the notion of topological Kuranishi chart (in \cite[Definition~2.1.1]{MW:top} and Definition~\ref{def:tchart} below) to include a smooth structure and trivial isotropy.

\begin{defn}\label{def:chart}
Let $F\subset X$ be a nonempty open subset.
A {\bf Kuranishi chart} for $X$ with {\bf footprint} $F$ is a tuple $\bK = (U,E,s,\psi)$ consisting of
\begin{itemize}
\item
the {\bf domain} $U$, which is a finite dimensional differentiable manifold;
\item
the 
{\bf obstruction bundle} $\E=U\times E$, which is a trivial vector bundle given by a finite dimensional real vector space $E$, called the
{\bf obstruction space};
\item
the {\bf section} $\s: U\to U\times E, x\mapsto (x,s(x))$, which is given by a smooth map $s: U\to E$;
 \item
the {\bf footprint map} $\psi : s^{-1}(0) \to X$, which is a homeomorphism to the footprint ${\psi(s^{-1}(0))=F}$.
\end{itemize}
The {\bf dimension} of $\bK$ is $\dim \bK: = \dim U-\dim E$.
\end{defn}

\begin{rmk}\rm 
(i)  One could generalize the notion of Kuranishi chart by working with a nontrivial obstruction bundle over the domain 
(see e.g.\ \cite[\S5.2]{Mcn}),
but this complicates the notation and, by not fixing the trivialization, makes coordinate changes less unique. 
In the application to holomorphic curve moduli spaces, there are natural choices of trivialized obstruction bundles.
The section $s$ is then given by the generalized Cauchy--Riemann operator, and elements in the footprint are $J$-holomorphic maps modulo reparametrization.
\MS

\NI (ii) 
The constructions in \S\ref{ss:gw} and 
\cite{MW:GW,Mcn} 
are such that we cannot assume that $U$ is always an open subset of some Euclidean space. Instead, we assume it to be a differentiable manifold, i.e.\ a   second countable Hausdorff space that is locally homeomorphic to a fixed Euclidean space and has smooth transition maps. 
Note in particular that we do not allow manifolds with boundary or the more general notion of paracompact manifolds.
$\hfill\er$
\end{rmk}

As seen in the previous remark, the very first notion of Kuranishi chart already has many variations in the literature. 
However, our topological results for Kuranishi atlases in \cite{MW:top} apply to all these situations by working with the following notion of chart.

\begin{defn}\label{def:tchart} 
{\rm $\!\!$ \cite[Definition~2.1.1]{MW:top}}
A {\bf topological Kuranishi chart} for $X$ with open footprint $F\subset X$ is a tuple $\bK = (U,\E,
\s,\psi)$ consisting of
\begin{itemize}
\item
the {\bf domain} $U$, which is a separable, locally compact metric space;
\item
the {\bf obstruction ``bundle''} 
which is a continuous map 
$\pr :\E \to U$
from a separable, locally compact metric space $\E$,
together with a 
{\bf zero section} ${0: U \to \E}$, which is a continuous map with $\pr \circ 0 = \id_{U}$;
\item
the {\bf section} $\s: U\to \E$, which is a continuous map with $\pr \circ \s = \id_{U}$;
\item
the {\bf footprint map} $\psi : \s^{-1}(0) \to X$, which is a homeomorphism 
between the {\bf zero set} $\s^{-1}(0):=
\s^{-1}(\im 0)=\{x\in U \,|\, \s(x)=0(x)\}$ and 
the {\bf footprint} ${\psi(\s^{-1}(0))=F}$.
\end{itemize}
\end{defn}
 
\begin{rmk}\rm    \label{rmk:topchart}
(i) 
Every Kuranishi chart $\bK=(U,E,s,\psi)$ for $X$ in the sense of Definition~\ref{def:chart} induces a topological Kuranishi chart (which we denote with the same label)
$$
\bK = \bigl(\, U \,,\, \E=U\times E \,,\, \s(x)=(x,s(x)) \,,\, \psi \,\bigr)
$$
with projection $\pr :\E=U\times E\to U$ to the first factor and zero section $0: x \mapsto (x,0)$.
Indeed, note that finite dimensional manifolds are automatically separable (i.e.\ contain a countable dense subset), locally compact, and metrizable. 
\MS

\NI (ii)
A Kuranishi chart with nontrivial isotropy group $\Gamma$ as introduced in \cite{MW:iso} yields a topological Kuranishi chart with domain $\qu{U}{\Ga}$ and obstruction bundle $\E=\qu{U\times E}{\Ga}$.
$\hfill\er$
\end{rmk}

Most of the definitions below are direct specializations of notions for topological Kuranishi charts from \cite{MW:top}, 
so that throughout we will be able to draw from our topological refinement results.

Since we aim to define a regularization of $X$, the most important datum of a Kuranishi chart is its footprint.
So, as long as the footprint is unchanged, we can vary the domain $U$ and section $s$ without changing the chart in any important way. Nevertheless, we will always work with charts that have a fixed domain and section. In fact, our definition of a coordinate change between Kuranishi charts in the next section will crucially involve these domains.
Moreover, it will require the following notion of restriction of a Kuranishi chart to a smaller subset of its footprint.

\begin{defn} \label{def:restr} {\rm $\!\!$ \cite[Definition~2.1.3]{MW:top}}
Let $\bK$ be a Kuranishi chart and $F'\subset F$ an open subset of the footprint.
A {\bf restriction of $\bK$ to $\mathbf{\emph F\,'}$} is a Kuranishi chart of the form
$$
\bK' = \bK|_{U'} := \bigl(\, U' \,,\, E'=E  \,,\, s'= s|_{U'} \,,\, \psi'=\psi|_{U'\cap s^{-1}(0)}\, \bigr)
$$
given by a choice of open subset $U'\subset U$ of the domain such that $U'\cap s^{-1}(0)=\psi^{-1}(F')$.
In particular, $\bK'$ has footprint $\psi'(s'^{-1}(0'))=F'$.

%
%NOTE TO SELF: don't seem to need this notion any more, but for reference
%
%A {\bf pruning} of $\bK$ is a restriction to $F'=F$, that is a Kuranishi chart $\bK' = \bK|_{U'}$
%given by a choice of neighbourhood $U'\subset U$ of $s^{-1}(0)$.
%
\end{defn}

The following lemma shows that we may easily restrict to any open subset of the footprint.
In fact, we can often restrict to precompact domains, which provides a key tool for the construction of shrinkings and reductions of Kuranishi atlases in \cite{MW:top}.
Here we use the notation $V'\sqsubset V$ to mean that the inclusion $V'\hookrightarrow V$ is {\it precompact}. That is, ${\rm cl}_V(V')$ is compact, where ${\rm cl}_V(V')$ denotes the closure of $V'$ in the relative topology of $V$.  If both $V'$ and $V$ are contained in a compact space $X$, then $V'\sqsubset V$ is equivalent to the inclusion $\ov{V'}:={\rm cl}_X(V')\subset V$ of the closure of $V'$ w.r.t.\ the ambient topology.

\begin{lemma}\label{le:restr0}
Let $\bK$ be a Kuranishi chart. Then for any open subset $F'\subset F$ there exists a restriction $\bK'$ to $F'$ whose domain $U'$ is such that $\ov{U'}\cap s^{-1}(0) = \psi^{-1}(\ov{F'})$. If moreover $F'\sqsubset F$ is precompact, then $U'$ can be chosen to be precompact.
\end{lemma}

\begin{proof}
Applying \cite[Lemma~2.1.4]{MW:top} to the topological Kuranishi chart induced by $\bK$ yields the appropriate domain $U'$.
\end{proof}

%%%%%%%%%%%%%%%%%%%%%%%%%%%%%%%%%%%%%%%%%%%%%%%%%%%%%%%%%%%%%%
\subsection{Coordinate changes}\hspace{1mm}\\ \vspace{-3mm}
\label{ss:coord}  
%%%%%%%%%%%%%%%%%%%%%%%%%%%%%%%%%%%%%%%%%%%%%%%%%%%%%%%%%%%%%%

The following notion of coordinate change is key to the definition of Kuranishi atlases.
Here we 
start using notation that will also appear in our definition of Kuranishi atlases. For now, $\bK_I=(U_I,E_I,s_I,\psi_I)$ and $\bK_J=(U_J,E_J,s_J,\psi_J)$ just denote different Kuranishi charts for the same space $X$.
We begin with the notion of topological coordinate change from \cite{MW:top}, then specify to the smooth setting.

\begin{figure}[htbp] 
   \centering
   \includegraphics[width=4in]{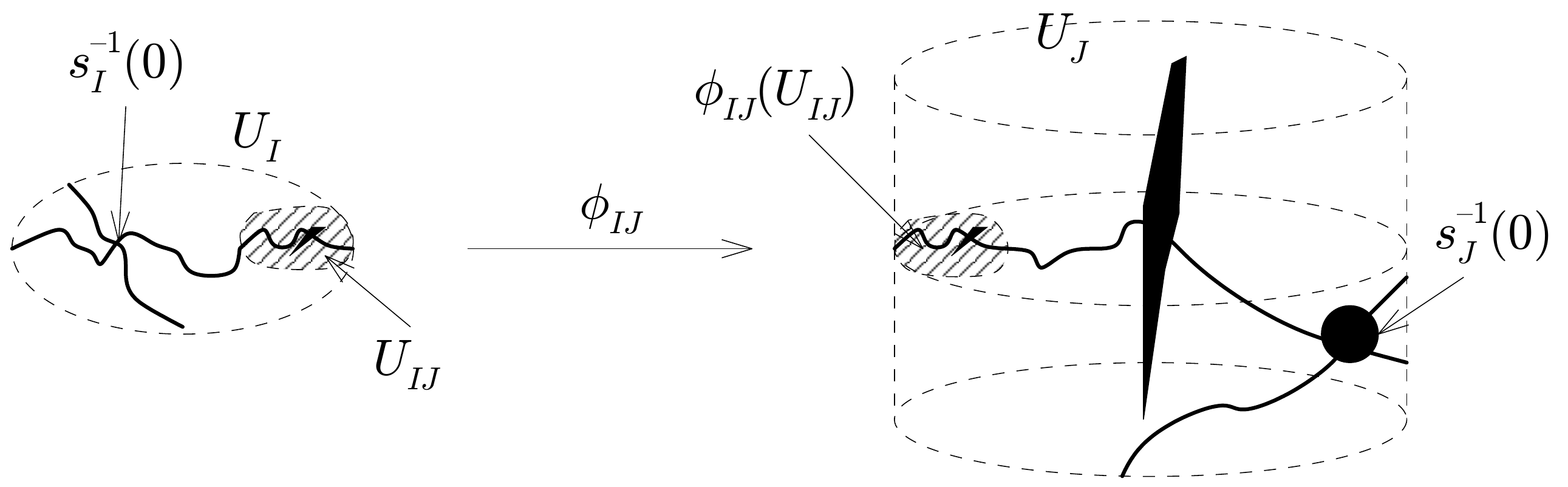}
   \caption{A coordinate change in which $\dim U_{J} =\dim U_I+ 1$.
   Both $U_{IJ}$ and its image $\phi(U_{IJ})$ are shaded.}
   \label{fig:2}
\end{figure}

\begin{defn}\label{def:tchange} {\rm $\!\!$ \cite[Definition~2.2.1]{MW:top}}
Let $\bK_I$ and $\bK_J$ be topological Kuranishi charts such that $F_I\cap F_J$ is nonempty.
A {\bf topological coordinate change} from $\bK_I$ to $\bK_J$ is a map $\Hat\Phi: \bK_I|_{U_{IJ}}\to \bK_J$ defined on a restriction of $\bK_I$ to $F_I\cap F_J$. More precisely:  
\begin{itemize}
\item
The {\bf domain} of the coordinate change is an open subset $U_{IJ}\subset U_I$ such that
$\s_I^{-1}(0_I)\cap U_{IJ} = \psi_I^{-1}(F_I\cap F_J)$.
\item 
The {\bf map} of the coordinate change is a topological embedding (i.e.\ homeomorphism to its image) 
$\Hat\Phi: \E_I|_{U_{IJ}}:=\pr_I^{-1}(U_{IJ}) \to \E_J$ that satisfies the following.
\begin{enumerate}
\item
It is a bundle map, i.e.\ we have $\pr_J \circ \Hat\Phi = \phi \circ\pr_I |_{\pr_I^{-1}(U_{IJ})}$ for a topological embedding $\phi: U_{IJ}\to U_J$, and it is linear in the sense that $0_J \circ \phi  = \Hat\Phi \circ 0_I|_{U_{IJ}}$.
\item
It intertwines the sections in the sense that $\s_J \circ \phi  = \Hat\Phi \circ \s_I|_{U_{IJ}}$.
\item
It restricts to the transition map induced from the footprints in $X$ in the sense that
$\phi|_{\psi_I^{-1}(F_I\cap F_J)}=\psi_J^{-1} \circ\psi_I : U_{IJ}\cap \s_I^{-1}(0_I) \to \s_J^{-1}(0_J)$.
\end{enumerate}
\end{itemize}
\end{defn}

The map $\Hat\Phi$ is not required to be locally surjective.
Indeed, the rank of the obstruction bundles $\E$ will typically be different for different charts.
However, the maps allowed as coordinate changes between smooth Kuranishi charts are carefully controlled in the normal direction by the index condition in the following.

\begin{defn}\label{def:change}
A {\bf coordinate change} between Kuranishi charts is a topological coordinate change
$\Hat\Phi: \bK_I|_{U_{IJ}}\to \bK_J$ that splits $\Hat\Phi=\phi\times \Hat\phi$ into a smooth embedding
$\phi: 
U_{IJ}\to U_J$ and a linear embedding $\Hat\phi: E_I\to E_J$
that satisfy the {\bf index condition} in (i),(ii) below.
\begin{enumerate}
\item
The embedding $\phi:U_{IJ}\to U_J$ identifies the kernels,
$$
\rd_u\phi \bigl(\ker\rd_u s_I \bigr) =  \ker\rd_{\phi(u)} s_J    \qquad \forall u\in U_{IJ}.
$$
\item
The linear embedding $\Hat\phi:E_I\to E_J$ identifies the cokernels,
$$
\forall u\in U_{IJ} : \qquad
E_I = \im\rd_u s_I \oplus C_{u,I}  \quad \Longrightarrow \quad E_J = \im \rd_{\phi(u)} s_J \oplus \Hat\phi(C_{u,I}).
$$
\end{enumerate}
\end{defn}

\begin{remark}\label{rmk:tchange} \rm
(i)
Combining the above definitions, we find that a coordinate change $\Hat\Phi:\bK_I|_{U_{IJ}}\to\bK_J$ is a restriction of $\bK_I$ to $F_I\cap F_J\neq\emptyset$ together with embeddings
$\phi: U_{IJ}\to U_J$, $\Hat\phi: E_I\to E_J$ as follows:
\begin{itemize}
\item
The {\bf domain} of the coordinate change is an open subset $U_{IJ}\subset U_I$ such that
$s_I^{-1}(0)\cap U_{IJ} = \psi_I^{-1}(F_I\cap F_J)$.
\item 
The {\bf map} of the coordinate change is a pair $\Phi=(\phi,\Hat\phi)$ of a smooth embedding
$\phi: U_{IJ}\to U_J$ and a linear embedding $\Hat\phi: E_I\to E_J$ that satisfy the following.
\begin{enumerate}
\item 
They satisfy the index condition in Definition~\ref{def:change} above.
\item
They intertwine the sections in the sense that $s_J \circ \phi  = \Hat\phi \circ s_I|_{U_{IJ}}$.
\item
The embedding $\phi$ restricts to the transition map induced from the footprints in $X$ in the sense that
$\phi|_{\psi_I^{-1}(F_I\cap F_J)}=\psi_J^{-1} \circ\psi_I : U_{IJ}\cap s_I^{-1}(0) \to s_J^{-1}(0)$.
\end{enumerate}
\end{itemize}
In particular, the following diagrams commute:
\begin{align} 
\label{eq:map-square}
& \qquad\qquad
 \begin{array} {ccc}
{E_I|_{U_{IJ}}}& \stackrel{\Hat\phi} \longrightarrow &
{E_J} 
\phantom{\int_Quark}  \\
\phantom{sp} \uparrow {s_I}&&\uparrow {s_J} \phantom{spac}\\
\phantom{s}{U_{IJ}} & \stackrel{\phi} \longrightarrow &{U_J} \phantom{spacei}
\end{array} 
\qquad
 \begin{array} {ccc}
{U_{IJ}\cap s_I^{-1}(0)} & \stackrel{\phi} \longrightarrow &{s_J^{-1}(0)} \phantom{\int_Quark} \\
\phantom{spa} \downarrow{\psi_I}&&\downarrow{\psi_J} \phantom{space} \\
\phantom{s}{X} & \stackrel{{\rm Id}} \longrightarrow &{X}. \phantom{spaceiiii}
\end{array}
\end{align}

\MS\NI
(ii)
Every coordinate change $\Hat\Phi:\bK_I|_{U_{IJ}}\to\bK_J$ in the sense of Definition~\ref{def:change} or (i) above induces a topological coordinate change (which we again denote $\Hat\Phi:\bK_I|_{U_{IJ}}\to\bK_J$)
between the induced topological charts given by the same domain $U_{IJ}$ and map 
$$
\Hat\Phi = \phi \times \Hat\phi \,:\;  \E_I|_{U_{IJ}}= U_{IJ}\times E_I \; \to \; U_J\times E_J = \E_J .
$$

\vspace{-6mm}
$\hfill\er$
\end{remark}

\begin{example}\label{ex:change}  \rm
Here is the prototypical example of a coordinate change with $I=\{1\}\subset J=\{1,2\}$ 
between charts on the finite set $X = \{-1,0,1\} $.
Consider the two Kuranishi charts with $U_1=(-2,2)$, $E_1=\R$, $s_1(x)=(x^2-1)x^2$ and 
$U_{12}=(-1,2)\times (-1,1)$, $E_{12}=\R^2$, $s_{12} (x,y) = \bigl( (x^2-1)x^2 , y \bigr)$, with footprint maps given by the obvious identification.
Thus their footprints are $\{0,1\}= F_{12}\subset F_1= \{-1,0,1\} = X$, and both charts
have dimension $0$ although their domains and obstruction spaces are not locally diffeomorphic.
A natural coordinate change that extends the identification $s_1^{-1}(0)\supset \{0,1\}\cong \{(0,0),(1,0)\} =s_{12}^{-1}(0)$ is the inclusion $\phi: x\mapsto (x,0)$ of $U_{1,12}:=(-1,2)$  onto $(-1,2)\times \{0\}\subset U_{12}$ together with $\Hat\phi: x\mapsto (x,0)$.
Then all required diagrams commute and the index condition holds as follows:
The kernel $\ker \rd_0 s_1 = \rT_0 U_1 = \R$,  resp.\ $\ker \rd_x s_1 =\{0\}$ for $x\neq 0$,
is identified by $\rd \phi$ with $\ker \rd_{(0,0)} s_{12} = \R\times \{0\}\subset \rT_{(0,0)} U_{12}$, resp.\
$\ker \rd_{(x,0)} s_{12} = \{(0,0)\}$.
Nontrivial cokernel only appears at $0\in U_{1,12}$ and $(0,0)\in U_2$, where $\Hat\phi(E_1) = \R\oplus \{0\}\subset E_{12}$ is a complement to the image $\{(0,v)\in E_{12} \,|\, v\in \R\}$ of $\rd_{(0,0)} s_{12}$.
$\hfill\er$
\end{example}

Note that coordinate changes are in general {\it unidirectional} since the map $U_{IJ}\to U_J$ is not assumed to have open image.
Note also that the footprint of the intermediate chart $\bK_I|_{U_{IJ}}$ is always the full intersection $F_I\cap F_J$. 
Moreover, in Kuranishi atlases we will only have coordinate changes when $F_J\subset F_I$, so that this intersection is $F_I\cap F_J=F_J$.
By abuse of notation, we often denote a coordinate change by $\Hat\Phi: \bK_I\to \bK_J$, thereby indicating the choice of a domain $U_{IJ}\subset U_I$ and 
maps $\phi:U_{IJ}\to U_J$, $\Hat\phi:E_I\to E_J$.
Further, for clarity we usually add subscripts, writing $\Hat\Phi_{IJ} = (\phi_{IJ},\Hat\phi_{IJ}): \bK_I\to \bK_J$.

The following lemma shows that the index condition is in fact equivalent to a tangent bundle condition which was first introduced, in a weaker version, by \cite{FO}, and formalized in the present version by \cite{J1}.
We have chosen to present it as an index condition, since that is closer to the basic motivating question of how to associate canonical (equivalence classes of) Kuranishi atlases to moduli spaces described in terms of nonlinear Fredholm operators, see 
Remark~\ref{rmk:smart}~(ii).

\begin{lemma} \label{le:change}
The index condition in Definition~\ref{def:change} is equivalent to the {\bf tangent bundle condition}, which requires isomorphisms for all $v=\phi_{IJ}(u)\in\phi_{IJ}(U_{IJ})$, 
\begin{equation}\label{tbc}
\rd_v s_J : \;\quotient{\rT_v U_J}{\rd_u\phi_{IJ}(\rT_u U_I)} \;\stackrel{\cong}\longrightarrow \; \quotient{E_J}{\Hat\phi_{IJ}(E_I)},
\end{equation}
or equivalently at all (suppressed) base points as above
\begin{equation}\label{eq-tbc}
E_J=\im\rd s_J + \im\Hat\phi_{IJ} \qquad\text{and}\qquad
\im\rd s_J \cap \im\Hat\phi_{IJ} = \Hat\phi_{IJ}(\im\rd s_I).
\end{equation}
Moreover, the tangent bundle condition implies that the charts $\bK_I, \bK_J$ have the same dimension, and that $\phi_{IJ}(U_{IJ})$ is an open subset of $s_J^{-1}(\Hat\phi_{IJ}(E_I))$ with
 \begin{equation}\label{inftame}
 \im\rd_u\phi_{IJ}=(\rd_v s_J)^{-1}\bigr(\Hat\phi_{IJ}(E_I)\bigl)\qquad \forall 
 v=\phi_{IJ}(u)\in\phi_{IJ}(U_{IJ}).
 \end{equation}
\end{lemma}

\begin{proof}
We will suppress most base points in the notation,
and for simplicity write $\phi,\Hat\phi$ instead of $\phi_{IJ}, \Hat \phi_{IJ}$.
The tangent bundle condition, in particular the assumption that \eqref{tbc} is well defined resp.\ injective, implies that $\rd s_J(\im\rd \phi)\subset \Hat\phi(E_I)$
and $\rd s_J^{-1}\bigl(\Hat\phi(E_I)\bigr)\subset \im\rd \phi$, which proves \eqref{inftame}.

To see that the tangent bundle condition \eqref{tbc} implies the index condition, first note that the compatibility with sections, $\Hat\phi\circ\rd s_I = \rd s_J \circ \rd\phi$ implies
$$
\Hat\phi(\im\rd s_I)\subset\im\rd s_J , \qquad \rd\phi \bigl(\ker\rd s_I \bigr) \subset  \ker\rd s_J  .
$$
Since $\phi$ and $\Hat\phi$ are embeddings, this implies dimension differences $d,d'\geq 0$ in
$$
\dim\im\rd s_I  + d = \dim\im\rd s_J, \qquad
\dim\ker\rd s_I + d' = \dim\ker\rd s_J .
$$
The fact that \eqref{tbc} is an isomorphism implies the equality of dimensions
\begin{align*}
\dim E_J  - \dim E_I
&= \dim U_J  - \dim U_I \\
&= \dim \ker\rd s_J + \dim\im\rd s_J - \dim \ker\rd s_I - \dim\im\rd s_I
\\ &= d+d' .
\end{align*}
Here the first line already implies $\dim E_J - \dim U_J = \dim E_I - \dim U_I$, so that the Kuranishi charts have equal dimensions.
Moreover, if we pick a representative space $C_I$ for the cokernel, i.e.\ $E_I = \im\rd s_I \oplus C_{I}$, 
then the surjectivity of $\rd s_J$ in \eqref{tbc}  gives
\begin{equation}\label{EJ}
E_J = \im \rd s_J + \Hat\phi(E_I)
= \im \rd s_J + \bigl( \Hat\phi(C_I) \oplus \Hat\phi(\im\rd s_I) \bigr) = \im \rd s_J + \Hat\phi(C_I) ,
\end{equation}
where
$$
\dim \Hat\phi(C_I) = \dim E_I - \dim\im\rd s_I
= \dim E_J - \dim\im\rd s_J - d' .
$$
Thus the sum \eqref{EJ} must be direct and $d'=0$, which implies the identification of cokernels and kernels.

Conversely, to see that the index condition implies the tangent bundle condition let again
$C_I\subset E_I$ be a complement of $\im\rd s_I$. Then compatibility of the sections $s_J \circ\phi = \Hat\phi \circ s_I$ implies
$$
\Hat\phi(E_I) = \Hat\phi (\im\rd s_I) \oplus \Hat\phi(C_I)
= \rd s_J(\im\rd\phi) \oplus \Hat\phi(C_I)  .
$$
Moreover, let $N_u\subset \rT U_J$ be a complement of $\im \rd_u \phi$,
then the identification of cokernels takes the form
$$
E_J = \rd s_J(N_u) \oplus \rd s_J(\im\rd\phi) \oplus \Hat\phi(C_I) = \rd s_J(N_u) \oplus \Hat\phi(E_I).
$$
This shows that \eqref{tbc} is surjective, and for injectivity it remains to check injectivity of $\rd s_J |_{N_u}$. The latter holds since the identification of kernels implies $\ker\rd s_J \subset \im\rd\phi$.

To check the equivalence of \eqref{eq-tbc} and \eqref{tbc} note that the first condition in \eqref{eq-tbc} is the surjectivity of \eqref{tbc}, while the injectivity is equivalent to $\im\rd s_J \cap \im\Hat\phi \subset \rd s_J(\im\rd\phi)$. 
The latter equals $\Hat\phi(\im\rd s_I)$ by the compatibility $s_J\circ\phi=\Hat\phi\circ s_I$ of sections. So \eqref{eq-tbc} implies \eqref{tbc}, and for the converse it remains to check that \eqref{tbc} implies equality of the above inclusion. This follows from a dimension count as in \eqref{EJ}.

Finally, to see that $\phi(U_{IJ})$ is an open subset of $s_J^{-1}(\Hat\phi(E_I))$, note that they are both submanifolds of $U_J$ (since the first is the image of an embedding and \eqref{tbc} says that the latter is cut out transversely), and \eqref{inftame} identifies their tangent spaces. 
\end{proof}

The next lemmas provide restrictions and compositions of coordinate changes.

\begin{lemma} \label{le:restrchange}
Let $\Hat\Phi:\bK_I|_{U_{IJ}}\to \bK_J$ be a coordinate change from $\bK_I$ to $\bK_J$, and let $\bK'_I=\bK_I|_{U'_I}$, $\bK'_J=\bK_J|_{U'_J}$ be restrictions of the Kuranishi charts to open subsets $F'_I\subset F_I, F'_J\subset F_J$ with $F_I'\cap F_J'\ne \emptyset$.
Then a {\bf restricted coordinate change} $\Hat\Phi|_{U'_{IJ}}: \bK'_I \to \bK'_J$
is given by any choice of open subset $U'_{IJ}\subset U_{IJ}$ of the domain such that
$$
U'_{IJ} \subset U'_I\cap \phi^{-1}(U'_J) , \qquad \psi_I(s_I^{-1}(0)\cap U'_{IJ}) = F'_I \cap F'_J ,
$$
and the map
$\Hat\Phi|_{U'_{IJ}} = \phi|_{U'_{IJ}}\times \Hat\phi : U'_{IJ}\times E_I \to U_J\times E_J$.
In particular, the index condition is preserved under restriction.
\end{lemma}

\begin{proof}
Applying \cite[Lemma~2.2.4]{MW:top} to the induced topological coordinate change $\Hat\Phi:\bK_I|_{U_{IJ}}\to \bK_J$ shows that the restriction is a topological coordinate change. 
It still splits in the way required by Definition~\ref{def:chart}, and the index condition is evidently preserved under restriction of the maps to open subdomains.
\end{proof}

\begin{lemma} \label{le:cccomp}
Let $\bK_I,\bK_J,\bK_K$ be Kuranishi charts such that $F_I\cap F_K \subset F_J$, and let $\Hat\Phi_{IJ}: \bK_I\to \bK_J$ and $\Hat\Phi_{JK}: \bK_J\to \bK_K$ be coordinate changes.
(That is, we are given restrictions
$\bK_I|_{U_{IJ}}$ to $F_I\cap F_J$ and $\bK_J|_{U_{JK}}$ to $F_J\cap F_K$ and 
embeddings
$\phi_{IJ}: U_{IJ}\to U_J$, $\Hat\phi_{IJ}: E_I \to E_J$, 
$\phi_{JK}: U_{JK}\to U_K$, $\Hat\phi_{JK}: E_J\to E_K$.)
Then the following holds.
\begin{enumerate}
\item
The domain $U_{IJK}:=\phi_{IJ}^{-1}(U_{JK}) \subset U_I$ defines a restriction $\bK_I|_{U_{IJK}}$
to $F_I \cap F_K$.
\item
The compositions ${\phi_{IJK}:=\phi_{JK}\circ\phi_{IJ}: U_{IJK}\to U_K}$ and
$\Hat\phi_{IJK}:=\Hat\phi_{JK}\circ\Hat\phi_{IJ}: E_I \to E_K$ 
define a coordinate change ${\Hat\Phi_{IJK}:\bK_I|_{U_{IJK}}\to\bK_K}$.
\end{enumerate}
We denote the induced {\bf composite coordinate change} $\Hat\Phi_{IJK}$ by
$$
\Hat\Phi_{JK}\circ \Hat\Phi_{IJ}
:=\Hat\Phi_{IJK} : \; \bK_I|_{U_{IJK}} \; \to\; \bK_K.
$$
\end{lemma}
\begin{proof}
Applying \cite[Lemma~2.2.5]{MW:top} to the induced topological coordinate changes shows that $\Hat\Phi_{IJK}$ is a topological coordinate change. 
By construction it splits in the way required by Definition~\ref{def:chart},
so it remains to show that the index condition is preserved under composition.
The kernel identifications
$\rd\phi_{IJ} \bigl(\ker\rd s_I \bigr) =  \ker\rd s_J $ and $\rd\phi_{JK} \bigl(\ker\rd s_J \bigr) =  \ker\rd s_K$ imply
$$
\rd\bigl(\phi_{JK} \circ \phi_{IJ} \bigr) \bigl(\ker\rd s_I \bigr) =
\bigl(\rd\phi_{JK} \circ \rd\phi_{IJ} \bigr) \bigl(\ker\rd s_I \bigr) =  \ker\rd s_K .
$$
Assuming $E_I = \im\rd s_I \oplus C_{I}$, the cokernel identification of $\Phi_{IJ}$ implies
$ E_J = \im \rd s_J \oplus C_J$ with  $C_J=\Hat\phi(C_I)$, so that  the cokernel identification of $\Phi_{JK}$ implies
$$
E_K = \im \rd s_K \oplus \Hat\phi_K(C_J) = \im \rd s_K \oplus (\Hat\phi_K\circ\Hat\phi_J)(C_I).
$$
Thus $\Hat\Phi_{IJK}$ identifies kernels and cokernels, and hence is a smooth coordinate change.
 \end{proof}

Finally, we introduce two notions of equivalence between coordinate changes that may not have the same domain, and show that they are compatible with composition.

\begin{defn} \label{def:overlap} {\rm $\!\!$ \cite[Definition~2.2.6]{MW:top}}
Let $\Hat\Phi^\al :\bK_I|_{U^\al_{IJ}}\to \bK_J$ and  $\Hat\Phi^\be:\bK_I|_{U^\be_{IJ}}\to \bK_J$ be coordinate changes.
\begin{itemlist}
\item
We say the coordinate changes are {\bf equal on the overlap} and write $\Hat\Phi^\al\approx\Hat\Phi^\be$, if the restrictions of Lemma~\ref{le:restrchange} to $U'_{IJ}:=U^\al_{IJ}\cap U^\be_{IJ}$ yield equal maps
$$
\Hat\Phi^\al|_{U'_{IJ}} = \Hat\Phi^\be|_{U'_{IJ}}  .
$$
\item
We say that $\Hat\Phi^\be$ {\bf extends} $\Hat\Phi^\al$ and write $\Hat\Phi^\al\subset\Hat\Phi^\be$,
if $U_{IJ}^\al\subset U_{IJ}^\be$ and the restriction of Lemma~\ref{le:restrchange} yields equal maps
$$
\Hat\Phi^\be|_{U_{IJ}^\al} = \Hat\Phi^\al .
$$
\end{itemlist}
\end{defn}

%for reference - don't seem to need this
%
%\begin{lemma}  \label{le:cccompoverlap}
%Let $\bK_I,\bK_J,\bK_K$ be Kuranishi charts such that $F_I\cap F_K\subset F_J$,  and suppose $\Hat\Phi^\al_{IJ}\approx \Hat\Phi^\be_{IJ}:\bK_I\to \bK_J$ and $\Hat\Phi^\al_{JK}\approx\Hat\Phi^\be_{JK}: \bK_J\to \bK_K$ are coordinate changes that are equal on the overlap.
%Then their compositions as defined in Lemma~\ref{le:cccomp} are equal on the overlap, $\Hat\Phi^\al_{JK}\circ \Hat\Phi^\al_{IJ}\approx \Hat\Phi^\be_{JK}\circ \Hat\Phi^\be_{IJ}$.
%
%Moreover, if $\Hat\Phi^\al_{IJ}\subset \Hat\Phi^\be_{IJ}$ and $\Hat\Phi^\al_{JK}\subset\Hat\Phi^\be_{JK}$ are extensions, then
%$\Hat\Phi^\al_{JK}\circ \Hat\Phi^\al_{IJ}\subset \Hat\Phi^\be_{JK}\circ \Hat\Phi^\be_{IJ}$
%is an extension as well.
%\end{lemma}
%\begin{proof}
%This follows directly from applying \cite[TODO]{MW:top} to the induced topological coordinate changes.
%\end{proof}

%%%%%%%%%%%%%%%%%%%%%%%%%%%%%%%%%%%%%%%%%%%%%%%%%%%%%%%%%%%%%%%%
\section{Kuranishi atlases and cobordisms with trivial isotropy}\label{s:Ks}
%%%%%%%%%%%%%%%%%%%%%%%%%%%%%%%%%%%%%%%%%%%%%%%%%%%%%%%%%%%%%%%%

With the preliminaries of \S\ref{s:chart} in hand, we can now define a notion of Kuranishi atlas with trivial isotropy on a compact metrizable space $X$, which will be fixed throughout this section.
As before, we work exclusively in the case of trivial isotropy and hence drop this qualifier from the wording.
We will however need to make distinctions between smooth Kuranishi atlases and the more general topological Kuranishi atlases
introduced in \cite{MW:top}. 
Here and in the following we will at times add the qualifier `smooth' to stress the fact that we are in the context with smooth structure on domains, rather than the purely topological context.
We first define the notion of Kuranishi atlas $\Kk$ and from it construct a virtual neighbourhood $|\Kk|$ for $X$. 
However, we will note that $|\Kk|$ need not be Hausdorff, that the maps from the domains $U_I$ of the charts into $|\Kk|$ need not be injective, and that $|\Kk|$ is neither metrizable or locally compact, except in very simple cases.
Moreover, in practice one can construct only weak Kuranishi atlases in the sense of Definition~\ref{def:K}, although they do often have the additivity property of Definition~\ref{def:Ku2}.
The main result of this section is then Theorem~\ref{thm:K}, which in particular states that given an additive weak Kuranishi atlas one can construct a Kuranishi atlas $\Kk$, whose neighbourhood $|\Kk|$ is Hausdorff,
has the injectivity property, and can be equipped with a second metric topology that is compatible with the given topology on the domains $U_I$.
Moreover, this refinement is well defined up to a notion of concordance that we develop as part of a general theory of Kuranishi cobordisms in \S\ref{ss:Kcobord}.

%%%%%%%%%%%%%%%%%%%%%%%%%%%%%%%%%%%%%%%%%%%%%%%%%%%%%%%%%%%
\subsection{Covering families, transition data,
and the virtual neighbourhood}\label{ss:Ksdef} \hspace{1mm}\\ \vspace{-3mm}
%%%%%%%%%%%%%%%%%%%%%%%%%%%%%%%%%%%%%%%%%%%%%%%%%%%%%%%%%%%

There are various ways that one might try to define a ``Kuranishi structure", but in practice every such structure on a compact moduli space of holomorphic curves is constructed from a covering family of basic charts with certain compatibility conditions akin to our notion of Kuranishi atlas.
We express the compatibility in terms of a further collection of charts for overlaps, and will discuss three different versions of a cocycle condition.
We compare our definition with others in Remark~\ref{rmk:otherK}.
The basic building blocks of our notion of Kuranishi atlases are the following.

\begin{defn}\label{def:Kfamily} {\rm $\!\!$ \cite[Definition~2.3.1]{MW:top}}
Let $X$ be a compact metrizable space.
\begin{itemlist}
\item
A {\bf covering family of basic charts} for $X$ is a finite collection $(\bK_i)_{i=1,\ldots,N}$ of Kuranishi charts for $X$ whose footprints cover $X=\bigcup_{i=1}^N F_i$.
\item
{\bf Transition data} for a covering family $(\bK_i)_{i=1,\ldots,N}$ is a collection of Kuranishi charts $(\bK_J)_{J\in\Ii_\Kk,|J|\ge 2}$ and coordinate changes $(\Hat\Phi_{I J})_{I,J\in\Ii_\Kk, I\subsetneq J}$ as follows:
\begin{enumerate}
\item
$\Ii_\Kk$ denotes the set of subsets $I\subset\{1,\ldots,N\}$ for which the intersection of footprints is nonempty,
$$
F_I:= \; {\textstyle \bigcap_{i\in I}} F_i  \;\neq \; \emptyset \;;
$$
\item
$\bK_J$ is a Kuranishi chart for $X$ with footprint $F_J=\bigcap_{i\in J}F_i$ for each $J\in\Ii_\Kk$ with $|J|\ge 2$, and for one element sets $J=\{i\}$ we denote $\bK_{\{i\}}:=\bK_i$;
\item
$\Hat\Phi_{I J}$ is a coordinate change $\bK_{I} \to \bK_{J}$ for every $I,J\in\Ii_\Kk$ with $I\subsetneq J$.
\end{enumerate}
\end{itemlist}
 \end{defn}

The transition data for a covering family automatically satisfies a cocycle condition on the zero sets, where due to the footprint maps to $X$ we have for $I\subset J \subset K$
$$
\phi_{J K}\circ \phi_{I J}
= \psi_K^{-1}\circ\psi_J\circ\psi_J^{-1}\circ\psi_I
= \psi_K^{-1}\circ\psi_I
= \phi_{I K}
\qquad \text{on}\; s_I^{-1}(0)\cap U_{IK} .
$$
Since there is no natural ambient topological space into which the entire domains of the Kuranishi charts map, the cocycle condition on the complement of the zero sets has to be added as an axiom. 
We will always impose $\Hat\phi_{J K}\circ \Hat\phi_{I J} = \Hat\phi_{I K}$
for the linear embeddings between obstruction spaces.
However for the embeddings between the domains of the charts there are three natural notions of cocycle condition with varying requirements on the domains of the coordinate changes.

\begin{defn}  \label{def:cocycle} {\rm $\!\!$ \cite[Definition~2.3.2]{MW:top}}
Let $\Kk=(\bK_I,\Hat\Phi_{I J})_{I,J\in\Ii_\Kk, I\subsetneq J}$ be a tuple of basic charts and transition data. Then for any $I,J,K\in\Ii_K$ with  $I\subsetneq J \subsetneq K$ we define the composed coordinate change $\Hat\Phi_{J K}\circ \Hat\Phi_{I J} : \bK_{I}  \to \bK_{K}$ as in Lemma~\ref{le:cccomp} with domain $\phi_{IJ}^{-1}(U_{JK})\subset U_I$.
We then use the notions of Definition~\ref{def:overlap} to say that the triple of coordinate changes
$\Hat\Phi_{I J}, \Hat\Phi_{J K}, \Hat\Phi_{I K}$ satisfies the
\begin{itemlist}
\item {\bf weak cocycle condition}
if $\Hat\Phi_{J K}\circ \Hat\Phi_{I J} \approx \Hat\Phi_{I K}$, i.e.\ the coordinate changes are equal on the overlap;
\item {\bf cocycle condition}
if $\Hat\Phi_{J K}\circ \Hat\Phi_{I J} \subset \Hat\Phi_{I K}$, i.e.\  $\Hat\Phi_{I K}$ extends the composed coordinate change;
\item {\bf strong cocycle condition}
if $\Hat\Phi_{J K}\circ \Hat\Phi_{I J} = \Hat\Phi_{I K}$ are equal as coordinate changes.
\end{itemlist}
More explicitly, 
each condition requires $\Hat\phi_{J K}\circ \Hat\phi_{I J} = \Hat\phi_{I K}$ and in addition
\begin{itemlist}
\item 
 the weak cocycle condition requires
 \begin{equation} \label{eq:wc}
\qquad
\phi_{J K}\circ \phi_{I J} = \phi_{I K}
\qquad \text{on}\;\;
\phi_{IJ}^{-1}(U_{JK}) \cap U_{IK} ;
\end{equation}
\item 
 the cocycle condition requires \eqref{eq:wc} and $U_{IJK}:=\phi_{IJ}^{-1}(U_{JK}) \subset U_{IK}$; 
\item
the strong cocycle condition requires \eqref{eq:wc} and $U_{IJK}:=\phi_{IJ}^{-1}(U_{JK}) = U_{IK}$.
\end{itemlist}
\end{defn}

The relevance of these versions is that the weak cocycle condition can be achieved in practice by constructions of finite dimensional reductions for holomorphic curve moduli spaces, whereas the strong cocycle condition is needed for our construction of a virtual moduli cycle in \S\ref{s:VMC} from perturbations of the sections in the Kuranishi charts.
The cocycle condition is an intermediate notion which is too strong to be constructed in practice and too weak to induce a VMC, but it is the minimal assumption under which we can below formulate Kuranishi atlases categorically and from this obtain a virtual neighbourhood of $X$ into which all Kuranishi domains map.

\begin{defn}\label{def:K} {\rm $\!\!$ \cite[Definitions~2.3.4, 3.1.1]{MW:top}}
A {\bf weak Kuranishi atlas of dimension $\mathbf d$} on a compact metrizable space $X$ is a tuple
$$
\Kk=\bigl(\bK_I,\Hat\Phi_{I J}\bigr)_{I, J\in\Ii_\Kk, I\subsetneq J}
$$
of a covering family of basic charts with transition data $\Kk=\bigl(\bK_I,\Hat\Phi_{I J}\bigr)_{I, J\in\Ii_\Kk, I\subsetneq J}$ as in Definition \ref{def:Kfamily}, 
consisting of $d$-dimensional Kuranishi charts and coordinate changes that satisfy the {\it weak cocycle condition} 
$\Hat\Phi_{J K}\circ \Hat\Phi_{I J} \approx \Hat\Phi_{I K}$ for every triple $I,J,K\in\Ii_K$ with $I\subsetneq J \subsetneq K$.

A {\bf Kuranishi atlas of dimension $\mathbf d$} on $X$ is a weak Kuranishi atlas of dimension $d$ whose coordinate changes satisfy the {\it cocycle condition} $\Hat\Phi_{J K}\circ \Hat\Phi_{I J} \subset \Hat\Phi_{I K}$ for every triple $I,J,K\in\Ii_K$ with $I\subsetneq J \subsetneq K$.
\end{defn}

While constructions of transition data in practice, e.g.\ \S\ref{ss:gw} and \cite{MW:GW,Mcn},
 only satisfy the weak cocycle condition, they use a sum construction which has the effect of adding the obstruction spaces and thus yields the following additivity property.
Here we simplify notation by writing $\Hat\Phi_{i I}:= \Hat\Phi_{\{i\} I}$ for the coordinate change $\bK_i =\bK_{\{i\}} \to \bK_I$ where $i\in I$.

\begin{defn}\label{def:Ku2}  
A (weak) Kuranishi atlas $\Kk$ is {\bf additive} if for each $I\in \Ii_\Kk$ the linear embeddings $\Hat\phi_{i I}:E_i \to E_I$ induce an isomorphism
$$
{\textstyle \prod_{i\in I}} \;\Hat\phi_{iI}: \; {\textstyle \prod_{i\in I}} \; E_i \;\stackrel{\cong}\longrightarrow \; E_I  ,
\qquad\text{or equivalently} \qquad
E_I = {\textstyle \bigoplus_{i\in I}} \; \Hat\phi_{iI}(E_i) .
$$
In this case we abbreviate notation by $s_J^{-1}(E_I): = s_J^{-1}\bigl(\Hat\phi_{IJ}(E_I)\bigr)$, and 
we set $E_\emptyset:=\{0\}$.
\end{defn}

The topological consequences of additivity will be explored in \S\ref{ss:tame}.

\begin{rmk}\label{rmk:Ku}\rm  
(i)
We have assumed from the beginning that $X$ is compact and metrizable.
Some version of compactness is essential in order for $X$ to define a VFC, but one might hope to weaken the metrizability assumption.  
However, any compact space $X$ that is covered by Kuranishi charts is automatically metrizable, see \cite[Remark~2.3.5]{MW:top}. 

\MS
\NI
(ii)
In \cite[Definitions~2.3.4, 3.1.1]{MW:top} we introduce the notion of {\bf (weak) topological Kuranishi atlas} as a tuple of basic charts and transition data as in Definition~\ref{def:Kfamily}, consisting of topological Kuranishi charts and topological coordinate changes that satisfy the (weak) cocycle condition.

\MS
\NI
(iii)
Every (weak) Kuranishi atlas $\Kk$ induces a (weak) topological Kuranishi atlas (which we again denote by $\Kk$), given by the induced topological Kuranishi charts and coordinate changes of Remarks~\ref{rmk:topchart}~(i)  and \ref{rmk:tchange}~(ii).
In particular, the (weak) cocycle condition transfers from smooth to topological case.
$\hfill\er$
\end{rmk}

It is useful to think of the domains and obstruction spaces of a Kuranishi atlas as forming the following categories.
These were introduced in \cite{MW:top} for the topological context, and now are defined as the categories associated to the induced topological Kuranishi atlases, which now have a more special form and additional smooth structure.

\begin{defn}\label{def:catKu} {\rm $\!\!$ \cite[Definition~2.3.6]{MW:top}}
Given a Kuranishi atlas $\Kk$ we define its {\bf domain category} $\bB_\Kk$ to consist of
the space of objects\footnote{
When forming categories such as $\bB_\Kk$, we  always take the space of objects 
to be the disjoint union of the domains $U_I$, even if we happen to have defined the sets $U_I$ 
as subsets of some larger space such as $\R^2$ or a space of maps as in the Gromov--Witten case.
Similarly, the morphism space is a disjoint union of the $U_{IJ}$ even though $U_{IJ}\subset U_I$ for all $J\supset I$.}
$$
\Obj_{\bB_\Kk}:= \bigsqcup_{I\in \Ii_\Kk} U_I \ = \ \bigl\{ (I,x) \,\big|\, I\in\Ii_\Kk, x\in U_I \bigr\}
$$
and the space of morphisms
$$
\Mor_{\bB_\Kk}:= \bigsqcup_{I,J\in \Ii_\Kk, I\subset J} U_{IJ} \ = \ \bigl\{ (I,J,x) \,\big|\, I,J\in\Ii_\Kk, I\subset J, x\in U_{IJ} \bigr\}.
$$
Here we denote $U_{II}:= U_I$ for $I=J$, and for $I\subsetneq J$ use
the domain $U_{IJ}\subset U_I$ of the restriction $\bK_I|_{U_{IJ}}$ to $F_J$
that is part of the coordinate change $\Hat\Phi_{IJ} : \bK_I|_{U_{IJ}}\to \bK_J$.

Source and target of these morphisms are given by
$$
(I,J,x)\in\Mor_{\bB_\Kk}\bigl((I,x),(J,\phi_{IJ}(x))\bigr),
$$
where $\phi_{IJ}: U_{IJ}\to U_J$ is the  embedding given by $\Hat\Phi_{I J}$, and we denote $\phi_{II}:={\rm id}_{U_I}$.
Composition\footnote
{
Note that we write compositions in the categorical ordering here.} is defined by
$$
\bigl(I,J,x\bigr)\circ \bigl(J,K,y\bigr)
:= \bigl(I,K,x\bigr)
$$
for any $I\subset J \subset K$ and $x\in U_{IJ}, y\in  U_{JK}$ such that $\phi_{IJ}(x)=y$.

The {\bf obstruction category} $\bE_\Kk$ is defined in complete analogy to $\bB_\Kk$ to consist of
the spaces of objects $\Obj_{\bE_\Kk}:=\bigsqcup_{I\in\Ii_\Kk} U_I \times E_I$ and morphisms
$$
\Mor_{\bE_\Kk}: = \bigl\{ (I,J,x,e) \,\big|\, I,J\in\Ii_\Kk, I\subset J,  x\in U_{IJ}, e\in E_I \bigr\},
$$
with source and target maps
$$
 (I,J,x,e) \mapsto (I,x,e) , \qquad  (I,J,x,e) \mapsto (J,\phi_{IJ}(x),\Hat\phi_{IJ}(e)).
$$
\end{defn}

Note that if $\Kk$ is only a weak Kuranishi atlas then we cannot define its domain category $\bB_{\Kk}$ as in Definition~\ref{def:catKu} since the given set of morphisms is not closed under composition.  We will deal with this by simply not considering this category unless $\Kk$ is a Kuranishi atlas, i.e.\ satisfies the standard cocycle condition in Definition~\ref{def:cocycle}.
In that case, the categories $\bB_\Kk, \bE_\Kk$ are well defined by 
Lemma~\ref{le:realization} below, and we can also express the further parts of a Kuranishi atlas in categorical terms:

\begin{itemlist}
\item
The obstruction category $\bE_\Kk$ is a trivial bundle over $\bB_\Kk$ in the sense that there is a functor
$\pr_\Kk:\bE_\Kk\to\bB_\Kk$ that is given on objects and morphisms by projections
$(I,x,e)\mapsto (I,x)$ and $(I,J,x,e)\mapsto(I,J,x)$ that are locally trivial with fibers $E_I$.

\item
The zero sections and sections $s_I$ induce two smooth sections of this bundle, i.e.\ functors $0_\Kk:\bB_\Kk\to \bE_\Kk$ and $\s_\Kk:\bB_\Kk\to \bE_\Kk$ which act smoothly on the spaces of objects and morphisms, and whose composite with the projection $\pr_\Kk: \bE_\Kk \to \bB_\Kk$ is the identity. More precisely, $\s_\Kk$ is given by $(I,x)\mapsto (I,x,s_I(x))$ on objects and by $(I,J,x)\mapsto (I,J,x,s_I(x))$ on morphisms, and analogously $0_\Kk$ is given by $(I,x)\mapsto (I,x,0)$ on objects and by $(I,J,x)\mapsto (I,J,x,0)$ on morphisms.

\item
The zero sets of the sections $\bigsqcup_{I\in\Ii_\Kk} \{I\}\times s_I^{-1}(0)\subset\Obj_{\bB_\Kk}$ form a very special strictly full subcategory $\s_\Kk^{-1}(0_\Kk)$ of $\bB_\Kk$. Namely, $\bB_\Kk$ splits into the subcategory $\s_\Kk^{-1}(0_\Kk)$ and its complement (given by the full subcategory with objects  $\{ (I,x) \,|\, s_I(x)\neq 0 \}$) in the sense that there are no morphisms of $\bB_\Kk$ between the underlying sets of objects.

\item
The footprint maps $\psi_I$ give rise to a surjective functor $\psi_\Kk: \s_\Kk^{-1}(0_\Kk) \to \bX$ 
to the category $\bX$ with object space $X$ and trivial morphism spaces.
It is given by $(I,x)\mapsto \psi_I(x)$ on objects and by $(I,J,x)\mapsto {\rm id}_{\psi_I(x)}$ on morphisms.
\end{itemlist}

\begin{remark}\label{rmk:Kgroupoid}\rm
All object and morphism spaces of the categories $\bB_\Kk$ and $\bE_\Kk$ are disjoint unions of smooth manifolds, and all structural maps are smooth.
Moreover, since $\Kk$ has trivial isotropy, all sets of morphisms in $\bB_\Kk$ or $\bE_\Kk$ between fixed objects consist of at most one element. 
However, because there could be coordinate changes  $\phi_{IJ}:U_{IJ}\to U_J$ with $\dim U_{IJ} < \dim U_J$, the target map  $t:\Mor_{\bB_\Kk}\to \Obj_{\bB_\Kk}$ is not in general a local diffeomorphism, although it is locally injective.
In other words, the category $\bB_\Kk$ is not  \'etale,\footnote
{
The notion of \'etale groupoid is reviewed in Remark~\ref{rmk:grp}.} 
though it has some similar features.  Moreover, one cannot in general complete $\bB_\Kk$ to a groupoid by adding inverses and compositions, while keeping the property that the morphism space is a union of smooth manifolds. 
The problem here is that the inclusion of inverses of the coordinate changes, and their compositions, may yield singular spaces of morphisms. Indeed, coordinate changes $\bK_I\to\bK_K$ and $\bK_J\to\bK_K$ with the same target chart are given by embeddings $\phi_{IK}:U_{IK}\to U_K$ and $\phi_{JK}:U_{JK}\to U_K$, whose images may not intersect transversely (for example, often their intersection is contained only in the zero set $\s_K^{-1}(0)$); yet this intersection would be a component of the space of morphisms from $U_{I}$ to $U_{J}$.
In the special case when all obstruction spaces $E_I$ are trivial, $E_I=\{0\}$, one can adjoin these compositions and inverses to $\bB_\Kk$, obtaining an  \'etale  proper groupoid whose realization is a manifold.
(This is one way of interpreting the proof in Proposition~\ref{prop:zeroS0} below that the zero set of a precompact transverse perturbation is a compact manifold.) 
$\hfill\er$
\end{remark}

The categorical formulation of a Kuranishi atlas $\Kk$ allows us to construct a topological space $|\Kk|$ 
from the topological realization of the category $\bB_\Kk$.\footnote{
As is usual in the theory of \'etale groupoids we take the realization of
the category $\bB_\Kk$ to be a quotient of its space of objects rather than the classifying space
of the category $\bB_\Kk$ (which is also sometimes called the topological realization).} 
We will see below that it contains a homeomorphic copy $\io_{\Kk}(X)\subset |\Kk|$ of $X$ and hence may be viewed as a virtual neighbourhood of $X$.

\begin{defn}  \label{def:Knbhd} {\rm $\!\!$ \cite[Definition~2.4.1]{MW:top}}
Let $\Kk$ be a  Kuranishi atlas on the compact space $X$.
Then the {\bf virtual neighbourhood} of $X$,
$$
|\Kk| := \Obj_{\bB_\Kk}/{\scriptstyle\sim}
$$
is the topological realization of the category $\bB_\Kk$, that is the quotient of the object space $\Obj_{\bB_\Kk}$ by the equivalence relation generated by
$$
\Mor_{\bB_\Kk}\bigl((I,x),(J,y)\bigr) \ne \emptyset \quad \Longrightarrow \quad
(I,x) \sim (J,y) .
$$
We denote by  $\pi_\Kk:\Obj_{\bB_\Kk}\to |\Kk|$ the natural projection $(I,x)\mapsto [I,x]$, where $[I,x]\in|\Kk|$ denotes the equivalence class containing $(I,x)$.
We moreover equip $|\Kk|$ with the quotient topology, in which $\pi_\Kk$ is continuous.
Similarly, we define
$$
|\bE_\Kk|:=\Obj_{\bE_\Kk} /{\scriptstyle\sim}
$$
to be the topological realization of the obstruction category $\bE_\Kk$.  The natural projection $\Obj_{\bE_\Kk}\to |\bE_\Kk|$ is denoted $\pi_{\bE_\Kk}$.
\end{defn}

Since this is the same construction as in 
\cite[Definition~2.4.1]{MW:top},the realizations of a smooth Kuranishi atlas and its induced topological Kuranishi atlas are naturally identified -- and both denoted by $|\Kk|$.
This will allow us to use the topological results for topological Kuranishi atlases from \cite{MW:top}.
For example, we now have two notions of zero set -- the realization 
of the subcategory $\s_\Kk^{-1}(0_\Kk)$ with its quotient topology, 
$$
\bigr| \s_\Kk^{-1}(0_\Kk)\bigr| \,:=\; 
\quotient{\s_\Kk^{-1}(0_\Kk)}{\sim_{\scriptscriptstyle \s_\Kk^{-1}(0_\Kk)}} ,
$$
and the zero set of the section $|\s_\Kk|$ with the relative topology induced from $|\Kk|$,
$$
|\s_\Kk|^{-1}(|0_\Kk|)  \,:=\;  \bigl\{p \in |\Kk| \,\big|\, |\s_\Kk|(p) = |0_\Kk|(p)  \bigr\}  \;\subset\; |\Kk| .
$$
The next lemma identifies these and also shows that the zero set $\bigl|\s_\Kk^{-1}(0_\Kk)\bigr|\cong |\s_\Kk|^{-1}(|0_\Kk|)$ is naturally homeomorphic to $X$. Hence $X$ embeds into the virtual neighbourhood $|\Kk|$.

\begin{lemma} \label{le:realization}
Let $\Kk$ be a  Kuranishi atlas on the compact space $X$.
\begin{enumerate}
\item
The categories $\bB_{\Kk}$ and $\bE_{\Kk}$  are well defined. 
\item
The functor ${\rm pr}_\Kk:\bE_\Kk\to\bB_\Kk$ induces a continuous map
$|{\rm pr}_\Kk|:|\bE_\Kk| \to |\Kk|$, 
which we call the {\bf obstruction bundle} of $\Kk$, although its fibers generally do not have vector space structure.\footnote
{
Proposition~\ref{prop:Khomeo} shows that the fibers do have a natural linear structure if $\Kk$ 
is a Kuranishi atlas that satisfies a natural additivity condition on its obstruction spaces as well as taming conditions on its domains.
}
However, the functors $0_\Kk$ and $\s_\Kk$ induce continuous maps
$|0_\Kk| : |\Kk| \to |\bE_\Kk|$, $|\s_\Kk|:|\Kk|\to |\bE_\Kk|$, which are sections in the sense that $|\pr_\Kk|\circ|\s_\Kk| =  |\pr_\Kk|\circ |0_\Kk|= {\rm id}_{|\Kk|}$.
\item 
There is a natural homeomorphism 
$\bigr| \s_\Kk^{-1}(0_\Kk)\bigr| \overset{\cong}{\longrightarrow}|\s_\Kk|^{-1}(|0_\Kk|)$
from the realization of the subcategory $\s_\Kk^{-1}(0_\Kk)$
to the zero set of the section $|\s_\Kk|$. 
\item  
The footprint functor $\psi_\Kk: \s_\Kk^{-1}(0_\Kk) \to \bX$ descends to a homeomorphism $|\psi_\Kk| :  |\s_\Kk|^{-1}(|0_\Kk|) \to X$.   Its inverse is given by
$$
\io_{\Kk}:= |\psi_\Kk|^{-1} : \; X\;\longrightarrow\;  |\s_\Kk|^{-1}(|0_\Kk|) \;\subset\; |\Kk|, \qquad
p \;\mapsto\; [(I,\psi_I^{-1}(p))] ,
$$
where $[(I,\psi_I^{-1}(p))]$ is independent of  
the choice of $I\in\Ii_\Kk$ with $p\in F_I$.
\end{enumerate}
\end{lemma}
\begin{proof}
Since the claims are all algebraic or topological in nature, they are proven by applying Lemma~2.3.7 from \cite{MW:top} for (i), and Lemma 2.4.2 for (ii), (iii), and (iv)  to the induced topological Kuranishi atlas, whose realization is identical with that of $\Kk$. 
\end{proof}

%SHORTER  (D version)
% 
%Note that the injectivity of $\io_{\Kk}:X\to|\Kk|$ in particular implies injectivity of the projection of the zero sets in fixed charts, $\pi_\Kk : s_I^{-1}(0) \to |\Kk|$. 
%However, as is shown by \cite[Example~2.4.3]{MW:top}, this injectivity  only holds on the zero set.  Further, even if the maps $\pi_{\Kk}:U_I\to |\Kk|$ and
%$\pi_{\bE_\Kk}:U_I\times E_I\to |\Kk|$  are injective, 
%the fibers of the bundle $\pr_\Kk:|\bE_\Kk|\to |\Kk|$ may fail to have a linear structure; see  Remark~\ref{rmk:LIN}.
%The topology of the quotient space may also  be rather wild:
%by \cite[Example~2.4.4]{MW:top}
%the quotient space $|\Kk|$ may not be Hausdorff in any neighbourhood of $\io_\Kk(X)$ even though the map $s\times t:\Mor_{\bB_\Kk}\to \Obj_{\bB_\Kk}\times \Obj_{\bB_\Kk}$ is proper. 
%In Section~\ref{ss:tame} below we will achieve the injectivity, Hausdorff, and linearity property by a subtle shrinking of the domains of charts and coordinate changes. However, we are still unable to make the Kuranishi neighbourhood $|\Kk|$ locally compact or even metrizable, due to the following natural example.
%

Note that the injectivity of $\io_{\Kk}:X\to|\Kk|$ in particular implies injectivity of the projection of the zero sets in fixed charts, $\pi_\Kk : s_I^{-1}(0) \to |\Kk|$. This injectivity however only holds on the zero set.

On $U_I\less s_I^{-1}(0)$, the projections $\pi_\Kk: U_I\to |\Kk|$ need not be injective,
as \cite[Example~2.4.3]{MW:top} shows. 
In fact, the following extends this example to show that even the fibers of the bundle $\pr_\Kk:|\bE_\Kk|\to |\Kk|$ may fail to have a linear structure. (Remark~\ref{rmk:LIN} describes another scenario where this linearity fails.)

\begin{figure}[htbp] 
   \centering
   \includegraphics[width=4in]{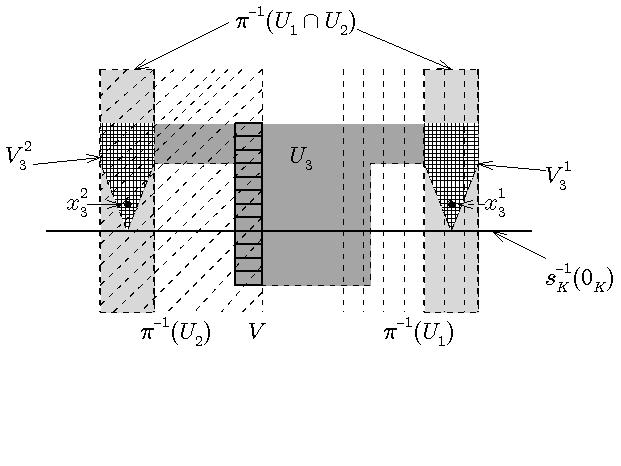}
   \vspace{-.7in}
   \caption{   
The domains $U_1,U_2\subset S^1\times\R$ lift injectively to the dashed subsets of $\R\times \R$. 
The lift $\pi^{-1}(U_1\cap U_2)\subset \R\times \R$ is shown as two light grey strips, whose intersections with the dark grey region $U_3$ are the two shaded sets $V_3^1, V_3^2$.  The points $x_3^1\neq x_3^2 \in U_3$ have the same image in $|\Kk|\subset S^1\times\R$.
An additional chart has domain $U'_4=V\cup V^2_3$, where $V$ is the barred subset of $U_3\cap \pi^{-1}(U_2)$.
}
   \label{fig:3}
\end{figure}

\begin{example}[Failure of Injectivity and Linearity]  \label{ex:nonlin}\rm 
The circle $X=S^1=\R/\Z$ has an open cover $S^1=F_1\cup F_2\cup F_3$ with $F_i=(\frac i3,\frac{i+2}3) \subset \R/\Z$ such that all pairwise intersections $F_{ij}:=F_i\cap F_j \neq \emptyset$ are nonempty, but the triple intersection $F_1\cap F_2\cap F_3$ is empty. 
Then a Kuranishi atlas with these footprints has to involve transition charts but the cocycle conditions are vacuous. We choose charts and transition data as follows:
For $i=1,2$ we use the basic charts $\bK_i$ given by
$$
U_i:=F_i\times (-1,1)\subset S^1\times \R, 
\qquad
E_i := \C,
\qquad
s_i(z,x)=x, 
\qquad
\psi_i(z,0)=z.
$$
These charts are related via the transition chart $\bK_{12}:=\bK_1|_{U_1\cap U_2}=\bK_2|_{U_1\cap U_2}$ and the coordinate changes $\Hat\Phi_{i,12}$ for $i=1,2$ given by the domain $U_{i,12}:=U_1\cap U_2\subset U_i$ and the identity maps $\phi_{12,i}:=\id_{U_1\cap U_2}$, $\Hat\phi_{i,12}:=\id_\C$.
The third basic chart $\bK_3$ is given by 
$$
U_3\subset (0,2)\times \R, 
\qquad
E_3 := \C,
\qquad
s_3(z,x)=\pi(x), 
\qquad
\psi_3(z,0)=\pi(z),
$$
where $U_3$ is chosen such that the projection $\pi:\R\times \R \to S^1\times \R$ embeds $U_3\cap (\R\times\{0\})=(1,\frac 23)\times\{0\}$ to $F_3\times\{0\}$. We can moreover choose $U_3$ so large that the lift under $\pi$ of $U_1\cap U_2$ meets $U_3$ in two components $\pi^{-1}(U_1\cap U_2)\cap U_3 = V_3^1 \sqcup V_3^2$ with $\pi(V_3^1)=\pi(V_3^2)$, but there are continuous lifts $\pi^{-1}: U_i \cap \pi(U_3) \to U_3$ with $V^i_3\subset\pi^{-1}(U_i)$; cf.\ Figure~\ref{fig:3}.
These intersections $V^i_3\subset U_3$ necessarily lie outside of the zero section $s_3^{-1}(0)=F_3\times\{0\}$, though their closure might intersect it.
We choose the transition charts $\bK_{i3}:= \bK_3|_{U_{i3}}$ as restrictions of $\bK_3$ to the domains $U_{i3}:= \pi^{-1}(U_1)\cap U_3$, with transition maps $\Hat\Phi_{3,i3}:=\id_{U_{i3}}\times \id_\C$
and $\Hat\Phi_{i,i3}:\bK_i|_{U_{i,i3}}\to\bK_3$ for $i=1,2$ given the lift
$\phi_{i,i3}:= \pi^{-1}$ on the domain $U_{i,i3}: = U_1\cap \pi(U_3)$
and the identity $\Hat\phi_{i,i3}:={\rm id}_{\C}$ on the identical obstruction spaces $E_i=E_3=\C$.

This defines a Kuranishi atlas of dimension $0$ on $S^1$ with vacuous cocycle condition, but the map $\pi_\Kk: U_3\to |\Kk|$ is not injective since any point $x_3^1\in V_3^1\subset U_3$ is identified $[x_3^1]=[x_3^2]\in|\Kk|$ with the corresponding point $x_3^2\in V_3^2$ with $\pi(x_3^1)=\pi(x_3^2)= y \in S^1\times\R$; see \cite[Example 2.4.3]{MW:top}.
To obtain a Kuranishi atlas $\Kk'$ in which the fibers of $\pr_{\Kk'}$ are not even contractible, 
we use the first three basic charts $\bK'_i:=\bK_i$ and associated transition data as above, except for setting $\Hat\phi\,'_{2,23}(\alpha +\hat\iota \beta) :=\alpha + 2 \hat\iota \beta$ (where we denote $\hat\iota: = \sqrt{-1}$ to prevent confusion with the index $i$).  As above the cocycle condition is trivially satisfied since there are no triple intersections of footprints.
Then with $x_3^i\in U'_3$ as above, the chain of equivalences between
$x_3^1\in U_{13}\subset U_3$, $y=\pi(x_3^i)\in U_{12} \subset U_i$, and 
$x_3^2\in U_{23}\subset U_3$
lifts to the obstruction space $E_3'=\C$ as
\begin{equation} \label{fiber2}
(3, x_3^1,\alpha + \hat\iota \beta) \;\sim\; (2,y, \alpha + \hat\iota \beta)
\;\sim\; (3,x_3^2,\alpha + 2\hat\iota \beta) .
\end{equation}
In order to also obtain the equivalences
\begin{equation} \label{fiber1}
(3, x_3^1,\alpha + \hat\iota \beta) \;\sim\; (3,x_3^2,\alpha + \hat\iota \beta)
\end{equation}
we add another basic chart $\bK_4'= \bK_3'|_{U'_4}$ with
domain indicated in Figure~\ref{fig:3},
$$
U'_4: = V \cup V_3^2 ,\qquad
V: = \pi^{-1} \bigl(F_{13}\times \R\bigr) \cap U'_3.
$$
This chart has footprint $F'_4=F_{13}$, so it requires no compatibility with $\bK'_2$, and for $I\subset\{1,3,4\}$ we always have $F'_I = F_{13}$. We define the transition charts as restrictions
$$
\bK'_{14}: = \bK'_1|_{\pi(U'_4)},\qquad \bK'_{34} = \bK'_3|_{U'_4} , \qquad
\bK_{134}: = \bK_3|_V .
$$
Then we obtain the coordinate changes $\Hat\Phi_{I,J}$ for $I\subsetneq J \subset\{1,3,4\}$ by setting $\Hat\phi\,'_{IJ} := \id_\C$ and $\phi_{IJ}'$ equal either to the identity or to $\pi^{-1}$, as appropriate, on the domains
\begin{align*}
U'_{1,14} := \pi(U'_4), \qquad
&
U'_{4,14}= U'_{3,34}= U'_{4,34} := U'_4, \\
U'_{1,134} = U'_{14,134} := \pi(V), \qquad
&
U'_{3,134} = U'_{4,134} = U'_{13,134} := V .
\end{align*}
To see that the cocycle condition holds, note that we only need to check it for the triples $(i,34,134), i=3,4$, $(j,14,134), j=1,4$, and $(k, 13,134), k=1,3$, and in all of these cases both $\phi_{JK}'\circ\phi_{IJ}'$ and $\phi_{IK}'$ have equal domain, given by $V$ or $\pi(V)$. This provides a chain of morphisms between the same objects as before,
$$
 (3,x_3^2)\sim (34,x_3^2)\sim (4,x_3^2) \sim (14,y) \sim (1,y)\sim
 (13,x^1_3)\sim (3,x_3^1),
$$
whose lift to the obstruction space is \eqref{fiber1} since we have $\Hat\phi\,'_{IJ}= \id_\C$ for all coordinate changes involved.
Therefore, the fiber of $|\pi_{\Kk'}|: |\bE_{\Kk'}|\to |\Kk'|$ over $[3,x_3^1]=[3,x_3^2]$ is
$$
|\pr_\Kk|^{-1}([3,x_3^1]) \;=\;\;  \quotient{\C}{\scriptstyle \bigl( \alpha + \hat\iota \beta \;\sim\; \alpha + 2\hat\iota \beta \bigr)}
\quad \cong\;  \R \times S^1 ,
$$
which does not have the structure of a vector space, and in fact is not even contractible.
$\hfill\er$
\end{example}

Another example in \cite[Example 2.4.4]{MW:top} shows that $|\Kk|$ may not be Hausdorff in any neighbourhood of $\io_\Kk(X)$ even though the map $s\times t:\Mor_{\bB_\Kk}\to \Obj_{\bB_\Kk}\times \Obj_{\bB_\Kk}$ is proper. 
In \S\ref{ss:tame} we will achieve the injectivity, Hausdorff, and linearity property by a subtle shrinking of the domains of charts and coordinate changes. However, we are still unable to make the Kuranishi neighbourhood $|\Kk|$ locally compact or even metrizable, due to the following natural example.

\begin{example}
[Failure of metrizability and local compactness]  \label{ex:Khomeo}
\rm
For simplicity we will give an example with noncompact $X = \R$. (A similar example can be constructed with $X = S^1$.)
We construct a Kuranishi atlas $\Kk$ on $X$ by two basic charts, 
$\bK_1 = (U_1=\R, E_1=\{0\}, s=0,\psi_1=\id)$ and
$$
\bK_2 = \bigl(U_2=(0,\infty)\times \R,\ E_2=\R, \ s_2(x,y)= y,\ \psi_2(x,y)= x\bigr),
$$
one transition chart $\bK_{12} = \bK_2|_{U_{12}}$ with domain $U_{12} := U_2$, and the coordinate changes $\Hat\Phi_{i,12}$ induced by the natural embeddings of the domains $U_{1,12} := (0,\infty)\hookrightarrow (0,\infty)\times\{0\}$ and $U_{2,12} := U_2\hookrightarrow U_2$.
Then as a set $|\Kk| = \bigl(U_1\sqcup U_2\sqcup U_{12}\bigr)/\sim$
can be identified with $\bigl(\R\times\{0\}\bigr) \cup \bigl( (0,\infty)\times\R\bigr) \subset \R^2$.
However, the quotient topology at $(0,0)\in|\Kk|$ is strictly stronger than the subspace topology.
That is, for any $O\subset\R^2$ open the induced subset $O\cap|\Kk|\subset|\Kk|$ is open, but some open subsets of $|\Kk|$ cannot be represented in this way.
In fact,  
for any $\eps>0$ and continuous function $f:(0,\eps)\to (0,\infty)$,
the set
$$
U_{f,\eps} \, :=\; \bigl\{ [x] \,\big|\, x\in U_1, |x|< \eps \}  \;\cup\; \bigl\{ [(x,y)] \,\big|\, (x,y)\in U_2,  |x|< \eps , |y|<f(x)\} \;\subset\; |\Kk|
$$
is open in the quotient topology. 
We use these in \cite[Example 2.4.5]{MW:top} to show that the point $[(0,0)]$ does not have a countable neighbourhood basis in the quotient topology.
Note also that the point $[(0,0)]\in|\Kk|$ has no compact neighbourhood with respect to the subspace topology from $\R^2$, and hence neither  with respect to the stronger quotient topology on $|\Kk|$. 
The same failure of local compactness and metrizability occurs for any Kuranishi atlas that involves coordinate changes between charts with domains of different dimension (more precisely the issue arises from an embedding $U_{IJ}\to U_J$ if $U_{IJ}\subset U_I$ is not just a connected component and $\dim U_I < \dim U_J$). In particular, the additive Kuranishi atlases that we will work with to achieve the Hausdorff property, will -- except in trivial cases -- always exclude local compactness or metrizability.
$\hfill\er$
\end{example}

\begin{rmk}\label{rmk:Khomeo}\rm  
For the Kuranishi atlas in Example~\ref{ex:Khomeo} there exists an exhausting sequence $\ov{\Aa^n}\subset \ov{\Aa^{n+1}}$ of closed subsets of $\bigsqcup_{I\in \Ii_\Kk} U_I$ with the properties
\begin{itemize}
\item  
each $\pi_\Kk(\ov{\Aa^n})$ contains $\iota_\Kk(X)$;
\item  
each $\pi_\Kk(\ov{\Aa^n})\subset |\Kk|$ is metrizable and locally compact in the subspace topology;
\item 
$\bigcup_{n\in\N} \ov{\Aa^n} = \bigsqcup_{I\in \Ii_\Kk} U_I$.
\end{itemize}
For example, we can take $\ov{\Aa^n}$ to be the disjoint union of the closed sets 
$$
\ov{A_1^n}= [-n,n]\subset U_1, \qquad 
\ov{A_{2}^n} : = \{(x,y)\in U_2 \,\big|\,  x \geq \tfrac 1n, |y| \leq  n\},
$$
and any closed subset $\ov{A_{12}^n} \subset \ov{A_2^n}$.
However, in the limit $[(0,0)]$ becomes a ``bad point'' because its neighbourhoods have to involve open subsets of $U_2$.  

In fact, if we altered Example~\ref{ex:Khomeo} to a Kuranishi atlas for the compact space $X=S^1$, then we could choose $\ov{\Aa^n}$ compact, so that the subspace and quotient topologies on  $\pi_\Kk(\ov{\Aa^n})$
coincide by Proposition~\ref{prop:Ktopl1}~(ii). We emphasize the subspace topology above because that is the one inherited by (open) subsets of $\ov{\Aa^n}$.  For example, the quotient topology on $\pi_\Kk(\Aa^n)$, where $\Aa^n: = \bigcup_I {\rm int}(\ov{A_I^n})$ has the same bad properties at $[(\frac 1n,0)]$ as the quotient topology on $|\Kk|$ has at $[(0,0)]$, while the subspace topology on $\pi_\Kk(\Aa_n)$ is metrizable.
 We prove in Proposition~\ref{prop:Ktopl1} that a similar statement holds for all $\Kk$,
 though there we only consider a fixed set $\ov\Aa$ since we have no need for an exhaustion of the domains.
$\hfill\er$
 \end{rmk}

Instead of a metric compatible with the quotient topology on $|\Kk|$, the construction of perturbations in \S\ref{ss:const} will use the following weaker notion of metrics on Kuranishi atlases from \cite{MW:top}.

\begin{defn}\label{def:metric}  {\rm $\!\!$ \cite[Definition~3.1.7]{MW:top}}
A Kuranishi atlas $\Kk$ is said to be {\bf metrizable} if there is a bounded metric $d$ on the set $|\Kk|$ such that for each $I\in \Ii_\Kk$ the 
pullback $d_I:=(\pi_\Kk|_{U_I})^*d$ is a metric on $U_I$  that induces the given topology.
In this situation we call $d$ an {\bf admissible metric} on $|\Kk|$. 
A {\bf metric Kuranishi atlas} is a pair $(\Kk,d)$ consisting of a metrizable Kuranishi atlas together with a choice of  admissible metric $d$.
For a metric Kuranishi atlas, we denote the $\de$-neighbourhoods of subsets $Q\subset |\Kk|$ resp.\ $A\subset U_I$ for $\de>0$ by
\begin{align*}
B_\de(Q) &\,:=\; \bigl\{w\in |\Kk|\ | \ \exists q\in Q : d(w,q)<\de \bigr\}, \\
B^I_\de(A) &\,:=\; \bigl\{x\in U_I\ | \ \exists a\in A : d_I(x,a)<\de \bigr\}.
\end{align*}
\end{defn}

Note that metrizability of a Kuranishi atlas in particular requires $\pi_\Kk|_{U_I}$ to be injective -- which will be achieved in \S\ref{ss:tame}.
It is also important to note that an admissible metric generally does not induce the quotient topology on $|\Kk|$, since this may not even be metrizable by Example~\ref{ex:Khomeo}.
The following shows that the metric topology on $|\Kk|$ is weaker (has fewer open sets) than the quotient topology.

\begin{lemma}\label{le:metric}  
Suppose that $d$ is an admissible metric on the virtual neighbourhood $|\Kk|$ of a Kuranishi atlas $\Kk$.
Then the following holds.
\begin{enumerate}
\item
The identity $\id_{|\Kk|} :|\Kk| \to (|\Kk|,d)$ is continuous as a map from the quotient topology to the metric topology on $|\Kk|$.
\item
In particular, each set $B_\de(Q)$ is open in the quotient topology on $|\Kk|$, so that the 
existence of an admissible metric implies that $|\Kk|$ is Hausdorff.
\item 
The embeddings $\phi_{IJ}$ that are part of the coordinate changes 
for $I\subsetneq J\in\Ii_\Kk$
are isometries when considered as maps $(U_{IJ},d_I)\to (U_J,d_J)$.
\end{enumerate}
\end{lemma}

\begin{proof} 
This follows from applying \cite[Lemma~3.1.8]{MW:top} to the induced topological Kuranishi atlas, whose virtual neighbourhood is identical to $|\Kk|$ and inherits the metric.
\end{proof}

One might hope to achieve the Hausdorff property by constructing an admissible metric, but the existence of the latter is highly nontrivial. 
Moreover, a weak Kuranishi atlas does not even have a well defined virtual neighbourhood on which one could attempt to construct a metric.
Instead, the shrinking in \S\ref{ss:tame} is a refinement process which constructs a Kuranishi atlas whose virtual neighbourhood has the Hausdorff property. Then we prove metrizability of certain subspaces, and finally obtain an admissible metric by pullback to a further refined Kuranishi atlas. 
We end this section by comparing our choice of definitions with the notions of Kuranishi structures in the current literature.

\begin{rmk}\label{rmk:otherK}\rm
(i)
We defined the notion of a Kuranishi atlas so that it is relatively easy to construct from an equivariant Fredholm section. The only condition that is difficult to satisfy is the cocycle condition since that involves making compatible choices of all the domains $U_{IJ}$.
However, we show in Theorem~\ref{thm:K} that, provided the obstruction bundles satisfy
the additivity condition, one can always construct a Kuranishi atlas from a tuple of charts and coordinate changes that satisfy the weak cocycle condition in Definition~\ref{def:cocycle}, which is much easier to satisfy in practice.
The additivity condition is also satisfied by the sum constructions for finite dimensional reductions of holomorphic curve moduli spaces in e.g.\ \cite{FO} and in fact is naturally satisfied by the Fredholm stabilizations that we systematically introduce in \cite{MW:iso,Mcn}.
\MS

\NI (ii)
A Kuranishi structure in the sense of \cite{FO,J1} is given in terms of conjugacy classes of germs of charts at every point of $X$ and some set of coordinate changes.
While this is a natural idea, we were not able to find a meaningful notion of compatible coordinate changes; see \S\ref{ss:alg}.
Recently, there seems to be a general understanding that explicit charts and coordinate changes are needed.
\MS

\NI (iii)
A Kuranishi structure in the sense of \cite[App.~A]{FOOO} consists of a Kuranishi chart $\bK_p$ at every point $p\in X$ and coordinate changes $\bK_q|_{U_{qp}}\to \bK_p$ whenever $q\in F_p$, which are required to satisfy a weak cocycle condition.
The idea from \cite{FO} for constructing such a Kuranishi structure also starts with a finite covering family of basic charts $(\bK_i)$. 
Then the chart at $p$ is obtained by a sum construction from the geometric data used to construct the charts $\bK_i$ with $p\in F_i$.
This construction is similar to the one outlined in \S\ref{ss:gw}, except that it is less explicit  
(with some more details in \cite{FOOO12}) and is only performed in a neighbourhood of a point.  
\MS

\NI (iv)
In the case of trivial isotropy, an abstract weak Kuranishi atlas in the sense of Definition~\ref{def:Ku2} induces a Kuranishi structure in the sense of \cite[App.~A]{FOOO} as follows.

Given a covering family of basic charts $(\bK_i)_{i=1,\dots,N}$ with footprints $F_i$ and transition data $(\bK_I,\Hat\Phi_{IJ})$, choose a family of compact subsets $C_i\subset F_i$ that also cover $X$.
Then for any $p\in X$ one obtains a Kuranishi chart $\bK_p$ by choosing a restriction of $\bK_{I_p}$ to $F_p:=\cap_{i\in I_p} F_i\less \cup_{i\notin I_p} C_i$, where $I_p: =  \{i \,|\, p\in C_i\}$.
This construction guarantees that for $q\in F_p$ we have $I_q\subset I_p$  and thus can restrict the coordinate change $\Hat\Phi_{I_q I_p}$ to a coordinate change from $\bK_q$ to $\bK_p$. 
The weak cocycle condition is preserved by these choices. 

Note however that neither this notion of a Kuranishi structure nor a weak Kuranishi atlas is sufficient for our approach to the construction of a VMC since that makes essential use of the additivity condition.
(As explained in \cite[\S6]{Mcn}, this condition can be weakened in some special circumstances.)

\MS

\NI (v)  The key step in the construction by Fukaya et al.\ of a virtual fundamental class  is the notion of a  ``good coordinate system". 
The first versions and proofs of existence (in \cite[Lemma 6.3]{FO}, which is also quoted by \cite{FOOO}) of this notion were based on composition of conjugacy classes of germs and did not address the cocycle condition. 
In its most recent version in~\cite{FOOO12}, this requires a finite cover of $X$ by a partially ordered set of (orbifold) charts $(\bK_I)_{I\in\Pp}$ and coordinate changes $\bK_I \to \bK_J$ for $I\leq J$, where the order is compatible with the overlaps of the footprints in the sense that $F_I\cap F_J\ne \emptyset$ implies $I\leq J$ or $J\leq I$.
In the language used here, this corresponds to an atlas that is both tame and reduced in the sense of \S\ref{ss:red}. 
Their construction aims to satisfy both these conditions simultaneously, while in our approach these questions are separated in order to clarify exactly what choices and constructions are  needed so that we can build an adequate cobordism theory. 

\MS
\NI (vi)
Another approach to the VFC, developed after the first versions of this paper appeared, is due to Pardon~\cite{pardon}.
His notion of implicit atlas is most similar to our notion of tame Kuranishi atlas with nontrivial isotropy in \cite{MW:iso,McL}, except that it does not require a differentiable structure.
It does however require large subsets of the chart domains to be topological manifolds, and so should not be mistaken for our notion of topological Kuranishi atlas. Instead, an implicit atlas induces a topological atlas by quotienting out the isotropy. 
One advantage of this approach is that one needs less sophisticated gluing theorems to construct charts near nodal curves. Instead of using perturbations of differentiable sections, \cite{pardon} then constructs the VFC in the \v{C}ech homology of $X$ using homological/sheaf theoretic tools.
Thus in this approach the explicit refinement, reduction, matching, and extension techniques in our construction of perturbations in \S\ref{s:red}
are replaced by delicate sheaf theoretic arguments that allow for patching local homological information into a virtual fundamental class. The result is a more abstract homological class than one represented by a virtual moduli cycle obtained from a perturbation.
$\hfill\er$  
 \end{rmk}

%%%%%%%%%%%%%%%%%%%%%%%%%%%%%%%%%%%%%%%%%%%%%%%%%%%%%%%%%%%%%%
\subsection{Cobordisms of Kuranishi atlases}\label{ss:Kcobord}  \hspace{1mm}\\ \vspace{-3mm}
%%%%%%%%%%%%%%%%%%%%%%%%%%%%%%%%%%%%%%%%%%%%%%%%%%%%%%%%%%%%%%

Because there are many choices involved in constructing a Kuranishi atlas, and holomorphic curve moduli spaces in addition depend on the choice of an almost complex structure, it is important to have suitable notions of equivalence.
Since we are only interested here in constructing the VMC as cobordism class, resp.\ the VFC as a homology class, a notion of uniqueness up to cobordism will suffice for our purposes.   
We will introduce in Definition~\ref{def:CKS} a general notion of {\bf Kuranishi cobordism}
to be a Kuranishi atlas on a compact metrizable space $Y$ with two collared boundary components $\p^0Y$  and $\p^1Y$.
However, if we are considering atlases over a fixed space $X$, then we will mostly work with two stronger equivalence relations: Firstly, a {\bf concordance} is given by a Kuranishi cobordism on the product $[0,1]\times X$. 
Secondly, {\bf commensurability} is the equivalence relation generated by calling two atlases $\Kk^0, \Kk^1$ directly commensurate if they are both contained in the same atlas $\Kk^{01}$ on $X$.
Lemma~\ref{lem:cobord1} will show that commensurability implies concordance.

\begin{defn}\label{def:Kcomm}
Two (weak) Kuranishi atlases $\Kk^0,\Kk^1$ on the same compact space $X$ are {\bf directly commensurate} if there exists a common extension $\Kk^{01}$. This means that $\Kk^{01}$ is a (weak) Kuranishi atlas on $X$ with basic charts $(\bK^\al_i)_{(\al,i)\in\Nn^{01}}$, where
$$
\Nn^{01} := \Nn^0 \sqcup \Nn^1 ; \qquad \Nn^\al := \bigl\{ (\al,i) \,\big|\, 0 \leq i \leq N^\al \bigr\} ,
$$
and transition data $(\bK^{01}_I,\Hat\Phi^{01}_{IJ})_{I,J\subset\Nn^{01}, I\subsetneq J}$ 
such that $\bK^{01}_I=\bK^0_I$ and $\Hat\Phi^{01}_{IJ}=\Hat\Phi^\al_{IJ}$ whenever $I,J\subset\Nn^\al$ for fixed $\al=0$ or $\al=1$.
We say they are  {\bf commensurate} if there is a sequence  $\Kk_0: = \Kk^0,\Kk_1,\dots,\Kk_\ell = \Kk^1$ of atlases on $X$ such that $\Kk_{i-1}$ is directly commensurate to $\Kk_{i}$ for $i = 1,\dots,\ell$.

Moreover, if $\Kk^0,\Kk^1$ are additive, we say they are  {\bf additively  commensurate} if there exists 
such a sequence with the additional property that all atlases involved are additive. 
\end{defn}

\begin{remark}\label{rmk:Kcomm}
\rm
While the construction of a Kuranishi atlas on a fixed Gromov--Witten moduli space $X$ depends on many choices (for example of slicing conditions and obstruction spaces for the basic charts), the resulting atlas will generally be unique up to commensurability.
Here the simplest construction of atlases (for example, that explained in \S\ref{ss:gw}) does not allow us to sum an arbitrary pair of charts, and so any two choices will at best yield Kuranishi atlases that are both directly commensurate to a third.
However, if we use Fredholm stabilization as in \cite{MW:iso,Mcn} then we can form arbitrary sum charts, which implies that any two Kuranishi atlases on $X$ that are constructed by this method are directly commensurate.

So our notion of commensurability is guided by the fact that, although all the atlases that we construct in applications will be directly commensurate, it is not clear that this relation is transitive in general. 
$\hfill\er$
\end{remark}

We next turn to the notion of cobordism between Kuranishi atlases.
In order to define this so that it is transitive, we will need a special form of charts and coordinate changes at the boundary that allows for gluing of cobordisms.
Thus we will define a Kuranishi cobordism to be a Kuranishi atlas over a space $Y$ whose designated ``boundary components" $\p^0Y, \p^1Y \subset Y$ have collared neighbourhoods as follows.

\begin{defn} \label{def:Ycob}  {\rm $\!\!$ \cite[Definition~4.1.1]{MW:top}}
A {\bf collared cobordism} $(Y, \io^0_Y,\io^1_Y)$ is a separable, locally compact, 
metrizable space $Y$ together with disjoint (possibly empty)
closed subsets $\p^0 Y,$ $ \p^1 Y\subset Y$ and maps
$$
\io_Y^0:  [0,\eps)\times  \p Y^0  \to Y, \qquad  \io_Y^1:  (1-\eps, 1]\times  \p Y^1  \to Y
$$
for some $\eps>0$ that are {\bf collared neighbourhoods} in the following sense: 
They extend the inclusions $\io_Y^0(0,\cdot) : \p^0 Y\hookrightarrow Y$, resp.\ $\io_Y^1(1,\cdot) : \p^1 Y\hookrightarrow Y$, and are homeomorphisms onto disjoint open neighbourhoods of $\p^0 Y\subset Y$, resp.\ $\p^1 Y\subset Y$. 

We call $\p^0 Y$ and $\p^1 Y$ the {\bf boundary components} of $(Y, \io^0_Y,\io^1_Y)$.
\end{defn}

For the next definition, it is useful to introduce the notation
\begin{equation}\label{eq:Naleps}
A_\delta^0: = [0,\delta)  \qquad\text{and} \qquad A_\delta^1: = (1-\delta,1] \qquad\text{ where }\  0<\delta<\tfrac 12
\end{equation}
for collar neighbourhoods of 
$0$ resp.\ $1$ in $[0,1]$.

\begin{defn}\label{def:collarset} {\rm $\!\!$ \cite[Definition~4.1.2]{MW:top}}
If $(Y, \io_Y^0, \io_Y^1)$ is a collared cobordism, we say that an open subset $F\subset Y$ is {\bf collared} 
if there is $0<\delta\le\eps$ such that for $\al\in \{0,1\}$ we have
$$
F \cap \im (\io_Y^\al)\ne \emptyset
\;\; \Longleftrightarrow\;\;
F \cap \im (\io_Y^\al)
= \io_Y^\al( A^\al_\delta\times \p^\al F) .
$$
Here we denote by
$
\partial^\al F :=  F \cap \p^\al Y 
$
the intersection with the ``boundary component'' $\p^\al Y$ and allow one or both of $\p^\al F$ to be empty.
\end{defn}

Note that a collared subset $F\subset Y$ with empty 
``boundary'' $\p^\al F=\emptyset$ is in fact disjoint from the open neighbourhood $\im\io^\al_Y$ of the corresponding ``boundary component'' $\p^\al Y$.

\begin{example}\label{ex:natcol}\rm
In general the ``boundary components" $\p^\al Y$ are by no means uniquely determined by $Y$ or 
topological boundaries of $Y$ in any sense, though the main example of a collared cobordism 
is $Y = [0,1]\times X$, which has the natural ``boundary components'' $\{0\}\times X$ and $\{1\}\times X$. 
In this case we always take $\io_Y^\al$ to be the canonical extensions of the 
inclusions $\io_Y^\al(\al,\cdot): \{\al\}\times X \to [0,1]\times X$, for some choice of $0<\eps<\frac 12$.
More generally, we might  
consider a union of moduli spaces
$$
Y = {\textstyle \bigcup_{t\in [0,1]}} \{t\}\times \oMm_{0,k}(M,A,J_t) ,
$$
where $J_t$ is a family of
almost complex structures that are constant for $t$ near $0$ and~$1$. 
This again has canonical ``boundary components'' $\p^\al Y=\{\al\}\times \oMm_{0,k}(M,A,J_\al)$ and collared neighbourhoods $\io_Y^\al\bigl(t,(\al,p)\bigr)=(t,p)$ for sufficiently small choice of $\eps>0$.
$\hfill\er$
\end{example}

When constructing Kuranishi cobordisms we require all charts and coordinate changes in a sufficiently small collar to be of a compatible product form as introduced below.

\begin{defn} \label{def:Cchart}  {\rm $\!\!$ \cite[Definition~4.1.4]{MW:top}}
Let $(Y, \io_Y^0, \io_Y^1)$ be a compact collared cobordism.
\begin{itemlist}
\item
Let $\bK^\al=(U^\al,E^\al,s^\al,\psi^\al)$ be a Kuranishi chart for $\p^\al Y$, and let $A\subset[0,1]$ be a relatively open interval. Then we define the {\bf product chart} for $[0,1] \times \p^\al Y$ with footprint $A\times F^\al$ 
as
$$
A\times \bK^\al  :=\bigl(A \times U^\al, E^\al, \, s^\al\circ{\rm pr}_{U^\al} ,\, \id_{A}\times \psi^\al \bigr) ,
$$
where $\pr^\al:A\times U^\al\to U^\al$ is the projection to the second factor.
\item
A {\bf Kuranishi chart with collared boundary} for $(Y, \io_Y^0, \io_Y^1)$
is a tuple $\bK = (U,E ,s,\psi)$ as in Definition~\ref{def:chart}, with the following collar form requirements:
\begin{enumerate}
\item
The footprint $F =\psi (s^{-1}(0))\subset  Y$ is collared and intersects at least one of the boundary components $\p^\al Y$.
\item 
The domain $U$ is a smooth manifold with boundary $\p U = \p^0 U \sqcup \p^1 U$ equipped with a structure of collared cobordism $(U,\io^0_U, \io^1_U)$ whose collar embeddings $\io^0_U, \io^1_U$ are tubular neighbourhood diffeomorphisms.
Each boundary component $\partial^\al U$  is nonempty iff $\p^\al F= F \cap \p^\al Y\ne \emptyset$. 
\item 
If $\partial^\al F \neq\emptyset$ then for some $\eps>0$ there is a Kuranishi chart $\partial^\al\bK$ for $\p^\al Y$ with footprint $\p^\al F$, domain $\p^\al U$, and obstruction space $E$, and an embedding of the product chart $A_\eps^\al\times \p^\al \bK$ into $\bK$ in the following sense:
The boundary embedding $\iota_U^\al:A_\eps^\al \times \partial^\al U \hookrightarrow  U$ intertwines the section and footprint maps of the charts $A_\eps^\al\times \p^\al \bK$ and $\bK$ as in  
Remark~\ref{rmk:tchange}
with $\phi=\iota_U^\al$, $\Hat\phi=\id_E$.
In particular, the footprint map is compatible with the boundary collars on $\bK$ and $Y$ in the sense that the following diagram commutes:
$$
  \begin{array} {ccc}
(\id_{A_\eps^\al}\times  s^\al)^{-1}(\id_{A_\eps^\al}\times 0^\al) & \stackrel{\io_U^\al} \longrightarrow &{s^{-1}(0)} \\
 \id_{A_\eps^\al}\times \psi^\al\downarrow\;\;\;\;\;\;\;\;\;&&\downarrow{\psi}  \\
\phantom{right}{A_\eps^\al\times \p^\al Y} & \stackrel{\io_Y^\al} \longrightarrow &{Y} \; .
\end{array}
$$
\end{enumerate}
\item
For any  Kuranishi chart with collared boundary for $(Y, \io_Y^0, \io_Y^1)$ we call the resulting uniquely determined  Kuranishi charts $\partial^\al\bK$ (with footprints in $\p^\al Y$) the {\bf restrictions of $\bK$ to the boundary}. 
\end{itemlist}
\end{defn}

We now define a coordinate change between charts on $Y$ that may have boundary.   Because in a Kuranishi atlas there is a coordinate change $\bK_I\to \bK_J$ only when $F_I\supset F_J$, we will  restrict to this case here.  (Definition~\ref{def:change} considered a more general scenario.)
In other words, we need not consider coordinate changes from a chart without boundary to a chart with boundary.

\begin{defn} \label{def:Ccc} {\rm $\!\!$ \cite[Definition~4.1.5]{MW:top}}
\begin{itemlist}
\item
Let $\Hat\Phi^\al_{IJ}:\bK^\al_I\to\bK^\al_J$ be a coordinate change between  Kuranishi charts for $\p^\al Y$, 
and let $A_I,A_J\subset[0,1]$ be relatively open intervals.
Then the {\bf product coordinate change}
$\id_{A_I\cap A_J} \times \Hat\Phi^\al_{IJ}  : (A_I\cap A_J) \times \bK^\al_I \to A_J \times  \bK^\al_J$ 
is given by
$$
\id_{A_I\cap A_J}\times \phi^\al_{IJ} : (A_I\cap A_J)\times U_{IJ} \to A_J\times  U_J
\qquad\text{and}\qquad 
\Hat\phi^\al_{IJ}:  E_I^\al \to  E_J^\al.
$$
\item
Let $\bK _I,\bK _J$ be Kuranishi charts on  $(Y, \io_Y^0, \io_Y^1)$ such that only $\bK _I$ or both
$\bK _I,\bK _J$ have collared boundary.
Then a {\bf coordinate change with collared boundary} $\Hat\Phi_{IJ} :\bK _I\to\bK _J$ 
with domain $U_{IJ}$ satisfies the conditions in Remark~\ref{rmk:tchange}, with the following boundary variations and collar form requirement:
\begin{enumerate}
\item
The domain is a collared subset $U _{IJ}\subset U _I$ in the sense of Definition~\ref{def:collarset},
so is a manifold with boundary, whose boundary components are $\partial^\al U _{IJ}:= U _{IJ} \cap \partial^\al U _I$.
\item
If $F_J\cap \p^\al Y \ne \emptyset$ then $F_I\cap \p^\al Y \ne \emptyset$ 
and there is a coordinate change $\partial^\al\Hat\Phi_{IJ} : \partial^\al\bK _I \to \partial^\al\bK _J$
with domain $\partial^\al U _{IJ}$ such that the restriction of $\Hat\Phi_{IJ}$ 
to $U_{IJ} \cap \io_{U_I}^\al(A^\al_\eps\times \partial^\al U _I)$
pulls back via the collar inclusions $\io^\al_{U_I}, \io^\al_{U_J}$ 
to the product $ {\rm id}_{A^\al_\eps} \times \partial^\al\Hat\Phi_{IJ}$
for some $\eps>0$.
In particular we have 
\begin{align*}
(\iota_{U_I}^\al)^{-1}(U _{IJ})
\cap \bigl(A^\al_\eps\times \partial^\al U _I \bigr)
&\;=\; A^\al_\eps\times \partial^\al U _{IJ}, \\
(\iota_{U_J}^\al)^{-1}(\im\phi _{IJ})
\cap \bigl(A^\al_\eps  \times \partial^\al U _J \bigr)
&\;=\;
 \phi _{IJ}( A^\al_\eps\times \partial^\al U _{IJ}) .
\end{align*}
\item
If $F_J\cap \p^\al Y= \emptyset$ but $F_I\cap \p^\al Y\ne \emptyset$ 
then we have $\p^\al U_{IJ}=\emptyset$ and hence also
$U _{IJ}\cap \iota_{U_I}^\al \bigl(A^\al_\eps\times \partial^\al U _I \bigr) = \emptyset$ for some $\eps>0$.
\end{enumerate}
\item
For any coordinate change with collared boundary $\Hat\Phi_{IJ} $ on $(Y, \io_Y^0, \io_Y^1)$  
we call the uniquely determined coordinate changes $\p^\al \Hat\Phi_{IJ}$ for  $\p^\al Y$ the {\bf restrictions of $\Hat\Phi_{IJ} $ to the boundary} for $\al=0,1$.
\end{itemlist}
\end{defn}

The above definitions of Kuranishi chart and coordinate change with collared boundary are special cases of corresponding definitions \cite{MW:top} in a purely topological setting.
Similarly, the following notions of (weak) Kuranishi cobordism/concordance -- given atlases in which we allow collared boundaries -- specialize the notions of {\bf (weak) topological Kuranishi cobordism/concordance} from \cite{MW:top}.

\begin{defn}\label{def:CKS} {\rm $\!\!$ \cite[Definition~4.1.6]{MW:top}}
A {\bf (weak) Kuranishi cobordism of dimension $\mathbf d$} on a compact collared cobordism  $(Y, \io_Y^0, \io_Y^1)$ is a tuple
$$
\Kk  = \bigl( \bK_{I} , \Hat\Phi_{IJ} \bigr)_{I,J\in \Ii_{\Kk}}
$$
of basic charts and transition data as in Definition~\ref{def:K} whose Kuranishi charts are of dimension $d$, 
with the following boundary variations and collar form requirements:
\begin{itemlist}
\item
The charts of $\Kk$ are either Kuranishi charts with collared boundary or standard Kuranishi charts whose footprints are precompactly contained in $Y\less (\p^0 Y\cup \p^1 Y)$.
\item
The coordinate changes $\Hat\Phi_{IJ}: \bK_{I} \to \bK_{J}$ 
are either standard coordinate changes on $Y\less (\p^0 Y\cup \p^1 Y)$  
between pairs of standard charts, or coordinate changes with collared boundary between 
pairs of charts, of which at least the first has collared boundary.
\end{itemlist}
A {\bf (weak)  Kuranishi concordance} is a (weak) Kuranishi cobordism on a collared cobordism of product type $Y = [0,1]\times X$ with canonical collars as in Example~\ref{ex:natcol}.

Moreover, a (weak) Kuranishi cobordism or concordance is said to be {\bf additive} if its obstruction spaces satisfy the additivity condition in Definition~\ref{def:Ku2}.
\end{defn}

\begin{rmk}\label{rmk:restrict}\rm  
(i)
Let $(Y, \io_Y^0, \io_Y^1)$ be a compact collared cobordism.
Then any (weak) Kuranishi cobordism $\Kk$ on $(Y, \io_Y^0, \io_Y^1)$ induces by restriction (weak) 
Kuranishi atlases $\partial^\al\Kk$ on each boundary component $\p^\al Y$
with
\begin{itemlist}
\item 
basic charts $\p^\al\bK_i$ given by restriction of  basic charts of $\Kk$ with $F_i\cap \p^\al Y\neq\emptyset$;
\item 
index set $\Ii_{\p^\al\Kk}=\{I\in\Ii_{\Kk}\,|\, F_I\cap  \p^\al Y\neq\emptyset\}$;
\item 
transition charts $\p^\al\bK_I$ given by restriction of transition charts of $\Kk$;
\item
coordinate changes $\p^\al\Hat\Phi_{IJ}$ given by restriction of coordinate changes of $\Kk$.
\end{itemlist}
In this case we say that {\bf $\Kk$ is a cobordism from $\p^0\Kk$ to $\p^1\Kk$}
and call $\p^\al\Kk$ the {\bf restrictions of $\Kk$ to the boundary}.
In the special case when $\Kk$ is in fact a Kuranishi concordance, we also say that 
{\bf $\Kk$ is a concordance from $\p^0\Kk$ to $\p^1\Kk$}.

\MS
\NI (ii) 
The restrictions $\p^\al \Kk$ of an additive Kuranishi cobordism $\Kk$ are also additive.

\MS\NI
(iii)
Any (weak) Kuranishi cobordism/concordance $\Kk$ induces a (weak) topological Kuranishi cobordism/concordance in the sense of \cite{MW:top}, which we again denote by $\Kk$. Moreover, this association is compatible with restrictions to the boundary.
$\hfill\er$
\end{rmk}

With this language in hand, we can now introduce the cobordism and concordance relations between Kuranishi atlases.

\begin{defn}\label{def:Kcobord} {\rm $\!\!$ \cite[Definition~4.1.8]{MW:top}}
Let $\Kk^0, \Kk^1$ be (weak)  Kuranishi atlases on compact metrizable spaces $X^0,X^1$.
\begin{itemlist}
\item 
$\Kk^0$ is {\bf (weakly) cobordant} to $\Kk^1$ if there exists a (weak)  Kuranishi cobordism $\Kk$ from $\Kk^0$ to $\Kk^1$. Equivalently, $\Kk$ is an atlas on a compact collared cobordism $(Y, \io_Y^0, \io_Y^1)$ with boundary components $\p^\al Y = X^\al$ and boundary restrictions $\p^\al\Kk=\Kk^\al$ for $\al = 0,1$.
To be more precise, there are injections $\iota^\al:\Ii_{\Kk^\al} \hookrightarrow \Ii_{\Kk}$ for $\al=0,1$ such that $\im\iota^\al=\Ii_{\partial^\al\Kk}$ and for all $I,J\in\Ii_{\Kk^\al}$ we have
$$
\bK^\al_I = \p^\al \bK_{\iota^\al(I)}, \qquad
\Hat\Phi^\al_{IJ} = \p^\al \Hat\Phi_{\iota^\al(I) \iota^\al (J)} .
$$
\item
$\Kk^0$ is {\bf (weakly) concordant} to $\Kk^1$ if there exists a (weak)  Kuranishi concordance $\Kk$ from $\Kk^0$ to $\Kk^1$.
Equivalently, the spaces $X^0=X^1=X$ are identical and $\Kk$ is a (weak)  Kuranishi cobordism on $Y = [0,1]\times X$ with boundary restrictions $\p^\al\Kk=\Kk^\al$ for $\al = 0,1$ as above.
\item
If $\Kk^0, \Kk^1$ are moreover additive, then they are {\bf additively (weak) cobordant (resp.\ concordant)} if there exists a (weak) Kuranishi cobordism (resp.\ concordance) as above that is additive.
\end{itemlist}
\end{defn}

In the following we will usually identify the index sets $\Ii_{\Kk^\al}$ of 
cobordant 
 Kuranishi atlases with the restricted index set $\Ii_{\partial^\al\Kk}$ in the cobordism index set $\Ii_{\Kk}$, so that $\Ii_{\Kk^0}, \Ii_{\Kk^1}\subset \Ii_{\Kk}$ are the (not necessarily disjoint) subsets of charts whose footprints intersect $\p^0 Y$ resp.\ $\p^1 Y$.

\begin{example} \label{ex:triv}\rm
Let $\Kk= \bigl( \bK_I, \Hat\Phi_{IJ}\bigr)_{I,J\in\Ii_\Kk}$ be a weak Kuranishi atlas on $X$.
Then the {\bf product Kuranishi concordance} $[0,1]\times \Kk$ from $\Kk$ to $\Kk$ is the weak 
Kuranishi  cobordism on $ [0,1]\times X$ consisting of the product charts $[0,1]\times \bK_I$ and the product coordinate changes $\id_{[0,1]}\times \Hat\Phi_{IJ}$ for $I,J\in\Ii_\Kk$.
Note that in this case all index sets are the same, 
$\Ii_{\Kk^0}=\Ii_{\Kk^1}=\Ii_{[0,1]\times \Kk}= \Ii_{\Kk}$.

Moreover, if $\Kk$ is an additive Kuranishi atlas, then the product Kuranishi concordance $[0,1]\times \Kk$ is an additive Kuranishi concordance.
$\hfill\er$
\end{example}

In the following we will extend the categorical formulation and  metrizability notion from \S\ref{ss:Ksdef} to Kuranishi cobordisms. 
For that purpose it will often be convenient to work with the following notion of uniform collar width.

\begin{remark} \rm \label{rmk:Ceps}   Let $\Kk$ be a (weak) Kuranishi cobordism.
Since the index set $\Ii_{\Kk}$ in Definition~\ref{def:CKS} is finite, there exists a uniform 
{\bf collar width} $\eps>0$ such that all collar embeddings  $\io^\al_Y$, $\io^\al_U$ are defined on a neighbourhood of $\ov{A_{\eps}^\al}$, all coordinate changes 
between charts with nonempty boundary
are of collar form on $B^\al=A_{\eps}^\al$, and all charts without 
boundary have footprint contained in 
$Y\less \bigcup_{\al=0,1} \io_Y^\al(\ov{ A_{\eps}^\al}\times \p^\al Y)$.
In particular,  the footprints of the charts with nonempty boundary cover a neighbourhood of 
$\bigcup_{\al=0,1} \io_Y^\al(\ov{ A_{\eps}^\al}\times \p^\al Y) \subset Y$.
$\hfill\er$
\end{remark}

\begin{rmk}\rm  \label{rmk:cobordreal}
Let $\Kk$ be a Kuranishi cobordism from $\Kk^0$ to $\Kk^1$.
Its associated categories $\bB_{\Kk}, \bE_{\Kk}$ with projection, section, and footprint functor, as well as their realizations $|\bB_{\Kk}|, |\bE_{\Kk}|$  are defined as for Kuranishi atlases without boundary in \S\ref{ss:Ksdef},
and form cobordisms in the following sense.

\begin{itemlist}
\item
We can think of the virtual neighbourhood $|\Kk|$ of $Y$ as a 
collared cobordism with boundary components $\p^0|\Kk|\cong|\Kk^0|$ and $\p^1|\Kk|\cong|\Kk^{1}|$ in sense of Definition~\ref{def:Ycob}, with the exception that $|\Kk|$ is usually not locally compact or metrizable.
More precisely, using Remark~\ref{rmk:Ceps} we have collared neighbourhoods for some $\eps>0$,
$$
\iota_{|\Kk|}^0: [0,\eps) \times  |\Kk^0|  \hookrightarrow |\Kk|,  
\qquad
\iota_{|\Kk|}^1: (1-\eps,1]\times  |\Kk^1|   \hookrightarrow |\Kk| .
$$
These are induced by the natural functors $\iota^\al_{\bB_\Kk} : A^\al_\eps\times \bB_{\Kk^\al} \to \bB_\Kk$
given by the inclusions  
$\iota^\al_{U_I} : A^\al_\eps\times U^\al_I \hookrightarrow U_I$ 
on objects and 
$\iota^\al_{U_{IJ}} : A^\al_\eps\times U^\al_{IJ}\hookrightarrow U_{IJ}$
on morphisms, where $A^\al_\eps$ is defined in \eqref{eq:Naleps}. The axioms on the interaction of the coordinate changes with the collar neighbourhoods imply that the functors map to full subcategories that split $\bB_\Kk$ 
in the sense that there are no morphisms between any other object and this subcategory.
Hence the functors $\iota^\al_{\bB_\Kk}$ descend to topological embeddings $\iota_{|\Kk|}^\al : \bigl| A^\al_\eps \times \bB_{\Kk^\al}\bigr| \to |\Kk|$, i.e.\ homeomorphisms onto open subsets of $|\Kk|$. 
Here the product topology on 
$A^\al_\eps \times\bigl| \bB_{\Kk^\al}\bigr| \cong \bigl| A^\al_\eps \times \bB_{\Kk^\al}\bigr|$ coincides with the quotient topology by 
%Note to self : alternative reference \cite[Prop.~2.101]{James}: 
\cite[Ex.~29.11]{Mun}: Applied to the topological space $X:=\bigsqcup_{I\in\Ii_\Kk} U_I$ with equivalence relation $Y\subset X\times X$ induced by the morphisms, and the locally compact Hausdorff space $T:=A^\al_\eps$, it asserts that the product topology on $T\times \qu{X}{Y}$ coincides with the quotient topology induced by the product relation $T\times Y \hookrightarrow (T\times X) \times (T\times X)$. 

To check that $\iota_{|\Kk|}^\al$ are collared neighbourhoods in the sense of Definition~\ref{def:Ycob}, note that $\io^\al_{|\Kk|}\bigl(\{\al\}\times |\Kk^\al|\bigr)$ is contained in the open image of $\iota_{|\Kk|}^\al$.
Moreover, to see that $\iota_{|\Kk|}^\al\bigl(\{\al\} \times  |\Kk^\al| \bigr)\subset |\Kk|$ is closed we verify that its complement has open preimage in $\bigsqcup_{I\in\Ii_\Kk} U_I$ by noting that each $\io^\al_{U^\al_I}\bigl(\{\al\}\times U^\al_I\bigr) \subset U_I$ is closed. 
\smallskip
 
\item
The ``obstruction bundle'' consists of an analogous collared cobordism $|\bE_{\Kk}|$ with boundary components $\p^\al |\bE_{\Kk}|\cong |\bE_{\Kk^\al}|$ and a projection $|\pr_{\Kk}|: |\bE_{\Kk}|\to |\Kk|$  that has product form on the collared boundary $|\pr_{\Kk}| \circ \iota^\al_{|\bE_\Kk|}= \iota^\al_{|\Kk|}\circ \bigl(\id_{A^\al_\eps} \times  |\pr_{\Kk^{\al}}|\bigr)$ induced by the ``obstruction bundles'' of the boundary components, $|\pr_{\Kk^\al}|: |\bE_{\Kk^\al}|\to |\Kk^\al|$.
\smallskip

\item
The embeddings $\io^\al_{|\Kk|}$ extend the natural map between footprints
$$
|\s_{\Kk}|^{-1}(|0_\Kk|)
\;\;
\xleftarrow{\io_\Kk}  \;\;Y \;\; \xleftarrow{\io_Y^\al}    \;\;
A_\eps^\al \times \p^\al Y
\;\;\xrightarrow {{\rm id}\times \iota_{\Kk^{\al}}}\;\;
 A_\eps^\al\times |\s_{\Kk^{\al}}|^{-1}(|0_{\Kk^{\al}}|)  .
$$

\item 
If $\Kk$ is a Kuranishi concordance on $Y = [0,1]\times X$ then the footprint functor to $[0,1]\times X$  induces a continuous surjection
$$
{\rm pr}_{[0,1]}\circ \psi_{\Kk}
: \; \s_{\Kk}^{-1}(0_\Kk) \;\to\; [0,1]\times X  \;\to\;  [0,1] .
$$
In general we do not assume that this extends to a functor $\bB_{\Kk}\to [0,1]$.
However, all the Kuranishi  concordances that we construct explicitly do have this property. 
 $\hfill\er$
\end{itemlist}
\end{rmk}

\begin{defn}\label{def:mCKS} {\rm $\!\!$ \cite[Definition~4.2.1]{MW:top}}
A {\bf metric Kuranishi cobordism} on $Y$ is a Kuranishi cobordism $\Kk$ equipped with a metric $d$ on $|\Kk|$ that satisfies the admissibility conditions of Definition~\ref{def:metric} and has a metric collar as follows:

There is $\eps>0$ such that for $\al=0,1$ 
the collaring maps $\io^\al_{|\Kk|}: A^\al_\eps\times |\p^\al\Kk|\to |\Kk|$ 
of Remark~\ref{rmk:cobordreal} are defined and pull back $d$  
to the product metric 
\begin{equation} \label{eq:epsprod}
(\io^\al_{|\Kk|})^* d 
\bigl((t,x),(t',x')\bigr) \;=\; |t'-t|  + d^\al(x,x') \qquad \text{on} \;\; A^\al_\eps\times |\p^\al\Kk| ,
\end{equation}
where the metric $d^\al$ on 
$|\p^\al\Kk|$ is given by pullback of the restriction of $d$ 
to $\p^\al |\Kk|  = \io^\al_{|\Kk|}\bigl(\{\al\}\times |\p^\al\Kk|\bigr)$, which we denote by
$$
d^\al \,:=\; d|_{|\p^\al \Kk|} \,:=\; \io^\al_{|\Kk|}(\al,\cdot)^* d.
$$ 
In addition, we require
for all $y\in |\Kk| \less \io^\al_{|\Kk|}  \bigl( A^\al_\eps\times  |\p^\al\Kk|\bigr)$ 
\begin{equation}\label{eq:epscoll}
d\bigl( y , \io^\al_{|\Kk|}(\al + t, x )\bigr) \;\ge\; 
\eps - |t|
\qquad \forall \;
(\al + t,x)\in  A^\al_\eps \times  |\p^\al\Kk| .
\end{equation}
We call a metric on 
$|\Kk|$ {\bf admissible} if it satisfies the conditions of Definition~\ref{def:metric}, 
{\bf $\eps$-collared} if it satisfies \eqref{eq:epsprod}, \eqref{eq:epscoll}, 
and {\bf collared} if it is $\eps$-collared for some $\eps>0$.
\end{defn}

\begin{remark}\rm
(i)
Condition \eqref{eq:epscoll} controls the distance between points $\io^\al_{|\Kk|}(\al + t,x)$ in the collar and 
points $y$ outside of the collar.  
In particular, if $\de<\eps-|t|$, then the $\de$-ball around $\io^\al_{|\Kk|}(\al + t,x)$ is contained in the $\eps$-collar, while the $\de$-ball around 
$y\in |\Kk| \less \io^\al_{|\Kk|}  \bigl( A^\al_\eps\times  |\Kk^\al|\bigr)$
does not intersect the $|t|$-collar $\io^\al_{|\Kk|}\bigl(A^\al_{|t|}\times |\Kk^{\al}|\bigr)$.

\MS\NI
(ii)
Any admissible metric $d$ on $|\Kk|$ for a Kuranishi atlas $\Kk$ induces an admissible 
collared metric $d_\R + d$ on $|[0,1]\times \Kk|\cong [0,1] \times |\Kk| $, given by 
$$
\bigr(d_\R + d \bigl)\bigl( (t,x) , (t',x') \bigr) = |t'-t| + d(x,x')  .
$$ 
For short, we call $d_\R + d$ a {\bf product metric.} 
$\hfill\er$
\end{remark}

All reasonable flavours of cobordism and concordance form equivalence relations.
The following lemma proves this in most of the cases that will be used in this paper.
There is one more important cobordism notion --  metric tame cobordism -- 
that is treated in Remark~\ref{rmk:metcob} below.

\begin{lemma}\label{lem:cobord1}
\begin{enumerate}
\item
Additive (weak) cobordism is an equivalence relation between additive (weak) Kuranishi atlases.
\item 
Additive (weak) concordance is an equivalence relation between additive (weak) Kuranishi atlases on a fixed space $X$.
\item
Commensurate additive weak Kuranishi atlases are additively weak concordant.
\end{enumerate}
\end{lemma}

\begin{proof} 
Filtered (weak) cobordism is reflexive by Remark~\ref{rmk:restrict2}~(iv), 
and additive concordance is reflexive by Remark~\ref{rmk:restrict2}~(iii).
Symmetry and transitivity of the relations in (i) and (ii) are proven for topological Kuranishi cobordisms (resp.\ concordances) in \cite[Lemma~4.1.16]{MW:top}. 
The proof extends to smooth additive Kuranishi cobordisms since the relabeling and reparametrization of boundary collars does not affect smoothness or additivity of obstruction spaces. 

 To prove transitivity in (i),
 \cite{MW:top} constructs a concatenation of two (weak) topological Kuranishi cobordisms -- the first $\Kk^{[0,1]}$ on $(Y^{[0,1]}, \io_{Y^{[0,1]}}^0, \io_{Y^{[0,1]}}^1)$ from an atlas $\Kk^0$ on $Y^0= \p^0 Y^{[0,1]}$  to an atlas $\Kk^1$ on $Y^1= \p^1 Y^{[0,1]}$, and the second $\Kk^{[1,2]}$ on $(Y^{[1,2]}, \io_{Y^{[1,2]}}^0, \io_{Y^{[1,2]}}^1)$ from the atlas $\Kk^1$ on $Y^1 = \p^0 Y^{[1,2]}$ to an atlas $\Kk^2$ on $Y^2 = \p^1Y^{[1,2]}$.   
The result is a (weak) topological Kuranishi cobordism $\Kk^{[0,2]}$, whose construction in the special case of smooth atlases $\Kk^{[0,1]}$, $\Kk^{[1,2]}$ can be summarized as follows:
\begin{itemlist}
\item
The underlying space is
$\displaystyle \;
 Y^{[0,2]} = Y^{[0,1]} \underset{\scriptscriptstyle Y^1}{\cup}\, Y^{[1,2]} \; : = \;\quotient{ Y^{[0,1]}{\sqcup} Y^{[1,2]}}
{ \scriptstyle
 \iota^1_{Y^{[0,1]}}(1,\cdot) \,\sim\, \iota^0_{Y^{[1,2]}}(1,\cdot) } 
$
\item
The index set is $\displaystyle \; \Ii_{\Kk^{[0,2]}} :=\;\Ii_{[0,1)}  \;\sqcup\; \Ii_{\Kk^1} \;\sqcup\; \Ii_{(1,2]} \;$ with
$\Ii_{[0,1)}:=\Ii_{\Kk^{[0,1]}}\less\iota_{\Kk^{[0,1]}}^1(\Ii_{\Kk^1})$, $\Ii_{(1,2]}:=\Ii_{\Kk^{[1,2]}}\less\iota_{\Kk^{[1,2]}}^0(\Ii_{\Kk^1})$.
\item
The charts are
$\bK^{[0,2]}_{I} := \bK^{[0,1]}_{I}$ for $I\in\Ii_{[0,1)}$, and $\bK^{[0,2]}_{I} := \bK^{[1,2]}_{I}$ for $I\in\Ii_{(1,2]}$.
For $I\in\Ii_{\Kk^1}$ denote by
$I^{01}\in\Ii_{\Kk^{[0,1]}}\less\Ii_{[0,1)}$, $I^{12}\in\Ii_{\Kk^{[1,2]}}\less\Ii_{(1,2]}$ the labels of the charts that restrict to $\bK_I$. Then $\bK^{[0,2]}_{I}$ is given
by boundary connected sum of domains
$$
U^{[0,2]}_{I} \;:=\; U^{[0,1]}_{I^{01}} \underset{\scriptscriptstyle U^1_I}{\cup}  U^{[1,2]}_{I^{12}}
\;:=\;
\quotient{ U^{[0,1]}_{I^{01}} \sqcup U^{[1,2]}_{I^{12}} }{ 
\iota^1_{U_{I^{01}}}(x) \sim \iota^0_{U_{I^{12}}}(x)\quad \forall x\in U^1_I , 
}  
$$
obstruction space $E^1_I$, and section resp.\ footprint map equal to 
$s^{[0,1]}_{I^{01}}$ resp.\ $\psi^{[0,1]}_{I^{01}}$ on $U^{[0,1]}_{I^{01}}$
and by $s^{[1,2]}_{I^{12}}$ resp.\ $\psi^{[1,2]}_{I^{12}}$ on $U^{[1,2]}_{I^{12}}$.
This is a smooth Kuranishi chart since the obstruction spaces
$E^{[0,1]}_{I^{01}} = E^1_I = E^{[1,2]}_{I^{12}}$ are identical by Definition~\ref{def:Cchart}~(iv) and the
sections $s^{[0,1]}_{I^{01}}, s^{[1,2]}_{I^{12}}$ are of identical product form on the boundary collar.
\vspace{.09in}
\item
The coordinate changes are $\Hat\Phi^{[0,2]}_{IJ} := \Hat\Phi^{[0,1]}_{IJ}$ for $I,J\in\Ii_{[0,1)}$ and $\Hat\Phi^{[0,2]}_{IJ} := \bK^{[1,2]}_{IJ}$ for $I,J\in\Ii_{(1,2]}$.
For $J\in\Ii_{[0,1)}$ and $I\in\Ii_{\Kk^1}$ corresponding to $I^{01}\in\Ii_{\Kk^{[0,1]}}$ with $I^{01}\subsetneq J$ the coordinate change is $\Hat\Phi^{[0,2]}_{IJ} := \Hat\Phi^{[0,1]}_{I^{01} J}$, and similarly for $J\in\Ii_{(1,2]}$,  $I\in\Ii_{\Kk^1}$.
For $I,J\in\Ii_{\Kk^1}$ the coordinate charts corresponding to
$I^{01}, J^{01}\in\Ii_{\Kk^{[0,1]}}$, $I^{12},J^{12}\in\Ii_{\Kk^{[1,2]}}$
fit together to give a glued coordinate change $\Hat\Phi^{[0,2]}_{IJ}$ with domain and embeddings
$$
U^{[0,2]}_{IJ} := U^{[0,1]}_{I^{01}J^{01}} \underset{\scriptstyle U^1_{IJ}}\cup U^{[1,2]}_{I^{12}J^{12}} ,
\qquad
\phi^{[0,2]}_{IJ} :=
\left\{ \begin{aligned}
\phi^{[0,1]}_{I^{01}}  \;\quad &\text{on} \; U^{[0,1]}_{I^{01}J^{01}}\\
\phi^{[1,2]}_{I^{12}}  \;\quad &\text{on} \; U^{[1,2]}_{I^{12}J^{12}}
\end{aligned}
\right\}, 
\Hat\phi^{[0,2]}_{IJ} := \Hat\phi^1_{IJ} .
$$
These are smooth by product form on the boundary collars. Moreover, the index condition is local, hence transfers from $\Kk^{[0,1]},\Kk^{[1,2]}$ to $\Kk^{[0,2]}$.
\end{itemlist}

This shows transitivity of the (weak) cobordism relation.
To check that $\Kk^{[0,2]}$ is also additive, first note that for $I\in \Ii_{\Kk^1}$ we necessarily have $\{i\}\in \Ii_{\Kk^1}$ for all $i\in I$, so that additivity follows directly from additivity of $\Kk^1$,
$$
E^{[0,2]}_I \;=\; E^1_I
\;=\; \bigoplus_{i\in I} \wh\phi^1_{iI}\bigl(E^1_i\bigr)
\;=\;\bigoplus_{i\in I} \wh\phi^{[0,2]}_{iI}\bigl(E^{[0,2]}_i\bigr) .
$$
For $I\in \Ii_{[0,1)}$ (and similarly for $I\in\Ii_{(1,2]}$) each basic chart index $i\in I$ either lies in $\Ii_{[0,1)}\subset \Ii_{\Kk^{[0,1]}}$ or in $\Ii_{\Kk^1}=\Ii_{\p^1\Kk^{[0,1]}}\subset \Ii_{\Kk^{[0,1]}}$, but in both cases our construction identifies $E^{[0,2]}_i=E^{[0,1]}_i$, so that additivity follows from additivity of $\Kk^{[0,1]}$,
$$
E^{[0,2]}_I \;=\; E^{[0,1]}_I
\;=\; \bigoplus_{i\in I} \wh\phi^{[0,1]}_{iI}\bigl(E^{[0,1]}_i\bigr)
\;=\;\bigoplus_{i\in I} \wh\phi^{[0,2]}_{iI}\bigl(E^{[0,2]}_i\bigr) .
$$
This completes the proof of~(i).

To prove transitivity in (ii) we apply the same gluing construction in the special case when both $\Kk^{[0,1]}$ and $\Kk^{[1,2]}$ are additive Kuranishi cobordisms over $[0,1]\times X$. 
The result is a Kuranishi cobordism $\Kk^{[0,2]}$ over 
$$
Y^{[0,2]}=[0,1]\times X \underset{(1,x)\sim (0,x)}{\cup} [0,1]\times X.
$$
To obtain the required  Kuranishi concordance, we compose all footprint maps with the homeomorphism $Y^{[0,2]}\overset{\sim}{\to} [0,1]\times X$ induced by the affine maps $[0,1]\overset{\sim}{\to}[0,\frac 12]$ and $[0,1]\overset{\sim}{\to}[\frac 12,1]$ 
on the two factors.

To prove (iii) consider additive weak Kuranishi atlases $\Kk^0,\Kk^1$ and a common additive weak extension $\Kk^{01}$
as in Definition~\ref{def:Kcomm}.  Then an additive weak Kuranishi concordance $\Kk^{[0,1]}$ from $\Kk^0$ to $\Kk^1$ is given by
\begin{itemize}
\item
index set
$\displaystyle \;
\Ii_{\Kk^{[0,1]}} := \Ii_{\Kk^{01}}  =  \bigl\{ I\subset\Nn^{01} \,\big|\, {\textstyle\bigcap_{i\in I}} F^{01}_i \neq \emptyset \bigr\}$;
\item
charts $\displaystyle\; \bK^{[0,1]}_I := A_I\times \bK^{01}_I $
with $A_I= [0,\tfrac 23)$ for $I\subset\Nn^0$, $A_I = (\tfrac 13,1]$ for $I \subset \Nn^1$, and $A_I=(\tfrac 13,\tfrac 23)$ otherwise;
\item
coordinate changes
$\displaystyle\; \Hat\Phi^{[0,1]}_{IJ} := \id_{A_I\cap A_J}\times \Hat\Phi^{01}_{IJ}\; $
on 
$\displaystyle\; U^{[0,1]}_{IJ} := (A_I\cap A_J)\times U^{01}_{IJ}$.
\end{itemize}
This proves (iii) in the case of direct commensurability since additivity and weak cocycle conditions follow from the corresponding properties of $\Kk^{01}$. In particular, note that additivity makes use of the fact that $\Kk^{[0,1]}$ has the same index set as the additive Kuranishi atlas $\Kk^{01}$.
General commensurability thus gives rise to a sequence of concordances, which can be composed as in (ii) to yield a single concordance that proves (iii).
\end{proof}

%%%%%%%%%%%%%%%%%%%%%%%%%%%%%%%%%%%%%%%%%%%%%%%%%%%%%%%%%%%%%%%
\subsection{Tameness, shrinkings, and the Hausdorff property}
\label{ss:tame}   \hspace{1mm}\\ \vspace{-3mm}
%%%%%%%%%%%%%%%%%%%%%%%%%%%%%%%%%%%%%%%%%%%%%%%%%%%%%%%%%%%%%%%

The weaker notion of Kuranishi atlas in Definition~\ref{def:K} is crucial for two reasons. Firstly, in the application to moduli spaces of holomorphic curves, as outlined in \S\ref{ss:gw}, it is not clear how to construct Kuranishi atlases that satisfy the cocycle condition.
Secondly, it is hard to preserve the cocycle condition while manipulating Kuranishi atlases, for example by shrinking as we do below.
On the other hand, the constructions of transition data on compactified holomorphic curve moduli spaces, e.g.\ in 
\cite{MW:GW,Mcn}
yield additive weak Kuranishi atlases in the sense of Definition~\ref{def:Ku2}.
The purpose of this section now is to provide tools for refining additive weak Kuranishi atlases to achieve a tameness condition, which in turn implies the Hausdorff property of the virtual neighbourhood.

We will be able to import most of these results from our work in \cite{MW:top} on topological Kuranishi atlases and cobordisms. However, the notion of additivity does not apply to (weak) topological Kuranishi atlases (or cobordisms) since the bundles $\E_I\to U_I$ generally do not have linear or even isomorphic fibers in any sense. Instead, \cite[Definition~3.1.3]{MW:top} generalizes
the notion of additivity to the notion of a {\bf filtration} given by closed (and in the cobordism case collared) subsets $\E_{IJ}\subset \E_J$ for all $I,J\in\Ii_\Kk$ with $I\subset J$, which satisfy
\begin{enumerate}
\item $\E_{JJ}= \E_J$ and $\E_{\emptyset J} = \im 0_J$ for all $J\in\Ii_\Kk$;
\item 
$\Hat\Phi_{JK}\bigl(
\pr_J^{-1}(U_{JK})\cap
\E_{IJ}\bigr) = \E_{IK}\cap \pr_K^{-1}(\im \phi_{JK})$ for all $I,J,K\in\Ii_\Kk$ with
${I\subset J\subsetneq K}$;
\item  $\E_{IJ}\cap \E_{HJ} = \E_{(I\cap H)J}$ for all $I,H,J\in\Ii_\Kk$ with $I, H \subset J$;
\item  $\im \phi_{IJ}$ is an open subset of $\s_J^{-1}(\E_{IJ})$
for all $I,J\in \Ii_\Kk$ with $I\subsetneq J$.
\end{enumerate}
Here and throughout, we will treat atlases and cobordisms together, and note that any definition or result on these will also apply to concordances -- as special cases of cobordisms.

The following shows that additivity implies the existence of a natural filtration.
Moreover, we show that additivity allows us to extend 
the automatic 
control of transition maps on the zero sets $s_J^{-1}(0)$ to a 
weaker control on larger parts of the Kuranishi domains $U_J$.

\begin{lemma} \label{le:Ku3} 
Let $\Kk$ be an additive weak Kuranishi atlas or cobordism.
\begin{itemlist}
\item[\rm (a)] 
The induced weak topological Kuranishi atlas has a filtration $\E_{IJ}: = U_J\times \Hat\phi_{IJ}(E_I)$, using the conventions $E_{\emptyset}:= \{0\}$ and $\Hat\phi_{JJ}:=\id_{E_J}$. 
\item[\rm (b)]
We have for any $H,I,J\in\Ii_\Kk$
\begin{align} \label{eq:CIJ}
 H,I\subset J
\quad\Longrightarrow\quad &
s_J^{-1}(E_I) \cap s_J^{-1}(E_H) = s_J^{-1}(E_{I\cap H}), \\
\label{eq:CIJ0}
H\cap I=\emptyset \quad\Longrightarrow\quad & s_J^{-1}(E_I) \cap s_J^{-1}(E_H) = s_J^{-1}(0).
\end{align}
\end{itemlist} 
\end{lemma}

\begin{proof}    
A first implication of additivity is that for $H,I,J \in \Ii_\Kk$ with $H,I\subset J$ we have
\begin{align} \label{eq:addd}
\Hat\phi_{IJ}(E_I)  \; \cap \; \Hat\phi_{HJ}(E_H) &\;=\; 
 \Hat\phi_{(I\cap H) J}(E_{I\cap H }) .
\end{align}
In the special case $I\cap H=\emptyset$ this holds in the sense that $\Hat\phi_{IJ}(E_I)  \cap \Hat\phi_{HJ}(E_H) =\{0\}$.
Indeed, for $H,I\subset J$ we have a direct sum ${\textstyle\bigoplus_{i\in I\cup H}}\,\Hat\phi_{iI}(E_i) \subset E_J$ and hence
\begin{align*}
\Hat\phi_{IJ}(E_I)\cap \Hat\phi_{HJ}(E_H)
&=
\Hat\phi_{IJ}\Bigl({\textstyle\bigoplus_{i\in I}}\,\Hat\phi_{iI}(E_i)\Bigr)\cap \Hat\phi_{HJ}\Bigl({\textstyle\bigoplus_{i\in H}}\,\Hat\phi_{iH}(E_i)\Bigr)  \\
&=  \Bigl({\textstyle\bigoplus_{i\in I}}\,\Hat\phi_{iJ}(E_i)\Bigr)\cap \Bigl({\textstyle\bigoplus_{i\in H}}\,\Hat\phi_{iJ}(E_i)\Bigr)  \\
&={\textstyle\bigoplus_{i\in I\cap H}}\,\Hat\phi_{iJ}(E_i)
= \Hat\phi_{(I\cap H) J} \Bigl( {\textstyle\bigoplus_{i\in I\cap H}}\,\Hat\phi_{i (I\cap H)}(E_i) \Bigr) \\
& = \Hat\phi_{(I\cap H) J} (E_{I\cap H})   .
\end{align*}
This proves \eqref{eq:addd} for $I\cap H\neq\emptyset$. In case $I\cap H=\emptyset$ the first three lines hold and result in ${\textstyle\bigoplus_{i\in I\cap H}}\,\Hat\phi_{iJ}(E_i) = \{0\}$, consistent with $E_\emptyset=\{0\}$.
Now (b) follows from applying $s_J^{-1}$ to \eqref{eq:addd}.
To check (a), first note that $U_J\times\Hat\phi_{IJ}(E_I)\subset U_J\times E_J$ is closed since $U_J\subset U_J$ and $\Hat\phi_{IJ}(E_I)\subset E_J$ are closed.
Property (i) holds by definition, and property (iii) is equivalent to \eqref{eq:addd}.
Moreover, because $\Hat\Phi_{JK}=\phi_{JK}\times\Hat\phi_{JK}$, 
property (ii) follows from the weak cocycle condition, 
\begin{align*}
\Hat\Phi_{JK}\bigl(U_{JK}\times \Hat\phi_{IJ}(E_I)\bigr) \;
&= \; \im\phi_{JK} \times \Hat\phi_{JK}\bigl(\Hat\phi_{IJ}(E_I)\bigr) \;=\;  \im\phi_{JK} \times \Hat\phi_{IK}(E_I) \\
&= \; \bigl(U_K\times \Hat\phi_{IK}(E_I)\bigr) \cap \bigl( \im \phi_{JK} \times E_K \bigr) \;=\; \E_{IK}\cap (\pr_K^{-1}(\im \phi_{JK})) .
\end{align*}
Property (iv) follows from the index condition via Lemma~\ref{le:change}.
Finally, in the case of a Kuranishi cobordism, the subsets $\E_{IJ}=U_J\times \Hat\phi_{IJ}(E_I) \subset U_J\times E_J$ inherit collar form directly from $U_J$.
\end{proof}

\begin{rmk}\label{rmk:Ku30}\rm   (i)
Even if we set $\E_{IJ}:=U_J\times \Hat\phi_{IJ} (E_I)$, the filtration properties (i)-(iv) above may not imply additivity. 
For example, if $\Ii_\Kk = \bigl\{1,2,\{12\}\bigr\}$, the filtration conditions do not imply that $E_{12}$ is generated by $\Hat\phi_{1,12}(E_1)$ and $\Hat\phi_{2,12}(E_2)$.  
Here is an explicit example. Let $\Kk$ be an atlas with three charts as follows.
The domains $U_1=(0,2)$, $U_2=(1,3)$, $U_{12}=(1,2)\times (-1,1)^2$ are related by coordinate changes with $U_{1,12}=U_{2,12}=(1,2)$ and $\phi_{i(12)}$ the obvious inclusions onto $(1,2)\times \{(0,0)\}\subset U_{12}$.
The obstruction spaces are $E_1=\R, E_2=\R, E_{12}=\R^3$, and the sections $s_1(x)=0$, $s_2(x)=0$, $s_{12}(x,y,z) = (y,y,z)$ are intertwined by the inclusions $\Hat\phi_{1(12)}(e)=(e,0,0)$, $\Hat\phi_{2(12)}(e)=(0,e,0)$.
One can check as in Example~\ref{ex:change} that the index condition holds.
Moreover, the identity \eqref{eq:addd} does hold although the atlas is not additive: 
$$
\Hat\phi_{1(12)}(E_1)\oplus  \Hat\phi_{1(12)}(E_1) = \R\times\R\times\{0\} \subsetneq E_{12}.
$$
However, \eqref{eq:addd} suffices to check --- as in the proof of Lemma~\ref{le:Ku3} --- that the sets 
$\bigl(\E_{IJ}: = U_J\times \Hat\phi_{IJ}(E_I)\bigr)_{I\subset J}$ define a filtration of $\Kk$. \MS

\NI  (ii) \ The atlases in Example~\ref{ex:nonlin} that illustrate the failure of injectivity and linearity are nonadditive.
The second example illustrating nonlinearity does not adapt to the additive case since under additivity  the inclusions $\Hat\phi_{IJ}:E_I\to E_J$ are fixed. (Nevertheless, linearity need not hold as we show in Remark~\ref{rmk:LIN}.)
However, the first example does adapt as follows. We will construct an atlas $\Kk'$ on $X=S^1$ with four basic charts $(\bK_i')_{i=1,\dots,4}$, whose footprints  $F_i': = F_i$ for $ i=1,2,3$ are as before and $F_4'=S^1$.  
For $i=1,2,3$ we take the trivial charts $\bK_i'= (U_i'=F_i, E_i'= \{0\}, s_i'\equiv 0, \psi_i'=\id)$ and for $i=4$ we define
$$
\bK_4' = \bigl(U_4'=S^1\times (-\eps,\eps),\ E_4'=\C,\  s_4'(z,x) = x, \ \psi_4'(z,0) = z\bigr).
$$ 
Then we may choose the transition charts $\bK_{i4}'$ for $i=1,2,3$ to be the basic charts $\bK_i$ from Example~\ref{ex:nonlin} with obstruction space $E_{i4}'=\C = E_i' \times E_4'$, with the exception that we remove $S^1\times [-\eps,\eps]$ from the domain $U_{34}'$ of $\bK_{34}'$, since this contains $\phi_{4,34}(U_4')$.
It is now easy to complete the definition of the additive atlas $\Kk'$ much as before so that the projection $\pi_{\Kk'}: U_{34}'\to |\Kk'|$ is not injective. 
$\hfill\er$
\end{rmk}

Now, given an additive weak Kuranishi atlas, we wish to achieve desirable topological properties of the virtual neighbourhood $|\Kk|=\bigsqcup_{I\in\Ii_\Kk} U_I / \!\!\sim$. This firstly requires a strengthening of the weak cocycle condition in order for the transition data even to induce an equivalence relation between the Kuranishi domains $U_I$. Secondly, in order for the quotient topology to be Hausdorff, the relation needs to be closed.
These properties are guaranteed -- for atlases as well as cobordisms -- by the following notion of tameness.

\begin{defn}\label{def:tame} {\rm $\!\!$ \cite[Definition~3.1.10]{MW:top}}
A (weak) Kuranishi atlas/cobordism is {\bf tame} if it is additive and for all $I,J,K\in\Ii_\Kk$ we have
\begin{align}\label{eq:tame1}
U_{IJ}\cap U_{IK}&\;=\; U_{I (J\cup K)}\qquad\qquad
\quad\;
\forall I\subset J,K ;\\
\label{eq:tame2}
\phi_{IJ}(U_{IK}) &\;=\; U_{JK}\cap s_J^{-1}(E_I)
 \qquad\forall I\subset J\subset K.
\end{align}
Here we allow equalities, using the notation $U_{II}:=U_I$ and $\phi_{II}:={\rm Id}_{U_I}$.
Further, to allow for the possibility that $J\cup K\notin\Ii_\Kk$, we define
$U_{IL}:=\emptyset$ for $L\subset \{1,\ldots,N\}$ with $L\notin \Ii_\Kk$.
Therefore \eqref{eq:tame1} includes the condition
$$
U_{IJ}\cap U_{IK}\ne \emptyset
\quad \Longrightarrow \quad F_J\cap F_K \ne \emptyset  \qquad \bigl( \quad \Longleftrightarrow\quad
J\cup K\in \Ii_\Kk \quad\bigr).
$$
\end{defn}

The notion of tameness generalizes the identities $F_J\cap F_K=F_{J\cup K}$ and $\psi_J^{-1}(F_{K}) = U_{JK}\cap s_J^{-1}(0)$ between the footprints and zero sets, which we can include into \eqref{eq:tame1} and \eqref{eq:tame2} as the case $I = \emptyset$, by using the notation 
\begin{equation}\label{eq:empty}
U_{\emptyset J}: = F_J,\qquad \phi_{\emptyset J}:=\psi_J^{-1}, 
\qquad E_{\emptyset}:= \{0\} .
\end{equation}
Indeed, the first tameness condition \eqref{eq:tame1} extends the identity for intersections of footprints -- which is equivalent to $\psi_I^{-1}(F_J)\cap \psi_I^{-1}(F_K) = \psi_I^{-1}(F_{J\cup K})$ for all $I\subset J,K$ 
-- to the domains of the transition maps in $U_I$. 
In particular, with $I\subset J\subset K$ it implies nesting of the domains of the transition maps,
\begin{equation}\label{eq:tame4}
U_{IK}\subset U_{IJ} \qquad\forall I\subset J \subset K.
\end{equation}
(This in turn generalizes the $I=\emptyset$ case $F_K\subset F_J$ for $J \subset K$.)
The second tameness condition \eqref{eq:tame2} extends the relation between footprints and zero sets -- equivalent to $\phi_{IJ}(\psi_I^{-1}(F_K)) = U_{JK}\cap s_J^{-1}(0)$ for all $I\subset J\subset K$ --
to a relation between domains of transition maps and preimages under the section of corresponding subbundles.
In particular, with $J=K$ it controls the image of the transition maps,
generalizing the $I=\emptyset$ case $\psi_J^{-1}(F_J) =  s_J^{-1}(0)$ to
$\phi_{IJ}(U_{IJ}) =  s_J^{-1}(E_I)$ for all $I\subset J$.

An important topological consequence of this identity is that the images of the transition maps $\phi_{IJ}$ 
are closed subsets of the Kuranishi domain $U_J$ for all $I\subset J$,
\begin{equation}\label{eq:tame3}
\im\phi_{IJ} =  s_J^{-1}(E_I) \subset U_J ,
\end{equation}
since $E_I\subset E_J$ is a closed subset and the section $s_J$ is continuous.
Moreover, \eqref{eq:tame3} strengthens the inclusion $\im\phi_{IJ}\subset  s_J^{-1}(E_I)$ 
which follows from the section compatibility \eqref{eq:map-square} of any coordinate change.
It is however neither necessary nor sufficient for the ``infinitesimal tameness'' condition
$\im\rd\phi_{IJ}=(\rd s_J)^{-1}\bigr(\Hat\phi_{IJ}(E_I)\bigl)$,
that is a part of the index condition \eqref{inftame}.
A first 
motivation for our notion of tameness is the following.

\begin{lemma}\label{le:tame0}
Suppose that the weak Kuranishi atlas or cobordism $\Kk$ is tame. Then the strong cocycle condition in Definition~\ref{def:cocycle} is satisfied, i.e.\ we have $\Hat\Phi_{JK}\circ \Hat\Phi_{IJ}=\Hat\Phi_{IK}$ for each $I,J\subset K$ with equality of domains $U_{IJ}\cap \phi_{IJ}^{-1}(U_{JK}) = U_{IK}$.
\end{lemma}
\begin{proof}
From the tameness conditions \eqref{eq:tame2} and \eqref{eq:tame3} we obtain for all $I\subset J\subset K$ 
$$
\phi_{IJ}(U_{IK})
= U_{JK}\cap s_J^{-1}(E_{I}) 
= U_{JK}\cap \phi_{IJ}(U_{IJ}) .
$$
Applying $\phi_{IJ}^{-1}$ to both sides and using \eqref{eq:tame4} implies equality of the domains.
Then the weak cocycle condition $\phi_{JK}\circ \phi_{IJ}=\phi_{IK}$ on the overlap of domains is identical to the strong cocycle condition.
\end{proof}

This Lemma shows in particular that a tame weak Kuranishi atlas/cobordism is in fact a Kuranishi atlas/cobordism, so that we will in the following label it as tame Kuranishi atlas/cobordism.
The following shows that this notion is compatible with cobordism notions and implies the desired topological properties of the virtual neighbourhood -- which is well defined due to the cocycle condition being satisfied by Lemma~\ref{le:tame0}.

\begin{rmk} \label{rmk:restrict2}\rm  

(i) The weak topological Kuranishi atlas/cobordism induced by a tame Kuranishi atlas/cobordism is tame in the sense of \cite[Definition~3.1.10]{MW:top}, since additivity by Lemma~\ref{le:Ku3}~(a) yields a canonical filtration with $\s_J^{-1}(\E_{IJ})=s_J^{-1}(E_I)$.
\MS\NI
(ii)
If $\Kk$ is a tame Kuranishi atlas, then $[0,1]\times \Kk$ is a tame Kuranishi concordance.

\MS
\NI (iii) 
If $\Kk$ is a tame Kuranishi cobordism, then both restrictions $\p^\al \Kk$ are also tame. 
$\hfill\er$  
\end{rmk}

The main motivation for our notion of tameness is the fact that it implies all the desirable topological properties of the virtual neighbourhood, as follows.

\begin{proposition}\label{prop:Khomeo}
Any tame Kuranishi atlas/cobordism $\Kk$ has the following Hausdorff, homeomorphism, and linearity properties:
\begin{itemize}
\item
The realizations $|\Kk|$ and $|\bE_{\Kk}|$ are Hausdorff in the quotient topology.
\item
For each $I\in \Ii_{\Kk} = \Ii_\Kk$ the projection maps $\pi_{\Kk}: U_I\to |\Kk|$ and
$\pi_{\Kk}: U_I\times E_I \to |\bE_{\Kk}|$ are homeomorphisms onto their images and fit into a commutative diagram
\begin{equation}\label{eq:lin}
\begin{array}{ccc} 
U_I\times E_I & \stackrel{\pi_{\Kk}}\longhookrightarrow & |\bE_{\Kk}|  \quad \\
\;\; \downarrow \scriptstyle \pr_I    & & \;\; \downarrow \scriptstyle |\pr_{\Kk}| \\
U_I &
\stackrel{\pi_{\Kk}} \longhookrightarrow  &|\Kk| \quad 
\end{array}
\end{equation}
\item
The horizontal maps in the above diagram intertwine the vector space structure on $E_I$ with a unique vector space structure on the fibers of $|\pr_{\Kk}|$.
\end{itemize}
\end{proposition}

\begin{proof} 
Up to the linearity property, this follows from applying \cite[Theorem~3.1.9]{MW:top}  to the induced tame topological Kuranishi atlas/cobordism.
The linearity property more precisely asserts that the requirement of compatibility of linear 
structures in \eqref{eq:lin} induces a unique linear structure on each fiber $|\pr_{\Kk}|^{-1}(p)$.
To prove this we need a couple of facts from \cite[Lemma~3.2.3]{MW:top} about the induced tame topological Kuranishi atlas/cobordism.

\begin{itemlist}
\item [(a)]  
The implication $(i)\Rightarrow(ii)$ in 
part (a) of that Lemma 
says that for any  $I,J\in\Ii_{\Kk}$ 
$$
\pi_{\Kk}(U_I)\cap\pi_{\Kk}(U_J) \neq \emptyset \qquad\Longrightarrow\qquad  I\cup J \in \Ii_{\Kk}  , \quad \pi_{\Kk}(U_I)\cap\pi_{\Kk}(U_J) \subset \pi_\Kk(U_{I\cup J }) .
$$
\item[(b)]  
Part (b) of 
that Lemma implies that
for $I,J\in\Ii_{\Kk}$ with $I\subset J$ we have the factorization
$\pi_{\bE_{\Kk}}\big|_{U_{IJ}\times E_{I}} = 
\pi_{\bE_{\Kk}}\big|_{U_J\times E_J} \circ (\phi_{IJ}\times\Hat\phi_{I J})$.
\item[(c)]
$\pi_{\Kk}:U_I \to |{\Kk}|$ and $\pi_{\Kk}: U_I\times E_I \to |\bE_{\Kk}|$ are injective for each $I\in\Ii_{\Kk}$
by part (c) of 
that Lemma.
\end{itemlist}

Now to prove linearity of the fiber over $p\in |{\Kk}|$ we denote the union of index sets for which $p\in \pi_{\Kk}(U_I)$ by
$$
I_p:= \bigcup_{I\in\Ii_{\Kk}, p\in \pi_{\Kk}(U_I)}  I
\qquad \subset \{1,\ldots,N\} .
$$
Repeated use of fact (a) above shows that $I_p\in\Ii_{\Kk}$.
Moreover, $x_p:=\pi_{\Kk}^{-1}(p)\cap U_{I_p}$ exists by (a) and is unique by (c).
Next, any element $[I,x,e]\in |\pr_{{\Kk}}|^{-1}(p)$ in the fiber is represented by some vector over $(I,x)\in \pi_{\Kk}^{-1}(p)$, so we have $I\subset I_p$ and $\phi_{I I_p}(x)=x_p$, and hence
$(I,x, e)\sim (I_p,x_p,\Hat\phi_{I I_p}(e))$. Thus $\pi_{\bE_{\Kk}}:\{x_p\}\times E_{I_p}\to |\pr_{\Kk}|^{-1}(p)$ is surjective, and by (c) also injective.
Thus the requirement of linearity for this bijection induces a unique linear structure on the fiber $|\pr_{\Kk}|^{-1}(p)$.
To see that this is compatible with the injection $\pi_{\bE_{\Kk}}:\{x\}\times E_{I}\to |\pr_{\Kk}|^{-1}(p)$ for any $(I,x)\sim(I_p,x_p)$ note again that $I\subset I_p$ since $I_p$ was defined to be maximal. 
Now by (b) the embedding factors as $\pi_{\bE_{\Kk}}\big|_{\{x\}\times E_{I}} = \pi_{\bE_{\Kk}}\big|_{\{x_p\}\times E_{I_p}} \circ \Hat\phi_{I I_p}$, where $\Hat\phi_{I I_p}$ is linear by definition of coordinate changes. Thus $\pi_{\bE_{\Kk}}\big|_{\{x\}\times E_{I}}$ is linear as well.
\end{proof}

\begin{rmk}\label{rmk:LIN}\rm
It is tempting to think that additivity alone is enough to imply that the fibers of $|\pr_\Kk|:|\bE_\Kk|\to|\Kk|$ are vector spaces. However, if the first tameness condition \eqref{eq:tame1} fails because there is $x\in (U_{IJ}\cap U_{IK}) \less U_{I(J\cup K)}$, then both $E_J$ and $E_K$ embed into the fiber $|\pr_\Kk|^{-1}([I,x]))$, but may not be summable, since such sums are well defined by additivity only in $E_{J\cup K}$.
$\hfill\er$
\end{rmk}

Finally, tameness of Kuranishi atlases will be achieved in Theorem~\ref{thm:K} below by the following notion of shrinking of the footprints along with the domains of charts and transition maps, which again is the direct translation of a notion for topological Kuranishi atlases in \cite{MW:top}.

\begin{defn}\label{def:shr} {\rm $\!\!$ \cite[Definition~3.3.2]{MW:top}}
Let $\Kk=(\bK_I,\Hat\Phi_{I J})_{I, J\in\Ii_\Kk, I\subsetneq J}$ be a weak Kuranishi atlas/cobordism.   Then a weak Kuranishi atlas/cobordism $\Kk'=(\bK_I',\Hat\Phi_{I J}')_{I, J\in\Ii_{\Kk'}, I\subsetneq J}$ is a {\bf shrinking} of $\Kk$~if the following holds.
\begin{enumerate}
\item  
The footprint cover $(F_i')_{i=1,\ldots,N'}$ is a shrinking of the cover $(F_i)_{i=1,\ldots,N}$ in the sense that the $F_i'\sqsubset F_i$ are precompact open subsets which cover $X= \bigcup_{i=1,\ldots,N} F'_i$, and are such that for all subsets $I\subset \{1,\ldots,N\}$ we have
\begin{equation} \label{same FI}
F_I: = {\textstyle\bigcap_{i\in I}} F_i \;\ne\; \emptyset
\qquad\Longrightarrow\qquad
F'_I: = {\textstyle\bigcap_{i\in I}} F'_i \;\ne\; \emptyset .
\end{equation}
In particular the numbers $N=N'$ of basic charts and index sets $\Ii_{\Kk'} = \Ii_\Kk$ agree.
\item
For each $I\in\Ii_\Kk$ the chart $\bK'_I$ is the restriction of $\bK_I$ to a precompact domain 
$U_I'\sqsubset U_I$ as in Definition \ref{def:restr}.
\item
For each $I,J\in\Ii_\Kk$ with $I\subsetneq J$ the coordinate change $\Hat\Phi_{IJ}'$ is the restriction of $\Hat\Phi_{IJ}$  to the open subset $U'_{IJ}: =  \phi_{IJ}^{-1}(U'_J)\cap U'_I$
 as in Lemma~\ref{le:restrchange}.
\end{enumerate}

A {\bf tame shrinking} of $\Kk$ is a shrinking $\Kk'$ that is tame in the sense of Definition~\ref{def:tame}.
A {\bf preshrunk tame shrinking} of $\Kk$ is a tame shrinking $\Kk_{sh}$ of $\Kk$ that is obtained as shrinking of an intermediate tame shrinking $\Kk'$ of $\Kk$. 
\end{defn}

Note that a shrinking is determined by the choice of domains $U'_I\sqsubset U_I$,
so can be considered as the restriction of $\Kk$ to the subset $\bigsqcup_{I\in\Ii_\Kk} U_I'\subset\Obj_{\bB_\Kk}$. 
Any such restriction of a weak Kuranishi atlas/cobordism preserves the weak cocycle condition (since it only requires equality on overlaps), and since restriction also preserves the index condition, the result of a shrinking is indeed a weak Kuranishi atlas/cobordism.
Similarly, any shrinking of an additive altas/cobordism is additive since the obstruction spaces and index sets are unaffected by the shrinking.
Finally, a preshrunk tame shrinking is a restriction to $U^{sh}_I\sqsubset U'_I\sqsubset U_I$, where both the induced transition domains $U^{sh}_{IJ}$ and the intermediate $U'_{IJ}$ satisfy the tameness conditions 
\eqref{eq:tame1}, \eqref{eq:tame2}.

\begin{thm}\label{thm:K}
Let $\Kk$ be an additive weak Kuranishi atlas/cobordism.
\begin{enumerate}
\item
There exists a preshrunk tame shrinking of $\Kk$.
\item  
Any preshrunk tame shrinking of $\Kk$ is a metrizable Kuranishi atlas/cobordism
and satisfies the Hausdorff, homeomorphism, and linearity properties in Proposition~\ref{prop:Khomeo}.
\item
Any two preshrunk tame shrinkings $\Kk^0_{sh}, \Kk^1_{sh}$ of an additive weak Kuranishi atlas $\Kk$ with choices of metrics $d^\al$ on $|\Kk^\al_{sh}|$ are concordant by a metric tame Kuranishi concordance $(\widetilde\Kk,d)$. That is, $\widetilde\Kk$ is a tame Kuranishi concordance with $\p^\al \widetilde\Kk = \Kk^\al_{sh}$ and $d$ is an admissible collared metric on $|\widetilde\Kk|$ with $d|_{|\p^\al \widetilde\Kk|} = d^\al$.
\item
Let $\Kk$ be an additive weak Kuranishi cobordism, let $\Kk^0_{sh}, \Kk^1_{sh}$ be preshrunk tame shrinkings of $\p^0\Kk$ and $\p^1\Kk$, and let admissible metrics $d^\al$ on $|\Kk^\al_{sh}|$ be given.
Then there is a preshrunk tame shrinking of $\Kk$ and an admissible collared metric $d$ on $|\Kk|$ that provide a metric tame Kuranishi cobordism $\widetilde\Kk$ with $\p^\al \widetilde\Kk = \Kk^\al_{sh}$ and $d|_{|\p^\al \widetilde\Kk|} = d^\al$.
\end{enumerate}
\end{thm}

\begin{proof} 
Proposition~\ref{prop:Khomeo} proves (ii) up to metrizability since a preshrunk tame shrinking is in particular a tame Kuranishi atlas/cobordism. Metrizability is shown in \cite[Proposition~3.3.8]{MW:top} for a Kuranishi atlas and in \cite[Theorem~4.2.6]{MW:top} for a Kuranishi cobordism. The latter also proves (iv), which in particular implies existence of preshrunk tame shrinkings (i) for Kuranishi cobordisms.
Existence and uniqueness of preshrunk tame shrinkings as in (i), (iii) is proven for the induced filtered weak topological Kuranishi atlas
in \cite[Theorem~3.1.9]{MW:top}. This transfers to the present smooth context since the virtual neighbourhoods are identified, so that metrizability notions are the same in the smooth and topological context. Moreover, shrinkings of the underlying topological atlas/cobordism induce shrinkings in the sense of Definition~\ref{def:shr}, and shrinkings preserve additivity, so that the topological notion of tameness (involving the canonical filtration) also implies the smooth notion (involving additivity).
\end{proof}

\begin{rmk}\label{rmk:metcob}\rm  A metric tame Kuranishi  cobordism is a metric Kuranishi  cobordism in the sense of Definition~\ref{def:mCKS} that is also tame.  
The construction of smooth cobordisms in Lemma~\ref{lem:cobord1} can be combined with constructions of metrics in \cite[Proposition~4.2.3]{MW:top}
to show that metric tame Kuranishi cobordism is an equivalence relation.  
\end{rmk}

We end this section with further topological properties of the virtual neighbourhood of a tame Kuranishi atlas that will be useful when constructing the virtual fundamental class in \S\ref{s:VMC}.  For that purpose we need to be careful in differentiating between the quotient and subspace topology on subsets of the virtual neighbourhood, as follows.

\begin{definition} \label{def:topologies} {\rm $\!\!$ \cite[Definition~3.1.14]{MW:top}}
For any subset $\Aa\subset \Obj_{\bB_\Kk}$ of the union of domains of a Kuranishi atlas/cobordism $\Kk$, we denote by 
$$
\|\Aa\|:=\pi_\Kk(\Aa)\subset|\Kk| , 
\qquad \mbox{ resp. } \quad
|\Aa|:=\pi_\Kk(\Aa)\cong \quot{\Aa}{\sim}\ ,
$$
the set $\pi_\Kk(\Aa)$ equipped with its subspace topology induced from the inclusion $\pi_\Kk(\Aa)\subset|\Kk|$, resp.\ its quotient topology induced from the inclusion $\Aa\subset \Obj_{\bB_\Kk}$ and the equivalence relation $\sim$ on $\Obj_{\bB_\Kk}$ (which is generated by all morphisms in $\bB_\Kk$, not just those between elements of $\Aa$).
\end{definition}

\begin{prop}\label{prop:Ktopl1}  
Let $\Kk$ be a tame Kuranishi atlas/cobordism.
\begin{enumerate}
\item
For any subset $\Aa\subset \Obj_{\bB_\Kk}$ the identity map $\id_{\pi_\Kk(\Aa)}: |\Aa| \to \|\Aa\|$ is continuous.
\item 
If $\Aa \sqsubset \Obj_{\bB_\Kk}$ is precompact, then both $|\ov\Aa|$ and $\|\ov\Aa\|$ are compact. In fact, the quotient and subspace topologies on $\pi_\Kk(\ov\Aa)$ coincide, that is $|\ov\Aa|=\|\ov\Aa\|$ as topological spaces.
\item
If $\Aa \sqsubset \Aa' \subset \Obj_{\bB_\Kk}$, then $\pi_\Kk(\ov{\Aa}) = \ov{\pi_\Kk(\Aa)}$ and $\pi_\Kk(\Aa) \sqsubset \pi_\Kk(\Aa')$ in the topological space $|\Kk|$.
\item
If  $\Aa \sqsubset \Obj_{\bB_\Kk}$ is precompact, then $\|\ov{\Aa}\|=|\ov\Aa|$ is metrizable; in particular this implies that $\|\Aa\|$ is metrizable.
\end{enumerate}
\end{prop}

\begin{proof}
This follows by applying \cite[Proposition~3.1.16]{MW:top}  to the induced tame topological Kuranishi atlas/cobordism and observing that both the quotient and the relative topologies induced on subsets $\pi_\Kk(\Aa)\subset|\Kk|$ are the same in the smooth and topological context.
\end{proof}

\begin{remark} \label{rmk:hom} \rm
In many cases we will be able to identify different topologies on subsets of the virtual neighbourhood $|\Kk|$ by appealing to the following elementary {\bf nesting uniqueness of compact Hausdorff topologies}:

Let $f:X\to Y$ be a continuous bijection from a compact topological space $X$ to a Hausdorff space $Y$. Then $f$ is in fact a homeomorphism. 
Indeed, it suffices to see that $f$ is a closed map, i.e.\ maps closed sets to closed sets, since that implies continuity of $f^{-1}$. But any closed subset of $X$ is also compact, and its image in $Y$ under the continuous map $f$ is also compact, hence closed since $Y$ is Hausdorff.

In particular, if $Z$ is a set with nested compact Hausdorff topologies $\Tt_1\subset\Tt_2$, then $\id_Z: (Z,\Tt_2)\to (Z,\Tt_1)$ is a continuous bijection, hence homeomorphism, i.e.\ $\Tt_1=\Tt_2$.
$\hfill\er$
\end{remark}

We end this section by noting that the additivity condition gives more information than the mere existence of the filtration needed to define tameness.  For example, the following lemma shows that we may obtain a limited transversality for the embeddings of the domains involved in coordinate changes, a property that is crucial to guarantee the existence of coherent (i.e.\ compatible with coordinate changes) perturbations of the canonical section $\s_\Kk$ of a tame Kuranishi atlas.
However, due to further technical complications, we will not use it directly in the constructions of \S\ref{ss:const}.

\begin{lemma}\label{le:phitrans} 
If $\Kk$ is a tame Kuranishi atlas, then for any  $J\in\Ii_\Kk$ and $H, I\subset J$ with $H \cap I\ne \emptyset$ we have
\begin{equation}\label{eq:tame5}
\im \phi_{H J}\cap\im \phi_{IJ}=\im \phi_{(H\cap I) J} ,
\end{equation}
and this intersection is transverse within the manifold $\im \phi_{(H\cup I)J}$.

In case $H\cap I=\emptyset$ we have the intersection identity\footnote{This intersection identity is consistent with \eqref{eq:empty} since $F_H\cap F_I \supset F_J$ so that $\s_J^{-1}(0) = \psi_J^{-1}(F_J) = \psi_J^{-1}(F_H\cap F_I) =  \im \phi_{\emptyset J}$.} 
$\im \phi_{H J} \cap \im \phi_{IJ} = s_J^{-1}(0)$, but no automatic transversality.
\end{lemma}
\begin{proof}  
Tameness \eqref{eq:tame3} identifies $\im \phi_{L J}=s_J^{-1}(E_{LJ})$ for $L=H,I,H\cap I$, so that the intersection identity holds by the filtration property \eqref{eq:CIJ}.
In case $H\cap I=\emptyset$ it gives $s_J^{-1}(E_{HJ})\cap s_J^{-1}(E_{IJ}) = s_J^{-1}(0)$ since $E_{\emptyset J}=\im 0$.

It remains to prove the transversality for $H\cap I \neq \emptyset$,
\begin{equation} \label{klaim}
\im\rd\phi_{(H\cap I)J} \; =\;  \im\rd\phi_{HJ} \;+\; \im\rd\phi_{IJ}
\end{equation}
at the appropriate base points coming from \eqref{eq:tame5}, which we suppress throughout.
For that purpose let us first consider the case $J=H\cup I$. 
Then the isomorphism $\rd s_I : \qu{\rT U_I}{\im\rd\phi_{LI}} \to \qu{E_I}{\im\Hat\phi_{LI}}$ from the tangent bundle condition \eqref{tbc} for $L:=H\cap I\in\Ii_\Kk$ implies that for any choice of complement $V_I$ in $\rT U_I = \im\rd\phi_{LI} \oplus V_I$ we have $E_I =  \im \Hat\phi_{LI}\oplus \rd s_I (V_I)$. Working analogously for $H$ in place of $I$, we obtain decompositions
\begin{equation}\label{tbcH}
\rT U_\bullet = \im\rd\phi_{L\bullet} \oplus V_\bullet 
\qquad\text{with}\qquad 
E_\bullet = \im \Hat\phi_{L\bullet} \oplus  \rd s_\bullet (V_\bullet) 
\qquad\text{for}\;\; \bullet = H,I.
\end{equation}
On the other hand, additivity gives the sum $E_J = \Hat\phi_{HJ}(E_H) + \Hat\phi_{IJ}(E_I)$ by subspaces whose intersection is
$$
\Hat\phi_{HJ}(E_H) \cap \Hat\phi_{IJ}(E_I) \;=\; \im\Hat\phi_{LJ} \;=\; \Hat\phi_{HJ}(\im\Hat\phi_{LH}) \;=\; 
\Hat\phi_{IJ}(\im\Hat\phi_{LI}),
$$ 
so that we obtain the direct decomposition
\begin{align*}
E_J 
&= \Hat\phi_{HJ}\bigl(\rd s_H(V_H)\bigr) \oplus \im\Hat\phi_{LJ} \oplus  \Hat\phi_{IJ}\bigl(\rd s_I(V_I)\bigr) \\
&= \rd s_J\bigl(\rd\phi_{HJ}(V_H)\bigr) \oplus \im\Hat\phi_{LJ} \oplus  \rd s_J(\rd \phi_{IJ}(V_I)\bigr) .
\end{align*}
Now the isomorphism  in  the tangent bundle condition \eqref{tbc} induced by $\rd s_J$  has the form
$$
\rd s_J : \qu{\rT U_J}{\im\rd\phi_{LJ}} \to \qu{E_J}{\im\Hat\phi_{LJ}} \cong \rd s_J\bigl(\rd\phi_{HJ}(V_H) 
\oplus \rd \phi_{IJ}(V_I)\bigr).
$$ 
Here the subspaces $\rd\phi_{HJ}(V_H),\rd \phi_{IJ}(V_I)$ intersect in $\{0\}$ since $\rd s_J$ maps them to complementary subspaces of $E_J$, and neither intersects $\ker \rd s_J$ nontrivially.
In fact, we have the inclusion $\ker\rd s_J = \rd\phi_{LJ}(\ker \rd s_L) \subset \im\rd\phi_{LJ}$ and $\rd\phi_{\bullet J}(V_\bullet)$ intersects $\im\rd\phi_{LJ}=\rd\phi_{\bullet J}(\im\rd\phi_{L\bullet})$ trivially for $\bullet = H,I$ by \eqref{tbcH}.
This means that $\rd\phi_{HJ}(V_H) \oplus \rd \phi_{IJ}(V_I)$ represents a subspace of $\qu{\rT U_J}{\im\rd\phi_{LJ}}$, but since the bijection $\rd s_J$ maps it onto $\qu{E_J}{\im\Hat\phi_{LJ}}$, this is in fact represents the whole space, i.e.\ is a complement of $\im\rd\phi_{LJ}$ in
$$
\rT U_J = \im\rd\phi_{LJ} \oplus \rd\phi_{HJ}(V_H) \oplus \rd\phi_{IJ}(V_I) .
$$ 
This decomposition proves \eqref{klaim} since $\im\rd\phi_{LJ}\subset\rd\phi_{\bullet J}$ for $\bullet = H,I$ as above.

For general $H,I\subset J$ we have $\rT U_{H\cup I} = \im\rd\phi_{H(H\cup I)} + \im\rd\phi_{I(H\cup I)}$ from above and can use the linearization of the cocycle condition, $\rd \phi_{(H\cup I) J} \circ\rd\phi_{\bullet (H\cup I)} = \rd \phi_{\bullet J}$ to obtain
$$
\im\rd\phi_{(H\cup I)J}
\;=\; 
\rd\phi_{(H\cup I)J}\bigl(  \im\rd\phi_{H(H\cup I)} \;+\; \im\rd\phi_{I(H\cup I)}  \bigr) 
\;=\; 
\im\rd\phi_{HJ}  \;+\; \im\rd\phi_{IJ} .
$$
This proves \eqref{klaim} in general, thus finishes the proof.
\end{proof}

%%%%%%%%%%%%%%%%%%%%%%%%%%%%%%%%%%%%%%%%%%%%%%%%
\section{Reductions and perturbation sections}\label{s:red}
%%%%%%%%%%%%%%%%%%%%%%%%%%%%%%%%%%%%%%%%%%%%%%%%%

The next step in the construction of the VFC is to show that the canonical section $\s_\Kk$ has suitable transverse perturbations. 
As a preliminary step, \S\ref{ss:red} provides reductions of the cover of the virtual neighbourhood $|\Kk|$ by the images of the domains $\pi_\Kk(U_I)$. The goal here is to obtain a  cover by a partially ordered set of Kuranishi charts, with coordinate changes governed by the partial order, in order to permit an iterative construction of perturbations.
In \S\ref{ss:sect} we introduce the notion of transverse perturbations in a reduction, and 
-- assuming their existence -- construct the VMC as a closed manifold up to compact cobordism, so far unoriented, from the associated perturbed zero sets. 
Here one difficulty here is to ensure compactness of the perturbed zero set despite $\io_\Kk(X)\subset |\Kk|$ generally having no precompact neighbourhood (see Example~\ref{ex:Khomeo}). 
In \S\ref{ss:red} we also explain how compactness is achieved by precompact perturbations, whose zero set is controlled by a nested reduction. 
The key technical results of this paper -- existence and uniqueness of 
precompact transverse perturbations -- are then proven in \S\ref{ss:const}.

%%%%%%%%%%%%%%%%%%%%%%%%%%%%%%%%%%%%%%%%%%%%%%%%%%%%
\subsection{Reductions and 
compactness}\label{ss:red}  \hspace{1mm}\\ \vspace{-3mm}
%%%%%%%%%%%%%%%%%%%%%%%%%%%%%%%%%%%%%%%%%%%%%%%%%%%%

Suppose that $\Kk$ is a tame Kuranishi atlas.
The cover of $X$ by the footprints $(F_I)_{I\in \Ii_\Kk}$ of all the Kuranishi charts
(both the basic charts and those that are part of the transitional data) is closed under intersection. This makes it easy to express compatibility of the charts, since the overlap of footprints of any two charts $\bK_I$ and $\bK_J$ is covered by another chart $\bK_{I\cup J}$.
However, this yields so many compatibility conditions that compatible perturbations of all $s_I:U_I\to E_I$ generally cannot achieve transversality.
For example, suppose that $I\cap J=\emptyset$ but $K: = I\cup J\in \Ii_\Kk$. Then for the perturbations
$\nu_I:U_I\to E_I$ and $\nu_J:U_J\to E_J$  to be compatible they must induce the same perturbation 
over the intersection $\im \phi_{IK}\cap \im \phi_{JK}\supset \psi_K^{-1}(F_K)$.  
In particular, when working with an additive atlas we must have for all $x\in\im \phi_{IK}\cap \im \phi_{JK}$
$$
\nu_K(x) = \Hat\phi_{IK}\circ  \nu_I (\phi_{IK}^{-1}(x)) =  \Hat\phi_{JK}\circ \nu_J (\phi_{JK}^{-1}(x))\;\in\; \Hat\phi_{IK}(E_I) \cap \Hat\phi_{JK}(E_J) \;=\;\{0\} .
$$
This means that $s_K^{-1}(0)$ is contained in any perturbed zero set $(s_K+\nu_K)^{-1}(0)$. So unless $s_K^{-1}(0)$ was a submanifold of correct dimension to begin with, transversality is impossible in this setting.
A purely topological refinement process from \cite{MW:top} helps us to avoid these difficulties, and also makes a step towards compactness, by reducing the domains of the Kuranishi charts to precompact subsets $V_I\sqsubset U_I$ such that all compatibility conditions between $\bK_I|_{V_I}$ and $\bK_J|_{V_J}$ are given by direct coordinate changes $\Hat\Phi_{IJ}$ or $\Hat\Phi_{JI}$.
This section transfers the results from \cite{MW:top} to the smooth setting.

The following notions of reduction make sense for general Kuranishi atlases and cobordisms, but we throughout assume tameness since that is the context for our perturbation constructions.
As always, we denote the closure of a set $Z\subset X$ by~$\ov Z$ and write $V\sqsubset U$ to mean that the inclusion $V\hookrightarrow U$ is precompact, i.e.\ the closure $\ov V \subset U$ is compact.

\begin{defn}\label{def:vicin}  {\rm $\!\!$ \cite[Definition~5.1.2]{MW:top}}
Let $\Kk$ be a tame Kuranishi atlas. 
A {\bf reduction} of 
$\Kk$ is an open subset $\Vv=\bigsqcup_{I\in \Ii_\Kk} V_I \subset \Obj_{\bB_\Kk}$ i.e.\ a tuple of (possibly empty) open subsets $V_I\subset U_I$, satisfying the following conditions:
\begin{enumerate}
\item
$V_I\sqsubset U_I $ for all $I\in\Ii_\Kk$, and if $V_I\ne \emptyset$ then $V_I\cap s_I^{-1}(0)\ne \emptyset$;
\item
if $\pi_\Kk(\ov{V_I})\cap \pi_\Kk(\ov{V_J})\ne \emptyset$ then
$I\subset J$ or $J\subset I$;
\item
the zero set $\iota_\Kk(X)= |\s_\Kk|^{-1}(|0_\Kk|)$ 
is contained in 
$
\pi_\Kk(\Vv) \;=\; {\textstyle{\bigcup}_{I\in \Ii_\Kk}  }\;\pi_\Kk(V_I).
$
\end{enumerate}
Given a reduction $\Vv$, we define the {\bf reduced domain category} $\bB_\Kk|_\Vv$ and the {\bf reduced obstruction category} $\bE_\Kk|_\Vv$ to be the full subcategories of $\bB_\Kk$ and $\bE_\Kk$ with objects $\bigsqcup_{I\in \Ii_\Kk} V_I$ resp.\ $\bigsqcup_{I\in \Ii_\Kk}  V_I\times E_I$ and denote by $\s_\Kk|_\Vv: 
\bB_\Kk|_\Vv\to \bE_\Kk|_\Vv$ the section given by restriction of $\s_\Kk$. 
\end{defn}

Uniqueness of the VFC will be based on the following relative notion of reduction.

\begin{defn} \label{def:cvicin} {\rm $\!\!$ \cite[Definition~5.1.3]{MW:top}}
Let $\Kk$ be a tame Kuranishi cobordism. 
Then a {\bf cobordism reduction} of $\Kk$ is an open subset $\Vv=\bigsqcup_{I\in\Ii_{\Kk}}V_I\subset \Obj_{\bB_{\Kk}}$ that satisfies the conditions of Definition~\ref{def:vicin} 
is collared in the following sense (somewhat stricter than Definition~\ref{def:collarset}).
\begin{enumerate}
\item[(iv)]
For each $\al\in\{0,1\}$ and $I\in 
\Ii_{\p^\al\Kk}\subset\Ii_{\Kk}$ 
there exists $\eps>0$ and a subset $\partial^\al V_I\subset \partial^\al U_I$ such that $\partial^\al V_I\ne \emptyset$ iff 
$V_I \cap \psi_I^{-1}\bigl(
\partial^\al F_I \bigr)\ne \emptyset$,
and 
$$
(\iota^\al_I)^{-1} \bigl( V_I \bigr) \cap \bigl(A^\al_\eps \times  \partial^\al U_I \bigr)
 \;=\; A^\al_\eps \times \partial^\al V_I .
$$
\end{enumerate}
We call 
$\partial^\al\Vv := \bigsqcup_{I\in\Ii_{\p^0\Kk}} \partial^\al V_I \subset \Obj_{\bB_{\p^\al\Kk}}$  
the {\bf restriction} of $\Vv$ to 
$\p^\al\Kk$.
\end{defn}

\begin{remark}\rm  \label{rmk:red}
(i)
The restrictions $\partial^\al\Vv$ of a reduction $\Vv$ of a Kuranishi cobordism $\Kk$ are reductions of the restricted Kuranishi atlases $\p^\al\Kk$ for $\al=0,1$.
In particular (i) holds because part~(iv) of Definition~\ref{def:cvicin} implies that if $\p^\al V_I\ne \emptyset$ then $\p^\al V_I \cap \psi_I^{-1}\bigl( \partial^\al F_I\bigr)\ne \emptyset$.

\MS\NI
(ii)
In some ways, the closest we come 
to constructing a ``good cover" in the sense of \cite{FO,J1} is the full subcategory $\bB_\Kk|_\Vv$ of the category $\bB_\Kk$.
Though it is not a Kuranishi atlas in its own right, we prove in Proposition~\ref{prop:red} that there is a well defined Kuranishi atlas $\Kk^\Vv$ with virtual neighbourhood $|\Kk^\Vv|\cong |\bB_\Kk|_\Vv|$ together with a faithful functor $\io^\Vv:\bB_{\Kk^\Vv}\to \bB_\Kk|_\Vv$ that induces an injection $|\Kk^\Vv|\to \pi_\Kk(\Vv)\subset|\Kk|$. Since the extra structure in $\Kk^\Vv$ has no real purpose for us, we use the simpler category $\bB_\Kk|_\Vv$ instead.

\MS\NI
(iii)
Although it is not true\footnote{
Consider e.g.\ the full subcategory given by the basic charts. It has objects $\bigsqcup_i U_i$ and only identity morphisms, yet usually will have nonempty overlaps $\pi_\Kk(U_i)\cap\pi_\Kk(U_j)\supset F_{\{i,j\}} \neq\emptyset$.} 
that the realization of every full subcategory of $\bB_\Kk$ injects into $|\Kk|$,
the subcategory $\bB_\Kk|_\Vv$
induced by a reduction 
does have this property.
In fact, \cite[Lemma~5.2.3]{MW:top} shows that for any full subcategory $\bC$ of the reduced domain category $\bB_\Kk|_\Vv$, the map $|\bC|\to |\Kk|$, induced by the inclusion of object spaces, is
a continuous injection.
Moreover, the realization $|\bC|$ is homeomorphic to its image $|\Obj_{\bC}|=\pi_\Kk(\Obj_{\bC})$ with the quotient topology.
Here the injectivity is due to the fact 
\begin{align}\label{eq:Ku2}
& |I|\le |J| \mbox{ and } 
\pi_\Kk(V_J) \cap \pi_\Kk(V_I)\ne\emptyset
\;\;\;  \Longrightarrow \\ \notag
& \qquad\qquad I \subset J \mbox{ and } V_J \cap \pi_\Kk^{-1}(\pi_\Kk(V_I)) = V_J\cap \phi_{IJ}(V_I\cap U_{IJ}),
\end{align}
where the right hand equality holds because we obtain $V_J \cap \pi_\Kk^{-1}(\pi_\Kk(V_I))\subset s_J^{-1}(E_I) = \im \phi_{IJ}$ from the tameness condition \eqref{eq:tame3}.

\MS\NI
(iv)
As explained in Definition~\ref{def:topologies}, the subset $\pi_\Kk(\Cc) \subset|\Kk|$ 
has two different topologies, its quotient topology and the subspace topology.  
If $(\Kk,d)$ is metric, there might conceivably be a third topology induced by 
restriction of the metric.  
Although we do not use this explicitly, \cite[Lemma~5.2.5]{MW:top} shows that the metric topology on $\pi_\Kk(\Cc)$ agrees with the subspace topology, so that we only have two different topologies in play.
\end{remark}

Before stating the existence and uniqueness result for reductions, we introduce the notion of a nested pair of reductions, which is extensively used both for the control of compactness of perturbed zero sets below
and for the construction of perturbations in \S\ref{ss:const}.

\begin{definition}\label{def:nest} {\rm $\!\!$ \cite[Definition~5.1.5]{MW:top}}
Let $\Kk$ be a tame Kuranishi atlas (or cobordism). Then we call a pair of subsets $\Cc,\Vv\subset\Obj_{\bB_\Kk}$ a {\bf nested (cobordism) reduction} if both are (cobordism) reductions of $\Kk$ and $\Cc\sqsubset \Vv$.
\end{definition}

\begin{thm} \label{thm:red}
\begin{enumerate}
\item
Any tame Kuranishi atlas $\Kk$ has a unique concordance class of reductions as follows:
There exists a reduction of $\Kk$ in the sense of Definition~\ref{def:vicin}, 
and given any two reductions $\Vv^0,\Vv^1$ of $\Kk$, there exists a cobordism reduction $\Vv$ of $[0,1]\times \Kk$ such that $\p^\al\Vv = \Vv^\al$ for $\al = 0,1$. 
\item
Any tame Kuranishi cobordism has a cobordism reduction in the sense of Definition~\ref{def:cvicin}.

\item
For any reduction $\Vv$ of a metric Kuranishi atlas $(\Kk,d)$ there exist smaller and larger nested reductions as follows:
\begin{enumerate}
\item
Given any open subset $\Ww\subset |\Kk|$ with respect to the metric topology such that $\io_\Kk(X)\subset\Ww$, there is a nested reduction $\Cc_\Ww\sqsubset\Vv$ such that $\pi_\Kk(\Cc_\Ww)\subset\Ww$.
\item
There exists $\de>0$ such that $\Vv \sqsubset \bigsqcup_{I\in\Ii_\Kk} B^I_\de(V_I)$ is a nested reduction, and we moreover have
$B_{2\de}^I({V_I})\sqsubset U_I$ for all $I\in\Ii_\Kk$, and for any $I,J\in\Ii_\Kk$
$$
B_{2\de}(\pi_\Kk({V_I}))\cap B_{2\de}(\pi_\Kk({V_J}))
 \neq \emptyset \qquad \Longrightarrow \qquad I\subset J \;\text{or} \; J\subset I. 
$$
\end{enumerate}
\item
For any cobordism reduction $\Vv$ of a metric Kuranishi cobordism $(\Kk,d)$
there exist nested reductions with prescribed boundary as follows:
\begin{enumerate}
\item
Let $\Ww\subset |\Kk|$ be a collared subset, open with respect to metric topology, and such that $\io_\Kk(X)\subset\Ww$, and let $\Cc^\al\sqsubset \p^\al \Vv$ for $\al = 0,1$ be nested cobordism reductions of $\p^\al\Kk$ with $\pi_{\p^\al\Kk}(\Cc^\al)\subset \Ww\cap |\p^\al\Kk|$.
Then there is a nested cobordism reduction $\Cc\sqsubset\Vv$ such that $\pi_\Kk(\Cc)\subset\Ww$ and $\p^\al \Cc = \Cc^\al$ for $\al=0,1$.
\item
There exists $\de>0$ such that $\Vv \sqsubset \bigsqcup_{I\in\Ii_\Kk} B^I_\de(V_I)$ is a nested cobordism reduction, and moreover the properties of {\rm (iii)(b)} hold.
\end{enumerate}
\end{enumerate}
\end{thm}
\begin{proof}
All results follow from \cite[Theorem~5.1.6]{MW:top} applied to the induced topological Kuranishi atlas/cobordism $\Kk$. Its virtual neighbourhood is identified with $|\Kk|$, so that the notion of admissible metrics is the same in the smooth and topological context.
\end{proof}

Now given a reduction $\Vv$, the functor $\s_\Kk$ restricts to a section $\s_\Kk|_\Vv: \bB_\Kk|_\Vv\to \bE_\Kk|_\Vv$ whose local zero sets $(\s_\Kk|_{\Vv})^{-1}(0)  := {\textstyle \bigsqcup_{I\in \Ii_\Kk}}(s_I|_{V_I})^{-1}(0) \subset\Obj_{\bB_\Kk}$ still cover $X$ by the reduction condition (iii), in the sense that 
$\pi_\Kk\bigl( (\s_\Kk|_{\Vv})^{-1}(0) \bigr)= \iota_\Kk(X) \subset |\Kk|$.
In particular that means that the zero set is contained in the image of any other reduction $\Cc$, i.e.\ $\pi_\Kk\bigl( (\s_\Kk|_{\Vv})^{-1}(0) \bigr)\subset\pi_\Kk(\Cc)$.
While $\pi_\Kk(\Cc)\subset|\Kk|$ is rarely an open neighbourhood of $\io_\Kk(X)$, it plays the role of a precompact neigbhourhood in that perturbations of the zero set 
$\pi_\Kk\bigl((\s_\Kk|_{\Vv}+\nu)^{-1}(0)\bigr)  := \pi_\Kk\bigl({\textstyle \bigsqcup_{I\in \Ii_\Kk}}(s_I|_{V_I}+\nu_I)^{-1}(0)\bigr)$ are constructed in \S\ref{ss:const} to remain contained in $\pi_\Kk(\Cc)$. Then the following sequential compactness of the perturbed zero set is a key step in the construction of the virtual moduli cycle in \S\ref{ss:sect}, \S\ref{ss:VFC}.

\begin{thm} \label{thm:zeroS0}
Let $\Kk$ be a tame Kuranishi atlas (or cobordism) with a (cobordism) reduction $\Vv$, and 
suppose that $\nu: \bB_\Kk|_\Vv \to \bE_\Kk|_\Vv$ is a precompact perturbation in the sense of Definition~\ref{def:precomp}, i.e.\ 
a continuous functor $\nu:\bB_\Kk|_\Vv\to\bE_\Kk|_\Vv$ such that $\pr_\Kk\circ\nu=\id_{\bB_\Kk|_\Vv}$ and $\pi_\Kk\bigl((\s_\Kk|_\Vv + \nu)^{-1}(0)\bigr)\subset\pi_\Kk(\Cc)$ for some nested reduction $\Cc\sqsubset \Vv$.
Then the realization $|(\s_\Kk|_\Vv+ \nu)^{-1}(0)|$ is a sequentially compact Hausdorff space. 
\end{thm}
\begin{proof} \hspace{1mm}
\cite[Theorem~5.2.2]{MW:top} proves this compactness for the induced topological Kuranishi atlas/cobordism, which induces the same quotient topology on $|(\s_\Kk|_\Vv+ \nu)^{-1}(0)|$ as the smooth Kuranishi atlas/cobordism.
\end{proof}

We end this subsection by showing that a reduction $\Vv$ of a tame topological/smooth  Kuranishi atlas $\Kk$ canonically induces a Kuranishi atlas $\Kk^\Vv$ which is typically not filtered or additive and hence not tame, but whose realization $|\Kk^\Vv|$ maps bijectively to $\pi_\Kk(\Vv)$.  
This result is not used in the construction of the VMC or VFC.

\begin{prop}\label{prop:red}
Let $\Vv$ be a reduction of a tame Kuranishi atlas $\Kk$.
Then there exists a canonical Kuranishi atlas $\Kk^\Vv$ which satisfies the strong cocycle condition
and is equipped with a canonical faithful functor $\io^\Vv: \bB_{\Kk^\Vv}\to \bB_{\Kk}$  that induces a continuous injection $|\Kk^\Vv|\to |\Kk|$ with image $\pi_\Kk(\Vv)$. 
\end{prop}

\begin{proof}
To begin the construction of $\Kk^\Vv$, note first that by condition (i) in Definition~\ref{def:vicin} the footprint $Z_I: = \psi_I\bigl(V_I\cap \s_I^{-1}(0_I)\bigr)$ is nonempty whenever $V_I\ne \emptyset$.
Further by (iii) the sets 
$(Z_I)_{I\in\Ii_\Kk}$ cover $X$.  Hence we can use the tuple of nontrivial reduced Kuranishi charts $(\bK_I|_{V_I})_{I\in \Ii_\Kk, V_I\ne \emptyset}$ as the covering family of basic charts in $\Kk^\Vv$.
Then the index set of the new Kuranishi atlas is
\begin{equation}\label{eq:IKV}
\Ii_{\Kk^\Vv}= \bigl\{C\subset \Ii_\Kk \,\big|\, Z_C: =  {\textstyle \bigcap_{I\in C}} Z_I\neq \emptyset  \bigr\}.
\end{equation}
By Definition~\ref{def:vicin}~(ii), each such subset $C\subset \Ii_\Kk$ that indexes basic charts with ${\bigcap_{I\in C} Z_I\neq \emptyset}$ can be totally ordered into a chain $I_1\subsetneq I_2\ldots\subsetneq I_n$ of elements in $\Ii_\Kk$; cf.\ \cite[Remark~5.1.1]{MW:top}.  Therefore $\Ii_{\Kk^\Vv}$ can be identified with the set $\Cc\subset 2^{\Ii_\Kk}$ of linearly ordered chains $C\subset \Ii_\Kk$ such that $Z_C\ne \emptyset$.
For $C = \bigl( I^C_1\subsetneq I^C_2\ldots\subsetneq I^C_{n^C}=:I^C_{\rm max} \bigr) \in\Ii_{\Kk^\Vv}$ we define the transition chart by restriction of the chart for
the maximal element $I^C_{\rm max}$
to the intersection of the domains of the chain:
\begin{equation}\label{eq:domKC}
\bK_C^\Vv := \bK_{I^C_{\rm max}}\big|_{V_C}
\qquad\text{with}\quad
V_C := \bigcap_{1\le k\le n_C}\; \phi_{I_k^C I^C_{\max}}\bigl(V_{I_k^C}\cap U_{I_k^C I^C_{\max}}\bigr) \;\subset\; V_{I^C_{\max}} .
\end{equation}
By [Lemma~3.2.3~(a)]{MW:top}, this domain satisfies $\pi_\Kk(V_C)= \bigcap_{I\in C} \pi_\Kk(V_I)$
and hence can be expressed as 
\begin{equation}\label{eq:redu2}
V_C \;=\; {\textstyle \bigcap_{I\in C}} \; \eps_{I^C_{\max}}(V_I)
\;=\; U_{I^C_{\max}} \cap \pi_\Kk^{-1} \Bigl( {\textstyle \bigcap_{I\in C}} \, \pi_\Kk(V_I) \Bigr).
\end{equation}
Next, coordinate changes are required only between $C,D\in\Ii_{\Kk^\Vv}$ with $C\subset D$ so that $Z_C\supset Z_D\ne \emptyset$ and, by the above, $\pi_\Kk(V_C)\subset\pi_\Kk(V_D)$.
Since the inclusion $C\subset D$ implies the inclusion of maximal elements $I^C_{\rm max} \subset I^D_{\rm max}$, we can define the coordinate change as the restriction
$$
\Hat\Phi_{CD}^\Vv:=\Hat\Phi_{I^C_{\rm max} I^D_{\rm max}}\big|_{U^\Vv_{CD}}
\;: \;
\bK_C^\Vv = \bK_{I^C_{\rm max}}\big|_{V_C}  \; \longrightarrow\;
\bK_D^\Vv =\bK_{I^D_{\rm max}}\big|_{V_D}
$$
in the sense of Lemma~\ref{le:restrchange} with maximal domain
\begin{equation}\label{eq:UCD}
U^\Vv_{CD}: = V_C\cap (\phi_{I^C_{\rm max} I^D_{\rm max}})^{-1}(V_D)  \;=\;
V_C\cap \pi_\Kk^{-1}(\pi_\Kk(V_D)) .
\end{equation}
Using \eqref{eq:redu2} and the fact that $C\subset D$ we can also rewrite this domain as
\begin{equation}\label{domUCD}
U^\Vv_{CD} \;=\;  U_{I^C_{\max}}\cap \pi_\Kk^{-1}\Bigl({\textstyle \bigcap_{I\in D}} \pi_\Kk(V_I)\Bigr)
\;=\;  U_{I^C_{\max}}\cap \pi_\Kk^{-1}\bigl(\pi_\Kk(V_D)\bigr).
\end{equation}
Hence, again using  [Lemma~3.2.3~(a)]{MW:top}, the domain of the composed coordinate change $\Hat\Phi_{DE}\circ \Hat\Phi_{CD}$ for $C\subset D\subset E$ is
\begin{align*}
U^\Vv_{CD} \cap (\phi^\Vv_{I^C_{\rm max} I^D_{\rm max}})^{-1}\bigl( U^\Vv_{DE} \bigr)
&\;=\;
U_{I^C_{\max}} \;\cap\; \pi_\Kk^{-1}\bigl(\pi_\Kk(V_D)\bigr)
\;\cap\;
(\phi^\Vv_{I^C_{\rm max} I^D_{\rm max}})^{-1}\Bigl( \pi_\Kk^{-1}\bigl(\pi_\Kk(V_E)\bigr) \Bigr) \\
&\;=\;
U_{I^C_{\max}} \;\cap\; \pi_\Kk^{-1}\bigl(\pi_\Kk(V_D) \cap \pi_\Kk(V_E)\bigr)\\
& \;=\;
U_{I^C_{\max}} \;\cap\; \pi_\Kk^{-1}\bigl( \pi_\Kk(V_E)\bigr) ,
\end{align*}
which equals the domain of $\Hat\Phi_{CE}$.
Now the strong cocycle condition for $\Kk^\Vv$ follows from the strong cocycle condition for $\Kk$, which holds by Lemma~\ref{le:tame0}. 
Moreover, smoothness and index condition of the charts and coordinate changes in $\Kk$ transfers to $\Kk^\Vv$, so that $\Kk^\Vv$ is a Kuranishi atlas.

Next, the inclusions $V_C\hookrightarrow U_{I^C_{\max}}$ induce a continuous map on the object spaces
$$
\iota^\Vv \,:\;
\Obj_{\bB_{\Kk^\Vv}} = {\textstyle \bigsqcup_{C\in\Ii_{\Kk^\Vv}}} V_C \; \longrightarrow \;  {\textstyle \bigsqcup_{I\in\Ii_\Kk}} U_I = \Obj_{\bB_\Kk}.
$$
Since $V_C\subset V_{I^C_{\max}}$ for all $C$, this map  has image $\bigsqcup_{I\in \Ii_\Kk} V_I$.  It is generally not injective. However, because the coordinate changes in
${\Kk^\Vv}$ are restrictions of those in  $\Kk$, this map on object spaces extends to a functor 
$\iota^\Vv : \bB_{\Kk^\Vv}\to \bB_{\Kk}$. 
This shows that the induced map $|\iota^\Vv|:|{\Kk^\Vv}|\to |\Kk|$ is continuous with image $\pi_\Kk(\Vv)$. Moreover, the functor $\iota^\Vv$ is faithful, i.e.\ for each $(C,x), (D,y)\in   \Obj_{\bB_{\Kk^\Vv}}$ the map
$$
\Mor_{\Kk^\Vv}\bigl((C,x),(D,y)\bigr)\;\longrightarrow\; \Mor_{\Kk}\bigl(\io^\Vv(C,x), \io^\Vv(D,y)\bigr)
$$
is injective.
To prove that $|\iota^\Vv|$ is injective 
we need to show for $x\in V_C, y\in V_D$ that
$$
(I^C_{\max}, x)\sim_\Kk (I^D_{\max},y) \;\Longrightarrow\;  (C,x)\sim_{\Kk^\Vv} (D,y).
$$
To see this, note that by assumption and \eqref{eq:redu2} we have
$\pi_\Kk(V_I)\ni \pi_\Kk(x) = \pi_\Kk(y) \in  \pi_\Kk(V_J)$ for all 
$I\in C$ and $J\in D$.
In particular, for each $I^C_k\in C, I^D_\ell\in D$ the intersection $\pi_\Kk(V_{I^C_{k}})\cap \pi_\Kk(V_{I^D_{\ell}})$ is nonempty.  Hence, by Definition~\ref{def:vicin}~(ii), the elements  $I^C_1,\ldots,I^C_{\max}, I^D_1,\ldots,I^D_{\max}$ of $\Ii_\Kk$ can be ordered into a chain $E:=C\vee D$ (after removing repeated elements) with maximal element $I^{E}_{\max}=I^{C}_{\max}$ or $I^{E}_{\max}=I^{D}_{\max}$, and such that $\pi_\Kk(x)=\pi_\Kk(y)\in \bigcap_{I\in C\vee D} \pi_\Kk(V_I)=\pi_\Kk(V_{C\vee D})$. 
In particular, $V_{C\vee D}$ is nonempty, so we have $E=C\vee D \in\Ii_{\Kk^\Vv}$ and
$x\in V_C\cap \pi_\Kk^{-1}(V_E) \subset  U_{I^C_{\max}I^E_{\max}}$ lies in the domain of $\phi^\Vv_{C (C\vee D)}=\phi_{I^C_{\max}I^E_{\max}}$, whereas $y\in V_D\cap \pi_\Kk^{-1}(V_E) \subset  U_{I^D_{\max}I^E_{\max}}$ lies in the domain of $\phi^\Vv_{D (C\vee D)}=\phi_{I^D_{\max}I^E_{\max}}$. 
Now  [Lemma~3.2.3~(a)]{MW:top} for $I^C_{\max}\subset I^E_{\max}$ and $I^D_{\max}\subset I^E_{\max}$ implies $\phi^\Vv_{C (C\vee D)}(x) = \phi^\Vv_{D (C\vee D)}(y)$.  
This proves $(C,x)\sim_{\Kk^\Vv} (D,y)$ as required, and thus completes the proof.
\end{proof}

\begin{remark}\label{rmk:KVv} \rm   
The resulting Kuranishi atlas $\Kk^\Vv$ is not filtered (or additive) because it has coordinate changes  $\bK_C^\Vv\to \bK_D^\Vv$ between charts that have the same obstruction space when $I^C_{\max} =I^D_{\max}$.  Further, the above proof shows that $\Kk^\Vv$ has the property that for any two charts $\bK_C^\Vv, \bK_D^\Vv$ with intersecting footprints $Z_C\cap Z_D\ne \emptyset$,
we must have $I^C_{\rm max} \subset I^D_{\rm max}$ or $I^D_{\rm max} \subset I^C_{\rm max}$, though possibly neither $C\subset D$ nor $D\subset C$. 
Assuming w.l.o.g.\ that $I^C_{\rm max} \subset I^D_{\rm max}$, there is a direct coordinate change 
$$
\Hat\Phi_{I^C_{\rm max} I^D_{\rm max}}|_{V_C\cap \pi_\Kk^{-1}(\pi_\Kk(V_D))}
 \;: \;  \bK_C^\Vv = \bK_{I^C_{\rm max}}\big|_{V_C}  \; \longrightarrow\; \bK_D^\Vv = \bK_{I^D_{\rm max}}\big|_{V_D} .
$$
It embeds one of the obstruction bundles  as a summand of the other; in this case
$E_C=E_{I^C_{\rm max}}\hookrightarrow \Hat\Phi_{I^C_{\rm max} I^D_{\rm max}}(E_C)\subset E_D = E_{I^D_{\rm max}}$.
Such a coordinate change is not explicitly included in the Kuranishi atlas  $\Kk^\Vv$ unless $C\subset D$. It does however appear, as in \cite[Lemma~3.2.3]{MW:top}, as the composite of a coordinate change $\bK_C^\Vv \to \bK_{C\vee D}^\Vv$ with the inverse of a coordinate change $\bK_D^\Vv \to \bK_{C\vee D}^\Vv$.
In this respect the subcategory $\bB_{\Kk}|_\Vv$ has much simpler structure, since the components  $V_I$ of its space of objects correspond to chains $C=(I)$ with just one element.
$\hfill\er$
\end{remark}

%%%%%%%%%%%%%%%%%%%%%%%%%%%%%%%%%%%%%%%%%%%%%%%%%%%%%%%%%%%%%%%%%%%%
\subsection{Perturbed zero sets}\label{ss:sect} \hspace{1mm}\\ \vspace{-3mm}
%%%%%%%%%%%%%%%%%%%%%%%%%%%%%%%%%%%%%%%%%%%%%%%%%%%%%%%%%%%%%%%%%%%%

Throughout this section, $\Kk$ is a fixed tame Kuranishi atlas on a compact metrizable space $X$ 
and $\Vv$ is a fixed reduction of $\Kk$.
We begin by introducing sections in the reduction and  an infinitesimal version of an admissibility condition 
for sections in \cite[A.1.21]{FOOO}.\footnote
{We work infinitesimally since the canonical section $\s_\Kk: = (s_I)$ may not satisfy an identity
$s_J = \Hat\phi_{IJ}(s_I) \oplus \id_{E_J/\Hat\phi_{IJ}(E_I)}$ in a tubular neighbourhood of $\phi_{IJ}(U_{IJ})\subset U_J$ identified with $U_{IJ}\times E_J/\Hat\phi_{IJ}(E_I)$ in the way described in  \cite{FOOO}.  The new definition used in \cite{FOOO12} is closer to ours.
}

\begin{defn}\label{def:sect} 
A  {\bf reduced section} of $\Kk$ is a smooth functor $\nu:\bB_\Kk|_\Vv\to\bE_\Kk|_\Vv$ between the reduced domain and obstruction categories of some reduction $\Vv$ of $\Kk$, such that $\pr_\Kk\circ\nu$ is the identity functor. 
That is, $\nu=(\nu_I)_{I\in\Ii_\Kk}$ is given by a family of smooth maps $\nu_I: V_I\to E_I$ such that for each $I\subsetneq J$ we have a commuting diagram
\begin{equation}\label{eq:comp}
\xymatrix{
 V_I\cap \phi_{IJ}^{-1}(V_J)   \ar@{->}[d]_{\phi_{IJ}} \ar@{->}[r]^{\qquad\nu_I}    &  E_I \ar@{->}[d]^{\Hat\phi_{IJ}}   \\
V_J \ar@{->}[r]^{\nu_J}  & E_J.
}
\end{equation}
We say that a reduced section $\nu$ is an {\bf admissible perturbation} of 
$\s_\Kk|_\Vv$ if  
\begin{equation}\label{eq:admiss}
\rd_y \nu_J(\rT_y V_J) \subset\im\Hat\phi_{IJ} \qquad \forall \; I\subsetneq J, \;y\in V_J\cap\phi_{IJ}(V_I\cap U_{IJ}) .
\end{equation}
\end{defn}

\begin{rmk}\label{rmk:sect} \rm  
(i)
Each reduced section  $\nu:\bB_\Kk|_\Vv\to\bE_\Kk|_\Vv$ induces a continuous map $|\nu|: |\Vv|\to |\bE_\Kk|$ such that $|\pr_\Kk|\circ |\nu| = \id$, where $|\pr_\Kk|$ is as in Theorem~\ref{thm:K}.
Each such map has the further property that $|\nu|\big|_{\pi_\Kk(V_I)}$ takes values in $\pi_\Kk(U_I\times E_I)$.  
\MS

\NI (ii) 
More generally, a section $\si$ of $\Kk$ is a functor $\bB_\Kk\to\bE_\Kk$ that satisfies the conditions of Definition~\ref{def:sect} with $\Vv$ replaced by $\Obj_{\bB_\Kk}$. 
But these compatibility conditions are now much more onerous.  
For example, except in the most trivial cases, the set $V_{12}\cap \bigcap_{i=1,2} \phi_{i,12}(U_{i,12}\less V_i)$ is nonempty, so that there is $x\in V_{12}$ with $\pi_\Kk(x)\in \bigl(\pi_\Kk(U_1)\cup \pi_\Kk(U_2)\bigl)\less \bigl( \pi_\Kk(V_1) \cup \pi_\Kk(V_2)\bigr)$.
A reduced section $\nu$ could take any value $\nu_{12}(x) \in E_{12}\cong \Hat\phi_{1,12}(E_1)\oplus \Hat\phi_{2,12}(E_2)$. On the other hand, a section $\si$ of $\Kk$ would have $\si(x)\in \bigcap_{i=1,2} \Hat\phi_{i,12}(E_i)=\{0\}$ since the compatibility conditions imply that  
$\si_{12}|_{\im\phi_{i,12}}$ takes values in $\Hat\phi_{i,12}(E_i)$.  
We cannot achieve transversality for $\s_\Kk+\si$ under such conditions, which explains why we consider reduced sections.
$\hfill\er$
\end{rmk}

The following simple example illustrates 
 the use of reductions.
 
 \begin{example}\rm 
We will construct a representative for the  VFC as the zero set of $\s_\Kk|_\Vv + \nu$ where $\nu$ is a reduced section such that that  $\s_\Kk|_\Vv + \nu$ is transverse to $0$.  Suppose that $X = \{p_1,p_2,p_3\}$ consists of three points, with footprint covering $F_1 = \{p_1,p_2\}, F_2 = \{p_2,p_3\}$.  Then the corresponding atlas has three charts $\bK_1,\bK_2, \bK_{12}$ where
$F_{12}  = \{p_2\}$.  Therefore we can choose the reduction $\Vv$ to extend the cover reduction $Z_1: = \{p_1\}, Z_2: =  \{p_3\}, Z_{12}: = \{p_2\}$.  Since these three sets are disjoint, we may suppose by \cite[Remark~5.3.6]{MW:top} that the three sets $\bigl(\pi_\Kk(V_I)\bigr)_{I\in \Ii_\Kk}$ are also disjoint.
Hence the perturbations $\nu_I$ can be chosen independently; there are no compatibility conditions between them.
$\hfill\er$
\end{example}

Note that the zero section $0_\Kk$ restricts to an admissible perturbation $0_\Vv:\bB_\Kk|_\Vv\to\bE_\Kk|_\Vv$ in the sense of the above definition.
Similarly, the canonical section $\s_\Kk$ of the Kuranishi atlas restricts to a section $\s_\Kk|_\Vv: \bB_\Kk|_\Vv\to\bE_\Kk|_\Vv$ of any reduction.
However, the canonical section is generally not admissible. In fact, as we saw in Lemma~\ref{le:change}, for all $y\in V_J\cap \phi_{IJ}(V_I\cap U_{IJ})$ the map 
$$
{\rm pr}_{E_I}^\perp\circ \rd_y s_J \,: \;\;  \quotient{\rT_y U_J} {\rT_y (\phi_{IJ}(U_{IJ}))} \; \longrightarrow \; \quotient{E_J}{\Hat\phi_{IJ}(E_I)} 
$$
is an isomorphism by the index condition \eqref{tbc}, while for an admissible section it is identically zero.
So for any reduction $\Vv$ and admissible perturbation $\nu$, the sum 
$$
\s_\Kk|_\Vv+\nu:=(s_I|_{V_I}+\nu_I)_{I\in\Ii_\Kk} \,:\; \bB_\Kk|_\Vv \;\to\; \bE_\Kk|_\Vv
$$
is a reduced section that satisfies the index condition
$$
{\rm pr}_{E_I}^\perp\circ\rd_y (s_J+\nu_J)\,: 
\;\; \quotient{\rT_y U_J}{\rT_y (\phi_{IJ}(U_{IJ}))} \;\stackrel{\cong}\longrightarrow \; \quotient{E_J}{\Hat\phi_{IJ}(E_I)}
\qquad\forall \; y\in V_J\cap \phi_{IJ}(V_I\cap U_{IJ}).
$$
We use this in the following lemma to show that transversality of the sections in Kuranishi charts is preserved under coordinate changes.
Admissibility is also needed for the notion of orientations; cf.\ Proposition~\ref{prop:orient}.

\begin{lemma}\label{le:transv}
Let $\nu$ be an admissible perturbation of  $\s_\Kk|_\Vv$. 
If $z\in V_I$ and $w\in V_J$ map to the same point in the virtual neighbourhood $\pi_\Kk(z)=\pi_\Kk(w)\in|\Kk|$, then $z$ is a transverse zero of $s_I|_{V_I}+\nu_I$ if and only if $w$ is a transverse zero of $s_J|_{V_J}+\nu_J$.
\end{lemma}
\begin{proof}
Note that, since the equivalence relation $\sim$ on $\Obj_{\bB_\Kk}$ is generated  by 
$\preceq$ and its inverse $\succeq$, it suffices to suppose that $(I,z)\preceq (J,w)$, i.e.\ $w=\phi_{IJ}(z)$. Now $s_J(w)=\Hat\phi_{IJ}(s_I(z))=0$ iff $s_I(z)=0$ since $\Hat\phi_{IJ}$ is injective.
Next, $z$ is a transverse zero of $s_I|_{V_I}+\nu_I$ exactly if $\rd_z (s_I+\nu_I): \rT_z U_I\to E_I$ is surjective. On the other hand, we have splittings $\rT_w U_J \cong \im\rd_z\phi_{IJ} \oplus \tfrac{\rT_w U_J} {\im\rd_z\phi_{IJ}} $ and $E_J \cong \Hat\phi_{IJ}(E_I) \oplus \tfrac{E_J}{\Hat\phi_{IJ}(E_I)}$
with respect to which the differential at $w$ has product form
\begin{equation}\label{eq:dnutrans}
\rd_w (s_J + \nu_J) \cong \bigl(\; \Hat\phi_{IJ}\circ\rd_z (s_I+\nu_I) \circ (\rd_z\phi_{IJ})^{-1} \,,\, \rd_w s_J \, \bigr) ,
\end{equation}
by the admissibility condition on $\nu_J$. 
Here the second factor is an isomorphism by the index condition \eqref{tbc}. 
Since $\Hat\phi_{IJ}$ and $(\rd_z\phi_{IJ})^{-1}$ are isomorphisms on the relevant domains, 
this proves equivalence of the transversality statements.
\end{proof}

\begin{defn}\label{def:sect2}
A {\bf transverse perturbation} of $\s_\Kk|_{\Vv}$
is a reduced section $\nu:\bB_\Kk|_\Vv\to\bE_\Kk|_\Vv$ whose sum with 
the restriction $\s_\Kk|_{\Vv}$ of the canonical section 
is transverse to the zero section $0_\Vv$, that is $s_I|_{V_I}+\nu_I\pitchfork 0$ for all $I\in\Ii_\Kk$.

Given a transverse perturbation $\nu$, we  define the {\bf perturbed zero set} $|\bZ^\nu|$ to be the realization of the full subcategory $\bZ^\nu$ of $\bB_\Kk$ with object space 
$$
(\s_\Kk|_{\Vv} + \nu)^{-1}(0)  
 := {\textstyle \bigsqcup_{I\in \Ii_\Kk}}(s_I|_{V_I}+\nu_I)^{-1}(0) \;\subset\;\Obj_{\bB_\Kk} . 
$$
That is, we equip
$$
|\bZ^\nu| : = \bigl|( \s_\Kk|_{\Vv}  + \nu)^{-1}(0)  
\bigr| \,=\; \quotient{ {\textstyle\bigsqcup_{I\in\Ii_\Kk} (s_I|_{V_I}+\nu_I)^{-1}(0) }}{\!\sim} 
$$
with the quotient topology generated by the morphisms of $\bB_\Kk|_\Vv$.
By Remark~\ref{rmk:red}~(iii) this is equivalent to the quotient topology induced by $\pi_\Kk$, 
and the inclusion $(\s_\Kk|_{\Vv} +\nu)^{-1}(0) \subset\Vv = \Obj_{\bB_\Kk|_\Vv}$ induces a continuous injection, which we denote by
\begin{equation}\label{eq:Zinject} 
i^\nu \,:\;  |\bZ^\nu| \;\longrightarrow\; |\Kk| .
\end{equation}
\end{defn}

To see that the above is well defined, recall that the canonical section restricts to a reduced section $\s_\Kk|_\Vv: \bB_\Kk|_\Vv\to\bE_\Kk|_\Vv$, so that the sum  $\s_\Kk|_\Vv+\nu$  is a reduced section as well, with a well defined zero set $(\s_\Kk|_\Vv + \nu)^{-1}(0)$.   
Moreover, since $\bZ^\nu$ is the realization of a full subcategory of $\bB_\Kk|_\Vv$,
Remark~\ref{rmk:red}~(iii) asserts that the map $i^\nu$ is a continuous injection to $|\Kk|$, and moreover a  homeomorphism from $|\bZ^\nu|$ to  $\pi_\Kk\bigl((\s_\Kk|_\Vv+\nu)^{-1}(0)\bigr)=|(\s_\Kk|_\Vv + \nu)^{-1}(0)|$ 
with respect to the quotient topology in the sense of Definition~\ref{def:topologies}. 
In particular, the continuous injection to the Hausdorff space $|\Kk|$ implies Hausdorffness of $|\bZ^\nu|$.
However, the image of $i^\nu$ is $\pi_\Kk\bigl((\s_\Kk|_\Vv+\nu)^{-1}(0)\bigr)$ 
with the relative topology induced by $|\Kk|$,  that is
$$
i^\nu ( |\bZ^\nu| )  =  \|(\s_\Kk|_\Vv+\nu)^{-1}(0)\| .
$$
So the perturbed zero set is equipped with two Hausdorff topologies -- the quotient topology on $|\bZ^\nu|\cong|(\s_\Kk|_\Vv+\nu)^{-1}(0)|$  and the relative topology on $\|(\s_\Kk|_\Vv+\nu)^{-1}(0)\|\subset|\Kk|$.
It remains to achieve local smoothness and compactness in one of the topologies. 
We will see below that local smoothness follows from transversality of the perturbation, though only in the topology of $|\bZ^\nu|$, which may contain smaller neighbourhoods than $\|(\s_\Kk|_\Vv + \nu)^{-1}(0)\|$.
On the other hand, compactness of $\|(\s_\Kk|_\Vv + \nu)^{-1}(0)\|$ seems easier to obtain than that of $|\bZ^\nu|$, which may have more open covers.
For the first, one could use the fact that $\|(\s_\Kk|_\Vv+\nu)^{-1}(0)\|\subset\|\Vv\|$ is precompact in $|\Kk|$ by Proposition~\ref{prop:Ktopl1}~(iii), so it would suffice to deduce closedness of $\|(\s_\Kk|_\Vv+\nu)^{-1}(0)\|\subset|\Kk|$. This would follow if the continuous map $|\s_\Kk|_\Vv+\nu|:\|\Vv\| \to |\bE_\Kk|_\Vv|$ had a continuous extension to $|\Kk|$ with no further zeros. 
However, such an extension may not exist. In fact, generally $\|\Vv\|\subset |\Kk|$ fails to be open, 
$\io_\Kk(X)\subset |\Kk|$ does not have any precompact neighbourhoods (see Example~\ref{ex:Khomeo}),
 and even those assumptions would not guarantee the existence of an extension.
So compactness of either $\|(\s_\Kk|_\Vv+\nu)^{-1}(0)\|$ or $|\bZ^\nu|$ will not hold in general without further hypotheses on the perturbation that force its zero set to be ``away from the boundary"  of $\|\Vv\|$
in the following sense,
which by Theorem~\ref{thm:zeroS0} directly implies sequential compactness of $|\bZ^\nu|$.

\begin{defn}\label{def:precomp}  {\rm $\!\!$ \cite[Definition~5.2.1]{MW:top}}
A reduced section $\nu: \bB_\Kk|_\Vv\to \bE_\Kk|_\Vv$ is said to be {\bf precompact} if 
its perturbed zero set is contained inside a nested reduction $\Cc\sqsubset \Vv$ in the sense that 
$\pi_\Kk\bigl((\s_\Kk|_\Vv + \nu)^{-1}(0)\bigr)\subset \pi_\Kk(\Cc)$, or equivalently
\begin{equation}\label{eq:zeroVCC}
(s_J|_{V_J} + \nu_J)^{-1}(0)
\;\subset\;  {\textstyle \bigcup_{H\supset J} } \,\phi_{JH}^{-1}(C_H) \; \cup \; {\textstyle \bigcup_{H\subsetneq J} }\,\phi_{HJ}(C_H) 
\qquad\forall J\in \Ii_\Kk.
\end{equation}
\end{defn}

The smoothness properties follow more directly from transversality of the perturbation.
The next lemma shows that for transverse perturbations the object space $\bigsqcup_I (s_I|_{V_I}+\nu_I)^{-1}(0) \subset \bigsqcup_I V_I$ of $\bZ^\nu$ is a smooth submanifold of dimension $d: = \dim \Kk$, and that the morphisms spaces are given by local diffeomorphisms.
Hence the category $\bZ^\nu$ can be extended to a groupoid by adding the inverses 
to the space of morphisms.

\begin{lemma} \label{le:stransv}
Let $\nu: \bB_\Kk|_\Vv\to\bE_\Kk|_\Vv$ be a transverse perturbation of  $\s_\Kk|_\Vv$. 
Then the domains of the perturbed zero set $(s_I|_{V_I}+\nu_I)^{-1}(0)\subset V_I$ are submanifolds for all $I\in\Ii_\Kk$; and for $I\subset J$ the map $\phi_{IJ}$ induces a diffeomorphism from ${V_J\cap U_{IJ}\cap (s_I|_{V_I}+\nu_I)^{-1}(0)}$ to an open subset of $(s_J|_{V_J}+\nu_J)^{-1}(0)$.
\end{lemma}
\begin{proof}
The submanifold structure of $(s_I|_{V_I}+\nu_I)^{-1}(0)\subset U_I$ follows from the transversality and the implicit function theorem, with the dimension given by the index $d:=\dim U_I-\dim E_I$.
For $I\subset J$ the embedding $\phi_{IJ}:U_{IJ}\to U_J$ then restricts to a smooth embedding  
\begin{equation}\label{eq:ZphiIJ}
\phi_{IJ}: U_{IJ}\cap(s_I|_{V_I}+\nu_I)^{-1}(0) \to (s_J|_{V_J}+\nu_J)^{-1}(0)
\end{equation}
by the functoriality of the perturbed sections.
Since this restriction of $\phi_{IJ}$ to this solution set is an embedding from an open subset of a $d$-dimensional manifold into a $d$-dimensional manifold, it has open image and is a diffeomorphism to this image.
\end{proof}

Assuming that precompact transverse perturbations exist (as we will show in Proposition~\ref{prop:ext}), we can deduce smoothness and compactness of the perturbed zero set.

\begin{prop} \label{prop:zeroS0}
Let $\Kk$ be a tame $d$-dimensional Kuranishi atlas with a reduction $ \Vv\sqsubset \Kk$, and 
suppose that $\nu: \bB_\Kk|_\Vv \to \bE_\Kk|_\Vv$ is a precompact transverse perturbation.
Then $|\bZ^\nu| = |(\s_\Kk|_\Vv+ \nu)^{-1}(0)|$ is a smooth closed $d$-dimensional manifold. 
Moreover, its quotient topology agrees with the subspace topology on ${\|(\s_\Kk|_\Vv+ \nu)^{-1}(0)\|\subset|\Kk|}$.
\end{prop}
\begin{proof}  
By Lemma~\ref{le:stransv}, $|\bZ^\nu|$ is made from the (disjoint) union $\bigsqcup_I \bigl(s_I|_{V_I}+\nu_I)^{-1}(0)\bigr)$ of $d$-dimensional manifolds via an equivalence relation given by the smooth local diffeomorphisms \eqref{eq:ZphiIJ}.   
From this we can deduce that $|\bZ^\nu|$ is second countable (i.e.\ its topology has a countable basis of neighbourhoods).
Indeed, a basis is given by the projection of countable bases of each manifold $(s_I|_{V_I}+\nu_I)^{-1}(0)$ to the quotient.
The images are open in the quotient space since the relation between different components of $(\s_\Kk|_\Vv+\nu)^{-1}(0)$ is given by local diffeomorphisms, taking open sets to open sets. In other words: The preimage of an open set in $|\bZ^\nu|$ is a disjoint union of open subsets of $(s_I|_{V_I}+\nu_I)^{-1}(0)$.
This can be used to express any open set as a union of the basis elements. 
It also shows that $|\bZ^\nu|$ is locally smooth, since any choice of lift $x\in (s_I|_{V_I}+\nu_I)^{-1}(0)$ of a given point $[x]\in |\bZ^\nu|$ lies in some chart $\Nn \hookrightarrow \R^d$, where $\Nn\subset (s_I|_{V_I}+\nu_I)^{-1}(0)$ is open; thus as above $[\Nn]\subset|\bZ^\nu|$ is open and provides a local homeomorphism to $\R^d$ near $[x]$.

Moreover, as noted above, the continuous injection $|\bZ^\nu| \to |\Kk|$ from 
Remark~\ref{rmk:red}~(iii)
transfers the Hausdorffness of $|\Kk|$ from Proposition~\ref{prop:Khomeo} to the realization $|\bZ^\nu|$.
Thus $|\bZ^\nu|$ is a second countable Hausdorff space that is locally homeomorphic to a 
$d$-dimensional manifold.  Hence it is a $d$-dimensional manifold, where we understand all manifolds to have empty boundary, since the charts are open sets in $\R^d$.
Now the sequential compactness of $|\bZ^\nu|$ established in Theorem~\ref{thm:zeroS0} implies compactness, since every manifold is metrizable. (In fact, second countability suffices for the equivalence of compactness and sequential compactness, see \cite[Theorem~5.5]{Kel}.)  
Therefore $|\bZ^\nu|$ is a closed manifold. 

Finally, the map \eqref{eq:Zinject}  is a continuous bijection between the compact space $|\bZ^\nu|$ and the Hausdorff space 
$\|(\s_\Kk|_\Vv+\nu)^{-1}(0)\|\subset|\Kk|$ with the relative topology induced by $|\Kk|$. As such it is 
automatically a homeomorphism $|\bZ^\nu|\cong \|(\s_\Kk|_\Vv+\nu)^{-1}(0)\|$, see Remark~\ref{rmk:hom}.
\end{proof}

We now extend these results to a tame Kuranishi cobordism $\Kk$ from $\Kk^0$ to $\Kk^1$ with cobordism reduction $\Vv$.
Recall from Definition~\ref{def:cvicin} that $\Vv$ induces reductions $\partial^\al\Vv := \bigsqcup_{I\in\Ii_{\Kk^\al}} \partial^\al V_I \subset \Obj_{\bB_{\Kk^\al}}$ of $\Kk^\al$ for $\al=0,1$, where we identify the index set $\Ii_{\Kk^\al}\cong\io^\al(\Ii_{\Kk^\al})$ 
with a subset of $\Ii_{\Kk}$.

\begin{defn} \label{def:csect}  
Let $\Kk$ be a tame Kuranishi cobordism with cobordism reduction $\Vv$.
A {\bf reduced cobordism section} of $\s_{\Kk}|_{\Vv}$ is a reduced section $\nu:\bB_{\Kk}|_\Vv\to\bE_{\Kk}|_{\Vv}$ as in Definition~\ref{def:sect} that in addition has product form in a collar neighbourhood of the boundary. 
That is, for $\al=0,1$ and $I\in \Ii_{\Kk^\al}\subset\Ii_{\Kk}$ there is $\eps>0$ and a map $\nu_I^\al: \p^\al V_I\to E_I$ such that 
$$
\nu_I \bigl( \io_I^\al ( t,x ) \bigr) 
= \nu_I^\al (x)  \qquad
\forall\, x\in \p^\al V_I , \   t\in A^\al_\eps .
$$ 
A {\bf precompact, transverse cobordism perturbation} of $\s_\Kk|_{\Vv}$ is a reduced cobordism section $\nu$ that satisfies  the transversality condition $s_I|_{V_I}+\nu_I \pitchfork 0$  on the interior of the domains $V_I$, and whose domain is part of a nested cobordism reduction $\Cc\sqsubset \Vv$ such that 
$\pi_\Kk\bigl((\s_\Kk|_\Vv + \nu)^{-1}(0)\bigr)\subset \pi_\Kk(\Cc)$.
We moreover call such $\nu$ {\bf admissible} if it satisfies \eqref{eq:admiss}.
\end{defn}

The product structure of $\nu$ in the collar ensures that the transversality of the perturbation extends to the boundary of the domains, as follows.

\begin{lemma}\label{le:ctransv}
If $\nu:\bB_{\Kk}|_\Vv\to\bE_{\Kk}|_{\Vv}$ is a precompact, transverse cobordism perturbation of $\s_{\Kk}|_{\Vv}$, then the {\bf restrictions} $\nu|_{\partial^\al\Vv} := \bigl(  \nu_I^\al \bigr)_{I\in\Ii_{\Kk^\al}}$ for $\al=0,1$ are precompact, transverse perturbations of the restricted canonical sections $\s_{\Kk^\al}|_{\partial^\al\Vv}$. If in addition $\nu$ is admissible, then so are the restrictions $\nu|_{\partial^\al\Vv}$.

Moreover, each perturbed section $s_I|_{V_I}+\nu_I : V_I \to E_I$ for $I\in\Ii_{\Kk^0}\cup\Ii_{\Kk^1}\subset \Ii_{\Kk}$ is transverse to $0$ as a map on a domain with boundary. That is, the kernel of its differential is transverse to the boundary 
$\partial V_I = \bigsqcup_{\al=0,1}\iota^\al_I ( \{\al\} \times \partial^\al V_I)$.
\end{lemma}

\begin{proof}
Precompactness transfers to the restriction since the restrictions of the nested cobordism reduction are nested reductions 
$\p^\al\Cc \sqsubset \p^\al\Vv$ for $\al=0,1$.
Similarly, admissibility transfers immediately by pullback of \eqref{eq:admiss} to the boundaries via $\io^\al_I :  \{\al\} \times \p^\al V_I \to V_I$. 
Transversality in (the interior of) a collar neighbourhood of the boundary $\io^\al_I(A^\al_\eps\times \partial^\al\Vv )\subset V_I$ is equivalent to transversality of the restriction $s|_{\partial^\al V_I}+\nu^\al_I \pitchfork 0$ since 
$\rd \bigl( \nu_I \circ \io_I^\al \bigr) = 0\ \rd t + \rd\nu_I^\al$.
Moreover, transversality of $s_I|_{V_I}+\nu_I : V_I \to E_I$
at the boundary of $V_I$, as a map on a domain with boundary, is equivalent under pullback with the embeddings $\io^\al_I$ to transversality of
$f := \bigl( s_I|_{V_I}+\nu_I \bigr)\circ\io^\al_I : 
A^\al_\eps \times \partial^\al V_I \to E_I$. 
For the latter, the kernel $\ker\rd_{s,x} f = \R\times \ker\rd_x \nu^\al_I  $ 
is indeed transverse to the boundary $ \{\al\}\times \rT_x \partial^\al V_I$ in  $\R\times \rT_x \partial^\al V_I $.
\end{proof}

With that, we can show that precompact transverse perturbations of the Kuranishi cobordism induce smooth cobordisms (up to orientations) between the perturbed zero sets of the restricted perturbations. 
 
\begin{lemma} \label{le:czeroS0}
Let $\nu: \bB_{\Kk}|_{\Vv} \to \bE_{\Kk}|_{\Vv}$ be a precompact, transverse cobordism perturbation.
Then $|\bZ^{\nu}|$, defined as in Definition~\ref{def:sect2}, is  a compact manifold whose boundary $\p|\bZ^\nu|$ is diffeomorphic to the disjoint union of $|\bZ^{\nu^0}|$ and $|\bZ^{\nu^1}|$, where $\nu^\al :=\nu|_{\partial^\al\Vv}$ are the restricted transverse perturbations of $\s_{\Kk^\al}|_{\p^\al \Vv}$ for $\al=0,1$.
\end{lemma}
\begin{proof}
The topological properties of $|\bZ^{\nu}|$ follow from the arguments in Proposition~\ref{prop:zeroS0},
and smoothness of the zero sets follows as in Lemma~\ref{le:stransv}. However, the zero sets $(s_I|_{V_I}+\nu_I)^{-1}(0)$ for $I\in\Ii_{\Kk^\al}\subset\Ii_{\Kk}$ are now submanifolds with boundary, by the implicit function theorem on the interior of $V_I$ together with the smooth product structure on the collar neighbourhoods $\io_I^\al ( A^\al_\eps\times \p^\al V_I)$ of the boundary. The latter follows from the smoothness of 
$(s_I|_{\p^\al V_I}+\nu^\al_I)^{-1}(0)$
from Lemma~\ref{le:stransv} and the embedding
$$
(s_I|_{V_I}+\nu_I)^{-1}(0)\;\cap\; \io_I^\al (\p^\al V_I \times \{\al\}) \;=\; \io^\al_I \bigl( 
\{\al\} \times (s_I|_{\p^\al V_I}+\nu^\al_I)^{-1}(0) \bigr)   .
$$
This gives $(\s_\Kk|_\Vv+\nu)^{-1}(0)$ the structure of a compact manifold with two disjoint boundary components for $\al=0,1$ given by
$$
\partial^\al \bigl( (\s_\Kk|_\Vv+\nu)^{-1}(0) \bigr) \;=\;  \underset{{I\in\Ii_{\Kk^\al}}}{\textstyle\bigsqcup} 
\io^\al_I \bigl(\{\al\}\times  (s_I|_{\p^\al V_I}+\nu^\al_I)^{-1}(0)  \bigr) ,
$$
which are diffeomorphic via $\partial^\al \Vv \ni (I,x) \mapsto \iota^\al_I(\al,x)$ 
to the submanifolds 
$$
(\s_\Kk|_{\Vv^\al}+\nu^\al)^{-1}(0)\subset \partial^\al \Vv
$$
 given by the restricted perturbations 
$\nu^\al=\bigl(\nu^\al_I\bigr)_{I\in\Ii_{\Kk^\al}}$.
By the collar form of the coordinate changes in $\Kk$ this induces fully faithful functors $j^\al$ from $\bZ^{\nu^\al}$ to the full subcategories of $\bZ^\nu$ with objects $\partial^\al \bigl( (\s_\Kk|_{\Vv}+\nu)^{-1}(0) \bigr)$.

Moreover, as in Lemma~\ref{le:stransv}, the morphisms are given by restrictions of the embeddings $\phi_{IJ}$, which are in fact local diffeomorphisms, and hence can be inverted to give $\bZ^\nu$ the structure of a groupoid.
Again using the collar form of the coordinate changes, there are no morphisms between $\partial^\al \bigl( (\s_\Kk|_\Vv+\nu)^{-1}(0) \bigr)$ and its complement in $(\s_\Kk|_\Vv+\nu)^{-1}(0)$, so the realization $|\bZ^{\nu}|$ inherits the structure of a compact manifold with boundary
$\partial |\bZ^{\nu}| = \bigsqcup_{\al=0,1} \partial^\al |\bZ^{\nu}|$ with two disjoint boundary components  
$$
\partial^\al |\bZ^{\nu}|  
\,:=\; \partial^\al |\Kk| \cap |\bZ^{\nu}|  
\;=\; \Bigl|  {\textstyle\bigsqcup}_{I\in\Ii_{\Kk^\al}} 
\io^\al_I \bigl(\{\al\} \times  (s_I|_{V_I^\al}+\nu^\al_I)^{-1}(0) 
\bigr) \Bigr| .
$$
Since the fully faithful functors $j^\al$ are diffeomorphisms between the object spaces, they then descend to diffeomorphisms to the boundary components,
$$
|j^\al| \,:\;  |\bZ^{\nu^\al}| \;=\;  
 \Bigl|  {\textstyle\bigsqcup}_{I\in\Ii_{\Kk^\al}} 
 (s_I|_{V^\al_I}+\nu^\al_I)^{-1}(0)  \Bigr|
 \;\overset{\cong}{\longrightarrow}\; \partial^\al |\bZ^{\nu}|  .
$$
Thus $|\bZ^{\nu}|$ is a (not yet oriented) cobordism between $|\bZ^{\nu^0}|$ and $|\bZ^{\nu^1}|$, as claimed.
\end{proof}

%%%%%%%%%%%%%%%%%%%%%%%%%%%%%%%%%%%%%%%%%%%%%%%%%%%%%%%%%%%%%%%%%%%%%%
\subsection{Construction of perturbations} \label{ss:const}\hspace{1mm}\\ \vspace{-3mm}
%%%%%%%%%%%%%%%%%%%%%%%%%%%%%%%%%%%%%%%%%%%%%%%%%%%%%%%%%%%%%%%%%%%%%%

In this section, we let $(\Kk,d)$ be a metric tame Kuranishi atlas (or cobordism) and $\Vv$ a (cobordism) reduction, and construct precompact transverse (cobordism) perturbations of the canonical section $\s_\Kk|_\Vv$.
In fact, we will construct a transverse perturbation with perturbed zero set contained in $\pi_\Kk(\Cc)$ for any given nested (cobordism) reduction $\Cc \sqsubset \Vv$.
This will be accomplished by an intricate construction that depends on the choice of two suitable constants $\de,\si>0$ depending on $\Cc\sqsubset\Vv$, 
and norms on the obstruction spaces. Consequently, the corresponding uniqueness statement requires not only the construction of transverse cobordism perturbations $\nu$ of $\s_\Kk|_\Vv$ in a nested cobordism reduction $\Cc \sqsubset \Vv$, and with given restrictions $\nu|_{\p^\al\Vv}$, but also an understanding of the dependence on the choice of constants $\de,\si>0$.
We begin by describing the setup, which will be used to construct perturbations for both Kuranishi atlases and Kuranishi cobordisms. It is important to have this in place before describing the iterative constructions because, firstly, the iteration depends on the above choice of constants, and secondly, even the statements about uniqueness and existence of perturbations in cobordisms need to take this intricate setup into consideration.
We begin by introducing a suitable notion of compatible norms on the obstruction spaces, which crucially uses the additivity of $\Kk$.

\begin{definition}\label{def:norm}
A choice of {\bf additive norms} on an additive Kuranishi atlas/cobordism $\Kk$ is a tuple of norms $\|\cdot\|=\bigl( \|\cdot\|_I \bigr)_{I\in\Ii_\Kk}$  on each obstruction space $E_I$ that are determined from a choice of norms $\|\cdot\|_i : E_i \to [0,\infty)$ on each basic obstruction space $E_i$ for $i=1,\dots N$ as follows: 
For any $I\in\Ii_\Kk$, the norm $\|\cdot\|_I: E_I \to [0,\infty)$ given by
\begin{equation} \label{eq:iI2}
\| e \|_I \;=\;
\Bigl\| {\textstyle \sum_{i\in I}} \Hat\phi_{iI} (e_i) \Bigr\| \,:=\; \max_{i\in I} \| e_i\|
\qquad
\forall e=  {\textstyle \sum_{i\in I}} \Hat\phi_{iI} (e_i) \in E_I
\end{equation}
is well defined due to the additivity isomorphism from Definition~\ref{def:Ku2},
\begin{equation} \label{eq:iI}
{\textstyle \prod_{i\in I}} \;\Hat\phi_{iI}: \; {\textstyle \prod_{i\in I}} \; E_i \;\stackrel{\cong}\longrightarrow \; E_I \;=\;\oplus_{i\in I} \Hat\phi_{iI}(E_i) .
\end{equation}
For additive norms $\|\cdot\|=\bigl( \|\cdot\|_I \bigr)_{I\in\Ii_\Kk}$ on a Kuranishi cobordism and $\al\in\{0,1\}$, we denote by 
$\p^\al\|\cdot\|:=
\bigl( \|\cdot\|_I \bigr)_{I\in\Ii_{\p^\al\Kk}}$ the induced additive norms on the boundary restriction $\p^\al\Kk$, given by the subset of norms for the index set $\Ii_{\p^\al\Kk}\subset \Ii_\Kk$. 
\end{definition}

In the following, we will drop the subscripts from the norms, since they will be evident from the context.
In that notation, note that our use of the maximum norm on the Cartesian product guarantees estimates of the components $\|e_i\| \leq \|e\|$.
This construction also guarantees that each embedding $\Hat\phi_{IJ}:(E_I,\|\cdot\|) \to (E_J,\|\cdot\|)$ is an isometry by the cocycle condition $\Hat\phi_{IJ}\circ\Hat\phi_{iI}=\Hat\phi_{iJ}$. 
Moreover, we will throughout use the supremum norm for functions, that is for any map $f_I:{\rm dom}(f_I)\to (E_I,\|\cdot\|)$ we use the unique decomposition $f_I = \sum_{i\in I} f^i_I$ into components $\bigl(f^i_I : {\rm dom}(f_I) \to \Hat\phi_{iI}(E_i)\bigr)_{i\in I}$ to denote
$$
\bigl\| f_I \bigr\| \,:=\; \sup_{x\in {\rm dom}(f_I)} \bigl\| f_I(x) \bigr\|  \;=\; \sup_{x\in {\rm dom}(f_I)} \max_{i\in I} \, \bigl\| f^i_I(x) \bigr\|   
 \;=:\, \max_{i\in I}\, \bigl\| f^i_I \bigr\|   .
$$
Next, recall from Lemma~\ref{le:metric} (which holds in complete analogy for metric Kuranishi cobordisms) that the metric $d$ on $|\Kk|$ induces metrics $d_I$ on each domain $U_I$ such that the coordinate changes $\phi_{IJ}:(U_{IJ},d_I)\to (U_J,d_J)$ are isometries.
In the following, we will make ample use of the notation $B_\de^I$ and $B_\de$ from Definition~\ref{def:metric} for $\de$-neighbourhoods in $U_I$ and $|\Kk|$ respectively.

\begin{rmk} \rm \label{rmk:iso}
The following perturbation constructions will also be applied -- with very minor adjustments -- to construct perturbations in a suitable context of nontrivial isotropy in \cite{MW:iso}. In order to provide a verifiably rigorous proof for the corresponding result \cite[Proposition~3.3.3]{MW:iso}, we indicate the necessary adjustment in a series of footnotes [$^{\rm NN}$] in the present section. 
These should only be read after becoming familiar with the construction of the pruned domain category $\Bb_\Kk|_\Vv^{\less\Ga}$ in \cite[Lemma~3.2.3]{MW:iso} and notion of admissible perturbation in \cite[Definition~3.2.4]{MW:iso}. The adjustments will be rather few after the following initial adjustments to the above setup:

\begin{itemlist}
\item
Associated to an atlas $\Kk$ with isotropy actions $\Ga_I\times U_I\to U_I$ is an intermediate atlas $\uKk$ whose domains are the quotients $\uU_I=\qu{U_I}{\Ga_I}$. Both are atlases for the same space $X$ with canonically identified virtual neighbourhoods $|\Kk|=|\uKk|$. 

\item
The reductions $\Cc=\bigsqcup C_I \sqsubset \Vv=\bigsqcup V_I$ are lifts of reductions $\uCc=\bigsqcup \uC_I \sqsubset \uVv=\bigsqcup \uV_I$ of the intermediate atlas $\uKk$, i.e.\ $C_I=\pi_I^{-1}(\uC_I)$, $V_I=\pi_I^{-1}(\uV_I)$.

\item
Although $\pi_\Kk:\Obj_{\Bb_\Kk} \to |\Kk|$ does not restrict to a functor on $\Bb_\Kk|_\Vv^{\less\Ga}$, we can work with $\pi_\Kk:\bigsqcup_{I\in\Ii_\Kk} U_I \to |\Kk|$ as continuous map. 
As in the case of trivial isotropy, we do not have a nicely controlled cover of sets $U_J\cap \pi_\Kk^{-1}(\pi_\Kk(\Cc))$ for $\Cc\subset \bigsqcup U_I$. However, when $\Cc =\bigsqcup C_I\subset\Vv= \bigsqcup V_I\subset  \bigsqcup U_I$ are lifts of reductions of $|\uKk|$ as in the previous item, then we obtain the analogue of \eqref{eq:VCC}, 
\begin{equation}\label{eq:VCC}
V_J\cap \pi_\Kk^{-1}(\pi_\Kk(\Cc))
\;=\; V_J \cap \bigl( \;{\textstyle \bigcup_{H\supset J} } \rho_{JH}(C_H) \cup {\textstyle \bigcup_{H\subsetneq J} }\rho_{HJ}^{-1}(C_H) \;\bigr) .
\end{equation}
Indeed, the reduction property $\pi_\Kk(V_J)$ only intersects $\pi_\Kk(C_H)$ for $H\supset J$ or $H\subset J$. The morphisms between $U_H$ and $U_J$ are then given by  $\rho_{JH}$ and $\Ga_J$ resp.\ $\rho_{HJ}$ and $\Ga_H$, and the isotropy groups are absorbed by the equivariance $\Ga_J \rho_{JH}(C_H) = \rho_{JH}(\Ga_H C_H )$ and fact that
$\Ga_H C_H = C_H = \pi_H^{-1}(\und C_H)$.

\item
We use ``equivariant norms'' that arise in the same way as the additive norms from choices of $\Ga_i$-invariant norms $\|\cdot\|_i$ on the basic obstruction spaces $E_i$.
This allows us to represent the zero sets of the sections $s_I: U_I\to E_I$ in terms of continuous functions on the intermediate category $\|\und s_I\|: \uU_I \to [0,\infty)$ given by $x \mapsto \|s_I(y)\|$ for any $y\in\pi_I^{-1}(x)$. Indeed, this yields $\|s_I\|=\|\und s_I\|\circ\pi_I$ and thus $\|\und s_I\|^{-1}(0)=\pi_I(s_I^{-1}(0))$.

\item
The metric $d$ on $|\Kk|=|\uKk|$ lifts to compatible metrics $\ud_I$ on the intermediate domains $\uU_I=\qu{U_I}{\Ga_I}$.
These then lift to $\Ga_I$-invariant pseudometrics $d_I$ on $U_I$. 

\item
Denoting $\de$-neighbourhoods in $\uU_I$ by $B^I_\de$, the corresponding $\de$-neighbourhoods of sets $S\subset U_I$ are related by
$\Hat B^I_\de(S) = \pi_I^{-1}\bigl(  B^I_\de(\pi_I(S)) \bigr)$. Similarly, we have the relation $B_\de(\pi_\Kk(S))=B_\de(\pi_{\uKk}(\und S))$ for $\de$-neighbourhoods in the virtual neighbourhood.

\item
In the following, all relationships between (or definitions/constructions of) subsets of $\Obj_{\Bb_\Kk}=\bigcup_{I\in\Ii_\Kk} U_I$ should be replaced by two statements -- one for subsets of $\Obj_{\Bb_\uKk}=\bigcup_{I\in\Ii_\Kk} \uU_I$ in the intermediate atlas $\uKk$, and one for subsets in the pruned domain category $\Bb_\Kk|_{\Vv}^{\less\Ga}$ with $B_\de$ replaced by $\Hat B_\de$. These two statements will always be equivalent via the projection $\pi_I$. Statements can then be checked by working in the intermediate category, but they will be applied on the level of the pruned domain category. 
Here it is crucial to know that the projections $\pi_I:U_I\to \uU_I$ are continuous (by definition of the quotient topology) and proper by \cite[Lemma 2.1.5]{MW:iso}.

\item
Our goal -- constructing a precompact, transverse, admissible (cobordism) perturbation 
 $\nu:\bB_\Kk|_\Vv\to\bE_\Kk|_\Vv$ -- remains essentially the same, with Definitions~\ref{def:sect}, \ref{def:sect2}, \ref{def:precomp} replaced by \cite[Definition~3.2.4]{MW:iso}.
This requires a functor $\nu:\bB_\Kk|_\Vv^{\less\Ga}\to\bE_\Kk|_\Vv^{\less\Ga}$ on slightly different categories, but writing it in terms of the maps $\nu=(\nu_I:V_I\to E_I)_{I\in\Ii_\Kk}$, the only difference is that the compatibility conditions \eqref{eq:comp},
\begin{equation}\label{eq:compatible}
\nu_J\big|_{N_{JI}} \; =\; 
 \Hat\phi_{IJ}\circ\nu_I\circ \phi_{IJ}^{-1}\big|_{N_{JI}}
\quad \text{on}\quad
N_{JI}:=V_J \cap \phi_{IJ}(V_I \cap U_{IJ})
\end{equation}
for all $I \subsetneq J$ are replaced by
\begin{equation}\label{eq:compatc}
\nu_J\big|_{\TV_{IJ}}\  =\ 
 \Hat\phi_{IJ}\circ\nu_I\circ \rho_{IJ}\big|_{\TV_{IJ}}
\quad \text{on}\quad
\TV_{IJ} := V_J\cap \rho_{IJ}^{-1}(V_I) ,
\end{equation}
and the precompactness condition $\pi_\Kk\bigl( (\s|_\Vv^{\less\Ga} + \nu)^{-1}(0) \bigr) \subset\pi_\Kk(\Cc)$ can be reformulated analogously to \eqref{eq:zeroVCC} as
\begin{equation}\label{eq:isoVCC}
(s_J|_{V_J} + \nu_J)^{-1}(0)
\;\subset\;  {\textstyle \bigcup_{H\supset J} } \,\rho_{JH}(C_H) \; \cup \; {\textstyle \bigcup_{H\subsetneq J} }\,\rho_{HJ}^{-1}(C_H) 
\qquad\forall J\in \Ii_\Kk.
\end{equation}
Here the setup in \cite{MW:iso} guarantees that $\rho_{IJ}:\TV_{IJ}\to V_I \cap \rho_{IJ}(V_J)\subset U_{IJ}$ is a regular covering (i.e.\ local diffeomorphism with fibers given by the free action of a finite group $\Ga_{J\less I}\cong \qu{\Ga_J}{\Ga_I}$) analogous to $\phi_{IJ}^{-1}:N_{IJ}\to V_I \cap \phi_{IJ}^{-1}(V_J)\subset U_{IJ}$, which is a regular covering with trivial fibers.
Thus in the following one should replace $\phi_{IJ}$ with $\rho_{IJ}^{-1}$ and identify $N_{IJ}=\TV_{IJ}$.
This translates \eqref{eq:zeroVCC} into \eqref{eq:isoVCC} and also the compatibility conditions \eqref{eq:compatible} continue to make sense and yield \eqref{eq:compatc}.
Note that this will make $\nu_J\big|_{\TV_{IJ}}$ automatically invariant under $\Ga_{J\less I}$. However, the notion of admissible perturbations does not require any compatibility with the full action of $\Ga_J$ or with the projection $\pi_J:U_J\to \uU_J$. 
These additional morphisms in $\Bb_\Kk|_{\Vv}$ are eliminated in $\Bb_\Kk|_{\Vv}^{\less\Ga}$ and are re-introduced when constructing the VMC/VFC in \cite{MW:iso} by means of weighting functions rather than multi-valued perturbations.
\end{itemlist}

\vspace{-6mm}
$\hfill\er$
\end{rmk}

To prepare for the iterative construction of perturbations, we need a nested sequence of reductions.
For that purpose, 
Theorem~\ref{thm:red}~(iii)~(b) provides $\de>0$ so that $B_{2\de}^I(V_I)\sqsubset U_I$  for all $I\in\Ii_\Kk$, and $B_{2\de}(\pi_\Kk(\ov{V_I}))\cap B_{2\de}(\pi_\Kk(\ov{V_J})) = \emptyset$ unless $I\subset J$ or $J\subset I$, and hence the precompact neighbourhoods 
\begin{equation}\label{eq:VIk}
V_I^k \,:=\; B^I_{2^{-k}\de}(V_I) 
\;\sqsubset\; U_I \qquad \text{for} \; k \geq 0
\end{equation}  
form further (cobordism) reductions, all of which contain $\Vv$.
Here we chose separation distance $2\de$ so that compatibility of the metrics ensures the strengthened version of the separation condition (ii) in Definition~\ref{def:vicin} for $I\not\subset J$ and $J\not\subset I$,
\begin{equation}\label{desep}
B_\de\bigl(\pi_\Kk(V^k_I)\bigr) \cap B_\de\bigl(\pi_\Kk(V^k_J)\bigr) \subset  
B_{\de + 2^{-k}\de}\bigl(\pi_\Kk(V_I)\bigr) \cap B_{\de + 2^{-k}\de}\bigl(\pi_\Kk(V_J)\bigr) = \emptyset .
\end{equation}
In case $I\subsetneq J$, 
\eqref{eq:Ku2}
then gives the identities
\begin{align}\label{eq:N}\notag
V^k_I \cap \pi_\Kk^{-1}(\pi_\Kk(V^k_J))& 
\;=\; V^k_I  \cap \phi_{IJ}^{-1}(V^k_J)  ,\\
V^k_J \cap \pi_\Kk^{-1}(\pi_\Kk(V^k_I)) 
&\;=\; V^k_J \cap \phi_{IJ}(V^k_I \cap U_{IJ})  
\;=:\, N^k_{JI}  
\end{align}
for the sets on which we will require compatibility of the perturbations $\nu_I$ and $\nu_J$.
The analogous identities hold for any combinations of the nested precompact open sets 
$$
C_I \;\sqsubset\; V_I \;\sqsubset\; \ldots V^{k'}_I \;\sqsubset\; V^{k}_I \ldots \;\sqsubset\; V^0_I  ,
$$
where $k'>k >0$ are any positive reals.
For the sets $N^k_{JI}\subset U_J$ introduced in \eqref{eq:N} above, note that by the compatibility of metrics we have inclusions for any $H\subsetneq J$,
$$
\pi_\Kk(B^J_\de(N^k_{JH}))
\subset 
B_\de \bigl(\pi_\Kk(N^k_{JH})\bigr)
\subset 
B_\de \bigl(\pi_\Kk(\phi_{HJ}(V^k_H\cap U_{HJ}))\bigr)
\subset B_{\de}\bigl( \pi_\Kk(V^k_H) \bigr) .
$$
So \eqref{desep} together with the injectivity of $\pi_\Kk|_{U_J}$ implies for any $H,I \subsetneq J$
[\footnote{
As an example of the translation mechanism for nontrivial isotropy in Remark~\ref{rmk:iso}, the above constructions can all be made on the intermediate category to yield
$B^J_\de(\uN^k_{JH}) \cap B^J_\de(\uN^k_{JI}) \;\Rightarrow\; H\subset I \;\text{or} \; I\subset H$
for $\uN^k_{JI}=\uV^k_J \cap \uphi_{IJ}(\uV^k_I \cap \uU_{IJ})$.
Taking the preimage under $\pi_J:U_J\to\uU_J$ then implies
$\Hat B^J_\de(N^k_{JH}) \cap \Hat B^J_\de(N^k_{JI}) \;\Rightarrow\; H\subset I \;\text{or} \; I\subset H$
for $N^k_{JI}=\pi_I^{-1}\bigl(\uV^k_J \cap \uphi_{IJ}(\uV^k_I \cap \uU_{IJ})\bigr)= 
V^k_J \cap \rho_{IJ}^{-1}(V^k_I \cap U_{IJ})$.
Moreover, each $N^k_{JI}$ is an open $\Ga_{J\less I}$-invariant subset of the domain $\TU_{IJ}$ of $\rho_{IJ}$, so that $\rho_{IJ}:N^k_{JI}\to V_I^k$ is a regular covering whose fibers are $\Ga_{J\less I}$-orbits. On the other hand, the sets $\TV_{IJ}$ on which the compatibility conditions \eqref{eq:compatc} are required, are open subsets of $N^k_{JI}$. 
}]
\begin{equation}\label{Nsep}
B^J_\de(N^k_{JH}) \cap B^J_\de(N^k_{JI}) \neq \emptyset \qquad \Longrightarrow \qquad H\subset I \;\text{or} \; I\subset H.
\end{equation}
Moreover, we have precompact inclusions for any $k'>k\geq 0$
\begin{equation} \label{preinc}
N^{k'}_{JI} \;=\;V^{k'}_J \cap  \phi_{IJ}(V^{k'}_I\cap U_{IJ}) \;\sqsubset\; V^k_J \cap \phi_{IJ}(V^k_I\cap U_{IJ})  \;=\;N^k_{JI} ,
\end{equation}
since $\phi_{IJ}$ is an embedding to the relatively closed subset $s_J^{-1}(E_I)\subset U_J$ and thus $\ov{\phi_{IJ}(V^{k'}_I\cap U_{IJ})} = \phi_{IJ}\bigl(\,\ov{V^{k'}_I}\cap U_{IJ}\bigr) \subset \phi_{IJ}(V^k_I\cap U_{IJ})$.
[\footnote{For \cite{MW:iso}, $\rho_{IJ}^{-1}|_{V^0_J}$ corresponds to finitely many embeddings with disjoint images in $s_J^{-1}(E_I)$.
}]
Next, we abbreviate
$$
N^k_J \, := \;{\textstyle \bigcup_{J\supsetneq I}} N^k_{JI} \;\subset\; V^k_J ,
$$  
and will call the union $N^{|J|}_J$ the {\it core} of $V^{|J|}_J$, since it is the part of this set on which we will prescribe $\nu_J$ in an iteration by a compatibility condition with the $\nu_I$ for $I\subsetneq J$.
In this iteration we will be working with quarter integers between $0$ and 
$$
M \,:=\; M_\Kk \,:=\; \max_{I\in\Ii_\Kk} |I|  ,
$$
and need to introduce another constant $\eta_0>0$ that controls the intersection with $\im\phi_{IJ}=\phi_{IJ}(U_{IJ})$ for all $I\subsetneq J$ as in Figure~\ref{fig:4},
\begin{equation}\label{eq:useful} 
\im \phi_{IJ} \;\cap\;  B^J_{2^{-k-\frac 12}\eta_0} \bigl( N_{JI}^{k+\frac 34} \bigr) \;\subset\; N^{k+\frac 12}_{JI} 
\qquad \forall \; k\in  \{0,1,\ldots,M\}.
\end{equation}
Since $\phi_{IJ}$ is an isometric embedding, this inclusion holds whenever 
$2^{-k-\frac 12}\eta_0 + 2^{-k-\frac 34}\de \leq 2^{-k-\frac 12}\de$ 
for all $k$. To minimize the number of choices in the construction of perturbations, we may thus simply fix $\eta_0$ in terms of $\de$ by
\begin{equation}\label{eq:eta0}
\eta_0 \,:=\; (1 -  2^{-\frac 14} ) \de.
\end{equation} 
Then we also have $2^{-k}\eta_0 + 2^{-k-\frac 12}\de <  2^{-k}\de$,
% note to self
%$\Leftrightarrow  2^{-k} - 2^{-k-\frac 14} + 2^{-k-\frac 12} < 2^{-k}$
%$\Leftrightarrow 2^{-\frac 12} < 2^{-\frac 14}$
which provides the inclusions
\begin{equation}\label{eq:fantastic}
B^I_{\eta_k}\Bigl(\;\ov{V_I^{k+\frac 12}}\;\Bigr) \;\subset\; V_I^k \qquad\text{for} \;\; k\geq 0, \; \eta_k:=2^{-k}\eta_0 .
\end{equation}

\begin{figure}[htbp] 
   \centering
   \includegraphics[width=3in]{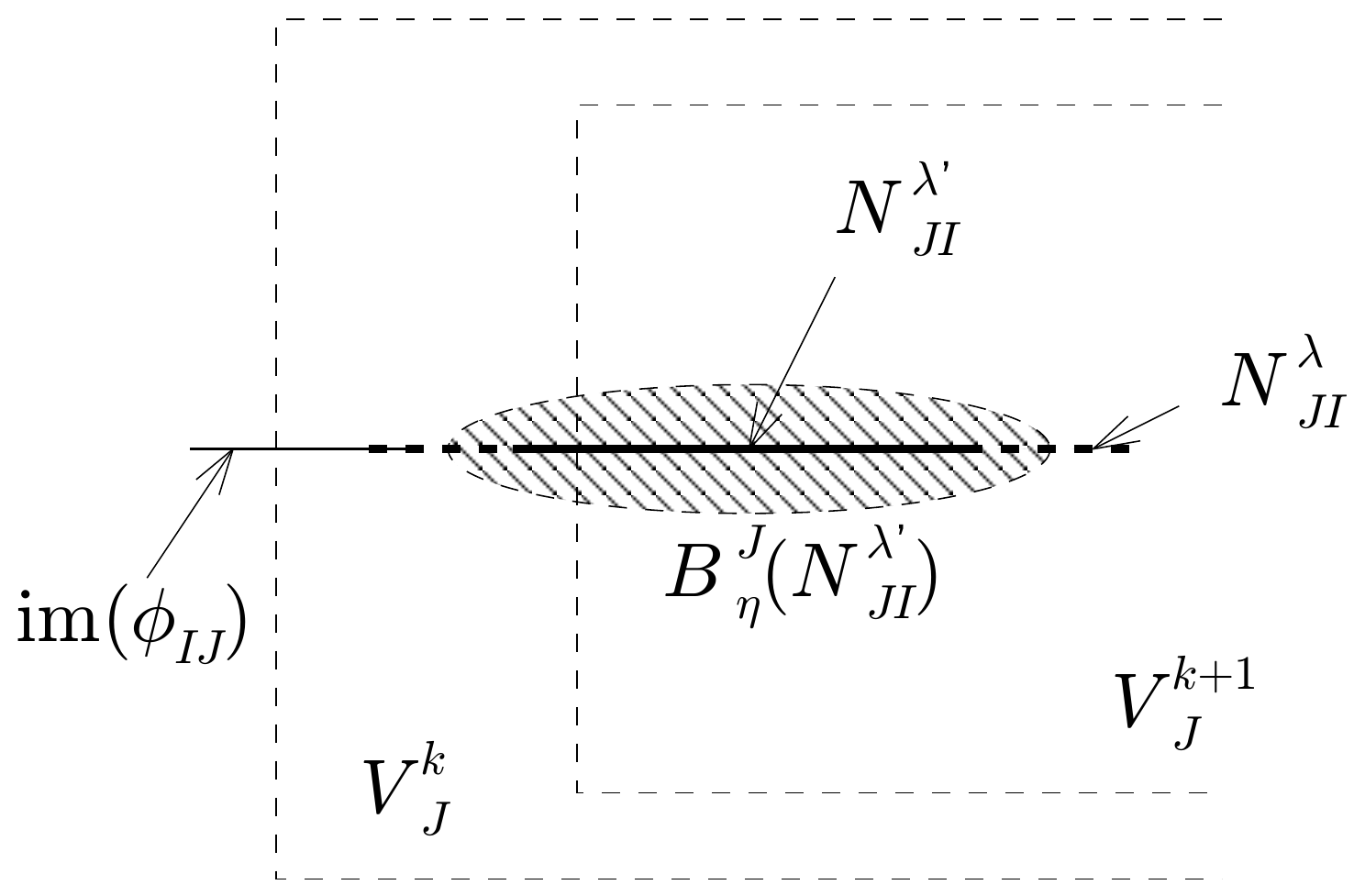} 
 \caption{
This figure illustrates the nested sets $V^{k+1}_J \sqsubset V^k_k$ and 
$ N^{\la'}_{JI}\sqsubset N^{\la}_{JI}\subset \im(\phi_{IJ})\cap V^k_J$ for $k+1>\la'=k+\frac 34 > \la =k+\frac 12>k$, the shaded neighbourhood $B^J_{\eta}(N^{\la'}_{JI})$ for $\eta = 2^{-\la}\eta_0$, and the inclusion given by \eqref{eq:useful}.
}
   \label{fig:4}
\end{figure}

%
%Continuing the preparations, let a nested (cobordism) reduction $\Cc \sqsubset \Vv$ be given. 
%Then we denote the open subset contained in $U_J\cap \pi_\Kk^{-1}(\pi_\Kk(\Cc))$ for $J\in\Ii_\Kk$ by
%\begin{equation} \label{ticj}
%\Ti C_J \,:=\;   {\textstyle \bigcup_{K\supset J}} \, \phi_{JK}^{-1}(C_K)   \;\subset\; U_J  .
%\end{equation}
%The assumption $\io_\Kk(X)\subset\pi_\Kk(\Cc)$ implies $s_J^{-1}(0)\subset \pi_\Kk^{-1}(\pi_\Kk(\Cc))$, and so by \eqref{eq:N} the zero set $s_J^{-1}(0)$ is contained in $\Ti C_J \;\cup\;  {\textstyle \bigcup_{I\subsetneq J}} \phi_{IJ}(C_I\cap U_{IJ})$, the union of an open set $\Ti C_J$ and a set in which the perturbed zero sets will be controlled by earlier iteration steps. [\footnote{
%In the case of \cite{MW:iso}, we use the identity \eqref{eq:VCC} to determine $V_J\cap  \pi_\Kk^{-1}(\pi_\Kk(\Cc))$.
%}]
%

Continuing the preparations, let a nested (cobordism) reduction $\Cc \sqsubset \Vv$ be given. 
Then precompactness w.r.t.\ $\Cc$ of a perturbation $\nu$ of $\s_\Kk|_\Vv$ requires as in \eqref{eq:zeroVCC}
$$
(s_J|_{V_J} + \nu_J)^{-1}(0)
\;\subset\;  {\textstyle \bigcup_{K\supset J} } \,\phi_{JK}^{-1}(C_K) \; \cup \; {\textstyle \bigcup_{I\subsetneq J} }\,\phi_{IJ}(C_I) 
\qquad\forall J\in \Ii_\Kk.
$$
To keep track of this requirement we denote the open part of $U_J\cap \pi_\Kk^{-1}(\pi_\Kk(\Cc))$ by
\begin{equation} \label{ticj}
\Ti C_J \,:=\;   {\textstyle \bigcup_{K\supset J}} \, \phi_{JK}^{-1}(C_K)   \;\subset\; U_J  .
\end{equation}
The assumption $\io_\Kk(X)\subset\pi_\Kk(\Cc)$ implies that the unperturbed solution set $s_J^{-1}(0)$ is contained in $\Ti C_J \;\cup\;  {\textstyle \bigcup_{I\subsetneq J}} \phi_{IJ}(C_I\cap U_{IJ})$ -- the union of an open set $\Ti C_J$ and a set in which the perturbed zero sets will be controlled by earlier iteration steps. [\footnote{
In the case of \cite{MW:iso}, precompactness has the analogous reformulation in \eqref{eq:isoVCC}.
}]

\begin{defn}  \label{def:admiss}
Given a reduction $\Vv$ of a metric Kuranishi atlas (or cobordism) $(\Kk,d)$, we set $\de_\Vv\in (0,\frac 14]$
to be the maximal constant such that any $\de< \de_\Vv$ satisfies the reduction properties of Theorem~\ref{thm:red}~(iii)~(b), that is 
\begin{align}
\label{eq:de1} 
B^I_{2\de}
(V_I)\sqsubset U_I\qquad &\forall I\in\Ii_\Kk , \\
\label{eq:dedisj}
B_{2\de}(\pi_\Kk({V_I}))\cap B_{2\de}(\pi_\Kk({V_J}))
 \neq \emptyset &\qquad \Longrightarrow \qquad I\subset J \;\text{or} \; J\subset I .
\end{align}
Given a nested reduction $\Cc\sqsubset\Vv$ of a metric Kuranishi atlas $(\Kk,d)$,
choice of additive norms $\|\cdot\|$, and $0<\de<\de_\Vv$, we set $\eta_{|J|-\frac 12} :=  2^{-|J|+\frac 12} \eta_0 = 2^{-|J|+\frac 12} (1 -  2^{-\frac 14} ) \de$ and {\rm [\footnote{
For \cite{MW:iso} this definition makes sense as is on the pruned domain category, and on the intermediate category should be read with $\|s_J(x)\|$ replaced by $\|\us_J\|(x)$; see Remark~\ref{rmk:iso}~(iii).
}]}
\begin{equation*}
\si(\Vv,\Cc,\|\cdot\|,\de) \,:=\; \min_{J\in\Ii_\Kk}
\inf \Bigl\{
\; \bigl\| s_J(x) \bigr\| \;\Big| \;
x\in \ov{V^{|J|}_J} \;\less\; \Bigl( \Ti C_J \cup {\textstyle \bigcup_{I\subsetneq J}} B^J_{\eta_{|J|-\frac 12}}\bigl(N^{|J|-\frac14}_{JI}\bigr) \Bigr) \Bigr\}  .
\end{equation*}
\end{defn}

In this language, the previous development of setup in this section shows that for any metric Kuranishi atlas or cobordism $(\Kk,d)$ we have $\de_\Vv>0$. We note some further properties of these constants. 
Note first that there is no general relation between $\si(\Vv,\Cc,\|\cdot\|,\de)$ and $\si(\Vv,\Cc,\|\cdot\|,\de')$ for $0<\de'<\de<\de_\Vv$ since both $V^{|J|}_J$ and $B^J_{\eta_{|J|-\frac 12}}\bigl(N^{|J|-\frac14}_{JI}\bigr)$ grow with growing $\de$, and hence the domains of the infimum are not nested in either way.
There are however simple relations if we scale the additive norms $\|\cdot\|=(\|\cdot\|_I)_{I\in\Ii_\Kk}$ by a common factor or have an inequality in the sense that $\|e\|_I\leq\|e\|'_I$ for all $I\in\Ii_\Kk, e\in E_I$,
\begin{align} \label{eq:sigmascale}
\si(\Vv,\Cc,c \|\cdot\|,\de) &= c \cdot \si(\Vv,\Cc,\|\cdot\|,\de) \qquad \forall c>0,  \\
\|\cdot\| \leq \|\cdot\|' \; &\Rightarrow\; \si(\Vv,\Cc,\|\cdot\|,\de) \leq \si(\Vv,\Cc,\|\cdot\|',\de). \notag
\end{align}

\begin{lemma}\label{le:admin}
\begin{enumerate}
\item
Let $\Cc\sqsubset\Vv$ be a nested reduction of a metric Kuranishi atlas or cobordism, and let $\de<\de_\Vv$, then we have $\si(\Vv,\Cc,\|\cdot\|,\de)>0$.
\item 
For any reduction $\Vv$ of a metric Kuranishi atlas $(\Kk,d)$ we have $\de_\Vv=\de_{[0,1]\times \Vv}$, 
and this constant is always smaller than the collar width of $[0,1]\times \Kk$ with respect to the product metric.
\item
Given a metric Kuranishi cobordism $(\Kk,d)$ we equip the Kuranishi atlases $\p^\al\Kk$ for $\al=0,1$ with the restricted metrics $d\big|_{|\p^\al\Kk|}$.
Then for any cobordism reduction $\Vv$ we have $\de_{\p^\al\Vv}\geq \de_\Vv$
for $\al=0,1$.
\item
In the setting of (iii), let $\eps>0$ be the collar width of $(\Kk,d)$.
Then the neighbourhood of radius $r<\eps$ of any $\eps$-collared set $W\subset U_I$ 
(i.e.\ with $W\cap \io^\al_I(A^\al_\eps\times \p^\al U_I )=\io^\al_I(A^\al_\eps\times \p^\al W)$) is $(\eps-r)$-collared,
\begin{equation} \label{eq:Wnbhd}
B^I_r(W) \cap \io^\al_I\bigl(A^\al_{\eps-r}\times \p^\al U_I \bigr)  = \io^\al_I\bigl( A^\al_{\eps-r} \times B^{I,\al}_r(\p^\al W)  \bigr) , 
\end{equation}
with $\p^\al B^I_r(W) = B^{I,\al}_r(\p^\al W)$, where we denote by $B^{I,\al}_r$ the neighbourhoods in $\p^\al U_I$ induced by pullback of the metric $d_I$ with $\io^\al_I:\{\al\} \times \p^\al U_I \to U_I$.
\item
If in (iv) the collared sets $W\subset U_I$ are obtained as products with $[0,1]$ in a product Kuranishi cobordism $\Kk=[0,1]\times \Kk'$, then \eqref{eq:Wnbhd} holds for any $r>0$ with $\eps-r$ replaced by $1$.
\end{enumerate}
\end{lemma}
\begin{proof}
To check statement (i) it suffices to fix $J\in\Ii_\Kk$ and consider the continuous function $\|s_J\|$ over the compact set $\ov{V^{|J|}_J} \;\less\; \bigl( \Ti C_J \cup {\textstyle \bigcup_{I\subsetneq J}} B^J_{\eta_{|J|-\frac 12}}\bigl(N^{|J|-\frac14}_{JI}\bigr) \bigr)$. We claim that its infimum is positive since the domain is disjoint from $s_J^{-1}(0)$. Indeed, the reduction property $\io_\Kk(X)\subset\pi_\Kk(\Cc)$ implies $s_J^{-1}(0)\subset \Ti C_J \cup \bigcup_{I\subsetneq J}  \phi_{IJ}(C_I\cap U_{IJ})$, the intersections $\ov{V^{|J|}_J}\cap \phi_{IJ}(C_I\cap U_{IJ})$ are contained in $N^{k}_{JI}$ for any $k<|J|$ since $V^{|J|}_J \sqsubset V^k_J$ and $C_I\sqsubset V_I \subset V^k_I$, and we have $N^{k}_{JI}\subset  B^J_{\eta_{|J|-\frac 12}}\bigl(N^{|J|-\frac14}_{JI}\bigr)$ for $k\geq -|J|+\frac 12$.
[\footnote{
In the case of \cite{MW:iso}, we use \eqref{eq:VCC} to control $s_J^{-1}(0) \cap  \pi_\Kk^{-1}(\pi_\Kk(\Cc))$.
}]
% note to self
%$$
%2^{-k}\de \le \eta_{|J|-\frac 12} + 2^{-|J|+\frac14}\de 
%\quad\Leftrightarrow\quad
%2^{-k} \le
% 2^{-|J|+\frac 12}  (1 -  2^{-\frac 14} ) + 2^{-|J|+\frac14}
% =  2^{-|J|+\frac 12} 
%$$

Statement (ii) holds since all sets and metrics involved are of product form, and we chose $\de_\Vv\le \frac 14$.

Statement (iii) follows by pullback with $\io^\al_I|_{\{\al\} \times \p^\al U_I}$ since the $2^{-k}\de$-neighbourhood $(\p^\al V_I)^k$ of the boundary $\p^\al V_I$ within $\p^\al U_I$ is always contained in the boundary $\p^\al V^k_I$ of the $2^{-k}\de$-neighbourhood of the domain $V_I$.

To check (iv) note in particular the product forms $(\io^\al_I)^{-1}(V_I)=A^\al_\eps\times \p^\al V_I $ and
$(\io^\al_I)^*d_I = d_\R +d^\al_I$ on the $\eps$-collars, where $d^\al_I$ denotes the metric on $\p^\al U_I$ induced from the restriction of the metric on $|\Kk|$ to $|\p^\al\Kk|$.
Then for any $\p^\al W \subset \p^\al U_I$ the product form of the metric implies product form of the $r$-neighbourhoods 
$$
B^I_r\bigl(\io^\al_I(A^\al_\eps\times \p^\al W)\bigr)\cap \im\io^\al_I \;=\; \io^\al_I\bigl(A^\al_\eps\times  B^{I,\al}_r(\p^\al W)\bigr) .
$$
Moreover, for any $W' \subset U_I\less\im\io^\al_I$ and $r < \eps$, 
the collaring condition \eqref{eq:epscoll} on the metric implies that 
\begin{equation}\label{eq:nob}
B^I_r(W') \cap \io^\al_I(A^\al_{\eps-r}\times \p^\al U_I) = \emptyset .
\end{equation}
The identity \eqref{eq:Wnbhd} now follows from applying the above identities with $W'=W \less \im\io^\al_I$.

In the product case (v), the complement of the (closed) collars is empty, so there is no need for the second identity and hence for the restriction $r<\eps$.
\end{proof}

In the case of a metric tame Kuranishi atlas we will construct transverse perturbations $\nu = \bigl(\nu_I : V_I \to E_I \bigr)_{I\in\Ii_\Kk}$ by an iteration which constructs and controls each $\nu_I$ over the larger set ${V_I^{|I|}}$.  In order to prove uniqueness of the VMC, we will moreover need to interpolate between any two such perturbations by a similar iteration. We will use the following definition to keep track of the refined properties of the perturbations 
in this iteration.
\MS

\begin{defn}  \label{a-e}
Given a nested reduction $\Cc\sqsubset\Vv$ of a metric tame Kuranishi atlas 
$(\Kk,d)$, additive norms $\|\cdot\|$, and a choice of constants $0<\de<\de_\Vv$ and $0<\si\le\si(\Vv,\Cc,\|\cdot\|,\de)$, we say that a perturbation $\nu$ of $\s_\Kk|_\Vv$ is {\bf ${(\Vv,\Cc,\|\cdot\|,\de,\si)}$-adapted} if the sections $\nu_I:V_I\to E_I$ extend to sections over ${V^{|I|}_I}$ (also denoted $\nu_I$) so that the following conditions hold for every $k=1,\ldots, M_\Kk$ with 
$$
M_\Kk:= \max_{I\in\Ii_\Kk} |I|, \qquad
\eta_k:=2^{-k}\eta_0=2^{-k} (1-2^{-\frac 14})\de .
$$
\begin{itemize}
\item[a)]
The perturbations are compatible in the sense that the commuting diagrams in Definition~\ref{def:sect} hold on 
$\bigsqcup_{|I|\leq k} {V^k_I}$, that is {\rm [\footnote{
For \cite{MW:iso} this is replaced by \eqref{eq:compatc}.
}]}
$$
\qquad
\nu_I \circ \phi_{HI} |_{{V^k_H}\cap \phi_{HI}^{-1}({V^k_I})} \;=\; \Hat\phi_{HI} \circ \nu_H |_{{V^k_H}\cap \phi_{HI}^{-1}({V^k_I})} 
\qquad \text{for all} \; H\subsetneq I , |I|\leq k .
$$
\item[b)]
The perturbed sections are transverse, that is $(s_I|_{{V^k_I}} + \nu_I) \pitchfork 0$ for each $|I|\leq k$.
\item[c)]
The perturbations are {\bf strongly admissible} with radius $\eta_k$, that is for all $H\subsetneq I$ and $|I|\le k$ we have
$$
\qquad
\nu_I\bigl( B^I_{\eta_k}(N^{k}_{IH})\bigr) \;\subset\; \Hat\phi_{HI}(E_H) 
\qquad
\text{with}\;\;
N^k_{IH} = V^k_I \cap \phi_{HI}(V^k_H\cap U_{HI}) .
$$
In particular, the perturbations are admissible along the core $N^k_I$, that is we have $\im\rd_x\nu_I \subset \im\Hat\phi_{HI}$ at all $x\in N^k_{IH}$. {\rm [\footnote{
In the setting of \cite{MW:iso}, $\nu_I\bigl( \Hat B^I_{\eta_k}(N^{k}_{IH})\bigr) \subset \Hat\phi_{HI}(E_H)$ also implies admissibility because $\Hat B^I_{\eta_k}(N^{k}_{IH})= \pi_I^{-1}\bigl(B^I_{\eta_k}(\uN^{k}_{IH})\bigr)$ is also open and contains $\pi_I^{-1}(\uN^{k}_{IH})=N^{k}_{IH}$.
}]}
\item[d)]  
The perturbed zero sets are contained in $\pi_\Kk^{-1}\bigl(\pi_\Kk(\Cc)\bigr)$; more precisely
$$
(s_I |_{{V^k_I}}+ \nu_I)^{-1}(0) \;\subset\; {V^k_I} \cap \pi_\Kk^{-1}\bigl(\pi_\Kk(\Cc)\bigr)
\qquad
\forall |I|\leq k,
$$
or equivalently $s_I + \nu_I \neq 0$ on ${V^k_I} \less  \pi_\Kk^{-1}\bigl(\pi_\Kk(\Cc)\bigr)$.
\item[e)]
The perturbations are small, that is $\sup_{x\in {V^k_I}} \| \nu_I (x) \| < \si$
for $|I|\leq k$. 
\end{itemize}

Given a metric Kuranishi atlas $(\Kk,d)$, we say that a perturbation $\nu$  is {\bf adapted} if it is a $(\Vv,\Cc,\|\cdot\|,\de,\si)$-adapted perturbation $\nu$ of $\s_\Kk|_\Vv$ for some choice of nested reduction $\Cc\sqsubset\Vv$, additive norms $\|\cdot\|$, and constants $0<\de<\de_\Vv$ and $0<\si\le\si(\Vv,\Cc,\|\cdot\|,\de)$. 
\end{defn}

Next, we note some simple properties of these notions; in particular the fact that adapted perturbations are automatically admissible, precompact, and transverse.

\begin{lemma}\label{le:admin2}
Let $\nu$ be a $(\Vv,\Cc,\|\cdot\|,\de,\si)$-adapted perturbation of $\s_\Kk|_\Vv$. 
Then $\nu$ is an admissible, precompact, transverse perturbation with $\pi_\Kk( (\s_\Kk|_\Vv+\nu)^{-1}(0))\subset\pi_\Kk(\Cc)$.
Moreover, $\nu$ is $(\Vv,\Cc,\|\cdot\|',\de',\si')$-adapted for any choice of additive norms $\|\cdot\|'\leq \|\cdot\|$ and constants $0<\de'\leq\de$, $\si'\geq\si$ such that $\si'\leq \si(\Vv,\Cc,\|\cdot\|',\de')$.
\end{lemma}
\begin{proof}
First note that $\nu$ is an admissible reduced section in the sense of Definition~\ref{def:sect} by c) and d), and is transverse by b). 
Restriction of a) implies that it satisfies the zero set condition $s_I+\nu_I \neq 0$ on $V_I \less \pi_\Kk^{-1}(\pi_\Kk(\Cc))$, and hence $\pi_\Kk( (\s_\Kk|_\Vv+\nu)^{-1}(0))\subset\pi_\Kk(\Cc)$, which in particular implies precompactness in the sense of Definition~\ref{def:precomp}.

To see that $\nu$ is also admissible with respect to the data $(\Vv,\Cc,\|\cdot\|',\de')$, first note that this statement only makes sense for $0<\de'<\de_\Vv$ and $0<\si'\le\si(\Vv,\Cc,\|\cdot\|',\de')$, which is ensured by the assumptions.
Next, the reduced subsets ${V'}_I^{|I|}$ defined by $0<\de'\le\de$ are contained in $V_I^{|I|}$, so that the extensions of $\nu_I$ to $V_I^{|I|}$ restrict to sections over the ${V'}_I^{|I|}$ that still satisfy the compatibility, transversality, and zero set conditions a),b),d). They are strongly admissible since we have $\eta'_k\leq \eta_k$ and ${N'}^k_{IH} = {V'}^k_I \cap \phi_{HI}({V'}^k_H\cap U_{HI})\subset N^k_{IH}$. Finally, e) is satisfied since 
$\sup_{x\in {{V'}^k_I}} \| \nu_I (x) \|' \leq \sup_{x\in {V^k_I}} \| \nu_I (x) \| < \si \leq \si'$.
\end{proof}

Using these notions, we now prove a refined version of the existence of admissible, precompact, transverse perturbations in every metric tame Kuranishi atlas.

\begin{prop}\label{prop:ext}
Let $(\Kk,d)$ be metric tame Kuranishi atlas with nested reduction $\Cc \sqsubset \Vv$ and additive norms $\|\cdot\|$.
Then for any $0<\de<\de_\Vv$ and $0<\si\le\si(\Vv,\Cc,\|\cdot\|,\de)$ there exists a $(\Vv,\Cc,\|\cdot\|,\de,\si)$-adapted perturbation $\nu$ of $\s_\Kk|_{\Vv}$.  In particular, $\nu$ is admissible, precompact, and transverse, and its perturbed zero set $|\bZ^\nu|=|(\s_\Kk|_\Vv+\nu)^{-1}(0)|$ is compact with $\pi_\Kk\bigl((\s_\Kk|_\Vv+\nu)^{-1}(0)\bigr)$ contained in $\pi_\Kk(\Cc)$.
\end{prop}

%\comment{TBD: either say that compactness doesn't apply for \cite{MW:iso} or figure out how to quote the topology result}

\begin{proof}
We will construct $\nu_I: V^{|I|}_I\to E_I$ by an iteration over $k=0,\ldots,M= \max_{I\in\Ii_\Kk} |I|$, where in step $k$ we will define $\nu_I : V^k_I \to E_I$ for all $|I| = k$ that, together with the $\nu_I|_{V^k_I}$ for $|I|<k$ obtained by restriction from earlier steps, satisfy conditions a)-e) of Definition~\ref{a-e}.
Restriction to $V_I\subset V^{|I|}_I$ then yields a $(\Vv,\Cc,\|\cdot\|,\de,\si)$-adapted perturbation $\nu$ of $\s_\Kk|_\Vv$, which by Lemma~\ref{le:admin2} is automatically an admissible, precompact, transverse perturbation with $\pi_\Kk( (\s_\Kk|_\Vv+\nu)^{-1}(0))\subset\pi_\Kk(\Cc)$. 
Compactness of $|(\s_\Kk|_\Vv+\nu)^{-1}(0)|$ then follows from Proposition~\ref{prop:zeroS0}.
So it remains to perform the iteration.

For $k=0$ the conditions a)-e) are trivially satisfied since there are no index sets $I\in\Ii_\Kk$ with $|I|\leq 0$. Now suppose that $\bigl(\nu_I : V^k_I\to E_I\bigr)_{I\in\Ii_\Kk, |I|\leq k}$  are constructed such that a)-e) hold. In the next step we can then construct $\nu_J$ independently for each $J\in\Ii_\Kk$ with $|J|=k+1$, since for any two such $J,J'$ we have $\pi_\Kk(V_J^{k+1}) \cap \pi_\Kk(V_{J'}^{k+1})=\emptyset$ unless $J=J'$ by \eqref{desep}, and so the constructions for $J\neq J'$ are not related by the commuting diagrams in condition a).

\MS\NI
{\bf Construction for fixed $\mathbf {|J|=k+1}$:} 
We begin by noting that a) requires for all $I\subsetneq J$ 
\begin{equation} \label{some nu}
\nu_J|_{N^{k+1}_{JI}} \;=\; 
\nu_J|_{V^{k+1}_J \cap \phi_{IJ}(V^{k+1}_I\cap U_{IJ})} \;=\; \Hat\phi_{IJ}\circ \nu_I\circ\phi_{IJ}^{-1} |_{N^{k+1}_{JI}} .
\end{equation}
To see that these conditions are compatible, we note that for $H\neq I\subsetneq J$ with $\phi_{HJ}(V^k_H\cap U_{IJ}) \cap \phi_{IJ}(V^k_I\cap U_{IJ})\neq \emptyset $ property 
\eqref{Nsep}
implies $H\subsetneq I$ or $I\subsetneq H$. Assuming w.l.o.g.\ the first, we obtain compatibility from the strong cocycle condition in Definition~\ref{def:cocycle} and property a) for $H\subsetneq I$, 
\begin{align*}
&\Hat\phi_{HJ}\circ \nu_H\circ\phi_{HJ}^{-1} |_{V^k_J \cap \phi_{IJ}(V^k_I\cap U_{IJ})\cap \phi_{HJ}(V^k_H\cap U_{HJ})} \\
&=
\Hat\phi_{IJ}\circ \bigl(\Hat\phi_{HI}\circ \nu_H \bigr) |_{\phi_{HJ}^{-1}(V^k_J) \cap \phi_{HI}^{-1}(V^k_I) \cap V^k_H} \circ \phi_{HI}^{-1} \circ \phi_{IJ}^{-1} \\
&=
\Hat\phi_{IJ}\circ \bigl(\nu_I \circ \phi_{HI}\bigr) \circ \phi_{HI}^{-1} \circ \phi_{IJ}^{-1} 
\;=\;
\Hat\phi_{IJ}\circ \nu_I\circ\phi_{IJ}^{-1} .
\end{align*}
Here we checked compatibility on the domains $N^k_{JI}$, thus defining a map 
[\footnote{ \label{35}
In the setting of \cite{MW:iso}, the compatibility follows from the strong cocycle condition in terms of the regular coverings, $\rho_{HJ} = \rho_{HI}\circ \rho_{IJ}$ on $\TU_{HJ}=\TU_{IJ}\cap \rho_{IJ}^{-1}(\TU_{HI})$.
Indeed, the intersection domain is contained in $\rho_{IJ}^{-1}(U_{IJ})\cap \rho_{HJ}^{-1}(U_{HJ})= \TU_{IJ} \cap \TU_{HJ}$.
The map $\mu_J$ is then defined on each $N^k_{JI}=V^k_J\cap\rho_{IJ}^{-1}(V^k_I)$ by pulling back $\nu_I$ to the $|\Ga_{J\less I}|$ disjoint lifts under the regular covering $\rho_{IJ}: N^k_{JI} \to \rho_{IJ}(V^k_J)\cap V^k_I$.
}]
\begin{equation}\label{eq:nuJ'}
\mu_J \,:\;  N^k_J = 
{\textstyle \bigcup_{I\subsetneq J}} 
N^k_{JI} \;\longrightarrow\; E_J , \qquad 
\mu_J|_{N^k_{JI}} := \Hat\phi_{IJ}\circ \nu_I\circ\phi_{IJ}^{-1} .
\end{equation}
Note moreover that for $x=\phi_{IJ}(y)\in N^k_{JI}$ we have $\|\mu_J(x)\|= \|\nu_I(y)\|$ by the compatible construction of norms on the obstruction spaces. Taking the supremum over $N^k_J$ this implies
$$
\|\mu_J\| \,:=\; \sup_{y\in N^k_J} \|\mu_J(y)\| \;\leq\; \sup_{I\subsetneq J} \sup_{x\in V^k_I} \|\nu_I(x)\| \;<\; \si.
$$  
The construction of $\nu_J$ on $V^{k+1}_J$ then has three more steps.

\begin{itemize} 
\item {\bf Construction of extension:}
We construct an extension of the restriction of $\mu_J$ from \eqref{eq:nuJ'} to the enlarged core $N_J^{k+\frac 12}$. More precisely, we construct a smooth map $\Ti\nu_J : V^k_J \to E_J$ that satisfies
\begin{equation}\label{tinu}
\Ti\nu_J|_{N_J^{k+\frac 12}} \;=\; \mu_J|_{N_J^{k+\frac 12}} , \qquad\quad
\|\Ti\nu_J \| \;\leq \; \|\mu_J\| \;<\; \si ,
\end{equation}
and the strong admissibility condition on a larger domain than required in c),
\begin{equation}\label{value} 
\Ti\nu_J \bigl( B^J_{\eta_{k+\frac 12}}\bigl(N^{k+\frac 12}_{JI}\bigr) \bigr) \;\subset\; \Hat\phi_{IJ}(E_I) 
\qquad \forall\; I\subsetneq J .
\end{equation}
(In case $k=0$, i.e.\ $|J|=1$, the map $\mu_J$ is defined on the empty set, so that we may simply set $\Ti\nu_J:=0$.)

\vspace{.03in}

\item {\bf Zero set condition:} 
We show that \eqref{value} and the control over $\|\Ti\nu_J\|$ imply the strengthened control of d) over the 
zero set of $s_J + \Ti\nu_J$, in particular
$$
\bigl(s_J|_{{V^{k+1}_J}} + \Ti\nu_J\bigr)^{-1}(0) \;\less\;  B^J_{\eta_{k+\frac 12}}\bigl(N^{k+\frac34}_J\bigr)   \;\subset\; \Ti C_J .
$$
(In case $k=0$, i.e.\ $|J|=1$, this condition reads
$s_J|_{{V^1_J}}^{-1}(0) \subset \Ti C_J = U_J\cap \pi_\Kk^{-1}(\pi_\Kk(\Cc))$, 
which is satisfied since $F_J\subset\pi_\Kk(\Cc)$ by construction of $\Cc$.)
\vspace{.03in}

\item {\bf Transversality:}  
We make a final perturbation $\nu_\pitchfork$ to obtain transversality for $s_J + \Ti\nu_J + \nu_\pitchfork$, 
while preserving conditions a),c),d), and then set $\nu_J: = \Ti\nu_J + \nu_\pitchfork$.
Moreover, taking $\|\nu_\pitchfork\|< \si - \|\Ti\nu_J\|$ ensures e).
(In the case $k=0$ this is the first nontrivial step.)
\end{itemize}

\MS\NI
{\bf Construction of extensions:}  
 To construct $\Ti\nu_J$ in case $k\geq 1$  it suffices, in the notation of \eqref{eq:iI}, to extend each component $\mu^j_J$ for fixed $j\in J$.
For that purpose we iteratively construct smooth maps $\Ti\mu_\ell^j: W_\ell \to \Hat\phi_{jJ}(E_j)$ on the open sets 
\begin{equation}\label{eq:W}
W_\ell \,:=\; {\textstyle \bigcup _{I\subsetneq J,|I|\le \ell}}\, B^J_{r_\ell}(N^{k+\frac 12}_{JI}) 
\;=\;  B^J_{r_\ell}\bigl( {\textstyle \bigcup _{I\subsetneq J,|I|\le \ell}}\,N^{k+\frac 12}_{JI} \bigr) \;\subset\; U_J 
\end{equation}
with the radii 
$r_\ell:= \eta_k - \frac {\ell+1} {k+1} ( \eta_k-\eta_{k+\frac 12})$,
that satisfy the extended compatibility, admissibility, and smallness conditions
\vspace{.07in}
\begin{enumerate}
\item[(E:i)]
$\Ti\mu^j_\ell |_{N^{k+\frac 12}_{JI}} = \mu_J^j|_{N^{k+\frac 12}_{JI}}$ for all $I\subsetneq J$ with $|I|\leq \ell$ and $j\in I$;
\vspace{.07in}
\item[(E:ii)]
$\Ti\mu^j_\ell |_{B^J_{r_\ell}(N^{k+\frac 12}_{JI})} = 0$
for all $I\subsetneq J$ with $|I|\leq \ell$ and $j\notin I$;
\vspace{.07in}
\item[(E:iii)] 
$\bigl\|\Ti\mu^j_\ell \bigr\| \leq \|\mu^j_J\|$.
\end{enumerate}
\vspace{.07in}
Note here that the radii form a nested sequence $\eta_k=r_{-1} > r_0>r_1 \ldots > r_k = \eta_{k+\frac 12}$ and that when $\ell=k$ the function $\Ti\mu^j_k$ will satisfy  (E:i),(E:ii) for all $I\subsetneq J$,  and is defined on 
$W_k=B^J_{\eta_{k+\frac 12}}(N^{k+\frac 12}_J)\sqsupset N^{k+\frac 12}_J$.
So, after this iteration, we can define 
\begin{equation}\label{eq:betamu}
\Ti\nu_J:= \beta \bigl( {\textstyle\sum_{j\in J}} \, \Ti\mu^j_k\bigr),
\end{equation}
 where $\beta:U_J \to [0,1]$ is a smooth cutoff function with $\beta|_{N^{k+\frac 12}_J}\equiv 1$ and $\supp\beta\subset 
B^J_{\eta_{k+\frac 12}}(N^{k+\frac 12}_J)$, so that $\Ti\nu_J$ extends trivially to $U_J\less W_k$.
This has the required bound by (E:iii), satisfies \eqref{value} since $\Ti\nu_J^j |_{B^J_{\eta_{k+\frac 12}}(N^{k+\frac 12}_{JI})}\equiv 0$ for all $j\notin I$ by (E:ii). Finally, it has the required values on $N^{k+\frac 12}_J = \bigcup_{I\subsetneq J}N^{k+\frac 12}_{JI}$ since for each $I\subsetneq J$ the conditions (E:i), (E:ii) on $N^{k+\frac 12}_{JI}$ together with the fact $\mu_J(N_{JI}^{k+\frac 12}) \subset \Hat\phi_{IJ}(E_I)$ guarantee 
$$
\Ti \nu_J|_{N_{JI}^{k+\frac12}} \;=\; {\textstyle\sum_{j\in J}} \, \Ti\mu_k^j|_{N_{JI}^{k+\frac12}}
\;=\; {\textstyle\sum_{j\in I}} \, \mu_k^j|_{N_{JI}^{k+\frac12}} \;=\; \mu_J|_{N_{JI}^{k+\frac12}} .
$$ 
So it remains to perform the iteration over $\ell$, in which we now drop $j$ from the notation.
For $\ell=0$ the conditions are vacuous since $W_0=\emptyset$.
Now suppose that the construction is given on $W_\ell$. 
Then we cover $W_{\ell+1}$ by the open sets 
$$
B_L': = W_{\ell+1}\cap B^J_{r_{\ell-1}}(N^{k+\frac 12}_{JL})
\qquad
\text{for}\; L\subsetneq J, \; |L|=\ell+1,
$$ 
whose closures are pairwise disjoint by \eqref{Nsep} with $r_{\ell-1}<\delta$, and an open set 
$C_{\ell+1}\subset U_J$
covering the complement,
$$
C_{\ell+1} \,:=\; W_{\ell+1} \;\less\; {\textstyle \bigcup _{|L| = \ell+1}\, \ov{B^J_{r_{\ell}}(N^{k+\frac 12}_{JL})}} \;\sqsubset\; 
W_\ell \;\less\; {\textstyle \bigcup _{|L| = \ell+1}\, \ov{B^J_{r_{\ell+1}}(N^{k+\frac 12}_{JL})}} \;=:\, C_\ell ,
$$
which has a useful precompact inclusion into $C_\ell$,
as defined above,
by $r_{\ell+1}<r_\ell$.
This decomposition is chosen so that each $B^J_{r_{\ell+1}}(N^{k+\frac 12}_{JL})$ for $|L|=\ell+1$ (on which the conditions (E:i),(E:ii) for $I=L$ are nontrivial) has disjoint closure from $\ov{C_{\ell+1}}$ (a compact subset of the domain of $\Ti\mu_\ell$). 
Now pick a precompactly nested open set $C_{\ell+1} \sqsubset C' \sqsubset C_\ell$, 
in particular with $\ov{C'}\cap \ov{B^J_{r_{\ell+1}}(N^{k+\frac 12}_{JL})} = \emptyset$ for all $|L|=\ell +1$.
[\footnote{
For \cite{MW:iso} one can choose $\uC_{\ell+1} \sqsubset \uC' \sqsubset \uC_\ell \subset \uU_J$ and set $C':=\pi_J^{-1}(\uC')$ to ensure the required intersection properties of $C'$.
}]
Then we will obtain a smooth map $\Ti\mu_{\ell+1}: W_{\ell+1} \to \Hat\phi_{jJ}(E_j)$ by setting $\Ti\mu_{\ell+1}|_{C_{\ell+1}} := \Ti \mu_\ell|_{C_{\ell+1}}$  and separately constructing smooth maps $\Ti\mu_{\ell+1}: B'_L \to \Hat\phi_{jJ}(E_j)$ for each $|L|=\ell+1$ such that $\Ti\mu_{\ell+1}=\Ti \mu_\ell$ on $B'_L \cap C'$.
Indeed, this ensures equality of all derivatives on the intersection of the closures $\ov{B_L'} \cap \ov{C_{\ell+1}}$, since this set is contained in $\ov{B_L'} \cap C'$, which is a subset of $\ov{B_L' \cap C'}$ because $C'$ is open, 
and by construction we will have $\Ti\mu_{\ell+1}=\Ti \mu_\ell$ with all derivatives on $\ov{B_L' \cap C'}$. So it remains to construct the extension $\Ti\mu_{\ell+1}|_{B'_L}$ for a fixed $L\subsetneq J$.
For that purpose note that the subset on which this is prescribed as $\Ti\mu_\ell$, can be simplified by the separation property \eqref{Nsep}, 
\begin{equation} \label{simply}
B'_L\cap C' \;\subset\; 
\Bigl( B^J_{r_{\ell-1}}(N^{k+\frac 12}_{JL}) \less \ov{B^J_{r_{\ell+1}}(N^{k+\frac 12}_{JL})} \; \Bigr) \;\cap\; {\textstyle \bigcup _{I\subsetneq L} }B^J_{r_\ell}(N^{k+\frac 12}_{JI}) \;\subset\; W_\ell.
\end{equation}
To ensure (E:i) and (E:ii) for $|I|\leq \ell+1$ first note that $\Ti\mu_{\ell+1}|_{C_{\ell+1}}$ inherits these properties from $\Ti\mu_\ell$ because $C_{\ell+1}$ is disjoint from $B^J_{r_{\ell+1}}(N^{k+\frac 12}_{JI})$ for all $|I|=\ell+1$.
It remains to fix $L\subset J$ with $|L|=\ell+1$ and construct the map $\Ti\mu_{\ell+1}: B'_L \to \Hat\phi_{jJ}(E_j)$ as extension of $\Ti\mu_\ell|_{B'_L\cap C'}$ so that it satisfies properties (E:i)--(E:iii) for all $|I|\le\ell+1$. 
 
In case $j\notin L$ we have $\Ti\mu_\ell|_{B'_L\cap C'}=0$ by iteration hypothesis (E:ii) for each $I\subsetneq J$. 
So we obtain a smooth extension by $\Ti\mu_{\ell+1}:=0$, which satisfies (E:ii) and (E:iii), whereas (E:i) is not relevant.

In case $j\in L$ the conditions (E:i),(E:ii) only require consideration of $I\subsetneq L$ since otherwise $B'_L \cap B^J_{r_\ell}(N^{k+\frac 12}_{JI})=\emptyset$ by \eqref{Nsep}. So we need to construct a bounded smooth map $\Ti\mu_{\ell+1} : B'_L=W_{\ell+1} \cap B^J_{r_{\ell-1}}(N^{k+\frac 12}_{JL}) \to \Hat\phi_{jJ}(E_j)$ that satisfies 
\begin{itemize}
\item[(i)]
$\Ti\mu_{\ell+1}|_{N^{k+\frac 12}_{JL}} = \mu_J^j|_{N^{k+\frac 12}_{JL}}$;
\vspace{.07in}
\item[(i$'$)]
$\Ti\mu_{\ell+1}|_{N^{k+\frac 12}_{JI}} = \mu_J^j|_{N^{k+\frac 12}_{JI}}$ for all $I\subsetneq L$ with $j\in I$;
\vspace{.07in}
\item[(ii)]
$\Ti\mu_{\ell+1}|_{B^J_{r_{\ell+1}}(N^{k+\frac 12}_{JI})} = 0$ for all $I\subsetneq L$ with $j\notin I$;
\vspace{.07in}
\item[(iii)] 
$\bigl\| \Ti\mu_{\ell+1}\bigr\| \leq \|\mu_J^j\|$;
\vspace{.07in}
\item[(iv)] 
$\Ti\mu_{\ell+1}|_{B'_L\cap C'} =\Ti\mu_{\ell}|_{B'_L\cap C'}$.
\end{itemize}
\vspace{.07in}
Because every open cover of $B'_L$ has a locally finite subcovering, such extensions can be patched together by partitions of unity. Hence it suffices, for the given $j\in L\subsetneq J$, to construct smooth maps $\Ti\mu_z: B^J_{r_z}(z)\to \Hat\Phi_{jJ}(E_j)$ on some balls of positive radius $r_z>0$ around each fixed $z\in B'_L$, that satisfy the above requirements.

\NI $\bullet$ 
For $z\in W_\ell \less \ov{N^{k+\frac 12}_{JL}}$ we find $r_z>0$ such that $B^J_{r_z}(z)\subset W_\ell \less \ov{N^{k+\frac 12}_{JL}}$ lies in the domain of $\Ti\mu_\ell$ and the complement of $N^{k+\frac 12}_{JL}$, so that $\Ti\mu_z:=\Ti\mu_\ell|_{B^J_{r_z}(z)}$ is well defined and satisfies all conditions with $\|\Ti\mu_z\|\leq \|\Ti\mu_\ell\|$.

\NI $\bullet$ 
For $z\in B'_L\less \bigl( W_\ell \cup \ov{N^{k+\frac 12}_{JL}}\bigr)$, we claim that there is $r_z>0$ such that $B^J_{r_z}(z)$ is disjoint from the closed subsets $\bigcup_{I\subset L} \ov{N^{k+\frac 12}_{JI}}$ and $\ov{C'}\subset C_\ell\subset W_\ell$. 
This holds because $\bigcup_{I\subsetneq L} \ov{N^{k+\frac 12}_{JI}}\subset W_\ell$ by \eqref{eq:W}.
 Then $\Ti\mu_z:=0$ satisfies all conditions since its domain is in the complement of the domains on which (i), (i$'$), and (iv) are relevant.

\NI $\bullet$ 
Finally, for $z\in \ov{N^{k+\frac 12}_{JL}}$ recall that $\ov{N^{k+\frac 12}_{JL}} \sqsubset N^k_{JL}$ is a compact subset of the smooth submanifold $N^k_{JL}=V^k_J\cap \phi_{LJ}(V^k_L\cap U_{LJ}) \subset A_J$. So we can choose $r_z>0$ such that $B^J_{r_z}(z)$ lies in a submanifold chart for $N^k_{JL}$. Then we define $\Ti\mu_z$ by extending $\mu_J^j|_{B^J_{r_z}(z)\cap N^k_{JL}}$ to be constant in the normal directions. This guarantees (i) and $\|\Ti\mu_z\|\leq \|\mu_J^j\|$. It remains to prove the following claim.
\smallskip

\NI{\bf Claim:}  {\it For sufficiently small $r_z$, all the remaining conditions (i$\,'$), (ii), (iii), (iv)  hold.}
\MS

First, $N^{k}_{JL}$ is disjoint from $C_{\ell}\sqsupset C'$, so we can ensure that $B^J_{r_z}(z)$ lies in the complement of $C'$, and hence condition (iv) does not apply.
To begin to address (i$'$) and (ii) recall that for every $I\subsetneq L\subset J$ the strong cocycle condition of Lemma~\ref{le:tame0} implies that the open subset 
$N^{k+\frac 12}_{JI}\subset\im\phi_{IJ}  =\phi_{LJ}(U_{LJ}\cap\im\phi_{IL})$ is a submanifold of $\im\phi_{LJ}$, and by assumption $z$ lies in the open subset $N^k_{JL}\subset\im\phi_{LJ}$.
[\footnote{
For \cite{MW:iso} this amounts to the statement that $N^k_{IJ}\subset \TU_{IJ}$ is open and $\TU_{IJ}\subset\TU_{LJ}$ is a submanifold.
}]

In case $j\in I$ and $z\in N^{k+\frac 12}_{JI}\cap \ov{N^{k+\frac 12}_{JL}}$, we can thus choose $r_z$ sufficiently small to ensure that $B^J_{r_z}(z)\cap N^{k+\frac 12}_{JI}$ is contained in the open neighbourhood $N^k_{JL}\subset\im\phi_{LJ}$ of~$z$.
Then $\Ti\mu_z$ satisfies (i$'$) by $\Ti\mu_z=\mu^j_J$ on $B^J_{r_z}(z)\cap N^{k+\frac 12}_{JI}$.

In case $j\notin I$ condition (ii) requires $\Ti\mu_z$ to vanish on $B^J_{r_z}(z)\cap B^J_{r_{\ell+1}}(N^{k+\frac 12}_{JI})$. 
Here we have $r_{\ell+1}\leq r_1 <\eta_k$, so if $z\notin B^J_{\eta_k}(N^{k+\frac 12}_{JI})$, then we can make this intersection empty by choice of $r_z$, so that (ii) is vacuous. It remains to consider the case $z\in B^J_{\eta_k}(N^{k+\frac 12}_{JI})\cap \ov{N^{k+\frac 12}_{JL}}$, where $I\subset L\subset J$ as above. That is, we have $z=\phi_{LJ}(z_L)$ for some $z_L\in \ov{V^{k+\frac 12}_L}$ and $d_J(z,x_J)< \eta_k$ for some $x_J\in N^{k+\frac 12}_{JI}$. Moreover, this yields $x_J=\phi_{IJ}(x_I)$ for some $x_I\in V^{k+\frac 12}_I\cap U_{IJ}$. By tameness we also have $x_I\in U_{IL}$, and compatibility of the metrics then implies $d_L(z_L,x_L)=d_J(z,x_J)< \eta_k$ for $x_L:=\phi_{IL}(x_I)$.
This shows that $x_L$ lies in both $\phi_{IL}(V^{k+\frac 12}_I\cap U_{IJ})$ and $B_{\eta_k}( \ov{V^{k+\frac 12}_L} )$, where the latter is a subset of $V_L^k$ by \eqref{eq:fantastic}, and hence we deduce $x_L\in N^k_{LI}$.
Since $\nu_L( B^L_{\eta_k}(N^{k}_{LI})\bigr) \subset \Hat\phi_{IL}(E_I)$ by the induction hypothesis c), we obtain $\nu^j_L|_{B^L_{\eta_k}(x_L)} = 0$, so that the function $\mu^j_J$ of \eqref{eq:nuJ'} vanishes on 
$$
\phi_{LJ}(B^L_{\eta_k}(x_L)\cap U_{LJ})= B^J_{\eta_k}(x_J) \cap \phi_{LJ}(U_{LJ}) .
$$ 
Since $d_J(z,x_J) < \eta_k$ this set contains $z$, and thus $B^J_{r_z}(z)\cap\phi_{LJ}
(U_{LJ})$ for $r_z>0$ sufficiently small.
With that we have $\mu_J^j|_{B^J_{r_z}(z)\cap N^k_{JL}}=0$ and hence constant extension in normal direction yields $\Ti\mu_z = 0$, so that (ii) is satisfied.
This completes the construction of $\Ti\mu_z$ in this last case, and hence of $\Ti\mu_{\ell+1}$, and thus by iteration finishes the construction of the extension $\Ti\nu_J$.  
[\footnote{
For \cite{MW:iso} this last case is $j\notin I$ and we apply most of the above argument in the intermediate domain to see that for sufficiently small $r_z>0$ we either get $\Hat B^J_{r_z}(z)$ disjoint from $\Hat B^J_{\eta_k}(\uN^{k+\frac 12}_{JI})$ -- so (ii) is trivially satisfied -- or $\Hat B^J_{r_z}(z)\cap\TU_{LJ}\subset \Hat B^J_{\eta_k}(X_J) \cap \TU_{LJ}$ for $X_J=\pi_J^{-1}(x_J)$ the preimage of some $x_J=\uphi_{LJ}(x_L)\in\uN^{k+\frac 12}_{JI}\cap \uphi_{LJ}(\uN^k_{LI})$.
On the other hand, $\mu^j_J$ vanishes on $\rho_{LJ}^{-1}\bigl(\Hat B^L_{\eta_k}(N^{k}_{LI})\bigr)$, which contains $\Hat B^J_{\eta_k}(X_J)\cap \TU_{LJ}$ and thus $\Hat B^J_{r_z}(z)\cap\TU_{LJ}$, so that the construction yields $\Ti\mu_z = 0$.
}]

\MS

%
%\comment{DTODO: decide whether/where to put this and subsequent remark
%
%As in~\cite{Cast}, it follows that this construction can be adapted to situations in 
%which the domain $U_I$ is stratified smooth rather than smooth, provided that the atlas is constructed so that we retain sufficient differentiability control in the normal directions to the submanifolds  $\im \phi_{IJ}$ in the core.
%}
%

\NI {\bf Zero set condition:}
For the extended perturbation constructed above, we have $\bigl\|\Ti \nu_J\bigr\| \leq  \|\mu_J\| \le \sup_{I\subsetneq J} \sup_{x\in V^k_I} \|\nu_I(x)\|< \si$ by induction hypothesis e). We first consider the part of  the perturbed zero set near the core, and then look at the \lq\lq new part".  By \eqref{value}, 
the zero set near the core
$(s_J + \Ti\nu_J)^{-1}(0)\cap  B^J_{\eta_{k+\frac 12}}\bigl(N^{k+\frac34}_{JI}\bigr)$ 
consists of points with $s_J(x) = -\Ti\nu_J(x)\in \Hat\phi_{IJ}(E_I)$,
so must lie within $s_J^{-1}\bigl(\Hat\phi_{IJ}(E_I)\bigr) = \phi_{IJ}(U_{IJ})$.  
Hence \eqref{eq:useful} implies for all $I\subsetneq J$  the inclusion
\begin{equation} \label{eq:zeroset} 
(s_J + \Ti\nu_J)^{-1}(0)\;\cap\;  B^J_{\eta_{k+\frac 12}}\bigl(N^{k+\frac34}_{JI}\bigr)  
\;\subset\; N^{k+\frac12}_{JI} .
\end{equation}
Thus the inductive hypothesis d) together with the compatibility condition  $\Ti\nu_J =\mu_J$ on $N^{k+\frac12}_{JI}\subset \phi_{IJ}(V^k_I)$ from \eqref{tinu}, with $\mu_J$ given by \eqref{eq:nuJ'}, imply that
$s_J + \Ti\nu_J\ne 0$ on $N^{k+\frac12}_{JI}\less \pi_\Kk^{-1}(\pi_\Kk(\Cc))$. Therefore we have
$$
(s_J + \Ti\nu_J)^{-1}(0)\;\cap\;  B^J_{\eta_{k+\frac 12}}\bigl(N^{k+\frac34}_{JI}\bigr)  
\subset \pi_\Kk^{-1}(\pi_\Kk(\Cc)).
$$
Next, by Definition~\ref{def:admiss} we have
$$
\si < \si(\Vv,\Cc,\|\cdot\|,\de) \le \| s_J(x) \|  \qquad\forall x\in \ov{V^{k+1}_J} \;\less\; \Bigl( \Ti C_J \cup {\textstyle \bigcup_{I\subsetneq J}} B^J_{\eta_{k+\frac 12}}\bigl(N^{k+\frac34}_{JI}\bigr) \Bigr) .
$$
Thus if $x$ is in the complement in $\ov{V^{k+1}_J}$ of the neighbourhoods $B^J_{\eta_{k+\frac 12}}\bigl(N^{k+\frac34}_{JI}\bigr)$ which cover the core, then either $x\in \Ti C_J$ or $\|s_J(x)\|\geq \si_{J,\eta_{k+1}} > \|\Ti\nu_J(x)\|$. 
In particular, we obtain the inclusion
\begin{equation}\label{eq:include}
\bigl(s_J|_{\ov{V^{k+1}_J}} + \Ti\nu_J\bigr)^{-1}(0) \;\less\; {\textstyle\bigcup_{I\subsetneq J} } B^J_{\eta_{k+\frac 12}}\bigl(N^{k+\frac34}_{JI}\bigr)   \;\subset\; \Ti C_J.
\end{equation}
From this we can deduce a slightly stronger version of a) at level $k+1$, namely 
$$
(s_J|_{\ov{V^{k+1}_J}}+\Ti\nu_J)^{-1}(0)  \; \subset\; \pi_\Kk^{-1}(\pi_\Kk(\Cc)) \qquad \forall \; |J|\le k+1 .
$$
Indeed, the zero set of $(s+\Ti\nu_J)|_{\ov{V^{k+1}_J}}$ consists of an ``old part'' given by \eqref{eq:zeroset}, which lies in the enlarged core $N^{k+\frac 12}_J$, where by the above arguments we have $s_J + \Ti\nu_J\ne 0$ on $N^{k+\frac12}_{JI}\less \pi_\Kk^{-1}(\pi_\Kk(\Cc))$. The ``new part'' given by \eqref{eq:include} is in fact contained in the open part $\Ti C_J\subset U_J$ of $\pi_\Kk^{-1}(\pi_\Kk(\Cc))$.
[\footnote{In \cite{MW:iso}, the same arguments apply by viewing $\pi_\Kk$ not as functor but as continuous map $\pi_\Kk:\bigsqcup_{I\in\Ii_\Kk} U_I\to |\Kk|$ with the crucial identity \eqref{eq:VCC} for the lift $V_J\cap \pi_\Kk^{-1}(\pi_\Kk(\Cc))$ for any nested reductions $\Cc\subset\Vv\sqsubset \bigsqcup_{I\in\Ii_\Kk} U_I$ that are lifted from reductions of $|\uKk|$.
}]

\MS
\NI {\bf Transversality:}
Since the perturbation $\Ti\nu_J$ was constructed to be strongly admissible and hence admissible, the induction hypothesis b) together with Lemma~\ref{le:transv} and \eqref{tinu} imply that the transversality condition is already satisfied on the enlarged core, $(s_J + \Ti\nu_J)|_{N^{k+\frac12}_J} \pitchfork 0$. 
[\footnote{In \cite{MW:iso}, Lemma~\ref{le:transv} also applies to local inverses of $\rho_{IJ}$, as explained in \cite[Remark~3.3.2]{MW:iso}.
While the previous constructions preserve equivariance of the perturbations and make the extension $\Ti\nu_J$ invariant under $\Ga_{J\less I}$, the following additional perturbation construction strictly works on $V^k_J$ without regard to the $\Ga_J$-action and thus generally will not be equivariant.}
]
%\comment{DTODO:  please check the equivariance mumblings in the footnote}
In addition, \eqref{eq:zeroset} also implies that the perturbed section $s_J+\Ti\nu_J$ has no zeros on
$B^J_{\eta_{k+\frac 12}}\bigl(N^{k+\frac34}_{JI}\bigr) \less N^{k+\frac 12}_{JI}$,
so that we have transversality
$$
(s_J + \Ti\nu_J)|_{B^J_{\eta_{k+\frac 12}}(N^{k+\frac34}_J)} \; \pitchfork \; 0 
$$
on a neighbourhood $B:= B^J_{\eta_{k+\frac 12}}(N^{k+\frac34}_J) = \bigcup_{I\subsetneq J} B^J_{\eta_{k+\frac 12}}(N^{k+\frac34}_{JI})$ of the core $N:=N_J^{k+1}=\bigcup_{I\subsetneq J} N^{k+1}_{JI}$, on which compatibility a) requires $\nu_J|_N=\Ti\nu_J|_N$.
In fact, $B$ also precompactly contains the neighbourhood $B':= B^J_{\eta_{k+1}}(N^{k+1}_J)$ of $N$, so that strong admissibility c) can be satisfied by requiring $\nu_J|_{B'}=\Ti\nu_J|_{B'}$.

To sum up, the smooth map $\Ti\nu_J : V^{k+1}_J \to E_J$ fully satisfies the compatibility a), strong admissibility c), and strengthened zero set condition d). Moreover, $s_J+\Ti\nu_J$ extends to a smooth map on the compact closure $\ov{V^{k+1}_J}\subset U_J$, where it satisfies transversality $(s_J+\Ti\nu_J)|_B\pitchfork 0$ on the open set $B \subset \ov{V^{k+1}_J}$ and the zero set condition from \eqref{eq:include}, 
$$
(s_J+\Ti\nu_J)^{-1}(0) \cap (\ov{V^{k+1}_J}\less B)\; \subset\; O: = \ov{V^{k+1}_J} \cap \Ti C_J .
$$ 
The latter can be phrased as $\| s_J+\Ti\nu_J \| > 0$ on $( \ov{V^{k+1}_J}\less  B ) \less  O$, which is compact since $O$ is relatively open in $\ov{V^{k+1}_J}$.
Since $z \mapsto \| s_J(z) +\Ti\nu_J(z) \|$ is continuous,  it remains nonvanishing on $W\less O$ for some relatively open neighbourhood $W\subset \ov{V^{k+1}_J}$ of $\ov{V^{k+1}_J}\less B$. This extends the zero set condition to $(s_J+\Ti\nu_J)^{-1}(0) \cap W \subset O$.
We can moreover choose $W$ disjoint from the neighbourhood of the core $B' \sqsubset B$.
Now  we wish to find
a smooth perturbation $\nu_\pitchfork:\ov{V^{k+1}_J} \to E_J$ 
supported in $W$ that satisfies the following:
\begin{enumerate}
\item[(T:i)] it provides transversality $(s_J+\Ti\nu_J+\nu_\pitchfork)|_W \pitchfork 0$;
\item[(T:ii)] the perturbed zero set satisfies the inclusion  $(s_J+\Ti\nu_J+\nu_\pitchfork)^{-1}(0) \cap W \subset O$;
\item[(T:iii)] the perturbation is small:  $\|\nu_\pitchfork\|< \si - \|\Ti\nu_J\|$.
\end{enumerate}
To see that this exists, note that for $\nu_\pitchfork=0$ transversality holds outside the compact subset $\ov{V^{k+1}_J}\less B$ of $W$.  Hence by the Transversality Extension theorem in \cite[Chapter~2.3]{GuillP} we can 
fix a nested open precompact subset $\ov{V^{k+1}_J}\less B \subset P \sqsubset W$ and achieve transversality everywhere on $W$ by adding an arbitrarily small perturbation supported in $P$. This immediately provides (T:i).
Moreover, since $\|s_J +\Ti\nu_J\|$ has a positive maximum on the compact set $P\less O$, we can choose $\nu_\pitchfork$ sufficiently small to satisfy (T:ii) and (T:iii).
Setting 
$$
\nu_J:=\Ti\nu_J + \nu_\pitchfork \,:\; V^{k+1}_J \to E_J
$$
then finishes the construction since the choice of $\nu_\pitchfork$ ensures the zero set inclusion a) and transversality b) on $W$; the previous constructions for $\nu_J|_{V^{k+1}_J\less W}=\Ti\nu_J|_{V^{k+1}_J\less W}$ ensure a), b), and d) on $V^{k+1}_J\less W \supset B'$, and compatibility c) on the core $N \subset B'\subset V^{k+1}_J\less W$; and we achieve the smallness condition $\max_{I\subset J} \|\nu_I\| < \si$ required by e) by the inductive hypothesis together with
the triangle inequality:
$$
\|\nu_J\| =  \|\Ti\nu_J+\nu_\pitchfork\| \;<\;  \|\Ti\nu_J\| +  \si - \|\Ti\nu_J\| \; = \; \si.
$$
This completes the iterative step, and hence completes construction of the required $(\Vv,\Cc,\|\cdot\|,\de,\si)$-adapted section.  The last claim follows from Proposition~\ref{prop:zeroS0}.
\end{proof}

In order to prove uniqueness up to cobordism of the VMC, we moreover need to construct transverse cobordism perturbations with prescribed boundary values as in Definition~\ref{def:csect}.
We will perform this construction by an iteration as in Proposition~\ref{prop:ext}, with adjusted domains $V^k_J$ obtained by replacing $\de$ with $\frac 12 \de$. 
This is necessary since as before the construction of $\nu_J$ will proceed by extending the given perturbations from previous steps, $\mu_J$, and now also the given boundary values $\nu^\al_J$, and then restricting to a precompact subset.
However, the $(\Vv,\Cc,\|\cdot\|,\de,\si)$-adapted boundary values $\nu^\al_J$ on $\p^\al V_J$ only extend to admissible, precompact, transverse perturbations in a collar of $V^{|J|}_J$,
while the above construction of $\nu_J$ on $V^{|J|}_J$ by precompact restriction 
requires $\nu^\al_J$ to be defined on a set that precompactly contains a collar of $V^{|J|}_J$.
To deal with this we construct $\nu_J$ by restriction to $V^{|J|+1}_J\sqsubset V^{|J|}_J$, which by \eqref{eq:VIk} is the analog of $V^{|J|}_J$ when $\de$ is replaced by $\frac 12 \de $.
This means that, firstly, we have to adjust the smallness condition for the iterative construction of perturbations by introducing a variation of the constant $\si(\Vv,\Cc,\|\cdot\|,\de)$ of Definition~\ref{a-e}. 
Secondly, we need a further smallness condition on adapted perturbations if we wish to extend these to a Kuranishi cobordism. Fortunately, the latter construction will only be used on product Kuranishi cobordisms, which leads to the following definitions.

\begin{defn}  \label{a-e rel}
\begin{enumerate}
\item
Let $(\Kk,d)$ be a metric tame Kuranishi cobordism with nested cobordism reduction $\Cc\sqsubset\Vv$, and let $0<\de<\min\{\eps,\de_\Vv\}$, where $\eps$ is the smallest of the 
collar widths of $(\Kk,d)$ and the reductions $\Cc,\Vv$. Then we set
\begin{align*}
& \si' (\Vv,\Cc,\|\cdot\|,\de) \,:=\; \min_{J\in\Ii_\Kk} \inf \Bigl\{ \; \bigl\| s_J(x) \bigr\| \;\Big| \;
x\in \ov{V^{|J|+1}_J} \;\less\; \Bigl( \Ti C_J \cup {\textstyle \bigcup_{I\subsetneq J}} B^J_{\eta_{|J|+\frac 12}}\bigl(N^{|J|+\frac34}_{JI}\bigr) \Bigr) \Bigr\}  , \\
&\si_{\rm rel}(\Vv,\Cc,\|\cdot\|,\de) \,:=\; \\
 & \qquad 
 \qquad \qquad 
 \min\bigl\{ \si(\p^0\Vv,\p^0\Cc, \p^0\|\cdot\|, \de),
 \,\si(\p^1\Vv,\p^1\Cc,  \p^1\|\cdot\|,\de), \,\si'(\Vv,\Cc,\|\cdot\|,\de) \bigr\} .
\end{align*}
\item
Given a metric Kuranishi atlas $(\Kk,d)$, we say that a perturbation $\nu$ is {\bf strongly $(\Vv,\Cc)$-adapted} 
if it is a $(\Vv,\Cc,\|\cdot\|,\de,\si)$-adapted perturbation $\nu$ of $\s_\Kk|_\Vv$ for some choice of  
additive norms $\|\cdot\|$ and 
constants $0<\de<\de_\Vv$ and 
\begin{align*}
&\qquad 0\;<\;\si\;\le\;\si_{\rm rel}([0,1]\times \Vv,[0,1]\times  \Cc,\|\cdot\|,\de) \;=\;\\
& \qquad \qquad \qquad \qquad \qquad  \min\bigl\{
\si(\Vv,\Cc,\|\cdot\|,\de) , \si'([0,1]\times \Vv, [0,1]\times \Cc,\|\cdot\|,\de) \bigr\} ,
\end{align*}
where we use the product metric on $[0,1]\times |\Kk|$.
Further, we call it  {\bf strongly adapted}  if it is  {\bf strongly $(\Vv,\Cc)$-adapted} 
for some choice of nested reduction $\Cc\sqsubset\Vv$.
\end{enumerate}
\end{defn}

\begin{rmk}\rm   (i)
Note that any $(\Vv,\Cc,\|\cdot\|,\de,\si)$-adapted perturbation for fixed $\Vv,\Cc,\|\cdot\|,\de$ and sufficiently small $\si>0$ is in fact strongly adapted.
In fact, recalling the definition of $\si(\Vv,\Cc,\|\cdot\|,\de)$, and the product structure of all sets and maps involved in the definition of $\si'([0,1]\times \Vv,[0,1]\times  \Cc,\|\cdot\|,\de)$, we may rewrite the condition on $\si>0$ in the definition of strong adaptivity as
$$
 \si \,<\;   \bigl\| s_J(x) \bigr\| \qquad\forall\; 
 x\in \ov{V^{k}_J} \;\less\; \Bigl( \Ti C_J \cup {\textstyle \bigcup_{I\subsetneq J}} B^J_{\eta_{k-\frac 12}}\bigl(N^{k-\frac14}_{JI}\bigr) \Bigr)  ,\;
 J\in\Ii_\Kk, \; k\in\{|J|,|J|+1\} .
$$
This strengthened condition on $\si$  is needed in order to have an efficient way of extending a perturbation that is prescribed on the boundary of a cobordism to its interior; cf.\ the shift in index from $V_J^k$ to $V_J^{k+1}$ in the version of conditions (a-e) used in the proof of Proposition~\ref{prop:ext2} below.\MS

\NI (ii)  The notion of strong adaptivity is a crucial ingredient of the uniqueness statements in \S\ref{s:VMC}.
  Indeed, the first two steps  in the proof of Theorems~\ref{thm:VMC1} are based on
Proposition~\ref{prop:ext2} below, and amount to a proof of the following claim:

\NI {\it 
Let  $(\Kk,d)$ be a metric tame Kuranishi atlas with nested cobordism reduction $\Cc\sqsubset\Vv$.
Then any two strongly $(\Vv,\Cc)$-adapted perturbations $\nu^\al$ of $\s_{\Kk}|_{\Vv}$ for $\al=0,1$   are concordant in the sense that
they are the boundary restrictions of an admissible, precompact, transverse,
 cobordism perturbation  $\nu$ of $\s_{[0,1]\times \Kk}|_{[0,1]\times \Vv}$.}
\smallskip

\NI However, because we do not formally discuss the notion of cobordism for perturbations (and in particular prove that it is transitive), we prefer to prove this for their zero sets rather than for the perturbations themselves.
$\hfill\er$
\end{rmk}

\begin{prop}\label{prop:ext2}
Let $(\Kk,d)$ be a metric tame Kuranishi cobordism with nested cobordism reduction $\Cc\sqsubset\Vv$ and additive norms $\|\cdot\|$, and 
let $0<\de<\min\{\eps,\de_\Vv\}$, where $\eps$ is the collar width of $(\Kk,d)$
and the reductions $\Cc,\Vv$. 
Then we have $\si_{\rm rel}(\Vv,\Cc,\|\cdot\|,\de)>0$ and the following holds.
\begin{enumerate}
\item
Given any 
$0<\si\le \si_{\rm rel}(\Vv,\Cc,\|\cdot\|,\de)$, there exists an admissible, precompact, transverse cobordism perturbation $\nu$ of $\s_\Kk|_\Vv$ with $\pi_\Kk\bigl((\s_\Kk|_\Vv+\nu)^{-1}(0)\bigr)\subset \pi_\Kk(\Cc)$, whose restrictions $\nu|_{\p^\al\Vv}$  for $\al=0,1$ are $(\p^\al\Vv,\p^\al\Cc,\p^\al\|\cdot\|,\de,\si)$-adapted perturbations of $\s_{\p^\al\Kk}|_{\p^\al\Vv}$. 
\item
Given perturbations $\nu^\al$ of $\s_{\p^\al\Kk}|_{\p^\al\Vv}$ for $\al=0,1$ that are $(\p^\al\Vv,\p^\al\Cc,{\p^\al\|\cdot\|},\de,\si)$-adapted with $\si\le \si_{\rm rel}(\Vv,\Cc,\|\cdot\|,\de)$, the perturbation $\nu$ of $\s_\Kk|_\Vv$ in (i) can be constructed to have boundary values $\nu|_{\p^\al\Vv}=\nu^\al$ for $\al=0,1$. 
\item
In the case of a product cobordism $[0,1]\times \Kk$
with product metric and product reductions $[0,1]\times \Cc \sqsubset [0,1]\times \Vv$
both (i) and (ii) hold without requiring $\de$ to be bounded in terms of the collar width.
\end{enumerate}
\end{prop}

\begin{proof}[Proof of Proposition~\ref{prop:ext2}]
We have $\si_{\rm rel}(\Vv,\Cc,\|\cdot\|,\de)>0$ becuase Lemma~\ref{le:admin}~(i) implies 
$\si(\p^\al\Vv,\p^\al\Cc,{\p^\al\|\cdot\|},\de)>0$,
and $\si'>0$ by the arguments of Lemma~\ref{le:admin}~(i) applied to the shifted domains.

Next, we reduce (i) for given $0<\si \le \si_{\rm rel}(\Vv,\Cc,\|\cdot\|,\de)$ to (ii). For that purpose recall that 
$\de\le \de_{\p^\al\Vv}$ by Lemma~\ref{le:admin}~(iii) and $\si \le \si(\p^\al\Vv,\p^\al\Cc,\p^\al\|\cdot\|,\de)$ by definition of $\si_{\rm rel}(\Vv,\Cc,\|\cdot\|,\de)$.
Hence Proposition~\ref{prop:ext} provides $(\p^\al\Vv,\p^\al\Cc,\p^\al\|\cdot\|,\de,\si)$-adapted perturbations $\nu^\al$ of $\s_{\p^\al\Kk}|_{\p^\al\Vv}$ for $\al=0,1$. 
Now (ii) provides a cobordism perturbation $\nu$ with the given restrictions $\nu|_{\p^\al\Vv}=\nu^\al$,
which are $(\p^\al\Vv,\p^\al\Cc,\p^\al\|\cdot\|,\de,\si)$-adapted by construction.
So (i) follows from~(ii).

To prove (ii) recall that, by assumption, the given perturbations $\nu^\al$ of $\s_{\p^\al\Kk}|_{\p^\al\Vv}$ for $\al=0,1$ extend to $\nu^\al_I : (\p^\al V_I)^{|I|} \to E_I$ for all $I\in\Ii_{\p^\al\Kk}$ which satisfy conditions a)-e) of Definition~\ref{a-e} with 
the given constant $\si$.
Here by Lemma~\ref{le:admin}~(iv)
the domains of $\nu^\al_I$ are $(\p^\al V_I)^{|I|} = \p^\al V_I^{|I|}$, and these are the boundaries of the reductions $V_I^k$ which have collars 
$$
V_I^k \cap \io^\al_I\bigl(A^\al_{\eps-2^{-k}\de} \times  \p^\al U_I\bigr) 
 = \io^\al_I\bigl( A^\al_{\eps-2^{-k}\de}\times \p^\al V_I^k \bigr)  ,
$$
where the requirement $2^{-k}\de<\eps$ of Lemma~\ref{le:admin} for $k>0$ is ensured by the assumption $\de<\eps$.
In the case of a product cobordism with product reduction this holds for any $\de>0$ with ${\eps-2^{-k}\de}$ replaced by $\eps:=1$.
The same collar requirement holds for $C_I \sqsubset V_I$, and hence for any set such as $N^k_{JI}$ or $\Ti C_I$ constructed from these.
Now $\de<\eps$ also ensures $2^{-k}\eps \le \eps - 2^{-k}\de$ for $k\geq 1$, so that we may denote the $2^{-k}\eps$-collar of $V_I^k$ by
$$
N^k_{I,\al}  \,:=\; \io^\al_I\bigl(A^\al_{2^{-k}\eps} \times \p^\al V_I^k   \bigr) \;\subset\; V^k_I 
$$
and note the precompact inclusion $N^{k'}_{I,\al} \sqsubset N^k_{I,\al}$ for $k'>k$.

We will now construct the required cobordism perturbation $\nu$ by an iteration as in Proposition~\ref{prop:ext} with adjusted domains obtained by replacing $\de$ with $\frac 12 \de$. This is necessary since the given boundary value $\nu^\al_J$ by assumption only extends to a map $\nu^\al_J : V^{|J|}_J \to E_J$, but as before the construction of $\nu_J$ will proceed by restriction to a precompact subset of the domain of an extension $\Ti\nu_J$, where this agrees both with the push forward of previously defined  $(\nu_I)_{I\subsetneq J}$ and with the given boundary perturbations $\nu^\al_J$ in collar neighbourhoods. 
We achieve this by restriction to $V^{|J|+1}_J\sqsubset V^{|J|}_J$.
That is, in the $k$-th step we construct $\nu_J : V^{k+1}_J \to E_J$ for each $|J| = k$ that, together with the $\nu_I|_{V^{k+1}_I}$ for $|I|< k$ obtained by restriction from earlier steps, satisfies the following.

\begin{itemize}
\item[a)]
The perturbation is compatible with coordinate changes and collars, that is
$$
\quad
\nu_J |_{N^{k+1}_{JI}}  \;=\; \Hat\phi_{IJ} \circ \nu_I \circ \phi_{IJ}^{-1} |_{N^{k+1}_{JI}} 
\qquad \text{on}\quad
N^{k+1}_{JI} = V^{k+1}_J \cap \phi_{IJ}(V^{k+1}_I\cap U_{IJ})
$$
for all $I\subsetneq J$, and for each $\al\in\{0,1\}$ with $J\in\Ii_{\p^\al\Kk}$ we have 
$$
\nu_J |_{N^{k+1}_{J,\al}} \;=\; (\io^\al_J)^*\nu^\al_J
\qquad\text{on}\quad N^{k+1}_{J,\al}  = \io^\al_J\bigl(A^\al_{2^{-k-1} \eps}\times \p^\al V_J^{k+1} \bigr),
$$
where we abuse notation by defining $(\io^\al_J)^*\nu^\al_J : \io^\al_J(t,x) \mapsto \nu^\al_J(x)$. $\phantom{\bigg(}$
\item[b)]
The perturbed section is transverse, that is $(s_J|_{{V^{k+1}_J}} + \nu_J) \pitchfork 0$.
\item[c)]
The perturbation is {\bf strongly admissible} with radius $\eta_{k+1}= 2^{-k-1}(1-2^{-\frac 14})$,
$$
\qquad
\nu_J( B^J_{\eta_{k+1}}(N^{k+1}_{JI})\bigr) \;\subset\; \Hat\phi_{IJ}(E_I)
\qquad\forall \;I\subsetneq J .
$$
\item[d)]  
The perturbed zero set is contained in $\pi_\Kk^{-1}\bigl(\pi_\Kk(\Cc)\bigr)$; more precisely
$$
(s_J |_{{V^{k+1}_J}}+ \nu_J)^{-1}(0) \;\subset\; {V^{k+1}_J} \cap \pi_\Kk^{-1}\bigl(\pi_\Kk(\Cc)\bigr) .
$$
\item[e)]
The perturbation is small, that is $\sup_{x\in {V^{k+1}_J}} \| \nu_J (x) \| <  \si$. 
\end{itemize}
The final perturbation $\nu=(\nu_I|_{V_I})_{I\in\Ii_\Kk}$ of $\s_\Kk|_\Vv$ then has product form on collars of width $2^{-M_\Kk}\eps$ and thus is a cobordism perturbation, whose boundary restrictions are the given $\nu^\al$ by construction.
Moreover, $\nu$ will be admissible by c), transverse by b), and precompact by d) with $\pi_\Kk( (\s_\Kk|_\Vv+\nu)^{-1}(0))\subset\pi_\Kk(\Cc)$. Compactness of $|(\s_\Kk|_\Vv+\nu)^{-1}(0)|$ then follows from Lemma~\ref{le:czeroS0}.

For $k=0$, there are no indices $|J|=0$ to be considered.
Now suppose that $\bigl(\nu_I : V^{|I|+1}_I\to E_I\bigr)_{I\in\Ii_\Kk, |I|< k}$  are constructed such that a)-e) hold. Then for the iteration step it suffices as before to construct $\nu_J$ for a fixed $J\in\Ii_\Kk$ with $|J|=k$. In the following three construction steps we then unify the cases of $J\in\Ii_{\p^\al\Kk}$ for none, one, or both indices $\al$ by interpreting the collars $N^k_{J,\al}$ as empty sets unless $J\in\Ii_{\p^\al\Kk}$.

\MS\NI
{\bf Construction of extension for fixed $|J|=k$:}  
For each $k\geq 1$ we will construct an extension of a restriction of 
$$
\quad\mu_J : N_J^k \cup N^k_{J,0} \cup N^k_{J,1}   \;\longrightarrow\; E_J , \qquad
\mu_J|_{N^{k}_{JI}} := \Hat\phi_{IJ}\circ \nu_I\circ\phi_{IJ}^{-1}, \qquad
\mu_J|_{N^k_{J,\al}} :=  (\io^\al_J)^*\nu^\al_J ,
$$
where $N^k_{J,\al} = \io^\al_J\bigl( A^\al_{2^{-k}\eps} \times \p^\al V_J^k \bigr)$ is a collar of $V^k_J$. [
see [$^{\text{\ref{35}}}$]
]
More precisely, we construct a smooth map $\Ti\nu_J : V^k_J \to E_J$ that satisfies
\begin{equation}\label{ctinu}
\Ti\nu_J|_{N_{k+\frac 12}} = \mu_J|_{N_{k+\frac 12}} 
\qquad\text{on}\;\; N_{k+\frac 12} :=  N_J^{k+\frac 12} \cup N^{k+\frac 12}_{J,0} \cup N^{k+\frac 12}_{J,1} ,
\end{equation}
the bound $\|\Ti\nu_J \| \leq \|\mu_J\| < \si$, and the strong admissibility condition 
\begin{equation}\label{cvalue} 
\Ti\nu_J \bigl( B^J_{\eta_{k+\frac 12}}\bigl(N^{k+\frac 12}_{JI}\bigr) \bigr) \;\subset\; \Hat\phi_{IJ}(E_I) 
\qquad \forall\; I\subsetneq J .
\end{equation}

We proceed as in Proposition~\ref{prop:ext} for fixed $j\in J$ by iteratively constructing smooth maps $\Ti\mu^j_\ell: W_\ell \to \Hat\phi_{jJ}(E_j)$ for $\ell=0,\ldots,k-1$ on the adjusted open sets 
\begin{equation}\label{eq:cW}
W_\ell \,:=\; N^{k_\ell}_{J,0} \;\cup\; N^{k_\ell}_{J,1} \;\cup\;
 {\textstyle \bigcup _{I\subsetneq J,|I|\le \ell}}\, B^J_{r_\ell}(N^{k+\frac 12}_{JI}) 
\end{equation}
with $r_\ell:= \eta_k - \frac {\ell+1} {k} ( \eta_k-\eta_{k+\frac 13})$ and $k_\ell:= k + \frac {\ell+1} {3k}$, that satisfy the conditions
\begin{enumerate}
\item[(E:i)]
$\Ti\mu^j_\ell |_{N^{k+\frac 12}_{JI}} = \mu_J^j|_{N^{k+\frac 12}_{JI}}$ for all $I\subsetneq J$ with $|I|\leq \ell$ and $j\in I$;
\vspace{.07in}
\item[(E:ii)]
$\Ti\mu^j_\ell |_{B^J_{r_\ell}(N^{k+\frac 12}_{JI})} = 0$
for all $I\subsetneq J$ with $|I|\leq \ell$ and $j\notin I$;
\vspace{.07in}
\item[(E:iii)] 
$\bigl\|\Ti\mu^j_\ell \bigr\| \leq \|\mu^j_J\|$;
\vspace{.07in}
\item[(E:iv)]
$\Ti\mu^j_\ell = (\io^\al_J)^*\nu^{\al,j}_J$ on  $N^{k_\ell}_{J,\al} = \io^\al_J\bigl( A^\al_{2^{-k_\ell}\eps}\times \p^\al V_J^{k_\ell}  \bigr)$ 
for $\al\in\{0,1\}$ with $J\in\Ii_{\p^\al\Kk}$.
\end{enumerate}
These requirements make sense because  $\eta_{k+\frac 12}< r_\ell < \eta_k$ and $B^J_{\eta_{k}}(N^{k+\frac 12}_J) \subset V^k_J$ by \eqref{eq:fantastic}, so that the domain in (E:ii) is included in $V^k_J$ and is larger than that in \eqref{cvalue}.
After this iteration, we then obtain the extension $\Ti\nu_J:= \beta {\textstyle\sum_{j\in J}} \, \Ti\mu^j_{k-1}$ by multiplication with a smooth cutoff function $\beta:V^k_J \to [0,1]$ with $\beta|_{N_{k+\frac 12}}\equiv 1$ and $\supp\beta\subset N^{k+\frac 13}_{J,0} \cup N^{k+\frac 13}_{J,1} \cup B^J_{\eta_{k+\frac 13}}(N^{k+\frac 12}_J)$, where the latter contains the closure of $N_{k+\frac 12}=N^{k+\frac 12}_{J,0} \cup N^{k+\frac 12}_{J,1} \cup N^{k+\frac 12}_J$ in $V^k_J$, so that $\Ti\nu_J$ extends trivially to $V^k_J\less W_{k-1}$.

For the start of iteration at $\ell=0$,
the domain is $W_0= N^{k_0}_{J,0} \;\cup\; N^{k_0}_{J,1}$ with $k_0 = k + \frac 1{3k}$.
Conditions (E:i) and (E:ii) are vacuous since there are no index sets with $|I|\le 0$, and 
we can satisfy (E:iii) and (E:iv), by setting $\Ti\mu^j_0 (\iota^\al(t,x)) := \nu^{\al,j}_J(x)$.
Next, if the construction is given on $W_\ell$, then we cover $W_{\ell+1}$ by the open sets 
$B_L': = W_{\ell+1}\cap B^J_{r_{\ell-1}}(N^{k+\frac12}_{JL})$
for $L\subsetneq J$, $|L|=\ell+1$ and $C_{\ell+1}\subset W_\ell$
given below, and pick an open subset $C'\subset V^k_J$ such that
$$
C_{\ell+1} \,:=\; W_{\ell+1} \;\less\; {\textstyle \bigcup _{|L| = \ell+1}\, \ov{B^J_{r_{\ell}}(N^{k+\frac12}_{JL})}} \;\sqsubset\; C' \;\sqsubset\; W_\ell \;\less\; {\textstyle \bigcup _{|L| = \ell+1}\, \ov{B^J_{r_{\ell+1}}(N^{k+\frac12}_{JL})}} \;=:\, C_\ell .
$$
As before, this guarantees that $C'$ and $B^J_{r_{\ell+1}}(N^{k+\frac 12}_{JL})$ have disjoint closures for all $|L|=\ell +1$. Then we set $\Ti\mu_{\ell+1}|_{C_{\ell+1}} := \Ti \mu_\ell|_{C_{\ell+1}}$, which inherits properties (E:i)--(E:iv) from $\Ti\mu_\ell$ because $C_{\ell+1}$ is still disjoint from $B^J_{r_{\ell+1}}(N^{k+\frac 12}_{JL})$ for any $|I|=\ell+1$, and we have $N^{k_{\ell+1}}_{J,\al}\subset N^{k_\ell}_{J,\al}$.
So it remains to construct $\Ti\mu_{\ell+1}: B'_L \to \Hat\phi_{jJ}(E_j)$ for a fixed $L\subset J$, $|L|=\ell+1$ such that $\Ti\mu_{\ell+1}=\Ti \mu_\ell$ on $B'_L \cap C'$.

In case $j\notin L$ condition (E:iv) prescribes $\Ti\mu^j_\ell = (\io^\al_J)^*\nu^{\al,j}_J$ on the intersection
$$
B'_L \cap N^{k_{\ell+1}}_{J,\al}
\subset 
% note to self
%B^J_{r_{\ell-1}}(N^{k+\frac12}_{JL}
%\cap 
%\io^\al_J\bigl(\p^\al V_J^{k} \times N^\al_{2^{-k_{\ell+1}}\eps} \bigr)
%= 
\io^\al_J\bigl(A^\al_{2^{-k_{\ell+1}}\eps}  \times B^{J,\al}_{r_{\ell-1}}(\p^\al N^{k+\frac12}_{JL}) \bigr).
$$ 
Because $r_{\ell-1}<\eta_k$, strong admissibility for $\nu^\al_J$ on $B^{J,\al}_{\eta_k}(\p^\al N^k_{JL})$ implies that $(\io^\al_J)^*\nu^{\al,j}_J=0$ on this intersection.
Moreover, $B'_L\cap C'$ again is a subset of $\bigcup _{I\subsetneq L} B^J_{r_\ell}(N^{k+\frac 12}_{JI})$, where we have $\Ti\mu_\ell|_{B'_L\cap C'}=0$ by iteration hypothesis (E:ii) for each $I\subsetneq J$. 
Thus $\Ti\mu_{\ell+1}:=0$ satisfies all extension properties (E:i)--(E:iv) in this case.

In case $j\in L$ we may again patch together extensions by partitions of unity, so that it suffices to construct smooth maps $\Ti\mu_z: B^J_{r_z}(z)\to \Hat\Phi_{jJ}(E_j)$ on balls of positive radius $r_z>0$ around each fixed $z\in B'_L$, that satisfy
\begin{itemize}
\item[(i)]
$\Ti\mu_z = \mu_J^j$ on $B^J_{r_z}(z)\cap N^{k+\frac 12}_{JI}$ for all $I\subset L$ with $j\in I$ (including $I=L$);
\vspace{.07in}
\item[(ii)]
$\Ti\mu_z = 0$ on $B^J_{r_z}(z) \cap B^J_{r_{\ell+1}}(N^{k+\frac 12}_{JI})$ for all $I\subsetneq L$ with $j\notin I$;
\vspace{.07in}
\item[(iii)] 
$\bigl\| \Ti\mu_z\bigr\| \leq \|\mu_J^j\|$;
\vspace{.07in}
\item[(iv)] 
$\Ti\mu_z = (\io^\al_J)^*\nu^{\al,j}_J$ on  $B^J_{r_z}(z) \cap N^{k_{\ell+1}}_{J,\al}$ for $\al\in\{0,1\}$ with $J\in\Ii_{\p^\al\Kk}$;
\vspace{.07in}
\item[(v)] 
$\Ti\mu_z=\Ti\mu_{\ell}$ on $B^J_{r_z}(z)\cap B'_L\cap C'$.
\end{itemize}
\vspace{.07in}

For $z\in V^k_J \less \ov{N^{k_{\ell+1}}_{J,\al}}$, this is accomplished by the same constructions as in Proposition~\ref{prop:ext} by choosing $r_z>0$ such that $B^J_{r_z}(z)\cap N^{k_{\ell+1}}_{J,\al}=\emptyset$.
For $z\in \ov{N^{k_{\ell+1}}_{J,\al}}\subset N^{k_\ell}_{J,\al}$ we choose $r_z>0$ such that $B^J_{r_z}(z)\subset N^{k_{\ell}}_{J,\al}$. Then $\Ti\mu_z := \Ti\mu_\ell|_{B^J_{r_z}(z)}$ satisfies (v) by construction and (i)-(iv) by iteration hypothesis.

\MS

\NI {\bf Zero set condition:}
For the extended perturbation constructed above, we have $\bigl\|\Ti \nu_J\bigr\| \leq \max\{ \max_{I\subsetneq J} \|\nu_I\| , \|\nu^0_J\| , \|\nu^1_J\|  \} < \si $ by induction hypothesis e).
From \eqref{cvalue} and \eqref{eq:useful} we then obtain as in Proposition~\ref{prop:ext}
\begin{equation} \label{eq:czeroset} 
(s_J |_{V^k_J} + \Ti\nu_J)^{-1}(0)\;\cap\;  B^J_{\eta_{k+\frac 12}}\bigl(N^{k+\frac34}_{JI}\bigr)  
\;\subset\; N^{k+\frac12}_{JI} .
\end{equation}
Next, recall that we allowed only $\si>0$ such that
$$
\si \;\leq\; \inf \Bigl\{ \; \bigl\| s_J(x) \bigr\| \;\Big| \;
x\in \ov{V^{|J|+1}_J} \;\less\; \Bigl( \Ti C_J \cup {\textstyle \bigcup_{I\subsetneq J}} B^J_{\eta_{|J|+\frac 12}}\bigl(N^{|J|+\frac34}_{JI}\bigr) \Bigr) \Bigr\}  .
$$
Hence the same arguments as in the proof of Proposition~\ref{prop:ext} provide the inclusion
\begin{equation}\label{eq:cinclude}
\bigl(s_J|_{\ov{V^{k+1}_J}} + \Ti\nu_J\bigr)^{-1}(0) \;\less\; {\textstyle\bigcup_{I\subsetneq J} } B^J_{\eta_{k+\frac 12}}\bigl(N^{k+\frac34}_{JI}\bigr)   \;\subset\; \Ti C_J.
\end{equation}
Together with the induction hypothesis on $\Ti\nu_J =\mu_J=\Hat\phi_{IJ}\circ\nu_I\circ\phi_{IJ}^{-1}$ on $N^{k+\frac12}_{JI}$ this implies the zero set condition $(s_J|_{\ov{V^{k+1}_J}}+\Ti\nu_J)^{-1}(0) \subset\pi_\Kk^{-1}(\pi_\Kk(\Cc))$.

\MS
\NI {\bf Transversality:}
Admissibility together with induction hypothesis b) imply transversality $(s_J + \Ti\nu_J)|_{N^{k+\frac12}_J} \pitchfork 0$ on the enlarged core.
Since the 
perturbations $\nu^\al_J$ are transverse, this together with \eqref{eq:czeroset} implies that
transversality holds on the open set
$$
(s_J + \Ti\nu_J)|_{B} \; \pitchfork \; 0  , \qquad B:= B^J_{\eta_{k+\frac 12}}(N^{k+\frac34}_J) \cup N^{k+\frac 12}_{0,J} \cup N^{k+\frac 12}_{1,J} \;\subset\; V^k_J .
$$
Now $B$ precompactly contains the neighbourhood $B':= B^J_{\eta_{k+1}}(N^{k+1}_J)\cup N^{k+1}_{0,J} \cup N^{k+1}_{1,J} \subset V^k_J$ of the core and collar $N:= N^{k+1}_J \cup N^{k+1}_{0,J} \cup N^{k+1}_{1,J}$,
so that compatibility with the coordinate changes and collars in a) and strong admissibility in d) can be satisfied by requiring $\nu_J|_{B'}=\Ti\nu_J|_{B'}$.
In this abstract setting, we can finish the iterative step word by word as in Proposition~\ref{prop:ext}. This completes the construction of the required perturbation in case (ii).
Finally, we note that (iii) holds by Lemma~\ref{le:admin}~(ii).
This completes the proof.
\end{proof}

%
%\comment{DTODO - decide what to do with this
%
%In the stratified smooth case, we must take extra care when constructing the  extension $\Ti\mu_J$, which is done by choosing suitable local extensions $\Ti\mu_z$.    The only new point here is that, for consistency in the collar, we define $\Ti\mu_z$ by restriction from $\Ti\mu_\ell$ 
%when $z$ lies in the collar neighbourhood $\ov{N^{k_{\ell+1}}_{J,\al}}$.  
%Therefore if $\Ti\mu_\ell$ is appropriately smooth, its restriction will be as well.  
%Since the other points $z$ lie in the interior of the cobordism,  the difficulty
% is exactly as described in ... (remark in other TBD). 
%}
%

\section{From Kuranishi atlases to  the Virtual  Fundamental Class}\label{s:VMC}

In this section we first discuss orientations, and then finally put everything together to
construct the virtual moduli cycle (VMC) and virtual fundamental class (VFC) for a smooth 
oriented, tame Kuranishi atlas (as always with trivial isotropy) of dimension $d$ on a compact metrizable space $X$.
This will prove Theorem~B in the introduction.

%%%%%%%%%%%%%%%%%%%%%%%%%%%%%%%%%%%%%%%%%%%%%%%%%%%%%%
\subsection{Orientations} \label{ss:vorient}  \hspace{1mm}\\ \vspace{-3mm}
%%%%%%%%%%%%%%%%%%%%%%%%%%%%%%%%%%%%%%%%%%%%%%%%%%%%%%%%%%%%%%%%%%%%%%%%%%%%%%%%%%%%%%%%

This section develops the theory of orientations of Kuranishi atlases.
We use the method of determinant line bundles as in e.g.\ \cite[App.A.2]{MS},
but encountered compatibility issues of sign conventions in the literature, e.g.\ all editions of \cite{MS}.
We resolve these by using a different set of conventions most closely related to K-theory and thank Thomas Kragh for helpful discussions.
As shown in the recent work of  \cite{Z3}, these conventions are consistent with some important naturality properties, a fact which may prove useful in the future development of Kuranishi atlases.
 
While the relevant bundles and sections could just be described as tuples of bundles and sections over the domains of the Kuranishi charts, related by lifts of the coordinate changes, we take this opportunity to develop a general framework of vector bundles over Kuranishi atlases, which now no longer are assumed to be additive or tame.

\begin{defn} \label{def:bundle}
A {\bf vector bundle} $\La=\bigl(\La_I,\Ti\phi^\La_{IJ}\bigr)_{I,J\in\Ii_\Kk}$ {\bf over a weak Kuranishi atlas} $\Kk$ is a collection $(\La_I \to U_I)_{I\in \Ii_\Kk}$ of vector bundles together with lifts $\bigl(\Tilde \phi^\La_{IJ}: \La_I|_{U_{IJ}}\to \La_J\bigr)_{I\subsetneq J}$ of the coordinate changes $\phi_{IJ}$, that are linear isomorphisms on each fiber and satisfy the weak cocycle condition 
$\Tilde \phi^\La_{IK} =  \Tilde \phi^\La_{JK}\circ  \Tilde \phi^\La_{IJ}$ on $\phi^{-1}_{IJ}(U_{JK})\cap U_{IK}$ for all triples $I\subset J\subset K$.  

A {\bf section} of a bundle $\La$ over $\Kk$ is a collection of smooth sections $\si=\bigl( \si_I: U_I\to \La_I \bigr)_{I\in\Ii_\Kk}$ that are compatible with the bundle maps $\Ti\phi^\La_{IJ}$.
In particular, for a vector bundle $\La$ with section $\si$ there are commutative diagrams for each $I\subset J$,
\[
\xymatrix{
  \La_I|_{U_{IJ}}  \ar@{->}[d] \ar@{->}[r]^{\;\;\Tilde\phi^\La_{IJ}}   &  \La_J \ar@{->}[d]   \\
U_{IJ}\ar@{->}[r]^{\phi_{IJ}}  & U_J
}
\qquad\qquad\qquad
\xymatrix{
  \La_I|_{U_{IJ}}  \ar@{->}[r]^{\;\;\Tilde\phi^\La_{IJ}}    &  \La_J  \\
U_{IJ}  \ar@{->}[u]^{\si_I}   \ar@{->}[r]^{\phi_{IJ}}  & U_J  \ar@{->}[u]_{\si_J} .
}
\]
\end{defn}

The following notion of a product bundle will be the first example of a bundle over a Kuranishi cobordism. 

\begin{defn} \label{def:prodbun}
If $\La=\bigl(\La_I,\Ti\phi^\La_{IJ}\bigr)_{I,J\in\Ii_\Kk}$ is a bundle over $\Kk$ and $A\subset [0,1]$ is an interval, then the {\bf product bundle} $
A\times \La$ over $A\times \Kk$ is the tuple $\bigl(A\times \La_I, \id_A\times \Ti\phi^\La_{IJ}\bigr)_{I,J\in\Ii_\Kk}$. 
Here and in the following we denote by $A\times\La_I\to A\times U_I$ the pullback bundle of $\La_I\to U_I$ under the projection $A\times U_I\to U_I$.
\end{defn}

\begin{defn} \label{def:cbundle}
A {\bf vector bundle over a weak Kuranishi cobordism} $\Kk$ 
is a collection $\La=\bigl(\La_I,\Ti\phi^\La_{IJ}\bigr)_{I,J\in\Ii_\Kk}$ of vector bundles and bundle maps as in Definition~\ref{def:bundle}, together with a choice of isomorphism from its restriction to a
collar of the boundary to a product bundle.
More precisely, this requires for $\al=0,1$ the choice of a {\bf restricted vector bundle} $\La|_{\p^\al\Kk}= \bigl( \La^\al_I \to \partial^\al U_I, \Ti\phi^{\La,\al}_{IJ}\bigr)_{I,J \in \Ii_{\p^\al\Kk}}$ over $\p^\al\Kk$, and, for some $\eps>0$ less than the collar width of $\Kk$, a choice of lifts of the embeddings $\io^\al_I$ for $I\in\Ii_{\p^\al\Kk}$ to bundle isomorphisms 
$\ti\io^{\La,\al}_I : A^\al_\eps\times \La^\al_I \to \La_I|_{\im\io^\al_I}$ such that,  with $A: = A^\al_\eps$, the following diagrams commute
\[
\xymatrix{A\times  \La_I^\al \ar@{->}[d]   \ar@{->}[r]^{\ti\io^{\La,\al}_I}    &   \La_I|_{\im\io^\al_I} \ar@{->}[d]   \\
A\times \partial^\al U_I  \ar@{->}[r]^{\io^\al_I}   & \im\io^\al_I \subset U_I
}
\qquad\qquad
\xymatrix{ 
 A \times 
  \La_I^\al|_{A\times \p^\al U_{IJ}} 
  \ar@{->}[r]^{\ti\io^{\La,\al}_I} \ar@{->}[d]_{\id_A\times \Ti\phi^{\La,\al}_{IJ}}    & 
 \La_I |_{A\times \io^\al_I(\p^\al U_{IJ} )}
  \ar@{->}[d]^{\Ti\phi^\La_{IJ}}  \\
A\times   \La_J^\al \ar@{->}[r]^{\ti\io^\al_{J}}  &  \La_J|_{\im\io^\al_J}  
}
\]

A {\bf section} of a vector bundle $\La$ over a Kuranishi cobordism as above is a compatible collection $\bigl(\si_I:U_I\to \La_I\bigr)_{I\in\Ii_\Kk}$ of sections as in Definition~\ref{def:bundle} that in addition have product form in the collar. 
That is we require that for each $\al=0,1$ there is a {\bf restricted section} $\si|_{\p^\al\Kk}= ( \si^\al_I :\partial_\al U_I \to \La^\al_I)_{I\in\Ii_{\p^\al\Kk}}$ of $\La|_{\p^\al\Kk}$ such that for $\eps>0$ sufficiently small we have 
$(\ti\io^{\La,\al}_I)^*\si_I  = \id_{A^\al_\eps}\times \si^\al_I$.
 \end{defn}

In the above definition,  we implicitly work with an isomorphism $(\ti\io^{\La,\al}_I)_{I\in\Ii_{\p^\al\Kk}}$, that satisfies all but the product structure requirements of the following notion of isomorphisms on Kuranishi cobordisms.

\begin{defn} \label{def:buniso}
An {\bf isomorphism} $\Psi: \La\to \La'$ between vector bundles over $\Kk$ is a collection
$(\Psi_I: \La_I\to \La'_I)_{I\in \Ii_\Kk}$ of bundle isomorphisms covering the identity on $U_I$, that intertwine the transition maps, i.e.\ 
$\Ti\phi^{\La'}_{IJ}\circ\Psi_I|_{U_{IJ}} = \Psi_J \circ \Ti \phi^\La_{IJ}|_{U_{IJ}}$ for all $I\subset J$.

If $\Kk$ is a Kuranishi cobordism then we additionally require $\Psi$ to have product form in the collar. That is we require that for each $\al=0,1$ there is a restricted isomorphism $\Psi|_{\p^\al\Kk}= ( \Psi^\al_I :\La^\al_I \to \La'_I\,\!\!^\al)_{I\in\Ii_{\p^\al\Kk}}$ from $\La|_{\p^\al\Kk}$ to $\La'|_{\p^\al\Kk}$ such that for $\eps>0$ sufficiently small we have 
$\ti\io'_I\,\!\!^{\La,\al} \circ \bigl( \id_A\times \Psi^\al_I \bigr) = \Psi_I \circ \ti\io^{\La,\al}_I$ on $A^\al_\eps\times \partial^\al U_I$.
\end{defn}

\begin{remark}\rm
In the newly available language, Definition~\ref{def:cbundle} of a bundle on a Kuranishi cobordism requires isomorphisms (without product structure on the collar) for $\al=0,1$ from the product bundle 
$A^\al_\eps\times  \La|_{\p^\al\Kk}$ to the $\eps$-collar restriction 
$(\iota^\al_\eps)^*\La := \bigl((\iota^\al_I)^*\La_I , (\iota^\al_J)^* \circ \Ti\phi^\La_{IJ} \circ (\iota^\al_I)_* \bigr)_{I,J\in\Ii_{\p^\al\Kk}}$,  
given by the collection of pullback bundles and isomorphisms 
under the embeddings $\iota^\al_I : A^\al_\eps\times  \partial^\al U_I  \to U_I$.
$\hfill\er$
\end{remark}

Note that, although the compatibility conditions are the same, the canonical section 
$\s_\Kk = ( s_I : U_I\to E_I)_{I\in\Ii_\Kk}$ of a Kuranishi atlas does not form a section of a 
vector bundle since the obstruction spaces $E_I$ are in general not of the same dimension, 
hence no bundle isomorphisms $\Ti\phi^\La_{IJ}$ as above exist.
Nevertheless, we will see that, 
there is a natural bundle associated with the section $\s_\Kk$, namely its determinant line bundle, 
and that this line bundle is isomorphic to a bundle constructed by combining the determinant lines of the obstruction spaces $E_I$ and the domains $U_I$.
 
Here and in the following we will exclusively work with finite dimensional vector spaces.
First recall that the determinant line of a vector space $V$ is its maximal exterior power $\lm V := \wedge^{\dim V}\,V$, with $\wedge^0\,\{0\} :=\R$.
More generally, the {\bf determinant line of a linear map} $D:V\to W$ is defined to be 
$$
\det(D):= \lm\ker D \otimes \bigl( \lm \bigl( \qu{W}{\im D} \bigr) \bigr)^*.
$$ 
In order to construct isomorphisms between determinant lines, we will need to fix various conventions, in particular pertaining to the ordering of factors in their domains and targets.
We begin by noting that every isomorphism $F: Y \to Z$ between finite dimensional vector spaces induces an isomorphism
\begin{equation}\label{eq:laphi}
\La_F :\; \lm Y   \;\overset{\cong}{\longrightarrow}\; \lm Z , \qquad
y_1\wedge\ldots \wedge y_k \mapsto F(y_1)\wedge\ldots \wedge F(y_k) .
\end{equation}
For example, if $I\subsetneq J$ and  $x\in U_{IJ}$, it follows from the index condition in Definition~\ref{def:change} that the map 
for $x\in U_{IJ}$
\begin{equation}\label{eq:bunIJ}
\La_{IJ}(x): = \La_{\rd_x\phi_{IJ}} \otimes 
\bigl(\La_{[\Hat\phi_{IJ}]^{-1}}\bigr)^*
\, :\; \det(\rd_x s_I) \to \det(\rd_{\phi_{IJ}(x)} s_J)
\end{equation}
is an isomorphism, induced by the isomorphisms $\rd\phi_{IJ}:\ker\rd s_I\to\ker\rd s_J$ and
$[\Hat\phi_{IJ}] : \qu{E_I}{\im\rd s_I}\to\qu{E_J}{\im\rd s_J}$.
With this, we can define the determinant bundle $\det(\s_\Kk)$ of a Kuranishi atlas. A second, isomorphic, determinant line bundle $\det(\Kk)$ with fibers $\lm \rT_x U_I \otimes \bigl( \lm E_I \bigr)^*$ will be constructed in Proposition~\ref{prop:orient}.

\begin{defn} \label{def:det} 
The {\bf determinant line bundle} of a weak Kuranishi atlas (or cobordism) $\Kk$ is the vector 
bundle $\det(\s_\Kk)$ given by the line bundles 
$$
\det(\rd s_I):=\bigcup_{x\in U_I} \det(\rd_x s_I) \;\to\; U_I \qquad 
\text{for}\; I\in\Ii_\Kk, 
$$
and the isomorphisms $\La_{IJ}(x)$ in \eqref{eq:bunIJ} for $I\subsetneq J$ and $x\in U_{IJ}$.
\end{defn}

To show that  $\det(\s_\Kk)$ is well defined, in particular that $x\mapsto \La_{IJ}(x)$ is smooth, 
we introduce some further natural\footnote{
Here a ``natural" isomorphism is one that is functorial, i.e.\ it commutes with 
the action on both sides induced by a vector space isomorphism.}
isomorphisms and fix various ordering conventions.

\begin{itemlist}
\item
For any subspace $V'\subset V$ the {\bf splitting isomorphism}
\begin{equation}\label{eq:VW}
\lm V\cong \lm V'\otimes \lm\bigl( \qu{V}{V'}\bigr)
\end{equation}
is given by completing a basis $v_1,\ldots,v_k$ of $V'$ to a basis $v_1,\ldots,v_n$ of $V$ and
mapping $v_1\wedge \ldots \wedge v_n \mapsto (v_1\wedge \ldots \wedge v_k) \otimes ([v_{k+1}]\wedge \ldots\wedge [v_n])$.
\item
For each isomorphism $F:Y\overset{\cong}{\to} Z$ the {\bf contraction isomorphism} 
\begin{equation} \label{eq:quotable}
\mathfrak{c}_F \,:\; \lm Y  \otimes  \bigl( \lm Z \bigr)^* \;\overset{\cong}{\longrightarrow}\; \R , 
\end{equation}
is given by the map
$\bigl(y_1\wedge\ldots \wedge y_k\bigr) \otimes \eta \mapsto \eta\bigl(F(y_1)\wedge \ldots 
\wedge F(y_k)\bigr)$.
\item
For any space $V$ we use the {\bf duality isomorphism}
\begin{equation}\label{eq:dual} 
\lm V^* \;\overset{\cong}{\longrightarrow}\; (\lm V)^*, \qquad
 v_1^*\wedge\dots\wedge v_n^* 
\;\longmapsto\;
(v_1\wedge\dots\wedge v_n)^* ,
\end{equation}
which corresponds to the natural pairing
$$
 \lm V \otimes \lm V^*   \;\overset{\cong}{\longrightarrow}\;  \R , \qquad
 \bigl(v_1\wedge\dots\wedge v_n\bigr) \otimes \bigl(\eta_1\wedge\dots\wedge \eta_n\bigr)
 \;\mapsto\;
 \prod_{i=1}^n \eta_i(v_i) 
$$
via the general identification (which in the case of line bundles $A,B$ maps $\eta\neq0$ to a nonzero homomorphism, i.e.\ an isomorphism)
\begin{equation}\label{eq:homid}
\Hom(A\otimes B,\R) \;\overset{\cong}{\longrightarrow}\; \Hom(B, A^*)
,\qquad
H \;\longmapsto\; \bigl( \; b \mapsto H(\cdot \otimes b) \; \bigr) .
\end{equation}
\end{itemlist}

\MS\NI
Next, we combine the above isomorphisms to obtain a more elaborate 
contraction isomorphism.

\begin{lemma} \label{lem:get} 
Every linear map $F:V\to W$ together with an isomorphism $\phi:K\to \ker F$ induces an isomorphism
\begin{align}\label{Cfrak}
\mathfrak{C}^{\phi}_F \,:\; \lm V \otimes \bigl(\lm W \bigr)^* 
&\;\overset{\cong}{\longrightarrow}\;  \lm K \otimes \bigl(\lm \bigl( \qu{W}{F(V)}\bigr) \bigr)^*  
\end{align}
given by
\begin{align}
(v_1\wedge\dots v_n)\otimes(w_1\wedge\dots w_m)^* &\;\longmapsto\;
\bigl(\phi^{-1}(v_1)\wedge\dots \phi^{-1}(v_k)\bigr)\otimes \bigl( [w_1]\wedge\dots [w_{m-n+k}] \bigr)^* ,
\notag
\end{align}
where $v_1,\ldots,v_n$ is a basis for $V$ with ${\rm span}(v_1,\ldots,v_k)=\ker F$, and $w_1,\dots, w_m$ is a basis for $W$ whose last $n-k$ vectors are $w_{m-n+i}=F(v_i)$ for $i=k+1,\ldots,n$.

In particular, for every linear map $D:V\to W$ we may pick $\phi$ as the inclusion $K=\ker D\hookrightarrow V$ to obtain an isomorphism
$$
\mathfrak{C}_{D} \,:\;  \lm V \otimes \bigl(\lm W \bigr)^* \;\overset{\cong}{\longrightarrow}\;  \det(D) .
$$ 
\end{lemma}

\begin{proof}
We will construct $\mathfrak{C}^{\phi}_F$ by composition of several isomorphisms.
As a first step let $F(V)^\perp\subset W^*$ be the annihilator of $F(V)$ in $W^*$, then the splitting isomorphism \eqref{eq:VW} identifies $\lm W^*$ with $\lm ( F(V)^\perp )\otimes \lm \bigl(\qu{W^*}{F(V)^\perp}\bigr)$.
Next, we apply \eqref{eq:laphi} to the isomorphisms 
$F(V)^\perp \overset{\cong}{\to} \bigl(\qu{W}{F(V)}\bigr)^*$
and 
$\qu{W^*}{F(V)^\perp}\overset{\cong}{\to} F(V)^*$,
and apply the duality isomorphism \eqref{eq:dual} in all factors to obtain the isomorphism 
$$
S_W \,:\;
\bigl(\lm W\bigr)^* \;\overset{\cong}{\longrightarrow}\;  \bigl(\lm \bigl(\qu{W}{F(V)}\bigr)\bigr)^* \otimes \bigl( \lm F(V) \bigr)^*
$$
given by 
$(w_1\wedge \ldots  \wedge w_m)^* \mapsto ([w_1]\wedge \ldots\wedge [w_\ell])^* \otimes (w_{\ell+1}\wedge \ldots \wedge w_m)^*$
for any basis $w_1,\ldots,w_m$ of $W$ whose last elements $w_{\ell+i}$ for $i=1,\ldots,m-\ell=n-k$ span $F(V)$.
On the other hand, we apply the splitting isomorphism \eqref{eq:VW} for $\ker F\subset V$ and \eqref{eq:laphi} for $\phi^{-1}: \ker F\to K$ to obtain an isomorphism
$$
S_V \,:\;
\lm V \;\overset{\cong}{\longrightarrow}\;   \lm K \otimes \lm \bigl(\qu{V}{\phi(K)}\bigr)  
$$
given by 
$v_1\wedge \ldots \wedge v_n \mapsto (\phi^{-1}(v_1)\wedge \ldots\wedge \phi^{-1}(v_k)) \otimes ([v_{k+1}]\wedge \ldots \wedge [v_n])$
for any basis $v_1,\ldots,v_n$ of $V$ such that $v_1,\ldots,v_k$ spans $\ker F$.
Finally, note that $F$ descends to an isomorphism $[F] : \qu{V}{\phi(K)} \overset{\cong}{\to} F(V)$, so we wish to apply the contraction isomorphism 
$$
\mathfrak{c}_{[F]} : \lm \bigl(\qu{V}{\phi(K)}\bigr) \otimes \bigl( \lm F(V) \bigr)^* \to \R
$$
from \eqref{eq:quotable}.
Since these factors do not appear adjacent after applying $S_V\otimes S_W$, we compose $S_W$ with an additional reordering isomorphism -- noting that we do not introduce signs in switching factors here
$$
R \,:\; A \otimes B  \overset{\cong}{\longrightarrow}\;
B \otimes A , \qquad
 a \otimes b  \; \longmapsto\; b \otimes a .
$$
Finally, using the natural identification $\lm K \otimes \R\otimes  \bigl(\lm \bigl(\qu{W}{F(V)}\bigr)\bigr)^* \cong \lm K \otimes  \bigl(\lm \bigl(\qu{W}{F(V)}\bigr)\bigr)^*$ we obtain an isomorphism 
$$
\bigl( \id_{ \lm K }\otimes \mathfrak{c}_{[F]} \otimes \id_{ (\lm (\qu{W}{F(V)}))^*} \bigr) \circ \bigl( S_V \otimes (R\circ S_W) \bigr) .
$$ 
To see that it coincides with $\mathfrak{C}_F^\phi$ as described in the statement, note that -- using the bases as above -- it maps
$(v_1\wedge \ldots \wedge v_n) \otimes(w_1\wedge \ldots \wedge w_m)^*$ to $(\phi^{-1}(v_1)\wedge \ldots\wedge \phi^{-1}(v_k)) \otimes ([w_1]\wedge \ldots\wedge [w_\ell])^*$ multiplied with the factor
$(w_{\ell+1}\wedge \ldots \wedge w_m)^*\bigl(F(v_{k+1})\wedge \ldots \wedge F(v_n)\bigr)$, and that the latter equals $1$ if we choose $w_{\ell+i}=F(v_i)$ for $i=1,\ldots, n-k$. Note here that the existence of an isomorphism $F$ implies $m-\ell = n-k$, so that $m-n=\ell+k$, and hence $w_{m-n+(k+i)}=w_{\ell+i}$.
\end{proof}

\begin{prop}\label{prop:det0}  
For any weak Kuranishi atlas, $\det(\s_\Kk)$ is a well defined line bundle over $\Kk$.
Further, if $\Kk$ is a weak Kuranishi cobordism, then $\det(\s_\Kk)$ can be given product form on the collar of $\Kk$ with restrictions $\det(\s_\Kk)|_{\p^\al\Kk} = \det(\s_{\p^\al\Kk})$ for $\al= 0,1$.
The required bundle isomorphisms from the product $\det(\s_{\p^\al\Kk})\times A^\al_\eps$ to the collar restriction  $(\io^\al_\eps)^*\det(\s_\Kk)$ are given in \eqref{orient map}.
\end{prop}
\begin{proof}
To see that $\det(\s_\Kk)$ is a line bundle over $\Kk$, we first prove that each topological bundle $\det(\rd s_I)$ is a 
smooth line bundle, since it has compatible local trivializations $\det(\rd s_I)\cong\lm\ker(\rd s_I \oplus R_I)$ 
induced from constant linear injections $R_I:\R^{N}\to E_I$ 
which locally cover the cokernel, see  e.g.\ \cite[Appendix~A.2]{MS}. 
There are various natural ways to define these maps; the crucial choice is the sign in equation \eqref{Cfrak}.\footnote{
See \cite{Z3} for a discussion of the different  conventions.
Changing the sign in  \eqref{Cfrak} for example by the factor $(-1)^{n-k}$ affects the local trivializations (and hence the topology of the determinant bundle) because \eqref{Cfrak}  is applied below to the family of operators $F_x, x\in U_I,$ the dimension of whose kernels  varies with $x$.} 

At each point $x\in U_I$ we will use the contraction map $\mathfrak{C}^{\phi_x}_{F_x}$ of Lemma~\ref{lem:get} for the linear map $F_x$ and isomorphism to its kernel $\phi_x$, where
\begin{align*}
F_x  \,:\; \ker(\rd_x s_I \oplus R_I) &\; \to \; \im R_I \;\subset\; E_I , 
\qquad\qquad\qquad\qquad (v,r) \;\mapsto\; \rd_x s_I(v) ,
\\
\phi_x \,:\;\;\,  K := \ker \rd_x s_I &\;\to\; \ker( \rd_x s_I \oplus R_I)  \; \subset \; \rT_x U_I\oplus \R^N , 
\qquad
k\mapsto (k,0) .
\end{align*}
Note here that $\ker(\rd_x s_I \oplus R_I)=\bigl\{(v,r)\in \rT_x U_I\oplus \R^N \,\big|\, \rd_x s_I (v) = - R_I(r) \bigr\}$, so that $F_x$ indeed maps to $\im R_I$ with $F_x(v,r)=-R_I(r)$, and its image is $\im F_x = \im \rd_x s_I \cap \im R_I$.

If  we restrict $x$ to an open set $O\subset U_I$ on which $\rd_x s_I\oplus R_I$ is surjective, then the inclusion $\im R_I \hookrightarrow E_I$ induces an isomorphism
$$
\io_x \,: \; \qu{\im R_I}{\im \rd_x s_I \cap \im R_I} \; \overset{\cong}{\longrightarrow} \; \qu{E_I}{\im \rd_x s_I } .
$$
Indeed, $\io_x$ is surjective since $E_I = \im \rd_x s_I + \im R_I$ and injective by construction.
Hence \eqref{eq:laphi} together with dualization defines an isomorphism $\La_{\io_x}^* : \bigl( \lm \qu{E_I}{\im \rd_x s_I }\bigr)^* \to \bigl( \qu{\im R_I}{\im \rd_x s_I \cap \im R_I} \bigr)^*$, which we invert and compose with the contraction isomorphism of Lemma~\ref{lem:get} to obtain isomorphisms
\begin{align}\notag
T_{I,x} \, := \; \bigl(\id_{\lm\ker \rd_x s_I}\otimes (\La_{\io_x}^*)^{-1}\bigr) \circ {\mathfrak C}^{\phi_x}_{F_x} \;:\;\;\; & 
 \lm \ker(\rd_x s_I \oplus R_I) \otimes \bigl( \lm \im R_I \bigr)^*  \\
&\;\overset{\cong}{\longrightarrow}\;
\lm\ker \rd_x s_I \otimes \bigl(\lm \bigl(\qu{E_I}{\im \rd_x s_I }\bigr)\bigr)^*. \label{eq:TIx} 
\end{align}
Precomposing this with the isomorphism $\R\cong \lm \R^N\stackrel{\La_{R_I}}\cong \lm \im R_I$ from
\eqref{eq:laphi}, we obtain a trivialization of $\det(\rd s_I)|_O$ given by isomorphisms
\begin{align}\notag
\Hat T_{I,x} \,:\; 
\lm \ker(\rd_x s_I \oplus R_I) 
&\;\overset{\cong}{\longrightarrow}\; 
\qquad\qquad \det(\rd_x s_I)  \\ \label{eq:HatTx}
\ov v_1\wedge \ldots \wedge \ov v_n
&\;\longmapsto\; 
 (v_1\wedge\dots v_k)\otimes  \bigl( [R_I(e_1)]\wedge\dots [R_I(e_{N-n+k})] \bigr)^*,
\end{align}
where $\ov v_i=(v_i,r_i)$ is a basis of $\ker(\rd_x s_I \oplus R_I)$ such that $v_1,\ldots,v_k$ span $\ker \rd_x s_I$ (and hence $r_1=\ldots=r_k=0$), and $e_{1},\ldots, e_{N}$ is a positively ordered normalized basis of $\R^N$ (that is $e_1\wedge\ldots e_N = 1 \in \R \cong \lm\R^N$) such that $R_I(e_{N-n+i}) = \rd_x s_I(v_i)$ for $i = k+1,\ldots, n$. In particular, the last $n-k$ vectors span $\im \rd_x s_I \cap \im R_I \subset E_I$, and thus the first $N-n+k$ vectors $[R_I(e_1)],\dots, [R_I(e_{N-n+k})]$ span the cokernel $\qu{E_I}{\im\rd_x s_I}\cong\qu{\im R_I}{\im\rd_x s_I\cap \im R_I}$.

Next, we show that these trivializations do not depend on the choice of injection $R_I:\R^N\to E_I$.
Indeed, given another injection $R_I':\R^{N'}\to E_I$ that also maps onto the cokernel of $\rd s_I$,
we can choose a third injection $R_I'':\R^{N''}\to E_I$ that is surjective, and compare it to both of $R_I, R_I'$.  
Hence it suffices to consider the following two cases:

\begin{itemlist}
\item $N=N'$ and $R_I = R_I'\circ \io$ for a bijection $\io: \R^{N} \overset{\cong}{\to} \R^{N'}$;
\item $N<N'$ and $R_I=R_I'\circ\pr$ for the canonical projection $\pr: \R^{N'}\to \R^N\times\{0\} \cong\R^N$.
\end{itemlist}

In the second case denote by $\io: \R^{N}\to \R^N\times\{0\}\subset\R^{N'}$ the canonical injection, then in both cases we have $R_I = R_I'\circ \io$, and thus $\id \times \io$ induces an injection $\ker(\rd s_I\oplus R_I)\to
\ker(\rd s_I\oplus R_I')$ so that there is a well defined quotient bundle
$\qu{\ker(\rd s_I\oplus R_I')}{\ker(\rd s_I\oplus R_I)} \to U_I$.

In case $N<N'$ we claim that an appropriately scaled choice of local trivialization for this quotient over an open set $O\subset U_I$, on which both trivializations of $\det(\rd s_I)|_O$ are defined, induces a bundle isomorphism
$\Psi: \lm \ker(\rd s_I\oplus R_I)|_O\to \lm \ker(\rd s_I\oplus R_I')|_O$ 
that is compatible with the trivializations $\Hat T_I$ and $\Hat T_I': \lm \ker(\rd s_I\oplus R'_I)|_O \to \det(\rd s_I)|_O$ constructed as in \eqref{eq:HatTx}, that is $\Hat T_I = \Hat T_I' \circ \Psi$.

To define $\Psi$, let $n:=\dim \ker(\rd s_I\oplus R_I)$ and fix a trivialization of the quotient, that is a family of smooth sections $\bigl(\ov v^\Psi_{i}=(v^\Psi_{i},r^\Psi_{i})\bigr)_{i=n+1, \ldots,n'}$ of $\ker(\rd s_I\oplus R_I')|_O$ with 
$n':= n+N'-N$, that induces a basis for the quotient space at each point $x\in O$.
Here we may want to rescale $\ov v^\Psi_{n+1}$ by a nonzero real, as discussed below.
Then for fixed $x\in O$, any choice of basis $(\ov v_i)_{i=1,\ldots,n}$ of $\ker(\rd_x s_I\oplus R_I)$ induces a basis $(\id \times \io)(\ov v_1),\dots, (\id \times \io)(\ov v_n), \ov v^\Psi_{n+1}, \ldots, \ov v^\Psi_{n'}$ of $ \ker(\rd_x s_I\oplus R_I')$, and we define $\Psi$ by 
$$
\Psi_x \,: \; \ov v_1\wedge\dots\wedge \ov v_n \;\mapsto\; (\id \times \io)(\ov v_1)\wedge\dots\wedge (\id \times \io)(\ov v_n)\wedge \ov v^\Psi_{n+1}(x)\wedge \ldots \wedge \ov v^\Psi_{n'}(x) ,
$$
which varies smoothly with $x\in O$. It remains to show that, for appropriate choice of the sections $\ov v^\Psi_i$, we have $\Hat T_{I,x} = \Hat T_{I,x}' \circ \Psi_x$ for any fixed $x\in O$.
For that purpose we express the trivializations $\Hat T_{I,x}$ and $\Hat T'_{I,x}$ as in \eqref{eq:HatTx}. This construction begins by choosing a basis $(\ov v_i)_{i=1,\dots,n}$ of $\ker (\rd_x s_I \oplus R_I)$, where the first $k$ elements $\ov v_i=(v_i,0)$ span $\ker \rd_x s_I\times\{0\}$.
A compatible choice of basis $(\ov v'_i)_{i=1,\dots,n'}$ for $\ker (\rd_x s_I \oplus R'_I)$ is given by $\ov v'_i := (\id \times \io)(\ov v_i)$ for $i=1,\dots,n$, and $\ov v'_i:= \ov v^\Psi_i$ for $i=n+1,\ldots,n'$.
Note here that $\ov v'_i = \ov v_i$ for $i=1,\ldots,k$.
Next, one chooses a positively ordered normalized basis $e_{1},\ldots, e_{N}$ of $\R^N$ such that $R_I(e_{N-n+i}) = \rd_x s_I(v_i)$ for $i = k+1,\ldots, n$. 
Then the first $N-n+k$ vectors $[R_I(e_1)],\dots, [R_I(e_{N-n+k})]$ coincide with $[R'_I(\io(e_1))],\dots, [R'_I(\io(e_{N'-n'+k}))]$ and span the cokernel $\qu{E_I}{\im\rd_x s_I}$, and the last $n-k$ vectors span $\im \rd_x s_I \cap \im R_I \subset E_I$. So we obtain a corresponding basis $e'_1,\ldots, e'_{N'}$ of $\R^{N'}$ by taking $e'_i = \io(e_i)$ for $i=1,\ldots,N$ and $e'_{N+i} = (R'_I)^{-1}\bigl( \rd_x s_I(v^\Psi_{n+i}(x))$ for $i=1,\ldots, N'-N = n'-n$. 
To obtain the correct definition of $\Hat T'_{I,x}$, we then rescale $v^\Psi_{n'}$ by the reciprocal of 
\begin{align*}
\la(x) &\,:=\; \io(e_1)\wedge \ldots \wedge \io(e_N) \wedge (R'_I)^{-1}\bigl( \rd_x s_I(v^\Psi_{n+1}(x))\bigr) \wedge\ldots\wedge (R'_I)^{-1}\bigl( \rd_x s_I(v^\Psi_{n'}(x))\bigr) \\
& \;\in\; \lm \R^{N'} \;\cong \; \R ,
\end{align*}
such that $e'_1,\ldots, e'_{N'-1}, \la(x)^{-1} e'_{N'}$ becomes positively ordered and normalized.
Note here that $\la:O\to\R$ is a smooth nonvanishing function of $x$, depending only on the sections $v^\Psi_{n+1}(x), \ldots, v^\Psi_{n'}(x)$ since $\io(e_i)=(e_i,0)$ are a positively ordered normalized basis of $\R^N\times\{0\}\subset\R^{N'}$ for all $x\in O$.
Thus $v^\Psi_{n+1}(x), \ldots, \la(x)^{-1} v^\Psi_{n'}(x)$ defines a smooth trivialization of the quotient bundle $\qu{\ker(\rd s_I\oplus R_I')}{\ker(\rd s_I\oplus R_I)} \to O$, for which the induced map $\Psi$ now provides the claimed compatibility. Indeed, we have by construction
\begin{align}
\bigl( \Hat T'_{I,x}\circ\Psi_x \bigr) \bigl( \ov v_1\wedge\dots\wedge \ov v_n\bigr) &\;=\;  
(v_1\wedge\dots\wedge v_k)\otimes  \bigl([R_I'(e_1')]\wedge\dots\wedge [R_I'(e_{N'-n'+k}')] \bigr)^*  \notag \\ \notag
&\;=\; (v_1\wedge\dots\wedge v_k)\otimes 
\bigl([R_I(e_1)]\wedge\dots\wedge [R_I(e_{N-n+k})]\bigr)^* \\ \label{tpsit} 
&\;=\; \Hat T_{I,x}\bigl(\ov v_1\wedge\dots\wedge \ov v_n\bigr).
\end{align}

In case $N=N'$ we define an isomorphism $\Psi$ as above, which however does not depend on any choice of vectors $\ov v^\Psi_i$. 
Then in the above calculation of $\Hat T_{I,x}$ and $\Hat T'_{I,x}$, the factor $\la = \io(e_1) \wedge \ldots \wedge \io(e_{N})$ is constant (equal to the determinant of $\io = (R_I')^{-1}\circ R_I$), and hence $\la^{-1}\Psi$ intertwines the trivializations $\Hat T_I$ and $\Hat T_I'$. This completes the proof that the local trivializations of $\det(\rd s_I)$ do not depend on the choice of $R_I$.
In particular, $\det(\rd s_I)$ is a smooth line bundle over $U_I$ for each $I\in\Ii_\Kk$.

To complete the proof that $\det(\s_\Kk)$ is a vector bundle we must check that the 
lifts $\La_{IJ}$ given in \eqref{eq:bunIJ} of the coordinate changes $\Hat\Phi_{IJ}$ 
are smooth bundle isomorphisms.
Since the $\La_{IJ}(x)$ are constructed to be fiberwise isomorphisms, and the weak cocycle condition for the coordinate changes transfers directly to these bundle maps,
the nontrivial step is to check that $\La_{IJ}(x)$ varies smoothly with $x\in U_{IJ}$. 
For that purpose note that  
any trivialization $\Hat T_I$ near a given point $x_0\in U_{IJ}$ using a choice of $R_I$ as above, induces a trivialization $\Hat T_J$ of $\det(\rd s_J)$ near $\phi_{IJ}(x_0)\in U_J$ 
using the injection $R_J:=\Hat\phi_{IJ}\circ R_I$, since by the index condition $\Hat\phi_{IJ}$ identifies the cokernels. 
We will now show that these local trivializations transform $\La_{IJ}(x)$ into the isomorphisms $\La_{\rd_x \phi_{IJ} \oplus \id_{\R^{N}} }$ of \eqref{eq:laphi} induced by the smooth family of isomorphisms
$$
\rd_x\phi_{IJ} \oplus \id_{\R^{N}} \,:\;  \ker(\rd_x s_I \oplus R_I) \;\overset{\cong}{\longrightarrow}\; \ker\bigl(\rd_{\phi_{IJ}(x)} s_J \oplus (\Hat\phi_{IJ}\circ R_I)\bigr) .
$$
Namely, we prove the following.
\smallskip

\NI {\bf Claim:} {\it  The embeddings $\rd_x\phi_{IJ} \oplus \id_{\R^{N}}$ are surjective, and}
\begin{equation}\label{eq:etrans}
\Hat T_{J,\phi_{IJ}(x)}\circ \La_{\rd_x\phi_{IJ} \oplus \id_{\R^{N}}} \;=\; \La_{IJ}(x) \circ \Hat T_{I,x} .
\end{equation}
\MS

Notice that this will complete the proof.  Since the maps $\rd_x\phi_{IJ}$ vary smoothly with $x\in U_{IJ}$ near $x_0$, so do 
the  $\La_{\rd_x\phi_{IJ} \oplus \id_{\R^{N}} }$.   It follows that $\La_{IJ}$ is also smooth near $x_0$ 
 with respect to the 
smooth structure defined by the trivializations $\Hat T_{I,x}$ and $\Hat T_{J,\phi_{IJ}(x)}$.
\MS

\NI {\it Proof of Claim:}
The embeddings are surjective since, for $(v,z)\in \rT U_J \times \R^{N}$ with $\rd s_J(v)=-\Hat\phi_{IJ}(R_I(z))$, the tangent bundle condition $\im\rd s_J \cap \im\Hat\phi_{IJ}=\rd s_J (\im\rd\phi_{IJ})$ from Lemma~\ref{le:change}, the partial index condition $\ker\rd s_J \subset \im\rd\phi_{IJ}$, and injectivity of $\Hat\phi_{IJ}$ imply $v\in \im\rd s_I$ with $\rd s_I(v)=-R_I(z)$.
To prove \eqref{eq:etrans} we simply compare the explicit maps given in \eqref{eq:HatTx}.
So let $\ov v_i=(v_i,r_i)$ be a basis of $\ker(\rd_x s_I \oplus R_I)$ such that $v_1,\ldots,v_k$ span $\ker \rd_x s_I$.
Then, correspondingly, 
$\ov v'_i=\bigl(\rd_x\phi_{IJ} \oplus \id_{\R^{N}}\bigr)(\ov v_i)$ ia a basis of $\ker(\rd_{\phi_{IJ}(x)} s_J \oplus R_J)$ such that $v'_i=\rd_x\phi_{IJ}(v_i)$ for $i=1,\ldots,k$ span $\ker \rd_{\phi_{IJ}(x)} s_J$.
Next, let $e_{1},\ldots, e_{N}$ be a positively ordered normalized basis of $\R^N$ such that $R_I(e_{N-n+i}) = \rd_x s_I(v_i)$ for $i = k+1,\ldots, n$. Then, correspondingly, we have 
$$
R_J(e_{N-n+i}) = \Hat\phi_{IJ}\bigl(R_I(e_{N-n+i})\bigr) =   \Hat\phi_{IJ}\bigl(\rd_x s_I(v_i)\bigr)
 =  \rd_{\phi_{IJ}(x)}s_J \bigl( \rd_x\phi_{IJ}(v_i)\bigr)  =  \rd
 s_J ( v'_i) .
$$
Using these bases in \eqref{eq:HatTx} we can now verify \eqref{eq:etrans},
\begin{align*}
& \bigl(\La_{IJ}(x) \circ \Hat T_{I,x}\bigr)\bigl(\ov v_1\wedge \ldots \wedge \ov v_n\bigr) \\
&\;=\;
\La_{IJ}(x) \bigl(  (v_1\wedge\ldots \wedge v_k)\otimes  \bigl( [R_I(e_1)]\wedge\ldots\wedge [R_I(e_{N-n+k})] \bigr)^* \bigr) \\
&\;=\;
\bigl( \rd_x\phi_{IJ}(v_1)\wedge\ldots\wedge  \rd_x\phi_{IJ}(v_k)\bigr) \otimes  \bigl( [\Hat\phi_{IJ}(R_I(e_1))]\wedge\ldots \wedge [\Hat\phi_{IJ}(R_I(e_{N-n+k}))] \bigr)^* \bigr) \\
&\;=\;
( v'_1\wedge\ldots\wedge  v'_k) \otimes  \bigl( [R_J(e_1)]\wedge\ldots\wedge [R_J(e_{N-n+k})] \bigr)^* \bigr) \\
&\;=\;
\Hat T_{J,\phi_{IJ}(x)}\bigr)\bigl(\ov v'_1\wedge \ldots \wedge \ov v'_n\bigr) 
\;=\;
\bigl( \Hat T_{\phi_{IJ}(x)}\circ \bigl(\rd\phi_{IJ}(x) \oplus \id_{\R^{N}}\bigr) \bigr)\bigl(\ov v_1\wedge \ldots \wedge \ov v_n\bigr).
\end{align*}
This proves the Claim, and hence finishes the construction of $\det(\s_\Kk)$ for a weak Kuranishi atlas $\Kk$.

In the case of a weak Kuranishi cobordism $\Kk$, we moreover have to
construct bundle isomorphisms from collar restrictions to the product bundles $A^\al_\eps\times \det(\s_{\p^\al\Kk})$ to prove that $\det(\s_\Kk)$ is a line bundle in the sense of Definition~\ref{def:bundle} with the claimed restrictions.
That is, we have to construct bundle isomorphisms $\ti\io^{\La,\al}_I :A^\al_\eps\times  \det(\rd s^\al_I) \to \det(\rd s_I)|_{\im\io^\al_I}$ for $\al=0,1$, $I\in\Ii_{\p^\al\Kk}$, and $\eps>0$ less than the collar width of $\Kk$, and check the identities $\La_{IJ} \circ \ti\io^{\La,\al}_I = 
\ti\io^{\La,\al}_{J} \circ \bigl(
\id_{A^\al_\eps}\times \La^\al_{IJ}\bigr)$.

For that purpose recall that $(s_I\circ\io^\al_I )(t,x) = s^\al_I(x)$ for $(t,x)\in A^\al_\eps\times \partial^\al U_I $, so that we have a trivial identification $\id_{E_I} : \im \rd_x s^\al_I \to  \im\rd_{\io^\al_I(t,x)} s_I$ of the images and an isomorphism 
$\rd_{(t,x)}\io^\al_I : \R\times \ker \rd_x s^\al_I  \to  \ker\rd_{\io^\al_I(t,x)} s_I$.
The latter gives rise to an isomorphism given by wedging with the canonical positively oriented unit vector $1\in\R =\rT_t A^\al_\eps$,
\begin{equation} \label{wedge 1}
 \wedge_1 \,:\; \lm \ker \rd_x s^\al_I \;\to\; \lm \bigl( \R\times \ker \rd_x s^\al_I \bigr) , \qquad
\eta \;\mapsto\; 1\wedge \eta.
\end{equation}
Here and throughout we identify vectors $\eta_i\in\ker \rd s^\al_I$ with 
$(0,\eta_i)\in \R\times \ker \rd s^\al_I $ and also abbreviate $1:= (1,0)\in \R\times \ker \rd s^\al_I$.
This map now composes with the induced isomorphism $\La_{\rd_{(t,x)}\io^\al_I}$ from \eqref{eq:laphi} and can be combined with the identity on the cokernel factor to obtain fiberwise isomorphisms
\begin{equation}\label{orient map}
\ti\io^{\La,\al}_I (t,x) :=  \bigl( \La_{\rd_{(t,x)}\io^\al_I} \circ \wedge_1 \bigr) 
\otimes \bigl(\La_{\id_{E_I}}\bigr)^*
\;:\; \det(A^\al_\eps\times  \rd_x s^\al_I) \;\to\;  \det(\rd_{\io^\al_I(t,x)} s_I) .
\end{equation}
These isomorphisms vary smoothly with $(t,x)\in  A^\al_\eps\times \p^\al U_I$ since the compatible local trivializations $\lm\ker(\rd s^\al_I \oplus R_I) \to \det(\rd s^\al_I)$ and $\lm\ker(\rd s_I \oplus R_I) \to \det(\rd s_I)$ transform $\ti\io^{\La,\al}_I(t,x)$ to $\La_{\rd_{(t,x)}\io^\al_I\oplus\id_{\R^N}}\circ \wedge_1$.
Moreover, $(t,x)\to \ti\io^{\La,\al}_I(t,x)$ lifts $\io^\al_I$ and thus defines the required bundle isomorphism
$\ti\io^{\La,\al}_I :  A^\al_\eps\times \det(\rd s^\al_I) \to  \det(\rd s_I)|_{\im\io^\al_I}$ for each $I\in\Ii_{\p^\al\Kk}$.
Finally, the isomorphisms \eqref{orient map} intertwine $\La_{IJ} = \La_{\rd\phi_{IJ}} \otimes 
\bigl(\La_{[\Hat\phi_{IJ}]^{-1}}\bigr)^*$ and $\La^\al_{IJ} = \La_{\rd\phi^\al_{IJ}} \otimes \bigl(\La_{[\Hat\phi_{IJ}]^{-1}}\bigr)^*$ by the product form of the coordinate changes $\phi_{IJ}\circ\io^\al_I = \io^\al_ J\circ (\id_{A^\al_\eps}\times \phi^\al_{IJ})$, 
and because $\rd_{\io^\al_I(t,x)}\phi_{IJ}$ maps $\rd_{(t,x)} \io_I^\al (1)$ to $\rd_{(t,\phi^\al_{IJ}(x))} \io_J^\al (1)$,
both of which are wedged on by \eqref{orient map} from the left hand side.
(For an example of a detailed calculation see the end of the proof of Proposition~\ref{prop:orient1}.)
This finishes the proof.
\end{proof}

We next use the determinant bundle $\det(\s_\Kk)$ to define the notion of an orientation of a Kuranishi atlas.

\begin{defn}\label{def:orient} 
A  weak Kuranishi atlas or Kuranishi cobordism $\Kk$ is {\bf orientable} if there exists a nonvanishing section $\si$ of the bundle $\det(\s_\Kk)$ (i.e.\ with $\si_I^{-1}(0)=\emptyset$ for all $I\in\Ii_\Kk$).  An {\bf orientation} of $\Kk$ is a choice of nonvanishing section $\si$ of $\det(\s_\Kk)$. 
An {\bf oriented Kuranishi atlas or cobordism} is a pair $(\Kk,\si)$ consisting of a  Kuranishi atlas or cobordism and an orientation $\si$ of $\Kk$.

For an oriented Kuranishi cobordism $(\Kk,\si)$, the {\bf induced orientation of the boundary} $\p^\al\Kk$ for $\al=0$ resp.\ $\al=1$ is the orientation of $\p^\al\Kk$,
$$
\p^\al\si \,:=\; \Bigl( \bigl( (\ti\io^{\La,\al}_I)^{-1} \circ\si_I \circ \io^\al_I \bigr)\big|_{\{\al\} \times \partial^\al U_I } \Bigr)_{I\in\Ii_{\p^\al\Kk}}
$$
given by the isomorphism $(\ti\io^{\La,\al}_I)_{I\in\Ii_{\p^\al\Kk}}$ in \eqref{orient map} between a collar neighbourhood of the boundary in $\Kk$ and the product Kuranishi atlas $A^\al_\eps\times \p^\al\Kk$, followed by restriction to the boundary $\p^\al \Kk=\p^\al\bigl(
A^\al_\eps\times \p^\al \Kk\bigr)$, where we identify $\{\al\}\times \partial^\al U_I  \cong \partial^\al U_I$.

With that, we say that two oriented weak Kuranishi atlases $(\Kk^0,\si^0)$ and $(\Kk^1,\si^1)$ are {\bf oriented cobordant} if there exists a weak Kuranishi cobordism $\Kk$ from $\Kk^0$ to $\Kk^1$ and a section $\si$ of $\det(\s_\Kk)$ such that  $\partial^\al\si=\si^\al$ for $\al=0,1$.
\end{defn}

\begin{rmk}\label{rmk:orientb}\rm  
Here we have defined the induced orientation on the boundary $\p^\al \Kk$ of a cobordism so that it is completed to an orientation of the collar by adding the positive unit vector $1$ along $A^\al_\eps\subset \R$  rather than the more usual outward normal vector.
$\hfill\er$
\end{rmk}

\begin{lemma}\label{le:cK}
Let $(\Kk,\si)$ be an oriented weak Kuranishi atlas or cobordism. 
\begin{enumerate}\item
The orientation  $\si$ induces a canonical orientation $\si|_{\Kk'}:=(\si_I|_{U'_I})_{I\in\Ii_{\Kk'}}$ on each shrinking $\Kk'$ of $\Kk$ with domains $\bigl(U'_I\subset U_I\bigr)_{I\in\Ii_{\Kk'}}$.
\item
In the case of a Kuranishi cobordism $\Kk$, the restrictions to boundary and shrinking commute, that is
$(\si|_{\Kk'})|_{\p^\al\Kk'} = (\si|_{\p^\al\Kk})|_{\p^\al\Kk'}$.
\item 
In the case of a weak Kuranishi atlas $\Kk$, the orientation $\si$ on $\Kk$ induces an 
orientation  
$\si^{[0,1]}$ 
on $[0,1]\times \Kk$, 
which induces the given orientation $\p^\al\si^{[0,1]}=\si$ of the boundaries $\p^\al([0,1]\times \Kk) = \Kk$ for $\al=0,1$.
\end{enumerate}
\end{lemma}

\begin{proof}  
By definition, $\det(\s_{\Kk'})$ is the line bundle over $\Kk'$ consisting of the bundles $\det(\rd s'_I)=\det(\rd s_I)|_{U'_I}$ and the transition maps $\La'_{IJ}=\La_{IJ}|_{U'_{IJ}}$.
The restricted sections $\si_I|_{U'_I}$ of $\det(\rd s'_I)$ are hence compatible with the transition maps $\La'_{IJ}$ and have product form near the boundary in the case of a cobordism. Since they are nonvanishing, they define an orientation of $\Kk'$.
Commutation of restrictions holds since both $(\si|_{\Kk'})|_{\p^\al\Kk'}$ and $(\si|_{\p^\al\Kk})|_{\p^\al\Kk'}$ are given by $\si_I|_{\p^\al U'_I}$ with $\p^\al U'_I = \p^\al U_I \cap U'_I$.

For  part~(iii)  we consider an oriented, additive, weak Kuranishi atlas $(\Kk,\si)$ and begin by constructing an induced orientation of the product cobordism $[0,1]\times  \Kk$. For that purpose we use the bundle isomorphisms 
$$
\ti \io_I:=\wedge_1 \otimes (\La_{\id_{E_I}})^* \,:\; [0,1]\times  \det(\rd s_I) \;\to\; \det(\rd s'_I)
$$
with $s'_I(t,x)=s_I(x)$, covering $\io_I=\id_{[0,1]\times U_I}$.
These coincide with the maps defined in \eqref{orient map} for the interval $[0,1]$ instead of $A^\al_\eps$, so the proof of Proposition~\ref{prop:det0}
 shows that they provide an isomorphism $\ti \io$ from the product bundle $[0,1]\times \det(\s_\Kk)$ to the determinant bundle of the product $\det(\s_{[0,1]\times  \Kk})$.
 Now an orientation $\si$ of $\Kk$ determines an orientation 
$\si^{[0,1]}:=\ti\io_*\si$ of the product $[0,1]\times  \Kk$
given by $(\ti\io_*\si)_I(t,x)=\ti\io_I(t,x)\bigl(\si_I(x)\bigr)$.
Further, using $\ti\io^{\La,\al}:=\ti\io$ to define the collar structure on $\det(\s_{[0,1]\times  \Kk})$, the restrictions to both boundaries are $\p^\al( \ti\io_*\si) =\si$ since $\ti\io_I^{-1}\circ (\ti\io_*\si)_I
\circ \io^\al_I|_{U_I\times\{\al\}} = \si_I$.
\end{proof}

The arguments of Proposition~\ref{prop:det0} equally apply for any reduction $\Vv$ of a
 Kuranishi atlas $\Kk$ and an admissible perturbation $\nu$ to define a line bundle $\det(\s_\Kk|_\Vv + \nu)$ over $\Vv$ 
(i.e.\ a collection of vector bundles over $V_I$ for $I\in\Ii_\Kk$ together with lifts of the coordinate changes $\phi_{IJ}|_{\phi_{IJ}^{-1}(V_J)\cap V_I}$ that satisfy the weak cocycle condition),
or, equivalently, over the Kuranishi atlas $\Kk^\Vv$.
Instead of setting up a direct comparison between the bundles $\det(\s_\Kk|_\Vv + \nu)$  for different $\nu$,
we will work with a ``more universal" determinant bundle $\det(\Kk)$ over $\Kk$.
This will allow us to obtain compatible orientations of the determinant bundles over the perturbed zero set 
$\det(\s_\Kk|_\Vv + \nu)|_{(\s_\Kk|_\Vv+\nu)^{-1}(0)}$ for different transverse perturbations $\nu$.
We will construct the bundle $\det(\Kk)$ from the determinant bundles of the zero sections in each chart. 
However, since the zero section $0_\Kk$ does not satisfy the index condition, 
we need to construct different transition maps for $\det(\Kk)$, which will now depend on the section $\s_\Kk$.
For this purpose, we
again use contraction isomorphisms from Lemma~\ref{lem:get}.  
On the one hand, this provides families of isomorphisms
\begin{equation}\label{Cds}
\mathfrak{C}_{\rd_x s_I} \,:\; \lm \rT_x U_I \otimes \bigl( \lm E_I \bigr)^* \;\overset{\cong}{\longrightarrow}\;  \det(\rd_x s_I)
\qquad\text{for} \; x\in U_I .
\end{equation}
In fact, as we will see in the proof of Proposition~\ref{prop:orient} below, these maps are essentially the special cases of $T_{I,x}$ in \eqref{eq:TIx} in which $R_I$ is surjective.
On the other hand, recall that the tangent bundle condition \eqref{tbc} implies that $\rd s_J$ restricts to an isomorphism $\qu{\rT_y U_J}{\rd_x\phi_{IJ}(\rT_x U_I)}\overset{\cong}{\to} \qu{E_J}{\Hat\phi_{IJ}(E_I)}$ for $y=\phi_{IJ}(x)$. 
Therefore, if we choose a smooth normal bundle $N_{IJ} =\bigcup_{x\in U_{IJ}} N_{IJ,x}
\subset \phi_{IJ}^* \rT U_J
$ to 
the submanifold 
$\im \phi_{IJ}
\subset U_J$, then 
the subspaces 
$\rd_y s_J(N_{IJ,x})$ 
(where we always denote $y:=\phi_{IJ}(x)$ and vary $x\in U_{IJ}$)
form a smooth family of subspaces of $E_J$ that are complements  
to
$\Hat\phi_{IJ}(E_I)$.
Hence letting 
$\pr_{N_{IJ}}(x) : E_J \to \rd_y s_J(N_{IJ,x}) \subset E_J$
be the 
smooth family of 
projections with kernel  
$\Hat\phi_{IJ}(E_I)$,
we obtain 
a smooth family of 
linear maps
$$
F_x \,:= \;
\pr_{N_{IJ}}(x) 
\circ \rd_y s_J \,:\; \rT_y U_J \;\longrightarrow\; 
E_J
\qquad\text{for}\; x\in U_{IJ}, y=\phi_{IJ}(x)
$$
with images $\im F_x=\rd_y s_J(N_{IJ,x})$,
and isomorphisms to their kernel
$$
\phi_x  \,:= \;  \rd_x\phi_{IJ} \,:\; \rT_x U_I \;\overset{\cong}{\longrightarrow}\;  \ker F_x =  \rd_x\phi_{IJ}(\rT_x U_I) \;\subset\; \rT_y U_J .
$$
By Lemma~\ref{lem:get} these induce isomorphisms 
$$
\mathfrak{C}^{\phi_x}_{F_x} \,:\; \lm \rT_{\phi_{IJ}(x)} U_J \otimes \bigl(\lm E_J \bigr)^* 
 \;\overset{\cong}{\longrightarrow}\;  \lm \rT_x U_I \otimes  \Bigl(\lm \Bigl(\qq{E_J}
{\im F_x}
 \Bigr) \Bigr)^* .
$$
We may combine this with the dual of the  isomorphism $\lm \bigl(\qu{E_J}
{\rd_y s_J(N_{IJ,x})}
\bigr) \cong \lm E_I$ induced 
via \eqref{eq:laphi} by $\pr^\perp_{
N_{IJ}}(x) \circ\Hat\phi_{IJ} : E_I \to \qu{E_J}
{\rd_y s_J(N_{IJ,x})}
$ to obtain isomorphisms
\begin{align} \label{CIJ}
\mathfrak{C}_{IJ}(x) \,: \;
 \lm \rT_{\phi_{IJ}(x)} U_J \otimes \bigl(\lm E_J \bigr)^*  
\;\overset{\cong}{\longrightarrow}\;  \lm \rT_x U_I \otimes \bigl(\lm E_I \bigr)^*  
\end{align}
for $x\in U_{IJ}$, given by $\mathfrak{C}_{IJ}(x) :=  \bigl( \id_{ \lm \rT_x U_I} \otimes 
\La_{(\pr^\perp_{
N_{IJ}}(x)\circ\Hat\phi_{IJ})^{-1}}^*  \bigr) \circ \mathfrak{C}^{\phi_x}_{F_x}$.

\begin{prop}\label{prop:orient} 
\begin{enumerate}
\item 
Let $\Kk$ be a weak Kuranishi atlas. 
Then there is a well defined line bundle $\det(\Kk)$ over $\Kk$ given by the line bundles $\La^\Kk_I := \lm \rT U_I\otimes \bigl(\lm E_I\bigr)^* \to U_I$ for $I\in\Ii_\Kk$ and the transition maps $\mathfrak{C}_{IJ}^{-1}: \La^\Kk_I |_{U_{IJ}} \to \La^\Kk_J |_{\im\phi_{IJ}}$ from \eqref{CIJ} for $I\subsetneq J$. 
In particular, the latter isomorphisms are independent of the choice of 
normal bundle $N_{IJ}$.

Furthermore, the contractions $\mathfrak{C}_{\rd s_I}: \La^\Kk_I \to \det(\rd s_I)$ from \eqref{Cds} define an isomorphism $\Psi^{\s_\Kk}:=\bigl(\mathfrak{C}_{\rd s_I}\bigr)_{I\in\Ii_\Kk}$ from $\det(\Kk)$ to $\det(\s_\Kk)$.
\item 
If $\Kk$ is a weak Kuranishi cobordism, then the determinant bundle $\det(\Kk)$ defined as in (i)  
can be given a product structure on the collar 
such that its 
boundary restrictions are
$\det(\Kk)|{_{\p^\al\Kk}} = \det(\p^\al\Kk)$ for $\al= 0,1$. 

Further, the isomorphism $\Psi^{\s_\Kk}: \det(\Kk) \to \det(\s_\Kk)$ defined as in (i) has product structure on the collar with restrictions $\Psi^{\s_\Kk}|_{\p^\al \Kk}=\Psi^{s_{\p^\al\Kk}}$ for $\al=0,1$.
\end{enumerate}
\end{prop}

\begin{proof}
To begin, note that each $\La^\Kk_I = \lm \rT U_I \otimes \bigl(\lm E_I\bigr)^*$ is a smooth line bundle over $U_I$, since it inherits local trivializations 
from the tangent bundle $\rT U_I\to U_I$.

Next, we will show that the isomorphisms $\mathfrak{C}_{\rd_x s_I}$ from \eqref{Cds} are smooth in this trivialization, where $\det(\rd s_I)$ is trivialized via the maps $\Hat T_{I,x}$ as in Proposition~\ref{prop:det0} using an isomorphism $R_I: \R^{N}\to E_I$. For that purpose we introduce the isomorphisms
$$
G_x: \rT_x U_I\to \ker(\rd_x s_I \oplus  R_I),\qquad   v\mapsto \bigl(v, -R_I^{-1}(\rd_x s_I(v)\bigr),
$$  
and claim that the associated maps on determinant lines fit
into a commutative diagram
with $\mathfrak{C}_{\rd_x s_I}$ and the version of the trivialization $T_{I,x}$ from \eqref{eq:TIx}
\begin{equation}\label{ccord}
\xymatrix{  
\lm  \rT_x U_I \otimes \bigl( \lm E_I \bigr)^*  
\ar@{->}[d]_{\La_{G_x}\otimes \id}
 \ar@{->}[r]^{ \qquad\quad\mathfrak{C}_{\rd_x s_I}}  
 &
\det(\rd_x s_I)  \ar@{->}[d]^{\id} \\
\lm  \ker (\rd_x s_I \oplus R_I) \otimes \bigl(\lm E_I \bigr)^* \ar@{->}[r]^{ \qquad\qquad\quad\;\; T_{I,x}} 
& \det (\rd_x s_I).
}
\end{equation}
Here the trivialization $\Hat T_{I,x}$ of $\det(\rd_x s_I)$ is given by precomposing $T_{I,x}$ with the $x$-independent isomorphism $\La_{R_I^{-1}}^* : \bigl( \lm \R^N \bigr)^* \to \bigl(\lm E_I\bigr)^*$ in the second factor.
Since $G_x$ varies smoothly with $x$ in any trivialization of $\rT U$, this will prove that $\mathfrak C_{\rd s_I}$ is smooth with respect to the given trivializations.

To prove \eqref{ccord} we use the explicit formulas from Lemma~\ref{lem:get} and \eqref{eq:HatTx} at a fixed $x\in U_I$. 
So let $v_1,\ldots,v_n$ be a basis for $\rT_x U_I$ with ${\rm span}(v_1,\ldots,v_k)=\ker \rd_x s_I$, and let $w_1,\dots, w_N$ be a basis for $E_I$ with $w_{N-n+i}=\rd_x s_I(v_i)$ for $i=k+1,\ldots,n$.
Then $\ov v_i:=G_x(v_i)=\bigl(v_i,-R_I^{-1}(\rd_x s_I(v_i)\bigr) $ is a corresponding basis of $\ker(\rd_x s_I \oplus R_I)$. 
In this setting we can verify \eqref{ccord},
\begin{align*}
\mathfrak{C}_{\rd_x s_I} \bigl( (v_1\wedge\dots v_n)\otimes(w_1\wedge\dots w_N)^* \bigr)
&\;=\;
\bigl(v_1\wedge\dots v_k\bigr)\otimes \bigl( [w_1]\wedge\dots [w_{N-n+k}] \bigr)^* \\
&\;=\; T_{I,x}\bigl( (\ov v_1\wedge \ldots  \ov v_n)\otimes(w_1\wedge\dots w_N)^* \bigr) \\
&\;=\; T_{I,x}\bigl( \La_{G_x}(v_1\wedge \ldots  v_n)\otimes(w_1\wedge\dots w_N)^* \bigr).
\end{align*}
This proves the smoothness of the isomorphisms $\mathfrak C_{\rd s_I}$
so that we can
define preliminary transition maps 
\begin{equation}\label{tiphi}
\Ti\phi_{IJ}:= \mathfrak{C}_{\rd s_J}^{-1} \circ \La_{IJ} \circ \mathfrak{C}_{\rd s_I}
\,:\; \La^\Kk_I|_{U_{IJ}}\to \La^\Kk_J \qquad\text{for}\; I\subsetneq J \in \Ii_\Kk
\end{equation}
by the transition maps \eqref{eq:bunIJ} of $\det(\s_\Kk)$ and the isomorphisms \eqref{Cds}.
These define a line bundle $\La^\Kk:=\bigl(\La^\Kk_I, \Ti\phi^\La_{IJ} \bigr)_{I,J\in\Ii_\Kk}$ since the weak cocycle condition follows directly from that for the $\La_{IJ}$. Moreover, this automatically makes the family of bundle isomorphisms $\Psi^\Kk:=\bigl(\mathfrak{C}_{\rd s_I}\bigr)_{I\in\Ii_\Kk}$ an isomorphism from $\La^\Kk$ to $\det(\s_\Kk)$.

It remains to show that $\La^\Kk=\det(\Kk)$ and $\Psi^\Kk=\Psi^{\s_\Kk}$, i.e.\ we claim equality of transition maps $\Ti\phi^\La_{IJ}=\mathfrak{C}_{IJ}^{-1}$. This also shows that $\mathfrak{C}_{IJ}^{-1}$ and thus $\det(\Kk)$ is independent of the choice of 
normal bundle $N_{IJ}$ in \eqref{CIJ}.
So to finish the proof of (i), it suffices to establish the following commuting diagram at a fixed $x\in U_{IJ}$ with $y=\phi_{IJ}(x)$,
\begin{equation}\label{cclaim}
\xymatrix{
\lm  \rT_x U_I \otimes \bigl( \lm E_I \bigr)^* 
 \ar@{->}[r]^{ \qquad\quad\mathfrak{C}_{\rd_x s_I}}  
 &
\det(\rd_x s_I)  \ar@{->}[d]^{\La_{IJ}
(x)} \\
 \lm \rT_y U_J \otimes \bigl( \lm E_J \bigr)^*
 \ar@{->}[r]^{\qquad\quad\mathfrak{C}_{\rd_y s_J}}  
 \ar@{->}[u]^{\mathfrak{C}_{IJ}
 (x)} 
&
\det(\rd_y s_J) .
}
\end{equation}
Using \eqref{ccord}, for surjective maps $R_I:\R^N\to E_I$ and $R_J:\R^{N'}\to E_J$, and the compatibility of the trivialization $\Hat T_{J,y}$ with $R'_J:=\Hat T'_{J,y}$ arising from $\Hat\phi_{IJ}:\R^N\to E_J$, we can expand this diagram to 
 \[
 \xymatrix{
\lm  \rT_x U_I \otimes \bigl( \lm E_I \bigr)^* 
 \ar@{->}[r]^{\La_{G_x}\otimes \id \qquad}   
 & \;\; \lm  \ker (\rd_x s_I\oplus R_I) \otimes  \bigl( \lm E_I \bigr)^*  \ar@{->}[r]^{ \qquad\qquad\qquad T_{I,x} } 
 \ar@{->}[d]_{ \La_{\rd_x \phi_{IJ} \oplus \id_{\R^N}} \otimes ( \La_{R'_J\,\!^{-1}}^* \circ \La_{R_I}^*) }&
\det(\rd_x s_I)  \ar@{->}[d]^{\La_{IJ}
(x)} \\
 & \;\; \lm  \ker (\rd_y s_J\oplus R'_J) \otimes  \bigl( \lm \im R'_J \bigr)^*  \ar@{->}[d]_{\Psi_y\otimes  ( \La_{R_J^{-1}}^* \circ \La_{R'_J}^*) } \ar@{->}[r]^{ \qquad\qquad\qquad T'_{J,y}} 
&
\det(\rd_y s_J) \\
 \lm \rT_y U_J \otimes \bigl( \lm E_J \bigr)^* \ar@{->}[uu]^{\mathfrak{C}_{IJ}(x)} 
 \ar@{->}[r]^{\La_{G_y}\otimes \id \qquad }  
 & \;\; \lm  \ker (\rd_y s_J\oplus R_J) \otimes  \bigl( \lm E_J \bigr)^*  \ar@{->}[r]^{ \qquad\qquad\qquad T_{J,y}}
&
\det(\rd_y s_J)   \ar@{->}[u]_{\id} .
}
\]
Here the upper right square commutes by \eqref{eq:etrans}. To make the lower right square precise, and in particular to choose suitable $R_J$, we note that $E_J=\rd_y s_J + \im R'_J$ 
and $\rd_y s_J (\rT_y\im\phi_{IJ}) \subset \Hat\phi_{IJ}(E_I) =  \im R'_J$, so that given any normalized basis $e_1,\ldots,e_N\in \R^N$ we can complete the corresponding vectors $R'_J(e_i)$ to a basis for
$E_J$ by adding the vectors $\rd_y s_J(v^\Psi_{N+1}), \ldots, \rd_y s_J(v^\Psi_{N'})$,
where
$v^\Psi_{N+1}, \ldots, v^\Psi_{N'}\in N_{IJ,x}$ is a basis of the normal space $N_{IJ,x}\subset \rT_y U_J$ to $\rT_y\im\phi_{IJ}$ that was used to define $\mathfrak{C}_{IJ}(x)$.
Thus $R_J'$ extends  to a smooth family of bijections
\begin{align*}
R_J: = R_{J,x} 
 \,:\quad \R^N\times \R^{N'-N}  &\; \overset{\cong}{\longrightarrow}\;  \Hat\phi_{IJ}(E_I) \oplus 
 \rd s_y(N_{IJ,x})
 \;=\; E_J , \\
 (\ul r ; r_{N+1},\ldots,r_{N'}) &\;\longmapsto \;  \Hat\phi_{IJ}\bigl( R_I(\ul r) \bigr) + {\textstyle \sum_{i=N+1}^{N'}} r_i \cdot \rd_y s_J(v^\Psi_i) .
\end{align*}
We may choose the vectors $v^\Psi_{i}$ so that 
the $e_{i}:= R_J^{-1}\bigl(\rd_y s_J(v^\Psi_{i})\bigr)$ for $i=N+1,\ldots N'$ extend $e_1,\ldots e_N \in \R^N\cong \R^N \times\{0\}$ to a normalized basis of $\R^{N'}$.  Further, the vectors
 $\ov v^\Psi_{i}:= \bigl( v^\Psi_i , - e_i \bigr)$ span the complement of the embedding 
 $$
 \io: \ker (\rd_y s_J\oplus R'_J) \hookrightarrow  \ker (\rd_y s_J\oplus R_J),\qquad (v,\ul r) \mapsto (v,\ul r,0).
 $$ 
Hence \eqref{tpsit} (with $\la(x)=1$) shows that the isomorphism 
\begin{align*}
\Psi_y \,:\;  \lm \ker (\rd_y s_J\oplus R'_J) &\; \overset{\cong}{\longrightarrow}\;  \lm \ker (\rd_y s_J\oplus R_J) ,\\ 
\ov v_1 \wedge\ldots \wedge \ov v_n  &\;\longmapsto \;  
 \io(\ov v_1)\wedge\ldots \io(\ov v_n) \wedge \ov v^\Psi_{N+1} \wedge \ldots \ov v^\Psi_{N'}
\end{align*}
intertwines the trivializations $T'_{J,y}$ and $T_{J,y}$, that is the lower right square in the above diagram commutes. 

Now to prove that the entire diagram commutes it remains to identify $\mathfrak{C}_{IJ}(x)$ with the map given by composition of the other isomorphisms,  which is
 the tensor product of $\La_{R_I^{-1}}^*\circ \La_{R_J}^*$ (composed via $\lm \R^N\cong \R\cong\lm \R^{N'}$) on the obstruction spaces with the inverse of  
\begin{align*}
\lm \rT_x U_I &\;\overset{\cong}{\longrightarrow}\;  \lm \rT_y U_J   \\
v_1 \wedge\ldots \wedge v_n 
  &\;\longmapsto \;  
 \rd_x\phi_{IJ}(v_1) \wedge\ldots \rd_x\phi_{IJ} (v_{n}) \wedge v^\Psi_{N+1} \wedge \ldots v^\Psi_{N'} .
\end{align*}
Here we used the fact that $\io \circ \bigl(\rd_x\phi_{IJ}\oplus\id_{\R^N}\bigr) \circ G_x = G_y \circ \rd_x\phi_{IJ}$ and $\ov v^\Psi_{i}= G_y( v^\Psi_i )$.
Note moreover that we chose the vectors $v^\Psi_{i}\in \rT_y U_J$ to span the complement of $\im\rd_x\phi_{IJ}$, and hence $\rd_x\phi_{IJ}(v_1),\ldots,\rd_x\phi_{IJ} (v_{n}), v^\Psi_{N+1},\ldots ,v^\Psi_{N'}$ forms a basis of $ \rT_y U_J$. 
Moreover, note that $\bigl(w_i=R_J(e_i)\bigr)_{1\le i\le N'}$ 
is a basis for $E_J$ whose last $N'-N$ vectors are $w_{i}=\rd_y s_J(v^\Psi_i)\in 
\rd_y s_J(N_{IJ,x})$
for $i=N+1,\ldots,N'$. In these bases the explicit formulas \eqref{CIJ} and \eqref{Cfrak} give
\begin{align*}
\mathfrak{C}_{IJ}(x) \;:\; &  \bigl( \rd_x\phi_{IJ}(v_1) \wedge\ldots \rd_x\phi_{IJ} (v_{n}) \wedge v^\Psi_{N+1} \wedge \ldots v^\Psi_{N'}\bigr) \otimes \bigl(R_J(e_1)\wedge\dots R_J(e_{N'})\bigr)^*  \\
&\;\mapsto\;
(v_1\wedge\dots v_n)\otimes  \La_{(\pr^\perp_{
N_{IJ}}\circ\Hat\phi_{IJ})^{-1}}^*  \bigl( [\Hat\phi_{IJ}(R_I(e_1))]\wedge\dots [\Hat\phi_{IJ}(R_I(e_N))] \bigr)^* \\
&\;=\;
(v_1\wedge\dots v_n)\otimes  \bigl( R_I(e_1) \wedge\dots R_I(e_N) \bigr)^* .
\end{align*}
Here in the second factor we have $\La_{R_I^{-1}}\bigl( R_I(e_1) \wedge\dots R_I(e_N) \bigr) =1\in \lm\R^N$
and $\La_{R_J^{-1}}\bigl( R_J(e_1) \wedge\dots R_J(e_{N'}) \bigr)= 1 \in \lm\R^{N'}$, so this proves that  \eqref{cclaim} commutes.

For part (ii) the same arguments apply to define a bundle $\det(\Kk)$ and isomorphism $\Psi^{\s_\Kk}$, for which it remains to establish the product structure on a collar. However, we may use the isomorphisms $\mathfrak{C}_{\rd s_I}:\La^\Kk_I \to \det(\rd s_I)$ and $\mathfrak{C}_{\rd s_I|_{\p^\al U_I}}:\La^{\p^\al \Kk}_I \to \det(\rd s_I|_{\p^\al U_I})$ to pull back the isomorphisms 
$\ti\io^{\La,\al}_I : \det(\rd s_I)|_{\im\io^\al_I} \to  A^\al_\eps\times \det(\rd s_I|_{\p^\al U_I}) $ from Proposition~\ref{prop:det0} to isomorphisms
$$
\bigl( \id_{A^\al_\eps}\times  \mathfrak{C}_{\rd s_I|_{\p^\al U_I}}  \bigr) ^{-1}  \circ \ti\io^{\La,\al}_I  \circ \mathfrak{C}_{\rd s_I} \,:\;
\La^\Kk_I |_{\im\io^\al_I} \;\overset{\cong}{\longrightarrow}\; A^\al_\eps \times \La^{\p^\al\Kk}_I .
$$
This provides the product structure for $\det(\Kk)$. 
Moreover this construction was made such that $\Psi^{\s_\Kk}$ has product form in the same collar, and the restrictions are, as claimed, given by pullback of the restrictions of $\det(\s_\Kk)$. This completes the proof.
\end{proof}

\begin{prop}\label{prop:orient1} 
\begin{enumerate}
\item
Let $(\Kk,\si)$ be an oriented, tame Kuranishi atlas with reduction $\Vv$, and  let $\nu$ be an admissible, precompact, transverse perturbation of~$\s_\Kk|_\Vv$.  Then the zero set $|\bZ^\nu|$ 
inherits the structure of an oriented closed manifold.
\item
Let $(\Kk,\si)$ be an oriented, tame Kuranishi cobordism with reduction $\Vv$ and admissible, precompact, transverse perturbation $\nu$. Then the corresponding zero set $|\bZ^{\nu}|$ 
inherits the structure of
an oriented cobordism from $|\bZ^{\nu^0}|$ to $|\bZ^{\nu^1}|$ for $\nu^\al:=\nu|_{\partial^\al \Vv}$, with boundary orientations induced as in (i) by  $\si^\al:=\partial^\al \si$.
\end{enumerate}
\end{prop}
\begin{proof} 
We first show that the local zero sets $Z_I: = (s_I|_{V_I} + \nu)^{-1}(0)\subset V_I$  have a natural orientation.
By Lemma~\ref{le:stransv} they are submanifolds, and by transversality we have $\im (\rd_z s_I + \rd_z\nu_I)=E_I$ for each $z\in Z_I$, and thus $\lm \qu{E_I}{\im (\rd_z s_I + \rd_z\nu_I)}= \lm\{0\} = \R$, so that we have a natural isomorphism between the orientation bundle of $Z_I$ and the restriction of the determinant line bundle
\begin{align*}
\lm \rT Z_I &\;=\; {\textstyle \bigcup_{z\in Z_I}} \lm \ker (\rd_z s_I + \rd_z\nu_I) \\
&\; \cong\;  {\textstyle \bigcup_{z\in Z_I}} \lm \ker (\rd_z s_I + \rd_z\nu_I) \otimes \R 
\;=\; \det(s_I|_{V_I}+\nu_I)|_{Z_I} .
\end{align*}
Combining this with Proposition~\ref{prop:orient} and Lemma~\ref{lem:get} we obtain isomorphisms
$$
\mathfrak{C}^\nu_I(z) := \mathfrak{C}_{\rd (s_I +\nu_I)} \circ \mathfrak{C}_{\rd s_I}^{-1} \,:\;
\det(\rd s_I)|_{z} \;\longrightarrow\; \lm \rT_z Z_I  \qquad \text{for} \; z\in Z_I .
$$
To see that these are smooth,
recall that smoothness of $\mathfrak{C}_{\rd s_I}$ was proven in Proposition~\ref{prop:orient}.
The same arguments apply to $\mathfrak{C}_{\rd (s_I +\nu_I)}$.
Further, for $I\subsetneq J$ and $z\in Z_I\cap U_{IJ}$ these
isomorphisms 
are intertwined by the transition maps 
$$
\La_{IJ}(z)=\La_{\rd_z\phi_{IJ}} \otimes \bigl( \La_{\Hat\phi_{IJ}^{-1}}\bigr)^* : \det(\rd_z s_I) \to \det(\rd_{\phi_{IJ}(z)} s_J)
$$
and $\La_{\rd_z\phi_{IJ}} : \lm \rT_z Z_I \to \lm \rT_{\phi_{IJ}(z)} Z_J$.
To see this, one combines the commuting diagram \eqref{cclaim} with the analogous 
diagram over $Z_I\cap U_{IJ}$
\[
\xymatrix{
\lm  \rT U_I \otimes \bigl( \lm E_I \bigr)^* 
 \ar@{->}[rr]^{ \mathfrak{C}_{\rd (s_I+\nu_I)}\qquad}  
 & &
\det(\rd (s_I+\nu_I)) = \lm \rT Z_I \otimes \R
 \ar@{->}[d]^{\La_{IJ} = \La_{\rd\phi_{IJ}} \otimes \id_\R } \\
 \lm \rT U_J \otimes \bigl( \lm E_J \bigr)^*
 \ar@{->}[rr]^{\mathfrak{C}_{\rd (s_J+\nu_J)}\qquad}  
 \ar@{->}[u]^{\mathfrak{C}_{IJ}} 
& &
\det(\rd (s_J+\nu_J)) = \lm \rT Z_J \otimes \R.
}
\]
The latter diagram commutes by the arguments in Proposition~\ref{prop:orient} applied to $s_\bullet+\nu_\bullet$ because $\rd s_J$ and $\rd (s_J+\nu_J)$ induce the same map 
$\mathfrak{C}_{IJ}(z)$   
for $z\in Z_I\cap U_{IJ}$.
 Indeed, the admissibility of $\nu$ implies that  
$\im\rd_{\phi_{IJ}(z)} \nu_J\subset \Hat\phi_{IJ}(E_I)$ so that $F_z = \pr_{N_{IJ}}(z) \circ \rd_{\phi_{IJ}(z)} s_J =\pr_{N_{IJ}}(z) \circ \rd_{\phi_{IJ}(z)} (s_J+\nu_J)$ in the construction of $\mathfrak{C}_{IJ}(z)$.
Now the orientation $\bigl( \si_I : U_I \to \det(\rd s_I) \bigr)_{I\in\Ii_\Kk}$ of $\Kk$ induces nonvanishing sections $\si^\nu_I := \mathfrak{C}^\nu_I \circ \si_I : Z_I \to \lm \rT Z_I$ which, by the above discussion and the compatibility $\La_{IJ}\circ \si_I = \si_J$ are related by $\La_{\rd\phi_{IJ}} \circ\si^\nu_I = \si^\nu_J$, i.e.\ the orientations $\si^\nu_I$ in the charts of $|\bZ^\nu|$ are compatible with the transition maps $\phi_{IJ}|_{Z_I\cap U_{IJ}}$.
Hence this determines an orientation of $|\bZ^\nu|$. This proves~(i).

For a Kuranishi cobordism, the above constructions provide an orientation $|\si^\nu| : |\bZ^\nu| \to \lm \rT |\bZ^\nu|$ on the manifold with boundary $|\bZ^\nu|$. Moreover, Lemma~\ref{le:czeroS0} provides diffeomorphisms to the boundary components for $\al=0,1$
$$
|j^\al| \,:\;  |\bZ^{\nu^\al}| \;=\;   \Bigl| \underset{{I\in\Ii_{\Kk^\al}}}{\textstyle\bigcup} 
 (s_I+\nu^\al_I)^{-1}(0)  \Bigr| 
 \;\overset{\cong}{\longrightarrow}\; 
 \partial^\al |\bZ^{\nu}|  \;\subset \; \partial |\bZ^{\nu}|   ,
$$
which in the charts are given by $j^\al_I := \iota^\al_I (\al,\cdot) :  Z^\al_I  \to  Z_I$.
The latter lift to isomorphisms of determinant line bundles
$$
{\ti j}^\al_I := \La_{\rd\io^\al_I}\circ \wedge_1 \,:\; 
\lm \rT Z^\al_I \;\overset{\cong}{\longrightarrow}\; \lm \rT Z_I |_{\im\io^\al_I} ,
$$
given by the same expression as the restriction to $Z^\al_I\times \{\al\}$ of the map \eqref{orient map} on 
$ A^\al_\eps\times \det(\rd s^\al_I)$ in the case of trivial cokernel.
These are the expressions in the charts of an isomorphism of determinant line bundles
$$
|\ti j^\al|  := \La_{\rd |\io^\al|}\circ \wedge_1 \,:\;  \lm \rT |\bZ^{\nu^\al}|  \;\overset{\cong}{\longrightarrow}\; \lm \rT |\bZ^\nu| \bigr|_{|j^\al|( |\bZ^{\nu^\al}|) } ,
$$
which consists of the isomorphism induced by the collar neighbourhood embedding $|\io^\al| : A^\al_\eps \times  |\bZ^{\nu^\al}| \to  |\bZ^\nu|$ given by $\io^\al_\eps|_{A^\al_\eps\times Z^\al_I}$ in the charts together with 
the canonical isomorphism between the determinant line bundle of the boundary and the boundary restriction of the determinant line bundle of the collar neighbourhood,
$$
\wedge_1 \,:\; \lm \rT |\bZ^{\nu^\al}| \;\to\;  \lm \rT  \bigl(A^\al_\eps\times  |\bZ^{\nu^\al}|  \bigr) \big|_{\{\al\}\times |\bZ^{\nu^\al}| }
\;=\;  \R\times \bigl( \lm \rT  |\bZ^{\nu^\al}|  \bigr) 
, \quad
\eta \mapsto  
1\wedge\eta .
$$
Here, as before, we identify vectors $\eta_i \in \rT |\bZ^{\nu^\al}|$ with $(0,\eta_i) \in \R\times \rT |\bZ^{\nu^\al}| $ and abbreviate $1:=(1,0)  \in \R\times \rT |\bZ^{\nu^\al}|$.
The latter corresponds to the exterior normal $\rd |\io^1| (1,0) \in \rT |\bZ^\nu| \big|_{\p^1 |\bZ^\nu|}$ and the interior normal $\rd |\io^0| (1,0) \in \rT |\bZ^\nu| \big|_{\p^0 |\bZ^\nu|}$. Hence the boundary orientations\footnote
{
Here, contrary to the special choice for Kuranishi cobordisms discussed in Remark~\ref{rmk:orientb}, we use a more standard orientation convention for the manifold with boundary $|\bZ^\nu|$.
Namely, a positively ordered basis $\eta_1,\dots,\eta_k$ for the tangent space to the boundary is extended to a positively ordered basis $\eta_{out},\eta_1,\dots,\eta_k$ for the whole manifold by adjoining an outward unit vector $\eta_{out}$ as its first element.
 }
$\p^\al |\si^\nu| :|\bZ^{\nu^\al}| \to \lm \rT |\bZ^{\nu^\al}|$ induced from $|\si^\nu|$ on the two components for $\al=0,1$ differ by a sign,
$$
\p^0 |\si^\nu| := - \, |\ti j^0|^{-1} \circ  |\si^\nu| \circ |j^0| ,\qquad
\p^1 |\si^\nu| :=  |\ti j^1|^{-1} \circ  |\si^\nu| \circ |j^1| .
$$
On the other hand, the restricted orientations $\si^\al:=\partial^\al \si$ of $\Kk^\al$ also induce orientations $|\si^{\nu_\al}|$ of the boundary components $|\bZ^{\nu^\al}|$ by the construction in (i). 
Now to prove the claim that $|\bZ^\nu|$ is an oriented cobordism from $\bigl(|\bZ^{\nu^0}|,|\si^{\nu_0}|\bigr)$ to $\bigl(|\bZ^{\nu^1}|,|\si^{\nu_1}|\bigr)$, it remains to check that $\p^0 |\si^\nu| = - |\si^{\nu_0}|$ and $\p^1 |\si^\nu| = |\si^{\nu_1}|$. This is equivalent to $|\si^{\nu^\al}| = |\ti j^\al|^{-1} \circ  |\si^\nu| \circ |j^\al|$ for both $\al=0,1$. So, recalling the construction of $\p^\al\si_I=(\ti\io^{\La,\al}_I)^{-1} \circ \si_I \circ \io^\al_I\big|_{\{\al\} \times Z^\al_I }$ and $\si^\nu_I=\mathfrak{C}\nu_I \circ \si_I \big|_{Z_I}$ in local charts, we must show the following identity over
$\{\al\}\times Z^\al_I \cong Z^\al_I$ for all $I\in\Ii_{\p^\al\Kk}$
\begin{equation} \label{oclaim}
\mathfrak{C}_{\rd (s^\al_I +\nu^\al_I)} \circ \mathfrak{C}_{\rd s^\al_I}^{-1} \circ 
\bigl((\ti\io^{\La,\al}_I)^{-1} \circ \si_I \circ \io^\al_I\bigr)
\; =\;  
(\ti j^\al_I)^{-1} \circ  \bigl( \mathfrak{C}_{\rd (s_I +\nu_I)} \circ \mathfrak{C}_{\rd s_I}^{-1} \circ \si_I \bigr) \circ j^\al_I  .
\end{equation}
We will check this at a fixed point $z\in Z^\al_I$ in two steps. We first show that the contraction isomorphisms $\mathfrak{C}_{\rd s^\al_I}(z)$ and $\mathfrak{C}_{\rd s_I}(\io^\al_I(z,\al))$ intertwine the collar isomorphism 
$$
\ti\io^{\La,\al}_I = \bigl( \La_{\rd \io^\al_I} \circ \wedge_1 \bigr) \otimes \bigl(\La_{\id_{E_I}}\bigr)^*, \qquad \det(\rd s^\al_I) \to \det(\rd s_I)
$$ 
with the analogous collar isomorphism
$$
\Ti I^\al_I  := \bigl( \La_{\rd \io^\al_I} \circ \wedge_1 \bigr) \otimes \bigl(\La_{\id_{E_I}}\bigr)^*
\,:\; \lm \rT \partial^\al U_I \otimes \bigl( \lm E_I \bigr)^* 
\;\to\;
\lm \rT U_I \otimes \bigl( \lm E_I \bigr)^* .
$$
Indeed, we can use the product form of $s_I$ in terms of $s^\al_I$ to check the corresponding identity of maps 
$\lm \rT \partial^\al U_I \otimes \bigl( \lm E_I \bigr)^* \big|_{Z^\al_I}\to \det(\rd s_I)\big|_{\io^\al_I(Z^\al_I)}$ at a fixed vector of the form $\bigl( \eta_{\ker} \wedge \eta_{\perp} \bigr) \otimes  \bigl
( \zeta_{\coker} \wedge \zeta_{\im} \bigr)$ 
with $\eta_{\ker} \in \lm\ker\rd s^\al_I$ and $\zeta_{\im} \in \bigl(\lm \im\rd s^\al_I \bigr)^*$: 
\begin{align*}
&
\Bigl( \ti\io^{\La,\al}_I \circ \mathfrak{C}_{\rd s^\al_I} \Bigr)
\Bigl( \bigl( \eta_{\ker} \wedge \eta_{\perp} \bigr) \otimes  \bigl(
\zeta_{\coker} \wedge \zeta_{\im}
\bigr) \Bigr)  \\
&\;=\;
\ti\io^{\La,\al}_I\Bigl(
\mathfrak{c}_{\rd s^\al_I}\bigl( \eta_{\perp} ,  \zeta_{\im}   \bigr) \cdot 
\eta_{\ker} \otimes \zeta_{\coker} \Bigr)\\
&\;=\;
\mathfrak{c}_{\rd s^\al_I}\bigl( \eta_{\perp} ,  \zeta_{\im}   \bigr) \cdot 
\La_{\rd \io^\al_I}(1\wedge \eta_{\ker}) \otimes \zeta_{\coker} \\
&\;=\; 
\mathfrak{c}_{\rd s_I}\bigl( \La_{\rd \io^\al_I} (\eta_{\perp}) , \zeta_{\im}\bigr) \cdot 
 \La_{\rd \io^\al_I} \bigl( 1\wedge \eta_{\ker}  \bigr) \otimes  \zeta_{\coker} \\
&\;=\; 
\mathfrak{C}_{\rd s_I} \Bigl( \bigl( \La_{\rd \io^\al_I}(1\wedge \eta_{\ker} \bigr) 
\wedge  \La_{\rd \io^\al_I}(\eta_{\perp}) \bigr) \otimes  \bigl( \zeta_{\coker} \wedge \zeta_{\im}
 \bigr) \Bigr) \\
&\;=\; 
\mathfrak{C}_{\rd s_I} \Bigl( \La_{\rd \io^\al_I} 
\bigl( 1\wedge \eta_{\ker} \wedge \eta_{\perp} \bigr) \otimes 
 \bigl( \zeta_{\coker} \wedge \zeta_{\im}
  \bigr) \Bigr) \\
&\;=\;  
\Bigl( \mathfrak{C}_{\rd s_I} \circ \Ti I^\al_I   \Bigr) \Bigl( \bigl( \eta_{\ker} \wedge \eta_{\perp} \bigr) \otimes  \bigl(
\zeta_{\coker} \wedge \zeta_{\im} \bigr) \Bigr). 
\end{align*}
Secondly we check that the contraction isomorphisms for the surjective maps 
$\rd_{\io^\al_I(z,\al)}(s_I +\nu_I)$  and $\rd_z (s^\al_I +\nu^\al_I)$ intertwine $\Ti I^\al_I (z)$ with the boundary isomorphism ${\ti j}^\al_I = \La_{\rd\io^\al_I}\circ \wedge_1$ from $\lm \rT_z Z^\al_I$ to $\lm \rT_{\io^\al_I(z,\al)}Z_I $.
For that purpose we also use the product form of $\nu_I$ in terms of $\nu^\al_I$ to check the corresponding identity of maps $\lm \rT U_I \otimes \bigl( \lm E_I \bigr)^* \big|_{\io^\al_I(Z^\al_I)}  \to \lm \rT Z^\al_I$ at a fixed vector of the form $\bigl( \La_{\rd\io^\al_\eps}(1\wedge \eta_{\ker}) \wedge  \La_{\rd\io^\al_\eps}
(\eta_{\perp}) \bigr) \otimes  \zeta$ with  $\eta_{\ker} \in \lm\ker\rd (s^\al_I + \nu^\al_I)$: 
\begin{align*}
&  \Bigl( (\ti j^\al_I)^{-1} \circ  \mathfrak{C}_{\rd (s_I +\nu_I)} \Bigr) \Bigl( \bigl( \La_{\rd\io^\al_\eps}(1\wedge \eta_{\ker}) \wedge  \La_{\rd\io^\al_\eps}(\eta_{\perp}) \bigr) \otimes  \zeta \Bigr) 
  \\
&\;=\; 
\mathfrak{c}_{\rd (s_I +\nu_I)} \bigl(\La_{\rd\io^\al_\eps}( \eta_{\perp}) , \zeta \bigr)
 \cdot (\ti j^\al_I)^{-1} \bigl( \La_{\rd\io^\al_\eps}(1\wedge \eta_{\ker})\bigr) 
\\
&\;=\;  \mathfrak{c}_{\rd (s^\al_I +\nu^\al_I)} \bigl( \eta_{\perp} ,\zeta \bigr) \cdot
\eta_{\ker} 
\\
&\;=\;  \mathfrak{C}_{\rd (s^\al_I +\nu^\al_I)} \Bigl(  \bigl(\eta_{\ker} \wedge \eta_{\perp} \bigr) \otimes  \zeta   \Bigr)
\\
& \;=\; 
 \Bigl( \mathfrak{C}_{\rd (s^\al_I +\nu^\al_I)} \circ  \bigl(\Ti I^\al_I \bigr)^{-1} \Bigr) \Bigl( \bigl( \La_{\rd\io^\al_\eps}(1\wedge \eta_{\ker}) \wedge  \La_{\rd\io^\al_\eps}(\eta_{\perp}) \bigr) \otimes  \zeta \Bigr) .
\end{align*}
This proves \eqref{oclaim} and hence finishes the proof.
\end{proof}

%%%%%%%%%%%%%%%%%%%%%%%%%%%%%%%%%%%%%%%%%%%%%%%%%%%%%%%%%%%%%%%%%%%%%%%%
\subsection{Construction of Virtual Moduli Cycle and Fundamental Class}\label{ss:VFC}   \hspace{1mm}\\ \vspace{-3mm}
%%%%%%%%%%%%%%%%%%%%%%%%%%%%%%%%%%%%%%%%%%%%%%%%%%%%%%%%%%%%%%%%%%%%%%%%

We are finally in a position to prove Theorem B in the introduction.
We begin with its first part, which defines the virtual moduli cycle (VMC) as a cobordism class of closed oriented manifolds. 
After a discussion of \v{C}ech homology, we then construct the virtual fundamental class (VFC) as \v{C}ech homology class.

\begin{rmk} \rm
As in Remark~\ref{rmk:iso}, the following constructions of VMC/VFC will also be applied -- with very minor adjustments -- in the context of nontrivial isotropy in \cite{MW:iso}. In order to provide a verifiably rigorous proof for the corresponding result \cite[Theorems~3.3.4, 3.3.5]{MW:iso}, we indicate the necessary adjustment in a series of footnotes [$^{\rm NN}$] in the present section. These should only be read after becoming familiar with the construction of the pruned domain category $\Bb_\Kk|_\Vv^{\less\Ga}$ in \cite[Lemma~3.2.3]{MW:iso} and notion of admissible perturbation in \cite[Definition~3.2.4]{MW:iso}. The adjustments will be rather few after the setting is given as in Remark~\ref{rmk:iso}.
$\hfill\er$
\end{rmk}

\begin{thm}\label{thm:VMC1}
Let $X$ be a compact metrizable space.
\begin{enumerate}
\item
Let $\Kk$ be an oriented, additive weak Kuranishi atlas of dimension $D$ on $X$ (see Definitions~\ref{def:CKS}, \ref{def:Ku2}).
Then there exists a preshrunk tame shrinking $\Kk_{\rm sh}$ of $\Kk$ (see Definition~\ref{def:shr}), an admissible metric on $|\Kk_{\rm sh}|$ (see Definition~\ref{def:metric}), a reduction $\Vv$ of $\Kk_{\rm sh}$ (see Definition~\ref{def:vicin}), and a strongly adapted, admissible, precompact, transverse perturbation $\nu$ of $\s_{\Kk_{\rm sh}}|_\Vv$ (see Definitions~\ref{def:sect}, \ref{def:sect2}, \ref{def:precomp}). 
\item
For any choice of data as in (i), the perturbed zero set $|\bZ^\nu|$ 
is an oriented compact manifold (without boundary) of dimension $D$.
\item
Let $\Kk^0,\Kk^1$ be two oriented, additive weak Kuranishi atlases 
(on possibly different compact metrizable spaces) that are oriented, additively cobordant
(see Definition~\ref{def:Kcobord}). Then, for any choices of strongly adapted perturbations $\nu^\al$ in the sense of Definition~\ref{a-e rel} 
for $\al=0,1$, the perturbed zero sets are cobordant (as oriented closed manifolds), $|\bZ^{\nu^0}| \sim |\bZ^{\nu^1}|$.
\end{enumerate}
\end{thm}

\begin{proof}  
Part (i) of Theorem~\ref{thm:K} provides a preshrunk tame shrinking $\Kk_{\rm sh}$ of $\Kk$ which by part~(ii) of that theorem can be equipped with an admissible metric. 
The orientation of $\Kk$ then induces an orientation of $\Kk_{\rm sh}$ by Lemma~\ref{le:cK}. 
Moreover, Theorem~\ref{thm:red}~(i) provides a reduction $\Vv$ of $\Kk_{\rm sh}$, and by 
Theorem~\ref{thm:red}~(iii)~(a) with $\Ww=|\Kk|$ 
we find another reduction $\Cc$ with precompact inclusion $\Cc\sqsubset \Vv$,
i.e.\ a nested reduction. 
Then we may apply Proposition~\ref{prop:ext} with $\si=\si_{\rm rel}([0,1]\times \Vv,[0,1]\times  \Cc,\|\cdot\|,\de)$ to obtain a strongly adapted, admissible, transverse perturbation $\nu$ with $\pi_\Kk\bigl((\s_\Kk|_\Vv+\nu)^{-1}(0)\bigr) \subset \pi_{\Kk}(\Cc)$. This proves (i).

Part (ii) holds in this setting since Proposition~\ref{prop:zeroS0} shows that $|\bZ^\nu|$ is a smooth closed $D$-dimensional manifold, which is oriented by Proposition~\ref{prop:orient1}~(i). 

To prove (iii) we will use transitivity of the cobordism relation for oriented closed manifolds to prove increasing independence of choices in the following Steps 1--4.

\MS\NI
{\bf Step 1:} 
{\it
For a fixed oriented, metric, tame Kuranishi atlas $(\Kk,d)$, nested reductions $\Cc\sqsubset\Vv$, 
additive norms $\|\cdot\|$, 
and $0<\de<\de_\Vv$, $0<\si\leq \si_{\rm rel}([0,1]\times\Vv,[0,1]\times\Cc,\|\cdot\|,\de)$, the cobordism class of $|\bZ^\nu|$ is independent of the choice of $(\Vv,\Cc,\|\cdot\|,\de,\si)$-adapted perturbation~$\nu$.
}

\MS
To prove this we fix $(\Kk,d)$, $\Cc\sqsubset\Vv$, 
$\|\cdot\|$,
$\de$, and $\si$, consider two $(\Vv,\Cc,\|\cdot\|,\de,\si)$-adapted perturbations $\nu^0,\nu^1$, and need to find an oriented cobordism $|\bZ^{\nu^0}|  \sim |\bZ^{\nu^1}|$.
For that purpose we apply Proposition~\ref{prop:ext2}~(ii) to the Kuranishi cobordism $[0,1]\times \Kk$ with product metric and nested product reductions $[0,1]\times\Cc\sqsubset [0,1]\times\Vv$ to obtain an admissible, precompact, transverse cobordism perturbation $\nu^{01}$ of $\s_{[0,1]\times\Kk}|_{[0,1]\times\Vv}$ with boundary restrictions $\nu^{01}|_{\{\al\}\times \Vv}=\nu^\al$ for $\al=0,1$.
Here we use the fact that $\de_{[0,1]\times\Vv}=\de_\Vv>\de$ by Lemma~\ref{le:admin}~(ii). 
Moreover, by Lemma~\ref{le:cK}~(iii) the orientation of $\Kk$ induces an orientation of $[0,1]\times\Kk$, whose restriction to the boundaries $\p^\al([0,1]\times \Kk) =\Kk$ equals the given orientation on $\Kk$. 
Finally, Lemma~\ref{le:czeroS0} together with Proposition~\ref{prop:orient1}~(ii) imply that $|\bZ^{\nu^{01}}|$ is the required oriented cobordism from $|\bZ^{\nu^{01}|_{\{0\}\times \Vv}}|=|\bZ^{\nu^0}|$ to $|\bZ^{\nu^{01}|_{\{1\}\times\Vv}}|=|\bZ^{\nu^1}|$.

\MS\NI
{\bf Step 2:} {\it 
For a fixed oriented, metric, tame Kuranishi atlas $(\Kk,d)$ and nested reductions $\Cc\sqsubset\Vv$, the cobordism class of $|\bZ^\nu|$ is independent of the choice of strongly $(\Vv,\Cc)$-adapted perturbation~$\nu$. }

\MS
To prove this we fix $(\Kk,d)$ and $\Cc\sqsubset\Vv$ and consider two strongly $(\Vv,\Cc)$-adapted perturbations~$\nu^\al$ for $\al = 0,1$.
Thus $\nu^\al$ is $(\Vv,\Cc,\|\cdot\|^\al,\de^\al,\si^\al)$-adapted for some choices of additive norms $\|\cdot\|^\al$ and constants $0<\de^\al<\de_\Vv$ and  $0<\si^\al\leq \si_{\rm rel}{([0,1]\times\Vv,[0,1]\times\Cc,\|\cdot\|^\al,\de^\al)}$. 
Then we need to find an oriented cobordism $|\bZ^{\nu^0}| \sim |\bZ^{\nu^1}|$.
To do this, first note that we evidently have $\de:=\max(\de^0,\de^1)<\de_\Vv=\de_{[0,1]\times\Vv}$  
by Lemma~\ref{le:admin}~(ii) with respect to the product metric on $[0,1]\times |\Kk|$.
Moreover, since all norms on the finite dimensional obstruction spaces are equivalent, we can (e.g.\ by scaling) find additive norms $\|\cdot\|$ on $\Kk$ such that $\|\cdot\|^\al\leq\|\cdot\|$ for $\al=0,1$.
Now choose $\si \leq \min\{ \si^0,\si^1, \si_{\rm rel}([0,1]\times\Vv,[0,1]\times\Cc,\|\cdot\|,\de) \}$. Then Proposition~\ref{prop:ext2}~(i) provides an admissible, precompact, transverse cobordism perturbation $\nu^{01}$ of $\s_{[0,1]\times\Kk}|_{[0,1]\times\Vv}$, whose restrictions $\nu^{01}|_{\{\al\}\times \Vv}$ for $\al=0,1$ are $(\Vv,\Cc,\|\cdot\|,\de,\si)$-adapted perturbations of $\s_\Kk|_\Vv$.
Since $\de^\al\leq \de$, $\|\nu^{01}|_{\{\al\}\times \Vv}\|^\al \leq \|\nu^{01}|_{\{\al\}\times \Vv}\|<\si$ and $\si \leq \si^\al \leq \si_{\rm rel}([0,1]\times\Vv,[0,1]\times\Cc,\|\cdot\|,\de^\al)$, they are also $(\Vv,\Cc,\|\cdot\|,\de^\al,\si^\al)$-adapted; see Lemma~\ref{le:admin2}.
Then, as in Step~1, the perturbed zero set $|\bZ^{\nu^{01}}|$ is an oriented cobordism from $|\bZ^{\nu^{01}|_{\{0\}\times \Vv}}|$ to $|\bZ^{\nu^{01}|_{\{1\}\times \Vv}}|$.
Moreover, for fixed $\al\in\{0,1\}$ both the restriction $\nu^{01}|_{\{\al\}\times \Vv}$ and the given perturbation $\nu^\al$ are $(\Vv,\Cc,\|\cdot\|,\de^\al,\si^\al)$-adapted, so that Step 1 provides cobordisms $|\bZ^{\nu^0}|  \sim |\bZ^{\nu^{01}|_{\{0\}\times \Vv}}|$ and $|\bZ^{\nu^{01}|_{\{1\}\times \Vv}}| \sim |\bZ^{\nu^1}|$. By transitivity of the cobordism relation this proves $|\bZ^{\nu^0}| \sim |\bZ^{\nu^1}|$ as claimed.

\MS\NI
{\bf Step 3:} {\it
For a fixed oriented, tame Kuranishi atlas $\Kk$, the oriented cobordism class of $|\bZ^\nu|$ is independent of the choice of admissible metric and strongly adapted perturbation~$\nu$.
}

\MS
To prove this we fix $\Kk$ and consider two strongly $(\Vv^\al,\Cc^\al)$-adapted perturbations $\nu^\al$ with respect to nested reductions $\Cc^\al\sqsubset\Vv^\al$, additive norms $\|\cdot\|^\al$, 
and admissible metrics $d^\al$ for $\al=0,1$.
To find an oriented cobordism $|\bZ^{\nu^0}| \sim |\bZ^{\nu^1}|$ we begin by using 
\cite[Proposition~4.2.3]{MW:top} to find an admissible metric $d$ on $|[0,1]\times  \Kk|$ with $d|_{\{\al\}\times |\Kk|} =d^\al$.
As before, we also pick additive norms $\|\cdot\|$ on $\Kk$ such that $\|\cdot\|^\al\leq\|\cdot\|$ for $\al=0,1$.
Next, we use Theorem~\ref{thm:red}~(iv) with $\Ww=|[0,1]\times \Kk|$ to find a nested cobordism reduction $\Cc\sqsubset\Vv$ of $[0,1]\times \Kk$ with $\p^\al\Cc=\Cc^\al$ and $\p^\al\Vv=\Vv^\al$.
If we now pick any $0<\de<\de_\Vv$ smaller than the collar width of $d$, $\Vv$, and $\Cc$, then we automatically have $\de<\de_{\Vv^\al}$ by Lemma~\ref{le:admin}~(iii).
Then, for any $0<\si\leq\si_{\rm rel}(\Vv,\Cc,\|\cdot\|,\de)$, Proposition~\ref{prop:ext2}~(i) provides an admissible, precompact, transverse cobordism perturbation $\nu^{01}$ of $\s_{[0,1]\times  \Kk}|_\Vv$ whose restrictions $\nu^{01}|_{\p^\al\Vv}$ for $\al=0,1$ are $(\Vv^\al,\Cc^\al,\|\cdot\|,\de,\si)$-adapted perturbations of $\s_{\Kk}|_{\Vv^\al}$.
As in Step 1, the perturbed zero set $|\bZ^{\nu^{01}}|$ is an oriented cobordism from $|\bZ^{\nu^{01}|_{\p^0\Vv}}|$ to $|\bZ^{\nu^{01}|_{\p^1\Vv}}|$. 
Moreover, we can pick $\si\leq\si_{\rm rel}([0,1]\times  \Vv^\al,[0,1]\times  \Cc^\al,\|\cdot\|^\al,\de)$ 
for $\al=0,1$. 
Then we have $\| \nu^{01}|_{\p^\al\Vv} \|^\al \leq \| \nu^{01}|_{\p^\al\Vv} \|<\si$ so that each $\nu^{01}|_{\p^\al\Vv}$ is strongly $(\Vv^\al,\Cc^\al)$-adapted with respect to the metric  $d^\al$;
see Lemma~\ref{le:admin2}.
Now Step 2 applies for $\al=0$ as well as $\al=1$ to provide cobordisms $|\bZ^{\nu^0}|  \sim |\bZ^{\nu^{01}}|_{\p^0\Vv}|$ and $|\bZ^{\nu^{01}}|_{\p^1\Vv}| \sim |\bZ^{\nu^1}|$, which proves the claim by transitivity.

\MS\NI
{\bf Step 4:} 
{\it
Let $\Kk$ be an oriented, additive, weak Kuranishi cobordism, and for $\al=0,1$ let $\nu^\al$ be strongly adapted perturbations of some preshrunk tame shrinking $\Kk_{\rm sh}^\al$ of $\p^\al\Kk$ with respect to some choice of admissible metric on $|\p^\al\Kk|$. Then there is an oriented cobordism of compact manifolds $|\bZ^{\nu^0}|  \sim |\bZ^{\nu^1}|$.
}

\MS
This is proven along the lines of parts (i) and (ii) by first using Theorem~\ref{thm:K} to find a preshrunk tame shrinking $\Kk_{\rm sh}$ of $\Kk$ with $\p^\al\Kk_{\rm sh}=\Kk^\al_{\rm sh}$, and an admissible metric $d$ on $|\Kk_{\rm sh}|$. If we equip $\Kk_{\rm sh}$ with the orientation induced by $\Kk$, then by Lemma~\ref{le:cK} the induced boundary orientation on $\p^\al\Kk_{\rm sh}=\Kk^\al_{\rm sh}$ agrees with that induced by shrinking from $\Kk^\al$.
Next, 
Theorem~\ref{thm:red}~(ii)
provides a reduction $\Vv$ of $\Kk_{\rm sh}$, and by
Theorem~\ref{thm:red}~(iv) with $\Ww=|\Kk_{\rm sh}|$
we find a nested cobordism reduction $\Cc\sqsubset \Vv$. 
Now we may choose additive norms $\|\cdot\|$ on $\Kk_{\rm sh}$ and
 apply Proposition~\ref{prop:ext2}~(i) with 
$$
\si\;=\; \min\bigl\{  \si_{\rm rel}(\Vv,\Cc,\|\cdot\|,\de) ,\;  \min_{ \al=0,1} \si_{\rm rel}([0,1]\times  \p^\al\Vv, [0,1]\times  \p^\al\Cc,\p^\al\|\cdot\|,\de) \bigr\}
$$
to find an admissible, precompact, transverse cobordism perturbation $\nu$ of $\s_{\Kk_{\rm sh}}|_\Vv$, whose restrictions $\nu|_{\p^\al\Vv}$ for $\al=0,1$ are $(\p^\al\Vv,\p^\al\Cc,\p^\al\|\cdot\|,\de,\si)$-adapted perturbations of $\s_{\Kk^\al_{\rm sh}}|_{\p^\al\Vv}$. In particular, these are strongly adapted by the choice of $\si$.
Also, as in the previous steps, $|\bZ^{\nu}|$ is an oriented cobordism from $|\bZ^{\nu|_{\p^0\Vv}}|$ to $|\bZ^{\nu|_{\p^1\Vv}}|$. 
Finally, Step 2 applies to the fixed oriented, tame Kuranishi atlases $\Kk^\al_{\rm sh}$ for fixed $\al\in\{0,1\}$ to provide cobordisms $|\bZ^{\nu^0}|  \sim |\bZ^{\nu}|_{\p^0\Vv}|$ and $|\bZ^{\nu}|_{\p^1\Vv}| \sim |\bZ^{\nu^1}|$. By transitivity, this finishes the proof of Theorem~\ref{thm:VMC1}.
\end{proof}

One possible definition of the virtual fundamental class (VFC) is as the cobordism class of the zero set $|\bZ^\nu|$
constructed in the previous theorem. If we think of this as an abstract manifold and hence as representing an element in the $D$-dimensional oriented cobordism ring, it contains rather little information.
Such a notion is barely sufficient for the basic constructions of e.g.\ Floer differentials $\p_F$ from counts of moduli spaces with $D=0$, and proofs of algebraic relations such as $\p_F\circ\p_F=0$ by cobordisms with $D=1$.
If $X = \oMm_{g,k}(A,J)$ is the Gromov--Witten moduli space of $J$-holomorphic curves of genus $g$, homology class $A$, and with $k\geq 1$ marked points, one can construct the domains $U_I$ of the Kuranishi charts for $X$ to have elements that are tuples, one component of which is a  $k$-pointed stable map to $(M,\om)$, so that there are evaluation maps $\ev_I: U_I \to M^k$; see~\cite{MW:iso,Mcn} for example.
Further, the coordinate changes can be made compatible with these evaluation maps, and Kuranishi cobordisms can be constructed so that the evaluation maps extend over them. Hence, after shrinking to a tame Kuranishi atlas
(or cobordism) $\Kk_{sh}$, there is a continuous evaluation map
$$
\ev: |\Kk_{sh}| \to M^k
$$
both for the fixed tame shrinking used to define $|\bZ^\nu|$ and for any shrinking of a weak Kuranishi cobordism compatible with evaluation maps.  Therefore, for any admissible, precompact, transverse perturbation $\nu$, the map $\ev:|\bZ^\nu|\to M^k$ can be considered as a $D$-dimensional cycle (the virtual moduli cycle VMC) in the singular homology of $M^k$, or even as a cycle in the oriented bordism theory of $M^k$. Thus in this case a possible definition of the VFC is as the corresponding singular homology (or bordism) class in $M^k$.
One could also take into account the forgetful map 
to the Deligne--Mumford space $\oMm_{0,k}$ 
formed by the domains of the stable maps, 
to obtain
a class in the homology (or bordism) groups of $\oMm_{0,k}\times M^k$.

However, we will take a different route, interpreting the VFC more intrinsically as an 
element in the rational \v{C}ech homology $\check{H}_D(X;\Q)$ of
the compact metrizable space $X$.
As a first step, we associate to every oriented, metric, tame Kuranishi atlas a $D$-dimensional homology class in any open neighbourhood $\Ww\subset |\Kk|$ of $\io_\Kk(X)$. For that purpose recall from \eqref{eq:Zinject} that for any precompact, transverse perturbation $\nu$ of $\s_\Kk|_\Vv$, the inclusion 
$(\s_\Kk|_\Vv+\nu)^{-1}(0)\subset\Vv = \Obj_{\bB_\Kk|_\Vv}$ induces a continuous injection $i^\nu: |\bZ^\nu| \to |\Kk|$, which we now compose with the continuous bijection $|\Kk| \to (|\Kk|,d)$ from Lemma~\ref{le:metric} to obtain a continuous injection
\begin{equation}\label{ionu}
i^\nu \,:\;  |\bZ^\nu| \;\longrightarrow\; \bigl(|\Kk|,d\bigr) .
\end{equation}
Since $ |\bZ^\nu|$ is compact and the restriction of the metric topology to the image $i^\nu(|\bZ^\nu|)\subset(|\Kk|,d)$ is Hausdorff, this map is in fact a homeomorphism to its image; see Remark~\ref{rmk:hom}, and compare with Proposition~\ref{prop:zeroS0} which notes that $i^\nu:|\bZ^\nu|\to |\Kk|$ is a homeomorphism to its image.
If moreover $|\bZ^\nu|$ is oriented, then it has a fundamental class $\bigl[|\bZ^\nu|\bigr]\in H_D(|\bZ^\nu|)$.
Now we obtain a homology class by pushforward into any appropriate subset of $(|\Kk|,d)$,
$$
[i^\nu] :=  (i^\nu)_* \bigl[|\bZ^\nu|\bigr] \in H_D(\Ww) \qquad\text{for} \quad  i^\nu(|\bZ^\nu|) \subset \Ww \subset  \bigl(|\Kk|,d\bigr) .
$$
Analogously, any precompact, transverse perturbation $\nu^{01}$ of a metric, tame Kuranishi cobordism 
$\bigl(\Kk^{01}, d\bigr)$ gives rise to a topological embedding
\begin{equation}\label{cionu}
i^{\nu^{01}} \,:\;  |\bZ^{\nu^{01}}| \;\longrightarrow\; \bigl( |\Kk^{01}| , d \bigr) .
\end{equation}
Now by Lemma~\ref{le:czeroS0} the boundary 
$\partial |\bZ^{\nu^{01} }|$ 
of the cobordism $ |\bZ^{\nu^{01} }|$ has two disjoint (but not necessarily connected) components
$$
\p |\bZ^{\nu^{01} }| \;=\;  \p^0 |\bZ^{\nu^{01}}| \;\cup\;  \p^1|\bZ^{\nu^{01} }|, \qquad
\p^\al |\bZ^{\nu^{01}}| \,:=\;  \p^\al |\Kk^{01} |  \cap |\bZ^{\nu^{01} }| .
$$
In fact, we also showed there that the embeddings
\begin{equation}\label{eq:Jal}
J^\al:=\io^\al_{|\Kk^{01}|}(\al,\cdot)\;:\;\;  |\p^\al\Kk^{01} |  \;\longrightarrow \; \p^\al|\Kk^{01} | \subset |\Kk^{01} |
\end{equation}
(where we slightly abuse notation by identifying $ |\p^\al\Kk^{01} | $ with $\{\al\}\times  |\p^\al\Kk^{01} | $)
restrict to diffeomorphisms
$$
 |j^\al|  \,:\; |\bZ^{\nu^\al}| \;\longrightarrow\;\p^\al|\bZ^{\nu }|
\qquad\text{with}\quad
i^{\nu^{01} } \circ |j^\al| =  J^\al
 \circ i^{\nu^\al} ,
$$
where $\nu^\al:= \nu^{01} |_{\p^\al\Kk^{01} }$ are the restricted perturbations of the Kuranishi atlases $\p^\al\Kk^{01} $.
Moreover, Proposition~\ref{prop:orient1}~(ii) asserts that the boundary orientations on $\p^\al |\bZ^{\nu^{01} }|$ 
(which are induced by the orientation of the 
cobordism
$|\bZ^{\nu^{01} }|$ arising from the orientation of $\Kk^{01} $) are related to 
the orientations of $|\bZ^{\nu^\al}|$ (which are induced by the orientation of $\p^\al\Kk^{01} $ obtained by restriction from the orientation of $\Kk^{01} $) by
$$
|j^0| \,:\; |\bZ^{\nu^0}|^-  \;\overset{\cong}{\longrightarrow}\;  \p^0 |\bZ^{\nu^{01} }| \qquad
\text{and}\qquad
|j^1| \,:\;  |\bZ^{\nu^1}| \;\overset{\cong}{\longrightarrow}\; \p^1 |\bZ^{\nu^{01} }|.
$$
In terms of the fundamental classes this yields the identity
\begin{align*}
|j^1|_*\bigl[ |\bZ^{\nu^1}|\bigr]  - |j^0|_*\bigl[ |\bZ^{\nu^0}|\bigr] &\;=\; 
\bigl[\p^1 |\bZ^{\nu^{01} }|\bigr]  + \bigl[\p^0 |\bZ^{\nu^{01} }|\bigr] \\
&\;=\; \bigl[\p |\bZ^{\nu^{01} }|\bigr]  \;=\; \delta \bigl[|\bZ^{\nu^{01} }|\bigr] \;\in\; H_D( \p |\bZ^{\nu^{01} }|)
\end{align*}
for the boundary map $\delta: H_{D+1}(\bigl[ |\bZ^{\nu^{01} }|\bigr] , \p \bigl[ |\bZ^{\nu^{01} }|\bigr] ) \to H_D( \p |\bZ^{\nu^{01} }|)$ that is part of the long exact sequence for $\p |\bZ^{\nu^{01} }|\subset  |\bZ^{\nu^{01}}|$.
Inclusion to $|\bZ^{\nu^{01} }|$ now provides, by exactness of this sequence,
$|j^1|_*\bigl[ |\bZ^{\nu^1}|\bigr]  - |j^0|_*\bigl[ |\bZ^{\nu^0}|\bigr] = 0  \in H_D( |\bZ^{\nu^{01} }|)$.
Finally, we can push this forward by  $i^{\nu^{01} }$ to $H_D(|\Kk^{01} |)$ and use the identity
$i^{\nu^{01} } \circ |j^\al| = J^\al \circ i^{\nu^\al}$ to obtain
\begin{align*}
0 &\;=\; 
(i^{\nu^{01}})_*|j^1|_*\bigl[ |\bZ^{\nu^1}|\bigr]  - (i^{\nu^{01} })_*|j^0|_*\bigl[ |\bZ^{\nu^0}|\bigr] \\
&\;=\;
|J^1|_*(i^{\nu^{1}})_*\bigl[ |\bZ^{\nu^1}|\bigr]  - |J^0|_*(i^{\nu^{0}})_*\bigl[ |\bZ^{\nu^0}|\bigr] 
\;=\;
|J^1|_*\bigl[i^{\nu^{1}}\bigr]  - |J^0|_*\bigl[i^{\nu^{0}}\bigr] .
\end{align*}
The same holds in $H_D(\Ww^{01} )$ for any subset $\Ww^{01} \subset\bigl(|\Kk^{01} |,d \bigr)$ that contains $\io^{\nu^{01} }(|\bZ^{\nu^{01} }|)$, that is
\begin{equation} \label{homologous}
J^0_* [i^{\nu^0}] \;=\; J^1_* [i^{\nu^1}] \;\in\; H_D(\Ww^{01}) \qquad\text{when} \quad i^{\nu^{01} }(|\bZ^{\nu^{01} }|) \subset \Ww^{01}  \subset |\Kk^{01} | .
\end{equation}
This will be crucial for proving independence of the VFC from choices.

In the case of a product cobordism $\Kk^{01}=[0,1]\times \Kk$ with product metric and perturbation $\nu^{01}$, we can identify $|\Kk^{01}|\cong [0,1]\times  |\Kk|$ so that \eqref{cionu} also induces a cycle  
\begin{equation} \label{pionu}
\pr_{ |\Kk|}\circ i^{\nu^{01}} \, :\;  |\bZ^{\nu^{01}}|\; \longrightarrow\; (|\Kk|,d) ,
\end{equation}
whose boundary restrictions are $i^{\nu^\al} \circ |j^\al|^{-1}$, so that the above argument directly gives
\begin{equation} \label{pomologous}
[ i^{\nu^0} ] = [i^{\nu^1} ] \in H_D(\Ww) \qquad\text{when} \quad i^{\nu^{01}}(|\bZ^{\nu^{01}}|) \subset [0,1]\times  \Ww.
\end{equation}
Now we can associate a well defined virtual fundamental class to any choice of open neighbourhood  $\Ww$ of $X$ in the virtual neighbourhood $|\Kk|$ induced by a fixed tame Kuranishi atlas $\Kk$.

\begin{lemma}\label{le:VMC1}
Let $(\Kk,d)$ be an oriented, metric, tame Kuranishi atlas and let $\Ww\subset |\Kk|$ be an open subset with respect to the metric topology such that $\io_\Kk(X)\subset \Ww$. Then there exists a strongly adapted perturbation $\nu$ such that $\pi_\Kk((\s_\Kk|_\Vv+\nu)^{-1}(0))\subset \Ww$. More precisely, there is a strongly $(\Vv,\Cc)$-adapted perturbation $\nu$ for some nested reduction 
$\Cc\sqsubset\Vv$ such that $\pi_\Kk(\Cc)\subset\Ww$.
For any such perturbation, the inclusion of the perturbed zero set $i^\nu : |\bZ^\nu|\hookrightarrow \Ww \subset (|\Kk|,d)$ defines a singular homology class
$$
A^{(\Kk,d)}_{\Ww} \,:=\; \bigl[i^\nu : |\bZ^\nu| \to \Ww \bigr] \;\in\; H_D(\Ww)
$$
that is independent of the choice of reductions and perturbation.
\end{lemma}

\begin{proof}
To see that the required perturbations exist, 
we use Theorem~\ref{thm:red}~(iii)~(a) to choose a nested reduction $\Cc \sqsubset \Vv$ of $\Kk$ such that $\pi_\Kk(\Cc)\subset\Ww$.
Now Proposition~\ref{prop:ext} guarantees the existence of a strongly $(\Vv,\Cc)$-adapted perturbation $\nu$. By Proposition~\ref{prop:zeroS0} and Proposition~\ref{prop:orient1}~(i) its perturbed zero set is an oriented manifold $|\bZ^\nu|$. Moreover, the image of $i^{\nu} : |\bZ^{\nu}| \to (|\Kk|,d)$ is $\pi_\Kk\bigl((\s_\Kk|_\Vv+\nu)^{-1}(0)\bigr) \subset\pi_\Kk(\Cc) \subset \Ww$, so that by the discussion above $i^{\nu} : |\bZ^{\nu}| \to \Ww$ defines a cycle $\bigl[i^\nu \bigr] \in H_D(\Ww)$.

To prove independence of the choices, we need to show that $\bigl[i^{\nu^0}\bigr]=\bigl[i^{\nu^1}\bigr]$ for any two strongly $(\Vv^\al,\Cc^\al)$-adapted perturbations $\nu^\al$ of $\s_\Kk|_{\Vv^\al}$ with $\pi_\Kk(\Cc^\al)\subset \Ww$. 
We do this by modifying Steps 1--3
 in the proof of Theorem~\ref{thm:VMC1} so that 
at each step the cycle $i^{\nu^{01}}:|\bZ^{\nu^{01}}| \to |\Kk^{01}|=|[0,1]\times  \Kk|$ given by \eqref{cionu} takes values in $[0,1]\times \Ww\subset |[0,1]\times  \Kk|$.
Note that here we use the product metric on $[0,1]\times |\Kk|$ so that $[0,1]\times  \Ww$ is open.
Then in each step the composite map $\pr_{|\Kk|} \circ i^{\nu^{01}} : |\bZ^{\nu^{01}}| \to |\Kk|$ takes values in $\Ww$, so that \eqref{pomologous} applies to give 
$\bigl[i^{\nu^{01}|_{\p^0\Vv}}\bigr]=\bigl[i^{\nu^{01}|_{\p^1\Vv}}\bigr]\in H_D(\Ww)$.
By transitivity of equality in $H_D(\Ww)$, Steps 1--3 
then prove $\bigl[i^{\nu^0}\bigr]=\bigl[i^{\nu^1}\bigr]$.

In Steps 1 and 2, the required inclusion is automatic since the perturbations are constructed so that 
$|(\s_{[0,1]\times \Kk}|_{[0,1]\times \Vv}+\nu^{01})^{-1}(0)| \subset \pi_{[0,1]\times \Kk}([0,1]\times  \Cc)\subset [0,1]\times \Ww$, where the second inclusion follows from $\pi_\Kk(\Cc)\subset\Ww$.
We apply Step~3 with a a fixed metric $d^0=d^1=d$ and nested reductions $\Cc^\al\sqsubset\Vv^\al$ for $\al=0,1$ with $\pi_\Kk(\Cc^\al)\subset\Ww$.
Then we equip $|\Kk^{01}| \cong [0,1]\times  |\Kk|$ with the product metric and use Theorem~\ref{thm:red}~(iv)~(a) to choose a nested cobordism reduction $\Cc\sqsubset \Vv$ such that $\p^\al \Cc=\Cc^\al$, $\p^\al \Vv=\Vv^\al$, and $\pi_{[0,1]\times \Kk}(\Cc)\subset [0,1]\times  \Ww$.
Using the nested reduction $\Cc\sqsubset\Vv$ in choosing the cobordism perturbation $\nu^{01}$ then ensures that $i^{\nu^{01}}: |\bZ^{\nu^{01}}|\to |[0,1]\times \Kk|= [0,1]\times |\Kk|$  takes values in $\pi_{[0,1]\times  \Kk}(\Cc)\subset [0,1]\times \Ww$, as required to finish the proof.
\end{proof}

To construct the VFC as a homology class in $\io_\Kk(X)$ for tame Kuranishi atlases, and later in $X$,
we use rational \v{C}ech homology, rather than integral \v{C}ech or singular homology, because it has the following continuity property.

\begin{remark} \label{rmk:Cech}
Let $X$ be a compact subset of a metric space $Y$, and let
$(\Ww_k\subset Y)_{k\in\N}$ be a sequence of open subsets that is nested, $\Ww_k\subset\Ww_{k-1}$, such that $X=\bigcap_{k\in\N} \Ww_k$.
Then the system of maps $\check{H}_n(X;\Q)  \to \check{H}_n(\Ww_{k+1};\Q)\to \check{H}_n(\Ww_k;\Q)$ induces an isomorphism
$$
\check{H}_n(X;\Q) \;\overset{\cong}{\longrightarrow}\; \underset{\leftarrow }\lim\, \check{H}_n(\Ww_k;\Q)  .
$$

{\rm
To see that singular homology does not have this property, let $X\subset \R^2$ be the union of  the line segment $\{0\}\times [-1,1]$, the graph $\{(x,\sin \tfrac \pi x) \,|\, 0<x\le 1\}$, and an embedded curve joining $(0,1)$ to $(1,0)$ that is otherwise disjoint from the line segment and graph.
Then $H_1^{sing}(X;\Q) = 0$ since it is the abelianization of the trivial fundamental group.  However, $X$ has arbitrarily small neighbourhoods $U\subset\R^2$ with $H_1^{sing}(U;\Q) =\Q$.

Note that we cannot work with integral \v{C}ech homology since it does not even satisfy the exactness axiom (long exact sequence for a pair), because of problems with the inverse limit operation; see the discussion of \v{C}ech cohomology in Hatcher~\cite{Hat}, and \cite[Proposition~3F.5]{Hat} for properties of inverse limits.
However, rational \v{C}ech homology does satisfy the exactness axiom, and because it is dual to
\v{C}ech cohomology has the above stated continuity property by Spanier~\cite[Ch.6~Exercises~D]{Span}.

Further, rational \v{C}ech homology equals rational singular homology for finite simplicial complexes.
Hence the fundamental class of a closed oriented $n$-manifold $M$ can be considered as
an element $[M]\in \check{H}_n(M;\Q)$ in rational \v{C}ech homology and therefore pushes forward under a continuous map $f:M\to X$ to a well defined element $f_*([M])\in \check{H}_n(X;\Q)$.
Note finally that if one wants an integral theory with this continuity property, the correct theory to use is the Steenrod homology theory developed in Milnor~\cite{Mi}. $\hfill\er$
}
\end{remark}

We can now finish the proof of Theorem B in the introduction by constructing the virtual fundamental class in rational \v{C}ech homology, using the above continuity property.

\begin{thm}\label{thm:VMC2}
Let $\Kk$ be an oriented, additive weak Kuranishi atlas of dimension $D$ on a compact, metrizable space $X$.
\begin{enumerate}
\item
Let $\Kk_{\rm sh}$ be a preshrunk tame shrinking of $\Kk$ and $d$ an admissible metric on $|\Kk_{\rm sh}|$. Then there exists a nested sequence of open sets $\Ww_{k+1}\subset \Ww_k\subset \bigl(|\Kk_{\rm sh}|, d\bigr)$ such that $\bigcap_{k\in\N}\Ww_k = \io_{\Kk_{\rm sh}}(X)$.
Moreover, for any such sequence there is a sequence $\nu_k$ of strongly adapted perturbations of $\s_{\Kk_{\rm sh}}$ with respect to nested reductions $\Cc_k\sqsubset\Vv_k$ such that $\pi_\Kk(\Cc_k)\subset\Ww_k$ for all $k$. Then the embeddings
$$
i^{\nu_k} \,:\; |\bZ^{\nu_k}| \;\hookrightarrow \;   \Ww_k \;\subset\; \bigl(|\Kk_{\rm sh}|, d\bigr)
$$
induce a \v{C}ech homology class by inverse limit under the inclusions $\io_{\Kk_{\rm sh}}(X)\subset\Ww_{k+1}\subset\Ww_k$, 
$$
\underset{\leftarrow}\lim\, \bigl[ \,
i^{\nu_k} \, \bigr] \;\in\; \check{H}_D\bigl(i_{\Kk_{\rm sh}}(X);\Q \bigr),
$$
for the subspace $\io_{\Kk_{\rm sh}}(X) =|\s_{\Kk_{\rm sh}}|^{-1}(0)$ of the metric space $\bigl(|\Kk_{\rm sh}|,d\bigr)$.
\vspace{0.05in}
\item
The bijection $|\psi_{\Kk_{\rm sh}}| = \io_{\Kk_{\rm sh}}^{-1}: \io_{\Kk_{\rm sh}}(X) \to X$ from Lemma~\ref{le:realization}~(iv) is a homeomorphism with respect to the metric topology on $\io_{\Kk_{\rm sh}}(X)$ so that we can define the {\bf virtual fundamental class (VFC)} of $X$ as the pushforward
$$
[X]^{\rm vir}_\Kk \,:=\; |\psi_{\Kk_{\rm sh}}|_* \bigl( \, \underset{\leftarrow}\lim\, [\, i^{\nu_k} \,] \, \bigr)
\;\in\;
\check{H}_D(X;\Q) .
$$
It is independent of the choice of shrinkings, metric, nested open sets, reductions, and perturbations in (i),
and in fact depends on the weak Kuranishi atlas $\Kk$ on $X$ only up to oriented, additive concordance. 
\item
If $\Kk$ is an oriented, additive weak Kuranishi cobordism of dimension $D+1$ on a compact, metrizable collared cobordism $(Y,\io_Y^0,\io_Y^1)$, then the induced VFC's of the boundary components $\p^0 Y,\p^1 Y$ are homologous in $Y$, 
$$
(\io_Y^0)_*\bigl([\p^0 Y]^{\rm vir}_{\p^0\Kk}\bigr)
\;=\;
(\io_Y^1)_*\bigl([\p^1 Y]^{\rm vir}_{\p^1\Kk}\bigr)
\quad \in \check{H}_D(Y;\Q) .
$$

\end{enumerate}
\end{thm}

\begin{proof}
The existence of shrinkings and metric is guaranteed by Theorem~\ref{thm:VMC1}~(i).
We then obtain nested open sets converging to $\io_{\Kk_{\rm sh}}(X)$ by e.g.\ taking the $\frac 1k$-neighbourhoods $\Ww_k = B_{\frac 1 k}(\io_{\Kk_{\rm sh}}(X))$.
Given any such nested open sets $(\Ww_k)_{k\in\N}$, the existence of strongly adapted perturbations $\nu_k$ with respect to some nested reductions $\Cc_k\sqsubset\Vv_k$ with $\pi_{\Kk_{\rm sh}}(\Cc_k)\subset\Ww_k$ is proven in Lemma~\ref{le:VMC1}.
The latter also shows that the embeddings $i^{\nu_k} : |\bZ^{\nu_k}|\to \Ww_k$ define homology classes $A^{(\Kk_{\rm sh},d)}_{\Ww_k} = [i^{\nu_k} ]\in H_D(\Ww_k;\Q)$, 
which are independent of the choice of reductions $\Cc_k\sqsubset \Vv_k$ and strongly adapted perturbations $\nu_k$.
In particular, the pushforward $H_D(\Ww_{k+1};\Q)\to H_D(\Ww_k;\Q)$ by the inclusion $\Ii_{k+1}:\Ww_{k+1}\to\Ww_k$ maps $A^{(\Kk_{\rm sh},d)}_{\Ww_{k+1}}=[i^{\nu_{k+1}} ]$ to $A^{(\Kk_{\rm sh},d)}_{\Ww_k}$ since any strongly adapted perturbation $\nu_{k+1}$ with respect to a nested reduction $\Cc_{k+1}\sqsubset\Vv_{k+1}$ with $\Cc_{k+1}\subset\pi_{\Kk_{\rm sh}}^{-1}(\Ww_{k+1})$ can also be used as strongly adapted perturbation $\nu_k:=\nu_{k+1}$.
Then we obtain $i^{\nu_k} = \Ii_{k+1} \circ i^{\nu_{k+1}}$, and hence
$A^{(\Kk_{\rm sh},d)}_{\Ww_k}= [i^{\nu_k}] = (\Ii_{k+1})_* [i^{\nu_{k+1}}]$. This shows that the homology classes $A^{(\Kk_{\rm sh},d)}_{\Ww_k}$ form an inverse system and thus have a well defined inverse limit, completing the proof of (i),
$$
A^{(\Kk_{\rm sh},d)}_{(\Ww_k)_{k\in\N}}
\,:=\; \underset{\leftarrow}\lim\, \bigl[ \,i^{\nu_k} \, \bigr]
\;\in\; \check{H}_D\bigl(\io_{\Kk_{\rm sh}}(X);\Q \bigr).
$$
This defines $A^{(\Kk_{\rm sh},d)}_{(\Ww_k)_{k\in\N}}$ as a \v{C}ech homology class in the topological space $\bigl( \io_{\Kk_{\rm sh}}(X), d\bigr)$.

Towards proving (ii), recall first that $|\psi_{\Kk_{\rm sh}}|:\io_{\Kk_{\rm sh}}(X)\to X$ is a homeomorphism with respect to the relative topology induced from the inclusion $\io_{\Kk_{\rm sh}}(X)\subset|\Kk_{\rm sh}|$ by Lemma~\ref{le:realization}~(iv).
That the latter is equivalent to the metric topology on $\io_{\Kk_{\rm sh}}(X)$ follows as in Remark~\ref{rmk:hom} from the continuity of the identity map $|\Kk_{\rm sh}| \to \bigl(|\Kk_{\rm sh}|,d\bigr)$ (see Lemma~\ref{le:metric}), which restricts to a continuous bijection from the compact set $\io_{\Kk_{\rm sh}}(X)\subset|\Kk_{\rm sh}|$ to the Hausdorff space $\bigl( \io_{\Kk_{\rm sh}}(X), d\bigr)$, and thus is a homeomorphism.
To establish the independence of choices, we then argue as in the proof of Theorem~\ref{thm:VMC1}~(iii), with Lemma~\ref{le:VMC1} playing the role of Step~1.

\MS\NI
{\bf Step 2:} {\it
Let $(\Kk,d)$ be an oriented, metric, tame Kuranishi atlas, and let $(\Ww^\al_k)_{k\in\N}$ for $\al=0,1$ be two nested sequences of open sets $\Ww^\al_{k+1}\subset \Ww^\al_k\subset \bigl(|\Kk|,d\bigr)$ whose intersection is  $\bigcap_{k\in\N}\Ww^\al_k = \io_{\Kk}(X)$.
Then we have $A^{(\Kk,d)}_{(\Ww^0_k)_{k\in\N}} =A^{(\Kk,d)}_{(\Ww^1_k)_{k\in\N}}$, and hence
$$
A^{(\Kk,d)} \,:=\; A^{(\Kk,d)}_{(\Ww_k)_{k\in\N}}  \;\in\; \check{H}_D\bigl(\io_\Kk(X);\Q \bigr) ,
$$
given by any choice of nested open sets $(\Ww_k)_{k\in\N}$ converging to $\io_\Kk(X)$, is a well defined \v{C}ech homology class.
}

\MS
To see this note that the intersection $\Ww_k:=\Ww^0_k\cap \Ww^1_k$ is another nested sequence of open sets with $\bigcap_{k\in\N}\Ww_k = \io_{\Kk}(X)$.
We may choose a sequence of 
strongly 
adapted perturbations $\nu_k$ with respect to nested reductions $\Cc_k\sqsubset\Vv_k$ with $\pi_\Kk(\Cc_k)\subset\Ww_k$ to define $A^{(\Kk,d)}_{\Ww_k}=[i^{\nu_k}]$. The perturbations $\nu_k$ then also fit the requirements for the larger open sets $\Ww_k^\al$ and hence the inclusions $\Ii^\al_k : \Ww_k\to\Ww^\al_k$ push $A^{(\Kk,d)}_{\Ww_k}=[i^{\nu_k}]\in H_D(\Ww_k;\Q)$ forward to $A^{(\Kk,d)}_{\Ww^\al_k}=[\Ii^\al_k\circ i^{\nu_k}]\in H_D(\Ww^\al_k;\Q)$. Hence, by the definition of the inverse limit, we have equality
$$
A^{(\Kk,d)}_{(\Ww^0_k)_{k\in\N}} \;=\; A^{(\Kk,d)}_{(\Ww_k)_{k\in\N}} \;=\; A^{(\Kk,d)}_{(\Ww^1_k)_{k\in\N}}
 \;\in\; \check{H}_D\bigl(\io_{\Kk}(X);\Q \bigr) .
$$

\MS\NI
{\bf Step 3:} {\it
Let $\Kk$ be an oriented, metrizable, tame Kuranishi atlas with two admissible metrics $d^0,d^1$.
Then we have $A^{(\Kk,d^0)}= A^{(\Kk_,d^1)}$, and hence
$$
[X]^{\rm vir}_\Kk \,:=\; |\psi_\Kk|_* A^{(\Kk_{\rm sh},d)} \;\in\; \check{H}_D\bigl(X;\Q \bigr) ,
$$
given by any choice of metric, is a well defined \v{C}ech homology class.
}

\MS
As in Step 3 of Theorem~\ref{thm:VMC1}, we find an admissible collared metric $d$ on $|[0,1]\times \Kk|$ with $d|_{\{\al\}\times |\Kk|}=d^\al$.
Next, we proceed exactly as in the following Step 4 in the special case 
$\Kk^{01}_{\rm sh}=[0,1]\times  \Kk$ and $\Kk^0_{\rm sh}=\Kk^1_{\rm sh}=\Kk$ to find strongly adapted perturbations $\nu^\al_k$ of $(|\Kk|,d^\al)$ that define the \v{C}ech homology classes $A^{(\Kk,d^\al)} =\underset{\leftarrow}\lim\, \bigl[ \ i^{\nu^\al_k} \, \bigr]  \in \check H_D(\io_\Kk(X);\Q)$ and satisfy the identity  
$$
J^0_*\,\Bigl(  \underset{\leftarrow}\lim\, \bigl[ i^{\nu^0_k} \bigr] \Bigr)\;=\; J^1_* \,\Bigl(\underset{\leftarrow}\lim\, \bigl[ i^{\nu^1_k} \bigr] \Bigr)
\; \in\; \check H_D( \io_{\Kk^{01}_{\rm sh}}([0,1]\times   X); \Q )
$$
with the topological embeddings $J^\al: ( |\Kk| , d^\al ) \to ( |\Kk^{01}_{\rm sh}|, d)$ of \eqref{eq:Jal}.
To proceed we claim that the pushforwards by $J^\al$ restrict to the same isomorphism
\begin{equation}\label{J01}
\bigl(J^0\big|_{\io_\Kk(X)}\bigr)_* = \bigl(J^1\big|_{\io_\Kk(X)}\bigr)_*
 \;: \; \check H_D( \io_\Kk(X);\Q ) \;\longrightarrow\; \check H_D( \io_{[0,1]\times \Kk}([0,1]\times  X) ;\Q)
\end{equation}
on the compact set $\io_\Kk(X)$,
on which the two metric topologies induced by $d^0,d^1$ are the same,
since they both agree with the relative topology from $\io_\Kk(X)\subset|\Kk|$.
Indeed, the restrictions $J^\al\big|_{\io_\Kk(X)}$ for $\al = 0,1$ are homotopic  via the
continuous family of maps $J^t :\io_\Kk(X)\to \io_{\Kk^{01}_{sh}}([0,1]\times X)$ that arises from the continuous family of maps\footnote
{
There is no particular reason why the metric topology on $|[0,1]\times \Kk|$ should be a product topology in the canonical identification with $[0,1]\times  |\Kk|$.
In fact, the metrics $d^0$ and $d^1$ may well
induce different topologies on $|\Kk|$.
We avoid these issues by homotoping maps to $\io_{\Kk^{01}_{sh}}([0,1]\times  X)\cong [0,1]\times  X $, which always has the product topology by the remarks just before Step~2.
}
$I^t :X\to [0,1]\times  X$, $x\mapsto (t,x)$
by composition with the embeddings $\io_\Kk$ and $\io_{\Kk^{01}_{sh}}$, i.e.\
\[
J^t \,:
\xymatrix{
\io_\Kk(X) \ar@{->}[r]^{\quad |\psi_\Kk|} & X \ar@{->}[r]^{I^t\quad\;\;} & [0,1]\times X\ar@{->}[r]^{\io_{\Kk^{01}_{sh}}\quad\;\;} &\io_{\Kk^{01}_{sh}}
([0,1]\times X).
}
\]
These maps are continuous because $\io_\Kk=|\psi_\Kk|^{-1}$ and similarly $\io_{\Kk^{01}_{sh}}$ are homeomorphisms to their image with respect to the metric topology by the argument at the beginning of the proof of (ii).
Moreover, each $J^t$ is a homotopy equivalence because, up to homeomorphisms, it 
is equal to the homotopy equivalence $I^t$.
This proves \eqref{J01}, which we can then use to deduce the claimed identity
$$
A^{(\Kk,d^0)} \;=\;
\underset{\leftarrow}\lim\, \bigl[ \,
i^{\nu^0_k} \, \bigr]
 \;=\;
\underset{\leftarrow}\lim\, \bigl[ \,
i^{\nu^1_k} \, \bigr]
\;=\;A^{(\Kk,d^1)}
\quad\in \check H_D( \io_\Kk(X) ;\Q).
$$

\MS\NI
{\bf Step 4:} {\it
Let $\Kk$ be an oriented, additive, weak Kuranishi cobordism, and let $\Kk_{\rm sh}^\al$ be preshrunk tame shrinkings of $\p^\al\Kk$ for $\al=0,1$. Then we have
$$
(\io^0_{Y})_*\bigl([\p^0 Y]^{\rm vir}_{\Kk_{\rm sh}^0}\bigr) 
\;=\; 
(\io^1_{Y})_*\bigl([\p^1 Y]^{\rm vir}_{\Kk_{\rm sh}^1}\bigr)
\quad\in \check{H}_D(Y;\Q).
$$
}

As in Step 4 of Theorem~\ref{thm:VMC1}, we find a preshrunk tame shrinking $\Kk^{01}_{\rm sh}$ of $\Kk$ with $\p^\al\Kk^{01}_{\rm sh}=\Kk^\al_{\rm sh}$, and an admissible collared metric $d$ on $|\Kk^{01}_{\rm sh}|$.
We denote its restrictions to $|\Kk^\al_{\rm sh}|$ by $d^\al:= d|_{|\p^\al\Kk^{01}_{\rm sh}|}$.
Next, we proceed as in Lemma~\ref{le:VMC1} by choosing a nested cobordism reduction $\Cc \sqsubset \Vv$ of $\Kk^{01}_{\rm sh}$ and constructing nested cobordism reductions $\Cc_k \sqsubset \Vv$ by
$$
\Cc_k \,:=\; \Cc \cap \pi_{\Kk^{01}_{\rm sh}}^{-1}\bigl(\Ww_k) \;\sqsubset\; \Vv\qquad\text{with}\quad \Ww_k: =
B_{\frac 1k}(\io_{\Kk^{01}_{\rm sh}}(Y) \bigr) \;\subset\;|\Kk^{01}_{\rm sh}|,
$$
in addition discarding components $C_k\cap V_I$ that have empty intersection with $s_I^{-1}(0)$.
Indeed, each $\Ww_k$ and hence $\Cc_k$ is collared by \cite[Example~4.2.2~(iii)]{MW:top}, 
with boundaries given by the $\frac 1k$-neighbourhoods 
$\p^\al \Ww_k = B_{\frac 1k}^{d^\al}(\io_{\Kk^\al_{\rm sh}}(\p^\al Y) \bigr)\subset|\Kk^\al_{\rm sh}|$ with respect to the metrics $d^\al$ on $|\Kk^\al_{\rm sh}|$.
With that, Proposition~\ref{prop:ext2}~(i) guarantees the existence of admissible, precompact, transverse cobordism perturbations $\nu_k$ with $|(\s_{\Kk^{01}_{sh}}|_\Vv + \nu_k)^{-1}(0)| \subset \Ww_k$, and with boundary restrictions $\nu^\al_k:= \nu_k|_{\p^\al\Vv}$ that are strongly adapted perturbations of $(\Kk^\al_{\rm sh},d^\al)$ for $\al=0,1$.
Note here that these boundary restrictions satisfy the requirements of part (i) since $\bigcap_{k\in\N} \p^\al \Ww_k  = \io_{\Kk^\al_{\rm sh}}(\p^\al Y)$, thus they define the \v{C}ech homology classes
$$
A^{(\Kk^\al_{\rm sh},d^\al)} \;=\;
\underset{\leftarrow}\lim\, \bigl[ \,
i^{\nu^\al_k} \, \bigr] \;\in\; \check{H}_D\bigl(\io_{\Kk^\al_{\rm sh}}(\p^\al Y);\Q \bigr) .
$$
On the other hand, the homology classes $J^\al_* \bigl[ i^{\nu^\al_k}\bigr]$ also form two inverse systems in $H_D(|\Kk^{01}_{\rm sh}|;\Q)$, 
and as in \eqref{homologous} the chains $i^{\nu_k}: |\bZ^{\nu_k}| \to \Ww_k$ induce identities in the singular homology of $\Ww_k$,
$$
J^0 _* \bigl[ i^{\nu_k^0} \bigr] \;=\; J^1_* \bigl[ i^{\nu^1_k} \bigr] \; \in\; H_D(\Ww_k; \Q) ,  
$$
with the topological embeddings $J^\al: ( |\Kk^\al_{\rm sh}| , d^\al ) \to ( |\Kk^{01}_{\rm sh}|, d)$
of equation~\eqref{eq:Jal}.
Thus taking the inverse limit -- which commutes with pushforward -- we obtain
$$
J^0_*\,\Bigl(  \underset{\leftarrow}\lim\, \bigl[ i^{\nu^0_k} \bigr] \Bigr)\;=\; J^1_* \,\Bigl(\underset{\leftarrow}\lim\, \bigl[ i^{\nu^1_k}  \bigr] \Bigr)
\; \in\; \check H_D( \io_{\Kk^{01}_{\rm sh}}(Y) ;\Q ) .
$$
So further pushforward with the inverse $|\psi_{\Kk^{01}_{\rm sh}}|$ of the homeomorphism $\io_{\Kk^{01}_{\rm sh}}$ implies
$$
\bigl(|\psi_{\Kk^{01}_{\rm sh}}| \circ J^0\bigr)_*\,\Bigl(  \underset{\leftarrow}\lim\, \bigl[ i^{\nu^0_k} \bigr] \Bigr)\;=\; \bigl(|\psi_{\Kk^{01}_{\rm sh}}| \circ J^1\bigr)_* \,\Bigl(\underset{\leftarrow}\lim\, \bigl[ i^{\nu^1_k}  \bigr] \Bigr)
\; \in\; \check H_D(Y ; \Q) .
$$
To see that this proves Step~4, we use the fact that $|\psi_{\Kk^{01}_{\rm sh}}|$ is related to the analogous $|\psi_{\Kk^\al_{\rm sh}}| : \io_{\Kk^\al_{\rm sh}}(\p^\al Y) \to \p^\al Y$ by 
$$
|\psi_{\Kk^{01}_{\rm sh}}| \circ J^\al \big|_{\io_{\Kk^\al_{\rm sh}}(\p^\al Y)}
\;=\;
\io^\al_Y \circ |\psi_{\Kk^\al_{\rm sh}}| ,
$$
and hence we have for $\al=0,1$
$$
\bigl(|\psi_{\Kk^{01}_{\rm sh}}| \circ J^\al\bigr)_*\,\Bigl(  \underset{\leftarrow}\lim\, \bigl[ i^{\nu^\al_k} \bigr] \Bigr)
\;=\; 
(\io^\al_Y)_* \Bigl(  |\psi_{\Kk^\al_{\rm sh}}|_* \Bigl(  \underset{\leftarrow}\lim\, \bigl[ i^{\nu^\al_k} \bigr]\Bigr) \Bigr)
\;=\; 
(\io^\al_Y)_* \bigl[ \p^\al Y\bigr]^{\rm vir}_{\Kk^\al_{\rm sh}}.
$$

\MS\NI
{\bf Step 5:} {\it
Let $\Kk$ be an oriented, additive, weak Kuranishi concordance, and let $\Kk_{\rm sh}^\al$ be preshrunk tame shrinkings of $\p^\al\Kk$ for $\al=0,1$. Then we have
$$
[X]^{\rm vir}_{\Kk_{\rm sh}^0}=[X]^{\rm vir}_{\Kk_{\rm sh}^1}
\quad\in \check{H}_D(X;\Q).
$$
}

\MS
By Step~4 for the special case $Y=[0,1]\times X$ with the natural boundary embeddings
$\io^\al_Y=: I^\al :X\to \{\al\}\times X$ we obtain
$$
I^0_* \, \bigl([X]^{\rm vir}_{\Kk^0_{\rm sh}}\bigr)
\;=\; I^1_* \, \bigl([X]^{\rm vir}_{\Kk^1_{\rm sh}} \bigr)
\; \in\; \check{H}_D([0,1]\times  X ;\Q ) .
$$
Further, $I^0_* = I^1_* : \check{H}_D( X ;\Q) \to \check{H}_D( [0,1]\times  X ;\Q)$ are the same isomorphisms, because the two maps $I^0, I^1$ are both homotopy equivalences and homotopic to each other.
Hence we obtain the identity 
$[X]^{\rm vir}_{\Kk^0_{\rm sh}}  = [X]^{\rm vir}_{\Kk^1_{\rm sh}}$ in $\check{H}_D( X ; \Q)$, which
proves Step~5.

\MS
Finally, Step 5 implies uniqueness of the virtual fundamental cycle $[X]^{\rm vir}_\Kk\in \check{H}_D( X;\Q) $ 
for an oriented, additive weak Kuranishi atlas $\Kk$, since for any two choices of preshrunk tame shrinkings $\Kk^\al_{\rm sh}$ of $\Kk$ we can apply Step~5 to $[0,1]\times  \Kk$ to obtain $[X]^{\rm vir}_{\Kk_{\rm sh}^0}=[X]^{\rm vir}_{\Kk_{\rm sh}^1}$.
Moreover, given concordant oriented, additive, weak Kuranishi atlases $\Kk^0,\Kk^1$ there exists by assumption an oriented, additive, weak Kuranishi concordance  $\Kk$ with $\p^\al\Kk=\Kk^\al$.
If we pick any preshrunk tame shrinkings  $\Kk_{\rm sh}^\al$ of $\Kk^\al$ to define $[X]^{\rm vir}_{\Kk^\al}=[X]^{\rm vir}_{\Kk_{\rm sh}^\al}$, then Step~5 implies the claimed uniqueness under 
concordance claimed in (ii),
$$
[X]^{\rm vir}_{\Kk^0}\;=\;[X]^{\rm vir}_{\Kk_{\rm sh}^0}\;=\;[X]^{\rm vir}_{\Kk_{\rm sh}^1}\;=\;[X]^{\rm vir}_{\Kk^1}.
$$
Finally, (iii) follows from Step 4, which in the above notation simply becomes
$$
(\io^0_{Y})_*\bigl([\p^0 Y]^{\rm vir}_{\p^0\Kk}\bigr) 
\;=\; 
(\io^1_{Y})_*\bigl([\p^1 Y]^{\rm vir}_{\p^1\Kk}\bigr)
\quad\in \check{H}_D(Y;\Q).
$$
This completes the proof of Theorem~\ref{thm:VMC2}.
\end{proof}

\bibliographystyle{alpha}

\end{document}